\documentclass[reqno, 12pt]{amsart}
\makeatletter
\@namedef{subjclassname@2020}{\textup{2020} Mathematics Subject Classification}
\makeatother

\usepackage[T1]{fontenc}
\usepackage{amsmath}
\usepackage{amssymb}
\usepackage{bm}
\usepackage{mathrsfs}
\usepackage{dsfont}
\usepackage{accents}
\usepackage{stmaryrd}
\allowdisplaybreaks

\usepackage{xcolor} 
\definecolor{mygreen}{rgb}{0,0.7,0.3}
\definecolor{myblue}{rgb}{0,0.50,1.20}
\definecolor{orange}{rgb}{2.55,1.65,0}

\definecolor{fillred}{rgb}{1,0.9,0.9}
\definecolor{fillgreen}{rgb}{0.9,1,0.9}

\definecolor{refkey}{rgb}{0,0.7,0.3}
\definecolor{labelkey}{rgb}{1,0,0}

\usepackage[colorlinks, linkcolor=blue, citecolor=red, urlcolor=cyan, pagebackref]{hyperref}

\usepackage{float}
\usepackage{graphicx,color}
\usepackage{ascmac}
\usepackage{harpoon}

\usepackage{tabularx}
\usepackage{mathtools}

\usepackage{tikz}
\usetikzlibrary{calc,decorations.markings,patterns}
\usetikzlibrary{arrows.meta}
\usetikzlibrary{positioning}
\usetikzlibrary{cd, intersections, decorations.pathmorphing, arrows, decorations.pathreplacing}

\usepackage{cleveref}
\crefname{thm}{Theorem}{Theorems}
\crefname{theorem}{Theorem}{Theorems}
\crefname{cor}{Corollary}{Corollaries}
\crefname{corollary}{Corollary}{Corollaries}
\crefname{lem}{Lemma}{Lemmas}
\crefname{lemma}{Lemma}{Lemmas}
\crefname{prop}{Proposition}{Propositions}
\crefname{proposition}{Proposition}{Propositions}
\crefname{dfn}{Definition}{Definitions}
\crefname{definition}{Definition}{Definitions}
\crefname{ex}{Example}{Examples}
\crefname{example}{Example}{Examples}
\crefname{claim}{Claim}{Claims}
\crefname{conj}{Conjecture}{Conjectures}
\crefname{conv}{Notation}{Notations}
\crefname{rem}{Remark}{Remarks}
\crefname{remark}{Remark}{Remarks}
\crefname{prob}{Problem}{Problems}
\crefname{figure}{Figure}{Figures}
\crefname{table}{Table}{Tables}
\crefname{section}{Section}{Sections}
\crefname{subsection}{Section}{Sections}
\crefname{appendix}{Appendix}{Appendices}
\crefname{lemdef}{Lemma-Definition}{Lemma-Definitions}
\crefname{conv}{Convention}{Conventions}
\crefname{introthm}{Theorem}{Theorems}
\crefname{introcor}{Corollary}{Corollaries}
\crefname{introconj}{Conjecture}{Conjectures}

\usepackage{amsthm}

\newtheorem{thm}{Theorem}
\newtheorem{lem}[thm]{Lemma}
\newtheorem{prop}[thm]{Proposition}
\newtheorem{cor}[thm]{Corollary}
\newtheorem{conj}[thm]{Conjecture}
\newtheorem{introthm}{Theorem}

\newtheorem{introconj}[introthm]{Conjecture}
\newtheorem{introprop}[introthm]{Proposition}

\theoremstyle{definition}
\newtheorem{dfn}[thm]{Definition}
\newtheorem{rem}[thm]{Remark}
\newtheorem{ex}[thm]{Example}
\newtheorem{lemdef}[thm]{Lemma-Definition}

\numberwithin{figure}{section}
\numberwithin{equation}{section}
\numberwithin{thm}{section}

\newcommand{\bZ}{\mathbb{Z}}

\newcommand{\bR}{\mathbb{R}}
\newcommand{\bC}{\mathbb{C}}

\newcommand{\bM}{\mathbb{M}}
\newcommand{\bN}{\mathbb{N}}
\newcommand{\bP}{\mathbb{P}}
\newcommand{\bG}{\mathbb{G}}

\newcommand{\bE}{\mathbb{E}}

\newcommand{\bs}{{\boldsymbol{s}}}
\newcommand{\sfs}{\mathsf{s}}
\newcommand{\sfe}{\mathsf{e}}
\newcommand{\sff}{\mathsf{f}}

\newcommand{\Exch}{\mathrm{Exch}}
\newcommand{\bExch}{\bE \mathrm{xch}}
\newcommand{\uf}{\mathrm{uf}}
\newcommand{\f}{\mathrm{f}}

\DeclareMathOperator{\rank}{\mathrm{rank}}

\newcommand{\A}{\mathcal{A}}

\newcommand{\cC}{\mathcal{C}}

\newcommand{\cF}{\mathcal{F}}

\newcommand{\cO}{\mathcal{O}}

\newcommand{\X}{\mathcal{X}}

\newcommand{\scS}{\mathscr{S}}
\newcommand{\SK}[1]{\mathscr{S}_{\mathfrak{sl}_{3},#1}^{A}}
\newcommand{\Skein}[1]{\mathscr{S}_{\mathfrak{sl}_3,#1}}
\newcommand{\Bweb}[1]{\mathsf{BWeb}_{\mathfrak{sl}_{3},{#1}}}
\newcommand{\Eweb}[1]{\mathsf{EWeb}_{\mathfrak{sl}_{3},{#1}}}
\newcommand{\Cweb}[1]{\mathsf{CWeb}_{\mathfrak{sl}_{3},{#1}}}
\newcommand{\gr}{\mathrm{gr}}

\newcommand{\interior}{\mathrm{int}}
\newcommand{\CA}{\mathscr{A}}
\newcommand{\UCA}{\mathscr{U}}
\newcommand{\bD}{{\boldsymbol{\Delta}}}
\DeclareMathOperator{\Hom}{\mathrm{Hom}}
\DeclareMathOperator{\coker}{\mathrm{coker}}
\DeclareMathOperator{\Spec}{\mathrm{Spec}}

\makeatletter
\newcommand{\oset}[3][0ex]{%
  \mathrel{\mathop{#3}\limits^{
    \vbox to#1{\kern-2\ex@
    \hbox{$\scriptstyle#2$}\vss}}}}
\makeatother
\newcommand{\overbar}[1]{\oset{#1}{-\!\!\!-\!\!\!-}}

\newcommand\qarrow[2]{\draw[->,shorten >=2pt,shorten <=2pt] (#1) -- (#2) [thick];} 
\newcommand\qdrarrow[2]{\draw[->,dashed,shorten >=2pt,shorten <=2pt,bend right=0.5cm] (#1) to (#2) [thick];} 
\newcommand\qdlarrow[2]{\draw[->,dashed,shorten >=2pt,shorten <=2pt,bend left=0.5cm] (#1) to (#2) [thick];} 
\newcommand\qarrowbr[2]{\draw[->,shorten >=2pt,shorten <=2pt,bend right=0.5cm] (#1) to (#2) [thick];} 

\def\centerarc(#1)(#2:#3:#4)
    {($(#1)+({#4*cos(#2)},{#4*sin(#2)})$) arc(#2:#3:#4) }

\tikzset{
  mid arrow/.style={postaction={decorate,decoration={
        markings,
        mark=at position .5 with {\arrow[#1]{stealth}}
      }}},
}

\newcommand{\quiverplus}[3]{
\begin{scope}[>=latex]
{\color{mygreen}
    \path(#1) coordinate(x1);
    \path(#2) coordinate(x2);
    \path(#3) coordinate(x3);
    \foreach \l in {1,2}
    {
        \draw($(x1)!0.333*\l!(x2)$) circle(2pt) coordinate(x12\l);
        \draw($(x2)!0.333*\l!(x3)$) circle(2pt) coordinate(x23\l);
        \draw($(x3)!0.333*\l!(x1)$) circle(2pt) coordinate(x31\l);
    }
    \path($(x1)!0.5!(x2)$) coordinate(H);
    \draw($(x3)!0.667!(H)$) circle(2pt) coordinate(G);
    \qarrow{x121}{G}
    \qarrow{x231}{G}
    \qarrow{x311}{G}
    \qarrow{G}{x122}
    \qarrow{G}{x232}
    \qarrow{G}{x312}
    \qarrow{x312}{x121}
    \qarrow{x122}{x231}
    \qarrow{x232}{x311}
}
\end{scope}
}

\newcommand{\quiverminus}[3]{
\begin{scope}[>=latex]
{\color{mygreen}
    \path(#1) coordinate(x1);
    \path(#2) coordinate(x2);
    \path(#3) coordinate(x3);
    \foreach \l in {1,2}
    {
        \draw($(x1)!0.333*\l!(x2)$) circle(2pt) coordinate(x12\l);
        \draw($(x2)!0.333*\l!(x3)$) circle(2pt) coordinate(x23\l);
        \draw($(x3)!0.333*\l!(x1)$) circle(2pt) coordinate(x31\l);
    }
    \path($(x1)!0.5!(x2)$) coordinate(H);
    \draw($(x3)!0.667!(H)$) circle(2pt) coordinate(G);
    \path($(x1)!0.5!(x2)$) coordinate(z12);
    \path($(x2)!0.5!(x3)$) coordinate(z23);
    \path($(x3)!0.5!(x1)$) coordinate(z31);
    \path($(G)!0.7!(z12)$) coordinate(y12);
    \path($(G)!0.7!(z23)$) coordinate(y23);
    \path($(G)!0.7!(z31)$) coordinate(y31);
    \qarrow{x122}{G}
    \qarrow{x232}{G}
    \qarrow{x312}{G}
    \qarrow{G}{x121}
    \qarrow{G}{x231}
    \qarrow{G}{x311}
    \draw[<-,shorten >=2pt,shorten <=2pt] (x232) ..controls (y23) and (y12).. (x121) [thick];
    \draw[<-,shorten >=2pt,shorten <=2pt] (x312) ..controls (y31) and (y23).. (x231) [thick];
    \draw[<-,shorten >=2pt,shorten <=2pt] (x122) ..controls (y12) and (y31).. (x311) [thick];
}
\end{scope}
}

\newcommand{\quiversquare}[4]{
{\color{mygreen}
    \path(#1) coordinate(x1);
    \path(#2) coordinate(x2);
    \path(#3) coordinate(x3);
	\path(#4) coordinate(x4);
    \foreach \l in {1,2}
    {
        \draw($(x1)!0.333*\l!(x2)$) circle(2pt) coordinate(x12\l);
        \draw($(x2)!0.333*\l!(x3)$) circle(2pt) coordinate(x23\l);
        \draw($(x3)!0.333*\l!(x4)$) circle(2pt) coordinate(x34\l);
		\draw($(x4)!0.333*\l!(x1)$) circle(2pt) coordinate(x41\l);
		\draw($(x1)!0.333*\l!(x3)$) circle(2pt) coordinate(y13\l);	
	    \draw($(x2)!0.333*\l!(x4)$) circle(2pt) coordinate(y24\l);
	}
}
}

\newcommand{\CoG}[3]{
    \path(#1) coordinate(x1);
    \path(#2) coordinate(x2);
    \path(#3) coordinate(x3);
    \path($(x1)!0.5!(x2)$) coordinate(H);
    \path($(x3)!0.667!(H)$) circle(2pt) coordinate(G);}
\newcommand{\triv}[3]{
    \CoG{#1}{#2}{#3}
    \draw[red,thick,->-] (#1) -- (G);
    \draw[red,thick,->-] (#2) -- (G);
    \draw[red,thick,->-] (#3) -- (G);
}   
\newcommand{\trivop}[3]{
    \CoG{#1}{#2}{#3}
    \draw[red,thick,-<-] (#1) -- (G);
    \draw[red,thick,-<-] (#2) -- (G);
    \draw[red,thick,-<-] (#3) -- (G);
}   
\newcommand{\Hweb}[4]{
\path($(#1)!0.5!(#3)$) coordinate(P);
\CoG{#1}{#2}{P}
    \draw[red,thick,->-] (#1) -- (G);
    \draw[red,thick,->-] (#2) -- (G);
    \path(G) coordinate (G');
\CoG{#3}{#4}{P}
    \draw[red,thick,-<-] (#3) -- (G);
    \draw[red,thick,-<-] (#4) -- (G);
\draw[red,thick,->-] (G) -- (G');
}

\newcommand{\bline}[3]{
    \path (#1)++(0,-#3) coordinate(m1);
    \path (#2)++(0,-#3) coordinate(m2);
    \filldraw[gray!30] (m1) -- (#1) -- (#2) -- (m2) --cycle;
    \draw[thick] (#1) -- (#2);
}

\newcommand{\hweb}[1]{
    \begin{scope}[rotate=#1]
        \coordinate (P1) at (135:10);
        \coordinate (P2) at (225:10);
        \coordinate (P3) at (315:10);
        \coordinate (P4) at (45:10);
        \coordinate (H1) at (90:5);
        \coordinate (H3) at (270:5);
        \draw[thick, red, ->-={.5}{scale=.7}] (H1) -- (P1);
        \draw[thick, red, ->-={.5}{scale=.7}] (H1) -- (P4);
        \draw[thick, red, ->-={.5}{scale=.7}] (H1) -- (H3);
        \draw[thick, red, -<-={.5}{scale=.7}] (H3) -- (P2);
        \draw[thick, red, -<-={.5}{scale=.7}] (H3) -- (P3);
    \end{scope}
}

\newcommand{\bwebneg}[1]{
    \begin{scope}[rotate=#1]
        \coordinate (P1) at (135:10);
        \coordinate (P2) at (225:10);
        \draw[thick, red, ->-={.7}{scale=.7}] (P2) -- (P1);
    \end{scope}
}

\newcommand{\bwebpos}[1]{
    \begin{scope}[rotate=#1]
        \coordinate (P1) at (135:10);
        \coordinate (P2) at (225:10);
        \draw[thick, red, ->-={.7}{scale=.7}] (P1) -- (P2);
    \end{scope}
}

\newcommand{\twebplus}[1]{
    \begin{scope}[rotate=#1]
        \coordinate (P1) at (135:10);
        \coordinate (P2) at (225:10);
        \coordinate (P4) at (45:10);
        \coordinate (T1) at (135:2);
        \draw[thick, red, ->-={.7}{scale=.7}] (P1) -- (T1);
        \draw[thick, red, ->-={.7}{scale=.7}] (P2) -- (T1);
        \draw[thick, red, ->-={.7}{scale=.7}] (P4) -- (T1);
    \end{scope}
}

\newcommand{\twebminus}[1]{
    \begin{scope}[rotate=#1]
        \coordinate (P1) at (135:10);
        \coordinate (P2) at (225:10);
        \coordinate (P4) at (45:10);
        \coordinate (T1) at (135:2);
        \draw[thick, red, -<-={.7}{scale=.7}] (P1) -- (T1);
        \draw[thick, red, -<-={.7}{scale=.7}] (P2) -- (T1);
        \draw[thick, red, -<-={.7}{scale=.7}] (P4) -- (T1);
    \end{scope}
}

\newcommand{\dweb}[1]{
    \begin{scope}[rotate=#1]
        \coordinate (P1) at (135:10);
        \coordinate (P3) at (315:10);
        \draw[thick, red, ->-={.4}{scale=.7}] (P3) -- (P1);
    \end{scope}
}

\newcommand{\dwebup}[1]{
    \begin{scope}[rotate=#1]
        \coordinate (P1) at (135:10);
        \coordinate (P2) at (225:10);
        \coordinate (P3) at (315:10);
        \coordinate (P4) at (45:10);
        \draw[thick, red, ->-={.4}{scale=.7}, rounded corners] (P3) -- ($(P2)!.65!(P4)$) -- (P1);
    \end{scope}
}

\newcommand{\dwebdown}[1]{
    \begin{scope}[rotate=#1]
        \coordinate (P1) at (135:10);
        \coordinate (P2) at (225:10);
        \coordinate (P3) at (315:10);
        \coordinate (P4) at (45:10);
        \draw[thick, red, ->-={.4}{scale=.7}, rounded corners] (P3) -- ($(P2)!.35!(P4)$) -- (P1);
    \end{scope}
}

\newcommand{\Quadweb}[2]{
    \mathord{\ 
    \tikz[baseline=-.6ex, scale=.07]{
        \coordinate (P1) at (135:10);
        \coordinate (P2) at (225:10);
        \coordinate (P3) at (315:10);
        \coordinate (P4) at (45:10);
        \draw[gray] (P1) -- (P2) -- (P3) -- (P4) -- (P1) -- cycle;
        #1
        \node[draw, fill=black, circle, inner sep=1] at (135:10) {};
        \node[draw, fill=black, circle, inner sep=1] at (225:10) {};
        \node[draw, fill=black, circle, inner sep=1] at (315:10) {};
        \node[draw, fill=black, circle, inner sep=1] at (45:10) {};
        \path (225:10) -- (315:10) node [pos=.5, below=3pt]{#2};
    }
\ }
}

\newcommand{\ttwebplus}{
    \begin{scope}
        \coordinate (P1) at (210:10);
        \coordinate (P2) at (330:10);
        \coordinate (P3) at (90:10);
        \coordinate (T1) at (0,0);
        \draw[thick, red, ->-={.7}{scale=.7}] (P1) -- (T1);
        \draw[thick, red, ->-={.7}{scale=.7}] (P2) -- (T1);
        \draw[thick, red, ->-={.7}{scale=.7}] (P3) -- (T1);
    \end{scope}
}

\newcommand{\ttwebminus}{
    \begin{scope}
        \coordinate (P1) at (210:10);
        \coordinate (P2) at (330:10);
        \coordinate (P3) at (90:10);
        \coordinate (T1) at (0,0);
        \draw[thick, red, -<-={.7}{scale=.7}] (P1) -- (T1);
        \draw[thick, red, -<-={.7}{scale=.7}] (P2) -- (T1);
        \draw[thick, red, -<-={.7}{scale=.7}] (P3) -- (T1);
    \end{scope}
}

\newcommand{\ttwebminusdown}{
    \begin{scope}
        \coordinate (P1) at (210:10);
        \coordinate (P2) at (330:10);
        \coordinate (P3) at (90:10);
        \coordinate (T1) at (2,4);
        \draw[thick, red, -<-={.7}{scale=.7}] (P1) -- (T1);
        \draw[thick, red, -<-={.7}{scale=.7}] (P2) -- (T1);
        \draw[thick, red, -<-={.7}{scale=.7}] (P3) -- (T1);
    \end{scope}
}

\newcommand{\tbwebpos}[1]{
    \begin{scope}[rotate=#1]
        \coordinate (P1) at (210:10);
        \coordinate (P2) at (330:10);
        \coordinate (P3) at (90:10);
        \coordinate (T1) at (0,0);
        \draw[thick, red, ->-={.7}{scale=.7}] (P1) -- (P2);
    \end{scope}
}

\newcommand{\tbwebneg}[1]{
    \begin{scope}[rotate=#1]
        \coordinate (P1) at (210:10);
        \coordinate (P2) at (330:10);
        \coordinate (P3) at (90:10);
        \coordinate (T1) at (0,0);
        \draw[thick, red, -<-={.3}{scale=.7}] (P1) -- (P2);
    \end{scope}
}

\newcommand{\Triweb}[2]{
    \mathord{\ 
    \tikz[baseline=-.6ex, scale=.07]{
        \coordinate (P1) at (210:10);
        \coordinate (P2) at (330:10);
        \coordinate (P3) at (90:10);
        \draw[gray] (P1) -- (P2) -- (P3) -- (P1) -- cycle;
        #1
        \node[draw, fill=black, circle, inner sep=1] at (210:10) {};
        \node[draw, fill=black, circle, inner sep=1] at (330:10) {};
        \node[draw, fill=black, circle, inner sep=1] at (90:10) {};
        \path (225:10) -- (315:10) node [pos=.5, below=3pt]{#2};
    }
\ }
}

\tikzset{->-/.style 2 args={
	postaction={decorate},
	decoration={markings, mark=at position #1 with {\arrow[thick, #2]{>}}} 
    },
    ->-/.default={0.5}{}
}
\tikzset{-<-/.style 2 args={
	postaction={decorate},
	decoration={markings, mark=at position #1 with {\arrow[thick, #2]{<}}} 
    },
    -<-/.default={0.5}{}
}

\tikzset{
	overarc/.style={
		white, double=red, double distance=1.2pt, line width=2.4pt
	}
}

\title[Skein and cluster algebras for $\mathfrak{sl}_3$]{Skein and cluster algebras of unpunctured surfaces for $\mathfrak{sl}_3$}

\author[Tsukasa Ishibashi]{Tsukasa Ishibashi}
\address{Tsukasa Ishibashi, Mathematical Institute, Tohoku University, 
6-3 Aoba, Aramaki, Aoba-ku, Sendai, Miyagi 980-8578, Japan.}
\email{tsukasa.ishibashi.a6@tohoku.ac.jp}
\urladdr{https://sites.google.com/view/tsukasa-ishibashi/home}

\author[Wataru Yuasa]{Wataru Yuasa}
\address{Wataru Yuasa\\
Osaka Central Advanced Mathematical Institute\\
Osaka Metropolitan University\\
3-3-138 Sugimoto-cho, Sumiyoshi-ku, Osaka 558-8585, JAPAN;\\
Research Institute for Mathematical Sciences\\
Kyoto University\\
Kitashirakawa Oiwake-cho, Sakyo-ku, Kyoto 606-8502, Japan}
\email{wyuasa@kurims.kyoto-u.ac.jp}
\urladdr{https://wataruyuasa.github.io/math/}

\subjclass[2020]{13F60, 57K31 (Primary), 57K20 (Secondary)}
\keywords{Cluster algebra; Skein algebra; Positivity}

\begin{document}

\begin{abstract}
For an unpunctured marked surface $\Sigma$, we consider a skein algebra $\mathscr{S}_{\mathfrak{sl}_{3},\Sigma}^{q}$ consisting of $\mathfrak{sl}_3$-webs on $\Sigma$ with the boundary skein relations at marked points. 
We construct a quantum cluster algebra $\mathscr{A}^q_{\mathfrak{sl}_3,\Sigma}$ inside the skew-field $\mathrm{Frac}\mathscr{S}_{\mathfrak{sl}_{3},\Sigma}^{q}$ of fractions, which quantizes the cluster $K_2$-structure on the moduli space $\mathcal{A}_{SL_3,\Sigma}$ of decorated $SL_3$-local systems on $\Sigma$. 
We show that the cluster algebra $\mathscr{A}^q_{\mathfrak{sl}_3,\Sigma}$ contains the boundary-localized skein algebra $\mathscr{S}_{\mathfrak{sl}_{3},\Sigma}^{q}[\partial^{-1}]$ as a subalgebra, and their natural structures, such as gradings and certain group actions, agree with each other. We also give an algorithm to compute the Laurent expressions of a given $\mathfrak{sl}_3$-web in certain clusters and discuss the positivity of coefficients. In particular, we show that the bracelets and the bangles along an oriented simple loop in $\Sigma$ have Laurent expressions with positive coefficients, hence give rise to quantum GS-universally positive Laurent polynomials.
\end{abstract}
\maketitle

\setcounter{tocdepth}{1}
\tableofcontents


\section{Introduction}\label{sec:intro}
Quantizations of the $SL_2(\bC)$-character variety $\Hom(\pi_1(\Sigma),SL_2(\bC)) \sslash SL_2(\bC)$ for a surface $\Sigma$ have been studied in several different ways. One is via the \emph{Kauffman bracket skein algebra}, first defined in \cite{BFK99} for a closed surface; \cite{PS19, BonahonWong11} for a surface with boundary; \cite{RogerYang14} for a punctured surface; \cite{Muller16} for an unpunctured marked surface.
In each of these works the Kauffman bracket skein relation provides a non-commutative deformation of the trace identity among the $SL_2(\bC)$-matrices, and therefore the skein algebra can be regarded as a deformation quantization of the $SL_2(\bC)$-character variety (or its suitable variants).
More generally, connections between the $SL_2(\bC)$-character variety and the Kauffman bracket skein algebra for a $3$-manifold has been studied in \cite{Bullock97,PS00}.

Another one is via the theory of \emph{cluster algebras} initiated by Fomin--Zelevinsky \cite{FZ-CA1}, whose quantum counterpart has been introduced by Berenstein--Zelevinsky \cite{BZ}. Cluster algebras are commutative rings $\CA_\sfs$ associated witha combinatorial data $\sfs$, a mutation class of \emph{seeds}. For a marked surface $\Sigma$ (\emph{i.e.}, an oriented compact surface with boundary together with a finite set of marked points), the moduli space $\A_{SL_2,\Sigma}$ of twisted decorated $SL_2$-local systems has a canonical cluster $K_2$-structure \cite{FG03} encoded in a mutation class $\sfs=\sfs(\mathfrak{sl}_2,\Sigma)$, via which we can identify a subring of the ring of regular functions on an open subspace $\A^\times_{SL_2,\Sigma}$ with the associated cluster algebra $\CA_{\mathfrak{sl}_2,\Sigma}$.
See \cref{rem:geometry of moduli} for more information.
Forgetting the decoration, we get the moduli space of twisted $SL_2(\bC)$-representations, which is identified with the $SL_2(\bC)$-character stack by fixing a spin structure on $\Sigma$. 

There is a general theory of \emph{quantum cluster algebra} introduced by Berenstein--Zelevinsky \cite{BZ}, which provides non-commutative deformations $\CA_{\sfs_q}$ of a given cluster algebra $\CA_\sfs$. The deformation depends on additional data called the \emph{compatibility matrices}, which give rise to a mutation class $\sfs_q$ of \emph{quantum seeds}. Such a deformation exists if the original mutation class $\sfs$ of seeds possesses full-rank exchange matrices. There is an accompanying algebra $\UCA_{\sfs_q}$ called the \emph{quantum upper cluster algebra}, which is an intersection of (typically infinitely many) quantum tori. 
We always have an inclusion $\CA_{\sfs_q} \subset \UCA_{\sfs_q}$ by the so-called \emph{quantum Laurent phenomenon}.

For the mutation class $\sfs(\mathfrak{sl}_2,\Sigma)$, the full-rank condition forces $\Sigma$ to be \emph{unpunctured}, by which we mean it has no punctures. 
In this case, a suitable choice of a mutation class $\sfs_q(\mathfrak{sl}_2,\Sigma)$ quantizing $\sfs(\mathfrak{sl}_2,\Sigma)$ has been made by Muller \cite{Muller16}, for which we have (upper) cluster algebras $\CA^q_{\mathfrak{sl}_2,\Sigma} \subset \UCA^q_{\mathfrak{sl}_2,\Sigma}$. Upon this choice, he showed that the two quantization schemes via skein and cluster algebras give the same result. 
More precisely, he introduced a skein algebra $\mathscr{S}^A_{\mathfrak{sl}_2,\Sigma}$ on an unpunctured marked surface by imposing certain \emph{boundary skein relations} and obtained the following comparison result:

\begin{introthm}[Muller \cite{Muller16}]\label{introthm:Muller}
For any (triangulable) unpunctured marked surface $\Sigma$, we have 
\begin{align*}
    \CA^q_{\mathfrak{sl}_2,\Sigma} \subset \mathscr{S}^q_{\mathfrak{sl}_2,\Sigma}[\partial^{-1}] \subset \UCA^q_{\mathfrak{sl}_2,\Sigma},
\end{align*}
where $\mathscr{S}^q_{\mathfrak{sl}_2,\Sigma}[\partial^{-1}]$ denotes the localized skein algebra at boundary intervals. Moreover if $\Sigma$ has at least two marked points, these inclusions are isomorphisms. We also have the following comparison of their natural structures:
\begin{itemize}
    \item the inclusions are $MC(\Sigma)$-equivariant;
    \item the bar-involution on $\UCA^q_{\mathfrak{sl}_2,\Sigma}$ restricts to the mirror-reflection on $\mathscr{S}^q_{\mathfrak{sl}_2,\Sigma}[\partial^{-1}]$;
    \item the ensemble grading (\emph{a.k.a.} the universal grading) on $\UCA^q_{\mathfrak{sl}_2,\Sigma}$ restricts to the endpoint grading on $\mathscr{S}^q_{\mathfrak{sl}_2,\Sigma}[\partial^{-1}]$.
\end{itemize}
\end{introthm}
Let us briefly mention Muller's strategy. 
\begin{description}
    \item[Step 1] For a given ideal triangulation of $\Sigma$, he constructed a quantum cluster inside the localized skein algebra $\mathscr{S}^q_{\mathfrak{sl}_2,\Sigma}[\partial^{-1}]$. They are shown to be mutation-equivalent to each other by identifying the quantum exchange relations with the skein relations, and hence generate the algebras $\CA^q_{\mathfrak{sl}_2,\Sigma}$ and $\UCA^q_{\mathfrak{sl}_2,\Sigma}$ in the skew-field $\mathrm{Frac}\mathscr{S}^q_{\mathfrak{sl}_2,\Sigma}$ of fractions. 
    \item[Step 2] Since all the quantum clusters $\CA^q_{\mathfrak{sl}_2,\Sigma}$ are associated with ideal triangulations and they are realized inside $\mathscr{S}^q_{\mathfrak{sl}_2,\Sigma}[\partial^{-1}]$, we immediately get the inclusion $\CA^q_{\mathfrak{sl}_2,\Sigma} \subset \mathscr{S}^q_{\mathfrak{sl}_2,\Sigma}[\partial^{-1}]$. 
    \item[Step 3] Then he also gave a way to express any element of the localized skein algebra as a quantum Laurent polynomial in each quantum cluster, which leads to the inclusion $\mathscr{S}^q_{\mathfrak{sl}_2,\Sigma}[\partial^{-1}] \subset \UCA^q_{\mathfrak{sl}_2,\Sigma}$.
    \item[Step 4] Finally he proved the coincidence $\CA^q_{\mathfrak{sl}_2,\Sigma}=\UCA^q_{\mathfrak{sl}_2,\Sigma}$ when $\Sigma$ has at least two marked points.
\end{description}

Here is a comment on the localizations: one can as well consider the quantum (upper) cluster algebras with frozen variables not being invertible, in which case they are expected to coincide with $\mathscr{S}_{\mathfrak{sl}_2,\Sigma}^q$. It amounts to consider a partial compactification of the cluster $K_2$-variety.

\subsection{Comparison of the skein and cluster algebras for \texorpdfstring{$\mathfrak{sl}_3$}{sl3}}
Our far-reaching goal is to find higher-rank analogues of the Muller's result for semisimple Lie algebras $\mathfrak{g}$ other than $\mathfrak{sl}_2$. For a simply-connected semisimple algebraic group $G$ with Lie algebra $\mathfrak{g}$, the moduli space $\A_{G,\Sigma}$ of twisted decorated $G$-local systems \cite{FG03} has  a canonical $K_2$-structure, which is encoded in a mutation class $\sfs(\mathfrak{g},\Sigma)$ constructed in \cite{FG03} for $\mathfrak{sl}_n$; \cite{Le16} for classical Lie algebras; \cite{GS19} in general.
On the other hand, the higher-rank analogues of the skein theory has been studied by Kuperberg~\cite{Kuperberg96} for rank two Lie algebras, Murakami-Ohtsuki-Yamada~\cite{MOY98}, Sikora~\cite{Sikora05} and Morrison~\cite{Morrison07} for $\mathfrak{sl}_n$.

Our aim in this paper is to establish the $\mathfrak{sl}_3$-case via a specialization ($a=1$) of the skein algebra $\SK{\Sigma}$ introduced by Frohman--Sikora \cite{FrohmanSikora20}.
The skein algebra $\SK{\Sigma}$ is spanned by certain $\mathfrak{sl}_3$-webs on an unpunctured marked surface $\Sigma$, subject to certain boundary skein relations as well as the usual $\mathfrak{sl}_3$-skein relations (see \cref{def:skeinrel} and \cref{def:bskeinrel}).
Following Muller's strategy, we first construct quantum clusters associated with \emph{decorated triangulations} as specific collections of elements in the $\mathfrak{sl}_3$-skein algebra, which we call \emph{web clusters}. A web cluster is defined to be a collection of \emph{elementary webs} which $q$-commute with each other, with the prescribed cardinality (see \cref{elementary} and \cref{webcluster}). 
Then we show that these quantum clusters are mutation-equivalent to each other by identifying the quantum exchange relations relating these quantum clusters with the skein relations. Thus they generate a canonical mutation class $\sfs_q(\mathfrak{sl}_3,\Sigma)$, which defines the quantum (upper) cluster algebras $\CA^q_{\mathfrak{sl}_3,\Sigma}$ and $\UCA^q_{\mathfrak{sl}_3,\Sigma}$ in the skew-field $\mathrm{Frac}\mathscr{S}^q_{\mathfrak{sl}_3,\Sigma}$ of fractions. 
Then we obtain the following $\mathfrak{sl}_3$-analogue of \cref{introthm:Muller}:

\begin{introthm}[Comparison of skein and cluster algebras:
\cref{subsec:S_in_A}]\label{introthm:comparison}  
For a connected (triangulable) unpunctured marked surface $\Sigma$ with at least two marked points, we have 
\begin{align*}
 \mathscr{S}^q_{\mathfrak{sl}_3,\Sigma}[\partial^{-1}] \subset \CA^q_{\mathfrak{sl}_3,\Sigma} \subset \UCA^q_{\mathfrak{sl}_3,\Sigma}, 
\end{align*}
where $\mathscr{S}^q_{\mathfrak{sl}_3,\Sigma}[\partial^{-1}]$ denotes the $\mathfrak{sl}_3$-skein algebra localized at boundary webs (\cref{def:localized_skein_algebras}). 
Moreover,
\begin{itemize}
    \item the inclusions are $MC(\Sigma) \times \mathrm{Out}(SL_3)$-equivariant;
    \item the bar-involution on $\UCA^q_{\mathfrak{sl}_3,\Sigma}$ restricts to the mirror-reflection on $\mathscr{S}^q_{\mathfrak{sl}_3,\Sigma}[\partial^{-1}]$;
    \item the ensemble grading on $\UCA^q_{\mathfrak{sl}_3,\Sigma}$ restricts to the endpoint grading on $\mathscr{S}^q_{\mathfrak{sl}_3,\Sigma}[\partial^{-1}]$.
\end{itemize}
\end{introthm}
The (partially conjectural) correspondence of some notions is summarized below.

\bigskip 

\begin{tabular}{|l|l|}
\hline
  Quantum cluster algebra $\CA^q_{\mathfrak{sl}_3,\Sigma}$ & Skein algebra $\SK{\Sigma}$ \\ \hline
  clusters  & web clusters \\
  cluser variables & elementary webs \\
  quantum exchange relations & skein relations \\
  (a $\bZ_q$-basis) & graphical basis \\
  bar-involution & mirror-reflection \\
  ensemble grading & endpoint grading \\
  \hline
\end{tabular}

\bigskip

We remark that the classical counterpart ($q^{\frac{1}{2}}=1$) of the correspondence of exchange and skein relations has been discovered by Fomin--Pylyavskyy in their work \cite{FP14,FP16} on the cluster structures of certain algebras related to $SL_3$-invariants. See \cref{rem:FP} for a relation to our setting. They have already shown that the language of webs provides a powerful tool to describe the combinatorics of mutations, and made a sequence of insightful conjectures. Our notion of basis webs (resp. elementary webs) is a quantum version of the ``web invariants'' (resp. ``indecomposable webs'') in \cite{FP16}. 
See below for relevant conjectures in our quantum setting. 

In \cref{introthm:comparison}, the most non-trivial inclusion is the first one. 
In the $\mathfrak{sl}_3$-case (or more higher cases), a crucial difficulty arises from the fact that the mutation class $\sfs(\mathfrak{sl}_3,\Sigma)$ is typically of \emph{infinite mutation type}, meaning that it possesses infinitely many non-isomorphic quivers. In particular, the (quantum) clusters and their mutations do not necessarily come from geometric objects such as decorated triangulations and their flips. 
Hence Step~2 in the Muller's argument does not follow immediately. Rather, we establish the converse inclusion $\mathscr{S}^q_{\mathfrak{sl}_3,\Sigma}[\partial^{-1}] \subset \CA^q_{\mathfrak{sl}_3,\Sigma}$ by a new method. We first give a good generating set of the skein algebra $\mathscr{S}^q_{\mathfrak{sl}_3,\Sigma}[\partial^{-1}]$ refining the one given by Frohman--Sikora \cite{FrohmanSikora20}, and then use a \lq\lq sticking trick'' of $\mathfrak{sl}_3$-webs to boundary intervals (\cref{lem:stickto}) to write each generator as a quantum polynomial of known cluster variables.

As in the $\mathfrak{sl}_2$-case, we expect the following:

\begin{introconj}\label{introconj:comparison}
$\mathscr{S}^q_{\mathfrak{sl}_3,\Sigma}[\partial^{-1}] =\CA^q_{\mathfrak{sl}_3,\Sigma} = \UCA^q_{\mathfrak{sl}_3,\Sigma}$.
\end{introconj}
In particular, we expect a one-to-one correspondence between the quantum clusters in $\CA^q_{\mathfrak{sl}_3,\Sigma}$ and the web clusters in $\mathscr{S}^q_{\mathfrak{sl}_3,\Sigma}$. The classical counterpart $q^{\frac{1}{2}}=1$ of this conjecture is proved by Ishibashi--Oya--Shen \cite{IOS}. 
A further expectation is the following:

\begin{introconj}
\label{introconj:basis}
The graphical basis $\Bweb{\Sigma}$ contains all the cluster monomials. 
\end{introconj}
Indeed, it is one of the central problems in cluster algebra to find a \lq\lq canonical'' basis which contains all the cluster monomials. See \cite{Qin} for a recent review on bases of quantum cluster algebras. 

The classical counterpart of \cref{introconj:basis} and the correspondence between the quantum clusters and web clusters have conjectured in \cite[Section 9]{FP16}, where they gave a series of more detailed statements. They also conjecture that the cluster monomials in the disk case are exactly those give rise to ``forest diagrams'' after the \emph{arborization algorithm} \cite[Conjecture 10.6]{FP16}. 
We remark that the quantum setting may contain an extra difficulty on elevations, and thus our conjectures do not follow immediately from the classical ones: a tip of such a difficulty appears in the proof of \cref{introthm:comparison} (see \cref{rem:classical-proof}), and in the elevation-preserving condition for \cref{introthm:positivity}. 

A natural strategy to prove \cref{introconj:comparison} is to generalize Step~4 in the Muller's strategy by showing the \emph{local acyclicity} (\cite{Muller_acyclic}) of the cluster algebra $\CA_{\mathfrak{sl}_3,\Sigma}$. 
In the $\mathfrak{sl}_2$-case, the local acyclicity has been proved by solving the \emph{Banff algorithm}, which we do not know whether also solvable for the $\mathfrak{sl}_3$-case. As a slight variant of the Banff algorithm, we formulate the following:

\begin{introprop}[\cref{thm:S=U}]\label{introthm:S=U}
The covering conjecture (\cref{conj:covering}) implies the equalities in \cref{introconj:comparison}.
\end{introprop}
Indeed, the covering conjecture is on a quantum analogue of the covering of the moduli spaces up to codimension $2$ considered in \cite{Shen20,IOS}.

\subsection{Quantum Laurent positivity of webs}
From the inclusion $\mathscr{S}^q_{\mathfrak{sl}_3,\Sigma}[\partial^{-1}] \subset \UCA^q_{\mathfrak{sl}_3,\Sigma}$ given by \cref{introthm:comparison},  
each web $x \in \mathscr{S}^q_{\mathfrak{sl}_3,\Sigma}[\partial^{-1}]$ gives rise to a \emph{quantum universally Laurent polynomial} in an \underline{arbitrary} quantum cluster. Namely, we know that it is represented as a quantum Laurent polynomial of quantum cluster variables and $q$ in any quantum cluster in an abstract way. Such an element is called a \emph{quantum universally positive Laurent polynomial} if its Laurent expression in each quantum cluster has positive integral coefficients. The search of webs with such positivity is motivated by the Fock--Goncharov duality conjecture: see \cref{subsec:future_direction} below. 

We first remark that the \lq\lq sticking trick'' to boundary intervals used in the inclusion $\mathscr{S}^q_{\mathfrak{sl}_3,\Sigma}[\partial^{-1}] \subset \CA^q_{\mathfrak{sl}_3,\Sigma}$ always involves negative signs: see \cref{lem:stickto}. Therefore the positivity nature is not clear from this construction. As another way of expansion of webs in a given web cluster, we will give a direct inclusion $\mathscr{S}^q_{\mathfrak{sl}_3,\Sigma}[\partial^{-1}] \subset \UCA^q_{\mathfrak{sl}_3,\Sigma}$. 
Such an expansion of a web $x \in \mathscr{S}^q_{\mathfrak{sl}_3,\Sigma}$ in a web cluster $C$ is obtained by multiplying an appropriate product of webs in $C$ and then
successively applying the $\mathfrak{sl}_3$-skein relations to resolve the intersections. Most of the relations used here have a manifest positivity, while the one \eqref{rel:bigon} causes a problem. Indeed, one is forced to use the relation \eqref{rel:bigon} in some situation during the process getting the cluster expansion of a web. In order to avoid the usage of \eqref{rel:bigon}, we only consider the \emph{elevation-preserving webs} (\cref{def:elevation-preserving web}) with respect to an ideal triangulation. 
For instance, the \emph{$n$-bracelet}\footnote{Our ``diagrammatic'' bracelet here is different from the bracelet in the literature, which is defined to be a certain Chebyshev polynomial of a loop. While the latter is invariant under the mirror-reflection, the former is not. They coincide when $q^{\frac{1}{2}}=1$.} (resp. the \emph{$n$-bangle}) along an oriented simple loop $\gamma$ in $\Sigma$, obtained by replacing the embedding of $\gamma$ with an embedding of the graph shown in the left (resp. right) in \cref{fig:bracelet}, is an elevation-preserving web for any ideal triangulation. 

An element of $\UCA^q_{\mathfrak{sl}_3,\Sigma}$ is called a \emph{quantum GS-universally positive Laurent polynomial} (after Goncharov--Shen) if it is represented as a positive Laurent polynomial in the quantum cluster associated with any decorated triangulation. 
The following is our result in this paper:

\begin{introthm}[Quantum Laurent positivity of webs: \cref{thm:positivity_cluster}]\label{introthm:positivity}
Any elevation-preserving web with respect to $\Delta$ is expressed as a positive Laurent polynomial in the quantum cluster associated with a decorated triangulation $\bD=(\Delta,\bs_\Delta)$ with the underlying triangulation $\Delta$. In particular, the bracelets and the bangles along an oriented simple loop in $\Sigma$ are quantum GS-universally positive Laurent polynomials.
\end{introthm}

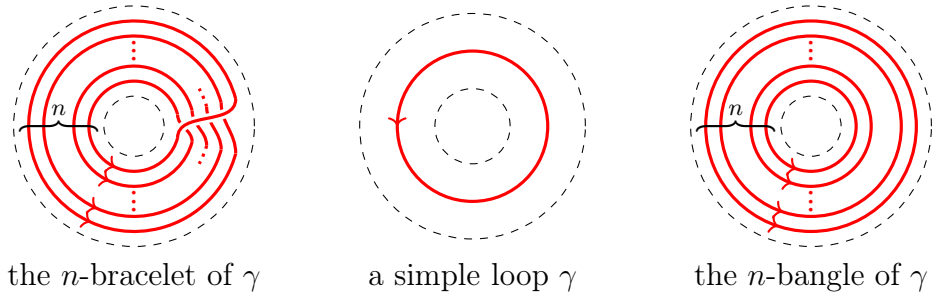
\begin{figure}
    \begin{tikzpicture}[scale=.1]
        \begin{scope}
            \draw[dashed] (0,0) circle [radius=5];
            \draw[dashed] (0,0) circle [radius=15];
            \draw[very thick, red, ->-] (0,0) circle [radius=10];
            \node at (0,-15) [below=5pt] {a simple loop $\gamma$};
        \end{scope}
        \begin{scope}[xshift=-45cm]
            \draw[dashed] (0,0) circle [radius=4];
            \draw[dashed] (0,0) circle [radius=16];
            \draw[very thick, red, ->-={.7}{red}] (15:14) arc (15:345:14);
            \draw[very thick, red, ->-={.7}{red}] (15:12) arc (15:345:12);
            \draw[very thick, dotted, red] (15:10) arc (15:30:10);
            \draw[very thick, dotted, red] (-15:10) arc (-15:-30:10);
            \draw[very thick, red, ->-={.7}{red}] (15:8) arc (15:345:8);
            \draw[very thick, red, ->-={.7}{red}] (15:6) arc (15:345:6);
            \draw[red, overarc, very thick] (-15:8) to[out=north, in=south] (15:6);
            \draw[red, overarc, very thick] (-15:10) to[out=north, in=south] (15:8);
            \draw[red, overarc, very thick] (-15:12) to[out=north, in=south] (15:10);
            \draw[red, overarc, very thick] (-15:14) to[out=north, in=south] (15:12);
            \draw[red, overarc, very thick] (-15:6) to[out=north, in=south] (15:14);
            \node at (180:10) [xscale=2.5, rotate=90] {$\}$};
            \node at (180:10) [above]{\scriptsize $n$};
            \node at (0,10) [scale=.3, red]{$\bullet$};
            \node at (0,10) [scale=.3, red, above=1pt]{$\bullet$};
            \node at (0,10) [scale=.3, red, below=1pt]{$\bullet$};
            \node at (0,-10) [scale=.3, red]{$\bullet$};
            \node at (0,-10) [scale=.3, red, above=1pt]{$\bullet$};
            \node at (0,-10) [scale=.3, red, below=1pt]{$\bullet$};
            \node at (0,-15) [below=5pt] {the $n$-bracelet of $\gamma$};
        \end{scope}
        \begin{scope}[xshift=45cm]
            \draw[dashed] (0,0) circle [radius=4];
            \draw[dashed] (0,0) circle [radius=16];
            \draw[very thick, red, ->-={.7}{red}] (0,0) circle [radius=14];
            \draw[very thick, red, ->-={.7}{red}] (0,0) circle [radius=12];
            \draw[very thick, red, ->-={.7}{red}] (0,0) circle [radius=8];
            \draw[very thick, red, ->-={.7}{red}] (0,0) circle [radius=6];
            \node at (180:10) [xscale=2.5, rotate=90] {$\}$};
            \node at (180:10) [above]{\scriptsize $n$};
            \node at (0,10) [scale=.3, red]{$\bullet$};
            \node at (0,10) [scale=.3, red, above=1pt]{$\bullet$};
            \node at (0,10) [scale=.3, red, below=1pt]{$\bullet$};
            \node at (0,-10) [scale=.3, red]{$\bullet$};
            \node at (0,-10) [scale=.3, red, above=1pt]{$\bullet$};
            \node at (0,-10) [scale=.3, red, below=1pt]{$\bullet$};
            \node at (0,-15) [below=5pt] {the $n$-bangle of $\gamma$};
        \end{scope}
    \end{tikzpicture}
    \caption{The middle shows a tubular neighborhood of an oriented simple loop $\gamma$. The $n$-bracelet (resp. $n$-bangle) along $\gamma$ is obtained by replacing it with the graph shown in the left (right).}
    \label{fig:bracelet}
\end{figure}

\subsection{Related works}\label{subsec:future_direction}

The theory of cluster ensembles \cite{FG09} produces a pair $(\A_\sfs,\X_\sfs)$ of positive schemes from a given mutation class $\sfs$ (see \cref{sec:FG}). Thanks to their positivity nature, we can form their $\bP$-valued points $(\A_\sfs(\bP),\X_\sfs(\bP))$ for any semifield $\bP$. 
The \emph{Fock--Goncharov duality conjecture} \cite[Section 4]{FG09} asks a construction of duality maps
\begin{align*}
    \A_\sfs(\bZ^T) \to \cO(\X_{\sfs^\vee}), \quad 
    \X_\sfs(\bZ^T) \to \cO(\A_{\sfs^\vee})
\end{align*}
which satisfy certain axioms, where $\sfs^\vee$ denotes the Langlands dual mutation class of $\sfs$ (we have $\sfs^\vee=\sfs$ for $\sfs=\sfs(\mathfrak{sl}_n,\Sigma)$). 
In particular, each $\bZ^T$-point gives rise to a universally positive Laurent polynomial on the dual side. 
One may ask their quantum aspects, by replacing $\cO(\X_{\sfs^\vee})$ with the Fock--Goncharov's quantized algebra $\cO_q(\X_{\sfs^\vee})$ \cite{FG08}, and replacing $\cO(\A_{\sfs^\vee})=\UCA_{\sfs^\vee}$ with $\cO_q(\A_{\sfs^\vee}):=\UCA_{\sfs^\vee_q}$ whenever the full-rank condition holds, though it involves a choice of a mutation class $\sfs^\vee_q$ of quantum seeds. 

For the mutation class $\sfs=\sfs(\mathfrak{sl}_2,\Sigma)$, such quantum duality maps are constructed via the skein theory. 
The classical duality maps are constructed by Fock--Goncharov \cite{FG03} when $\Sigma$ has empty boundary.
In the quantum setting, in the $\cO_q(\X_{\sfs^\vee})$-side, when $\Sigma$ has empty boundary, first Bonahon--Wong \cite{BonahonWong11} made a progress by defining a \emph{quantum trace map} from the ``stated'' Kauffman bracket skein algebra on $\Sigma$ to the ``square-root'' of the quantized algebra $\cO_q(\X_{\sfs^\vee})$. This has been upgraded to a quantum duality map $\A_\sfs(\bZ^T) \to \cO_q(\X_{\sfs^\vee})$ by Allegretti--Kim \cite{AllegrettiKim17}, by composing with a skein realization of each integral $\A$-lamination. On the $\cO_q(\A_{\sfs^\vee})$-side, when $\Sigma$ has no punctures, a quantum duality map $\X_\sfs(\bZ^T) \to \cO_q(\A_{\sfs^\vee})$ is worked out by Musiker--Schiffler--Williams \cite{MSW11,MSW13}, and finally established by Thurston \cite{ThurstonD14} by adding the loop elements to the Muller's work we mentioned above.
There are also related works in the skein theory side.
L\^{e}~\cite{TTQLe18, TTQLe19} gave an inclusion of the Muller's skein algebra ($\cO_q(\A_{\sfs^\vee})$-side) into 
the stated skein algebra ($\cO_q(\X_{\sfs^\vee})$-side), and obtained the explicit formula of the quantum trace map for certain simple loops.
Costantino--L\^{e}~\cite{CL} proved that the stated skein algebra of a biangle is isomorphic to the quantized coordinate ring of $\mathrm{SL}_{2}(\bC)$ and studied the correspondence of some algebraic structures.

The duality maps for the mutation class $\sfs=\sfs(\mathfrak{sl}_3,\Sigma)$ are recently intensively studied. Douglas--Sun \cite{DS20,DS20b} developed a theory on the bounded $\mathfrak{sl}_3$-laminations in terms of $\mathfrak{sl}_3$-webs (without endpoints on marked points) by defining their tropical $\A$-coordinates, based on the ideas of Xie \cite{Xie} from the viewpoint of Fock--Goncharov duality. 
Another coordinate systems of $\mathfrak{sl}_3$-webs are considered by Frohman--Sikora \cite{FrohmanSikora20}.
A stated $\mathfrak{sl}_{3}$-skein algebra was also defined by Higgins~\cite{Higgins20}.
Then a quantum duality map $\A_\sfs(\bZ^T) \to \cO_q(\X_{\sfs^\vee})$ has been established by Kim \cite{Kim20}, using the Douglas--Sun coordinates. 

Now our work in this paper can be regarded as a first step to constructing a duality map on the $\cO_q(\A_{\sfs^\vee})$-side. The tropical $\X$-coordinates of unbounded $\mathfrak{sl}_3$-laminations (an $\mathfrak{sl}_3$-analogue of the lamination shear coordinates) is introduced in \cite{IK}. For an unpunctured surface, a one-to-one correspondence between the integral $\mathfrak{sl}_3$-laminations and our basis webs is given there, proposing a conjectural construction of quantum duality map $\X_\bs(\bZ^T) \to \cO_q(\A_{\bs^\vee})$.

For the mutation class $\sfs=\sfs(\mathfrak{sp}_4,\Sigma)$, the comparison of quantum cluster and skein algebras as in this paper is carried out in \cite{IYsp4}. A study on the bounded $\mathfrak{sp}_4$-laminations is also on-going \cite{ISY}.

\subsection*{Organization of the paper}
In \cref{sect:skein}, we define the skein algebra $\SK{\Sigma}$ and investigate its basic structures. Expansion formulae and the positivity results are proved purely in terms of the skein theory in \cref{sec:expansion}.

In \cref{sect:qcluster}, we recall the general framework of the quantum cluster algebra. Here we partially use the terminology from the theory of cluster varieties \cite{FG09}, which is reviewed in \cref{sec:FG}. 
We also review the construction of the mutation class $\sfs(\mathfrak{sl}_3,\Sigma)$ related to the moduli space $\A_{SL_3,\Sigma}$. 

In \cref{sect:correspondance}, we construct the mutation class $\sfs_q(\mathfrak{sl}_3,\Sigma)$ of quantum seeds by realizing some of the quantum seeds in the skein algebra $\SK{\Sigma}$. Utilizing the results in \cref{sec:expansion}, we prove \cref{introthm:comparison,introthm:S=U,introthm:positivity} in \cref{sec:comparison}.

\bigskip

\noindent \textbf{Acknowledgements}
We are grateful to Greg Muller, Adam Sikora and Zhe Sun for valuable comments and insightful questions on the first version of this paper. 
T. I. is supported by JSPS KAKENHI Grant Number~JP20K22304.
W. Y. is supported by JSPS KAKENHI Grant Numbers~JP19J00252 and JP19K14528.
\bigskip

\subsection*{Notation on marked surfaces and their triangulations}\label{subsec:notation_marked_surface}

A \emph{marked surface} $(\Sigma,\bM)$ is a compact oriented surface $\Sigma$ with boundary equipped with a fixed non-empty finite set $\mathbb{M} \subset \Sigma$ of \emph{marked points}. 
When the choice of $\bM$ is clear from the context, we simply denote a marked surface by $\Sigma$. 
A marked point is called a \emph{puncture} if it lies in the interior of $\Sigma$, and a \emph{special point} otherwise. 
In this paper, we assume that there are no punctures, and hence $\mathbb{M} \subset \partial \Sigma$. We say that such a marked surface is \emph{unpunctured}. 
Moreover, assume the following conditions:
\begin{enumerate}
\item Each boundary component has at least one marked point.
\item 
$n(\Sigma):=-2\chi(\Sigma)+|\mathbb{M}|>0$.
\end{enumerate}
These conditions ensure that the marked surface $\Sigma$ has an ideal triangulation, that is, the isotopy class of a collection $\Delta$ of simple arcs connecting marked points whose interiors are mutually disjoint, which decomposes $\Sigma$ into triangles. 
The number $n(\Sigma)$ gives the number of triangles of any ideal triangulation $\Delta$. 
We call a connected component of the punctured boundary $\partial^\times \Sigma:=\partial\Sigma\setminus \mathbb{M}$ a \emph{boundary interval}, and denote by $\mathbb{B}$ the set of boundary intervals. Each boundary interval belongs to any ideal triangulation $\Delta$. We call an edge of $\Delta$ an \emph{interior edge} if it is not a boundary interval. Denote the set of edges (resp. interior edges, triangles) of $\Delta$ by $e(\Delta)$ (resp. $e_\interior(\Delta)$, $t(\Delta)$). 

It is sometimes useful to equip $\Delta$ with two distinguished points on the interior of each edge and one point in the interior of each triangle: see \cref{fig:sl3_triangulation}. The set of such points is denoted by $I(\Delta)=I_{\mathfrak{sl}_3}(\Delta)$.\footnote{This is the set of vertices of the $3$-triangulation \cite{FG03} associated with $\Delta$.} We refer to an ideal triangulation equipped with such collection of points as an \emph{$\mathfrak{sl}_3$-triangulation}. 
Let $I^{\mathrm{edge}}(\Delta)$ (resp. $I^{\mathrm{tri}}(\Delta)$) denote the set of points on edges (resp. faces of triangles) so that $I(\Delta)=I^{\mathrm{edge}}(\Delta) \sqcup I^{\mathrm{tri}}(\Delta)$, where we have a canonical identification $I^{\mathrm{tri}}(\Delta)=t(\Delta)$. 
For $k \in I^\mathrm{edge}(\Delta)$, let $k^\mathrm{op}$ denote the other point on the same edge. 
Let $I(\Delta)_\f \subset I^\mathrm{edge}(\Delta)$ be the subset consisting of the points on the boundary of $\Sigma$, and let $I(\Delta)_\uf:=I(\Delta) \setminus I(\Delta)_\f$.

More generally, we can consider an \emph{ideal cell decomposition} of $\Sigma$, which is a decomposition of $\Sigma$ into a union of polygons. When it is obtained from an ideal triangulation by removing $k$ interior edges, it is said to be \emph{of deficiency $k$}. 
In this paper, we only use an ideal cell decomposition of deficiency $0$ or $1$. The ideal cell decomposition of deficiency $1$ obtained from an ideal triangulation $\Delta$ by removing one interior edge $E$ is denoted by $(\Delta;E)$.

A \emph{decorated triangulation} $\bD=(\Delta,\bs_\Delta)$ will be an object to which we can concretely associate a web cluster in $\SK{\Sigma}$ (\cref{webcluster}) and
a quantum cluster in $\CA^q_{\mathfrak{sl}_3,\Sigma}$, respectively.
It consists of:
\begin{itemize}
    \item An ideal triangulation $\Delta$ of $\Sigma$;
    \item A function $\bs_\Delta: t(\Delta) \to \{+,-\}$.
\end{itemize}
In relation with the cluster theory on the moduli space $\A_{SL_3,\Sigma}$, 
the signature $\bs_\Delta(T)$ for $T \in t(\Delta)$ corresponds to the two possible choices of reduced words of the longest element of the Weyl group $W(\mathfrak{sl}_3)$. See \cref{rem:geometry of moduli} for a detail.

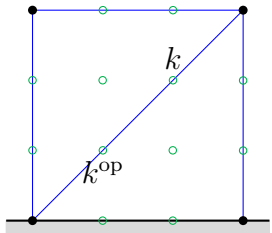
\begin{figure}[ht]
\begin{tikzpicture}[scale=0.7]
\draw[blue] (2,2) -- (-2,2) -- (-2,-2) -- (2,-2) --cycle;
\draw[blue] (2,2) -- (-2,-2);
\filldraw[gray!30] (-2.5,-2) -- (2.5,-2) -- (2.5,-2.3) -- (-2.5,-2.3) --cycle;
\draw[thick] (-2.5,-2) -- (2.5,-2);
\foreach \x in {45,135,225,315}
\path(\x:2.8284) node [fill, circle, inner sep=1.2pt]{};
{\color{mygreen}
\quiversquare{2,2}{-2,2}{-2,-2}{2,-2}
}
\draw(y131) node[above]{$k$};
\draw(y132) node[below]{$k^\mathrm{op}$};
\end{tikzpicture}
    \caption{A local picture of an $\mathfrak{sl}_3$-triangulation. By convention, a portion of $\partial \Sigma$ is drawn by a thick line together with a gray region indicating the ``outer side'' of $\Sigma$.}
    \label{fig:sl3_triangulation}
\end{figure}

\section{Skein algebras of unpunctured marked surfaces for \texorpdfstring{$\mathfrak{sl}_3$}{sl3}}\label{sect:skein}
A skein algebra of a connected compact oriented surface $\Sigma$ is the quotient of the algebra of links in the thickened surface $\Sigma\times[0,1]$ defined by certain skein relations.
Skein relations are obtained from representations of quantum groups associated with simple Lie algebras.
For $\mathfrak{sl}_2$, the skein relation is known as the Kauffman bracket skein relation, and the skein algebra is called the Kauffman bracket skein algebra.
Muller~\cite{Muller16} introduced the boundary Kauffman bracket skein relation for tangle diagrams on an unpunctured marked surface $(\Sigma,\bM)$ and defined the Kauffman bracket skein algebra of $(\Sigma,\bM)$.
In this section, we will introduce a skein algebra $\SK{\Sigma}$ of an unpunctured marked surface $(\Sigma,\bM)$ for $\mathfrak{sl}_3$, and observe the ``cluster'' structure of the skein algebra $\SK{T}$ of a triangle. See \cref{subsec:quadrilateral_skein} for a quadrilateral.
We will see that the skein algebra $\SK{\Sigma}$ has an Ore localization $\SK{\Sigma}[\Delta^{-1}]$ on each ideal triangulation, and is closely related to a quantum cluster algebra quantizing the mutation class $\sfs(\mathfrak{sl}_3,\Sigma)$.
Moreover, certain tangled trivalent graphs are expressed as positive Laurent polynomials in ``elementary webs'' in $\SK{\Sigma}[\Delta^{-1}]$.

\subsection{Skein algebras of unpunctured marked surfaces for \texorpdfstring{$\mathfrak{sl}_3$}{sl3}}
Let $\bN$ be the set of non-negative integers and $\bZ_{A}:=\bZ[A^{1/2},A^{-1/2}]$ the Laurent polynomial ring in a variable $A^{1/2}$.
In this subsection, there is no need to consider the conditions (1) and (2) for an unpunctured marked surface $(\Sigma,\bM)$ in \cref{sec:intro}.
\subsubsection{The boundary $\mathfrak{sl}_3$-skein relation}
The skein algebra $\SK{\Sigma}$ treats tangled trivalent graphs with endpoints in $\bM$, and its skein relations are defined by adding \emph{boundary $\mathfrak{sl}_3$-skein relations} to the $\mathfrak{sl}_3$-skein relations introduced in Kuperberg~\cite{Kuperberg96}.

A \emph{tangled trivalent graph} $G$ on $(\Sigma,\bM)$ is an immersion of an oriented uni-trivalent graph into $\Sigma$ satisfying the following conditions (1) -- (7):
\begin{enumerate}
	\item the valency of a vertex of the underlying graph is $1$ or $3$,
	\item the univalent vertices of $G$ are contained in $\bM$,
	\item the trivalent vertices of $G$ are distinct points in $\Sigma\setminus\partial\Sigma$,
	\item all intersection points of $G$ in $\Sigma\setminus\partial\Sigma$ are transverse double points of edges,
	\item an intersection point $p\in\Sigma\setminus\partial\Sigma$ of $G$ has \emph{over-/under-passing information} (we call such $p$ an \emph{internal crossings}),
	\item for an intersection point $p\in \bM$ of $G$, the set of univalent vertices on $p$ has a strict total order, which we call the \emph{elevation at $p$},
	\item the orientation of edges incident to a trivalent vertex is a sink
		\ \tikz[baseline=-.6ex, scale=.1]{
			\draw[dashed, fill=white] (0,0) circle [radius=5];
			\draw[very thick,red, -<-] (0:0) -- (90:5); 
			\draw[very thick, red, -<-] (0:0) -- (210:5); 
			\draw[very thick, red, -<-] (0:0) -- (-30:5);
		\ } 
		or a source
		\ 
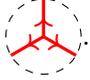
\tikz[baseline=-.6ex, scale=.1]{
			\draw[dashed, fill=white] (0,0) circle [radius=5];
			\draw[very thick, red, ->-] (0:0) -- (90:5); 
			\draw[very thick, red, ->-] (0:0) -- (210:5); 
			\draw[very thick, red, ->-] (0:0) -- (-30:5);
		\ }. 
		
\end{enumerate}
We denote the number of sinks in $G$ by $t_{+}(G)$, and sources by $t_{-}(G)$. 
The over-/under-passing information is indicated as
\ \tikz[baseline=-.6ex, scale=.1]{
	\draw[dashed, fill=white] (0,0) circle [radius=5];
	\draw[very thick, red] (-45:5) -- (135:5);
	\draw[red, overarc] (-135:5) -- (45:5);
\ }.
Two consecutive ordered univalent vertices $v_{1}<v_{2}$, whose half-edges $e_1$ and $e_2$ are incident to $p\in \bM$, are indicated as
\ \tikz[baseline=-.6ex, scale=.1, yshift=-4cm]{
	\coordinate (P) at (0,0);
	\draw[red, very thick] (P) -- (135:7);
	\fill[white] (P) circle [radius=2cm];
	\draw[red, very thick] (P) -- (45:7);
	\draw[dashed] (7,0) arc (0:180:7cm);
	\draw[gray, line width=2pt] (-7,0) -- (7,0);
	\draw[fill=black] (P) circle [radius=20pt];
	\node at (135:7) [left]{$e_{1}$};
	\node at (45:7) [right]{$e_{2}$};
\ } or
\ \tikz[baseline=-.6ex, scale=.1, yshift=-4cm]{
	\coordinate (P) at (0,0);
	\draw[red, very thick] (P) -- (45:7);
	\fill[white] (P) circle [radius=2cm];
	\draw[red, very thick] (P) -- (135:7);
	\draw[dashed] (7,0) arc (0:180:7cm);
	\draw[gray, line width=2pt] (-7,0) -- (7,0);
	\draw[fill=black] (P) circle [radius=20pt];
	\node at (45:7) [right]{$e_{1}$};
	\node at (135:7) [left]{$e_{2}$};
\ }.

We define skein relations for the tangled trivalent graphs on $(\Sigma,\bM)$.
\begin{dfn}[$\mathfrak{sl}_3$-skein relations~\cite{Kuperberg96}]\label{def:skeinrel}
	\begin{align}
		\mathord{
			\ \tikz[baseline=-.6ex, scale=.1]{
				\draw[dashed, fill=white] (0,0) circle [radius=7];
				\draw[red, very thick, ->-={.8}{red}] (-45:7) -- (135:7);
				\draw[overarc, ->-={.8}{red}] (-135:7) -- (45:7);
				}
		\ }
		&=A^{2}\mathord{
			\ \tikz[baseline=-.6ex, scale=.1]{
				\draw[dashed, fill=white] (0,0) circle [radius=7];
				\draw[very thick, red, ->-] (-45:7) to[out=north west, in=south] (3,0) to[out=north, in=south west] (45:7);
				\draw[very thick, red, ->-] (-135:7) to[out=north east, in=south] (-3,0) to[out=north, in=south east] (135:7);
			}
		\ }
		+A^{-1}\mathord{
			\ \tikz[baseline=-.6ex, scale=.1]{
				\draw[dashed, fill=white] (0,0) circle [radius=7];
				\draw[very thick, red, ->-] (-45:7) -- (0,-3);
				\draw[very thick, red, ->-] (-135:7) -- (0,-3);
				\draw[very thick, red, -<-] (45:7) -- (0,3);
				\draw[very thick, red, -<-] (135:7) -- (0,3);
				\draw[very thick, red, -<-] (0,-3) -- (0,3);
			}
		\ },\label{rel:plus}\\
		\mathord{
			\ \tikz[baseline=-.6ex, scale=.1]{
				\draw[dashed, fill=white] (0,0) circle [radius=7];
				\draw[very thick, red, ->-={.8}{red}] (-135:7) -- (45:7);
				\draw[overarc, ->-={.8}{red}] (-45:7) -- (135:7);
			}
		\ }
		&=A^{-2}\mathord{
			\ \tikz[baseline=-.6ex, scale=.1]{
				\draw[dashed, fill=white] (0,0) circle [radius=7];
				\draw[very thick, red, ->-] (-45:7) to[out=north west, in=south] (3,0) to[out=north, in=south west] (45:7);
				\draw[very thick, red, ->-] (-135:7) to[out=north east, in=south] (-3,0) to[out=north, in=south east] (135:7);
			}
		\ }
		+A\mathord{
			\ \tikz[baseline=-.6ex, scale=.1]{
				\draw[dashed, fill=white] (0,0) circle [radius=7];
				\draw[very thick, red, ->-] (-45:7) -- (0,-3);
				\draw[very thick, red, ->-] (-135:7) -- (0,-3);
				\draw[very thick, red, -<-] (45:7) -- (0,3);
				\draw[very thick, red, -<-] (135:7) -- (0,3);
				\draw[very thick, red, -<-] (0,-3) -- (0,3);
			}
		\ },\label{rel:minus}\\
		\mathord{
			\ \tikz[baseline=-.6ex, scale=.1]{
				\draw[dashed, fill=white] (0,0) circle [radius=7];
				\draw[very thick, red, -<-={.6}{}] (-45:7) -- (-45:3);
				\draw[very thick, red, ->-={.6}{}] (-135:7) -- (-135:3);
				\draw[very thick, red, ->-={.6}{}] (45:7) -- (45:3);
				\draw[very thick, red, -<-={.6}{}] (135:7) -- (135:3);
				\draw[very thick, red, -<-] (45:3) -- (135:3);
				\draw[very thick, red, ->-] (-45:3) -- (-135:3);
				\draw[very thick, red, -<-] (45:3) -- (-45:3);
				\draw[very thick, red, ->-] (135:3) -- (-135:3);
			}
		\ }
		&=\mathord{
			\ \tikz[baseline=-.6ex, scale=.1]{
				\draw[dashed, fill=white] (0,0) circle [radius=7];
				\draw[very thick, red, -<-] (-45:7) to[out=north west, in=south] (3,0) to[out=north, in=south west] (45:7);
				\draw[very thick, red, ->-] (-135:7) to[out=north east, in=south] (-3,0) to[out=north, in=south east] (135:7);
			}
		\ }
		+\mathord{
			\ \tikz[baseline=-.6ex, rotate=90, scale=.1]{
				\draw[dashed, fill=white] (0,0) circle [radius=7];
				\draw[very thick, red, ->-] (-45:7) to[out=north west, in=south] (3,0) to[out=north, in=south west] (45:7);
				\draw[very thick, red, -<-] (-135:7) to[out=north east, in=south] (-3,0) to[out=north, in=south east] (135:7);
			}
		\ },\label{rel:fourgon}\\
		\mathord{
			\ \tikz[baseline=-.6ex, scale=.1]{
				\draw[dashed, fill=white] (0,0) circle [radius=7];
				\draw[very thick, red, ->-] (0,-7) -- (0,-3);
				\draw[very thick, red, ->-] (0,3) -- (0,7);
				\draw[very thick, red, -<-] (0,-3) to[out=east, in=south] (3,0) to[out=north, in=east] (0,3);
				\draw[very thick, red, -<-] (0,-3) to[out=west, in=south] (-3,0) to[out=north, in=west] (0,3);
			}
		\ }
		&=-(A^3+A^{-3})\mathord{
			\ \tikz[baseline=-.6ex, scale=.1]{
				\draw[dashed, fill=white] (0,0) circle [radius=7];
				\draw[very thick, red, ->-] (0,-7) -- (0,7);
			}
		\ },\label{rel:bigon}\\
		\mathord{
			\ \tikz[baseline=-.6ex, scale=.1]{
				\draw[dashed, fill=white] (0,0) circle [radius=7];
				\draw[very thick, red, ->-] (0,0) circle [radius=3];
			}
		\ }
		&=(A^{6}+1+A^{-6})
		\mathord{
			\ \tikz[baseline=-.6ex, scale=.1]{
				\draw[dashed, fill=white] (0,0) circle [radius=7];
			}
		\ }
		=\mathord{
			\ \tikz[baseline=-.6ex, scale=.1]{
				\draw[dashed, fill=white] (0,0) circle [radius=7];
				\draw[very thick, red, -<-] (0,0) circle [radius=3];
			}
		\ }.\label{rel:circle}
	\end{align}
\end{dfn}

\begin{dfn}[boundary $\mathfrak{sl}_3$-skein relations~\cite{FrohmanSikora20}]\label{def:bskeinrel}
	\begin{align}
		\mathord{
			\ \tikz[baseline=-.6ex, scale=.1, yshift=-4cm]{
				\coordinate (P) at (0,0);
				\draw[very thick, red, ->-] (P) -- (135:10);
				\fill[white] (P) circle [radius=2cm];
				\draw[very thick, red, ->-] (P) -- (45:10);
				\draw[dashed] (10,0) arc (0:180:10cm);
				\bline{-10,0}{10,0}{2}
				\draw[fill=black] (P) circle [radius=20pt];
			\ }
		}
		&=A^{2}
		\mathord{
			\ \tikz[baseline=-.6ex, scale=.1, yshift=-4cm]{
				\coordinate (P) at (0,0);
				\draw[very thick, red, ->-] (P) -- (45:10);
				\fill[white] (P) circle [radius=2cm];
				\draw[very thick, red, ->-] (P) -- (135:10);
				\draw[dashed] (10,0) arc (0:180:10cm);
				\bline{-10,0}{10,0}{2}
				\draw[fill=black] (P) circle [radius=20pt];
			\ }
		}
		&\mathord{
			\ \tikz[baseline=-.6ex, scale=.1, yshift=-4cm]{
				\coordinate (P) at (0,0);
				\draw[very thick, red, -<-] (P) -- (135:10);
				\fill[white] (P) circle [radius=2cm];
				\draw[very thick, red, -<-] (P) -- (45:10);
				\draw[dashed] (10,0) arc (0:180:10cm);
				\bline{-10,0}{10,0}{2}
				\draw[fill=black] (P) circle [radius=20pt];
			\ }
		}
		&=A^{2}
		\mathord{
			\ \tikz[baseline=-.6ex, scale=.1, yshift=-4cm]{
				\coordinate (P) at (0,0);
				\draw[very thick, red, -<-] (P) -- (45:10);
				\fill[white] (P) circle [radius=2cm];
				\draw[very thick, red, -<-] (P) -- (135:10);
				\draw[dashed] (10,0) arc (0:180:10cm);
				\bline{-10,0}{10,0}{2}
				\draw[fill=black] (P) circle [radius=20pt];
			\ }
		}\label{rel:parallel}\\
		\mathord{
			\ \tikz[baseline=-.6ex, scale=.1, yshift=-4cm]{
				\coordinate (P) at (0,0);
				\draw[very thick, red, -<-] (P) -- (135:10);
				\fill[white] (P) circle [radius=2cm];
				\draw[very thick, red, ->-] (P) -- (45:10);
				\draw[dashed] (10,0) arc (0:180:10cm);
				\bline{-10,0}{10,0}{2}
				\draw[fill=black] (P) circle [radius=20pt];
			\ }
		}
		&=A
		\mathord{
			\ \tikz[baseline=-.6ex, scale=.1, yshift=-4cm]{
				\coordinate (P) at (0,0);
				\draw[very thick, red, ->-] (P) -- (45:10);
				\fill[white] (P) circle [radius=2cm];
				\draw[very thick, red, -<-] (P) -- (135:10);
				\draw[dashed] (10,0) arc (0:180:10cm);
				\bline{-10,0}{10,0}{2}
				\draw[fill=black] (P) circle [radius=20pt];
			\ }
		}
		&\mathord{
			\ \tikz[baseline=-.6ex, scale=.1, yshift=-4cm]{
				\coordinate (P) at (0,0);
				\draw[very thick, red, ->-] (P) -- (135:10);
				\fill[white] (P) circle [radius=2cm];
				\draw[very thick, red, -<-] (P) -- (45:10);
				\draw[dashed] (10,0) arc (0:180:10cm);
				\bline{-10,0}{10,0}{2}
				\draw[fill=black] (P) circle [radius=20pt];
			\ }
		}
		&=A
		\mathord{
			\ \tikz[baseline=-.6ex, scale=.1, yshift=-4cm]{
				\coordinate (P) at (0,0);
				\draw[very thick, red, -<-] (P) -- (45:10);
				\fill[white] (P) circle [radius=2cm];
				\draw[very thick, red, ->-] (P) -- (135:10);
				\draw[dashed] (10,0) arc (0:180:10cm);
				\bline{-10,0}{10,0}{2}
				\draw[fill=black] (P) circle [radius=20pt];
			\ }
		}\label{rel:antiparallel}\\
		\mathord{
			\ \tikz[baseline=-.6ex, scale=.1, yshift=-4cm]{
				\coordinate (P) at (0,0);
				\node[inner sep=0, outer sep=0, circle, fill=black] (R) at (45:7) {};
				\node[inner sep=0, outer sep=0, circle, fill=black] (L) at (135:7) {};
				\draw[very thick, red, ->-] (L) -- (R);
				\draw[very thick, red, -<-={.7}{}] (P) -- (L);
				\draw[very thick, red, -<-={.7}{}] (R) -- (45:10);
				\draw[very thick, red, ->-={.7}{}] (L) -- (135:10);
				\fill[white] (P) circle [radius=2cm];
				\draw[very thick, red, ->-={.7}{}] (P) -- (R);
				\draw[dashed] (10,0) arc (0:180:10cm);
				\bline{-10,0}{10,0}{2}
				\draw[fill=black] (P) circle [radius=20pt];
			\ }
		}
		&=
		\mathord{
			\ \tikz[baseline=-.6ex, scale=.1, yshift=-4cm]{
				\coordinate (P) at (0,0) {};
				\draw[very thick, red, ->-] (P) -- (135:10);
				\fill[white] (P) circle [radius=2cm];
				\draw[very thick, red, -<-] (P) -- (45:10);
				\draw[dashed] (10,0) arc (0:180:10cm);
				\bline{-10,0}{10,0}{2}
				\draw[fill=black] (P) circle [radius=20pt];
			\ }
		}
		&\mathord{
			\ \tikz[baseline=-.6ex, scale=.1, yshift=-4cm]{
				\coordinate (P) at (0,0);
				\node[inner sep=0, outer sep=0, circle, fill=black] (R) at (45:7) {};
				\node[inner sep=0, outer sep=0, circle, fill=black] (L) at (135:7) {};
				\draw[very thick, red, -<-] (L) -- (R);
				\draw[very thick, red, ->-={.7}{}] (P) -- (L);
				\draw[very thick, red, ->-={.7}{}] (R) -- (45:10);
				\draw[very thick, red, -<-={.7}{}] (L) -- (135:10);
				\fill[white] (P) circle [radius=2cm];
				\draw[very thick, red, -<-={.7}{}] (P) -- (R);
				\draw[dashed] (10,0) arc (0:180:10cm);
				\bline{-10,0}{10,0}{2}
				\draw[fill=black] (P) circle [radius=20pt];
			\ }
		}
		&=
		\mathord{
			\ \tikz[baseline=-.6ex, scale=.1, yshift=-4cm]{
				\coordinate (P) at (0,0);
				\draw[very thick, red, -<-] (P) -- (135:10);
				\fill[white] (P) circle [radius=2cm];
				\draw[very thick, red, ->-] (P) -- (45:10);
				\draw[dashed] (10,0) arc (0:180:10cm);
				\bline{-10,0}{10,0}{2}
				\draw[fill=black] (P) circle [radius=20pt];
			\ }
		}\label{rel:pbigon}\\
		\mathord{
			\ \tikz[baseline=-.6ex, scale=.1, yshift=-4cm]{
				\coordinate (P) at (0,0);
				\node[inner sep=0, outer sep=0, circle, fill=black] (C) at (90:7) {};
				\draw[very thick, red, ->-] (P) to[out=north west, in=west] (C);
				\fill[white] (P) circle [radius=2cm];
				\draw[very thick, red, ->-] (P) to[out=north east, in=east] (C);
				\draw[very thick, red, -<-={.7}{}] (C) -- (90:10);
				\draw[dashed] (10,0) arc (0:180:10cm);
				\bline{-10,0}{10,0}{2}
				\draw[fill=black] (P) circle [radius=20pt];
			\ }
		}
		&=0
		&\mathord{
			\ \tikz[baseline=-.6ex, scale=.1, yshift=-4cm]{
				\coordinate (P) at (0,0);
				\node[inner sep=0, outer sep=0, circle, fill=black] (C) at (90:7) {};
				\draw[very thick, red, -<-] (P) to[out=north west, in=west] (C);
				\fill[white] (P) circle [radius=2cm];
				\draw[very thick, red, -<-] (P) to[out=north east, in=east] (C);
				\draw[very thick, red, ->-={.7}{}] (C) -- (90:10);
				\draw[dashed] (10,0) arc (0:180:10cm);
				\bline{-10,0}{10,0}{2}
				\draw[fill=black] (P) circle [radius=20pt];
			\ }
		}
		&=0\label{rel:degbigon}\\
		\mathord{
			\ \tikz[baseline=-.6ex, scale=.1, yshift=-4cm]{
				\coordinate (P) at (0,0);
				\node[inner sep=0, outer sep=0, circle, fill=black] (C) at (90:7) {};
				\draw[very thick, red, -<-] (P) to[out=north west, in=west] (C);
				\fill[white] (P) circle [radius=2cm];
				\draw[red, very thick] (P) to[out=north east, in=east] (C);
				\draw[dashed] (10,0) arc (0:180:10cm);
				\bline{-10,0}{10,0}{2}
				\draw[fill=black] (P) circle [radius=20pt];
			\ }
		}
		&=0
		&\mathord{
			\ \tikz[baseline=-.6ex, scale=.1, yshift=-4cm]{
				\coordinate (P) at (0,0) {};
				\node[inner sep=0, outer sep=0, circle, fill=black] (C) at (90:7) {};
				\draw[very thick, red, ->-] (P) to[out=north west, in=west] (C);
				\fill[white] (P) circle [radius=2cm];
				\draw[red, very thick] (P) to[out=north east, in=east] (C);
				\draw[dashed] (10,0) arc (0:180:10cm);
				\bline{-10,0}{10,0}{2}
				\draw[fill=black] (P) circle [radius=20pt];
			\ }
		}
		&=0\label{rel:pcircle}
	\end{align}
\end{dfn}

It is easy to see that the $\mathfrak{sl}_3$-skein relations, boundary $\mathfrak{sl}_3$-skein relations, and the boundary fixing isotopy realize the following \emph{Reidemeister moves} (R1'), (R2), (R3), (R4) and (bR).
\begin{dfn}[Reidemeister moves]
	\begin{align*}
		\mathord{
			\ \tikz[baseline=-.6ex, scale=.1]{
				\draw[dashed] (0,0) circle [radius=7];
				\draw[red, very thick] (3,-2) to[out=south, in=east] (2,-3) to[out=west, in=south] (-1,0) to[out=north, in=west] (2,3) to[out=east, in=north] (3,2);
				\draw[overarc] (0,-7) to[out=north, in=west] (2,-1) to[out=east, in=north] (3,-2);
				\draw[overarc] (3,2) to[out=south, in=east] (2,1) to[out=west, in=south] (0,7);
			}
		\ }
		&\mathord{
			\tikz[baseline=-.6ex, scale=.1]{
				\draw[<->] (0,0)--(10,0);
			}
		} 
		\mathord{
			\ \tikz[baseline=-.6ex, scale=.1]{
				\draw[dashed] (0,0) circle [radius=7];
				\draw[red, very thick] (90:7) to (-90:7);
			}
		\ }\tag{R1'}\\
		\mathord{
			\ \tikz[baseline=-.6ex, scale=.1]{
				\draw[dashed] (0,0) circle [radius=7];
				\draw[red, very thick] (135:7) to[out=south east, in=west] (0,-2) to[out=east, in=south west](45:7);
				\draw[overarc] (-135:7) to[out=north east, in=west] (0,2) to[out=east, in=north west] (-45:7);
			}
		\ }
		&\mathord{
			\tikz[baseline=-.6ex, scale=.1]{
				\draw[<->] (0,0)--(10,0);
			}
		} 
		\mathord{
			\ \tikz[baseline=-.6ex, scale=.1]{
				\draw[dashed] (0,0) circle [radius=7];
				\draw[red, very thick] (135:7) to[out=south east, in=west](0,2) to[out=east, in=south west](45:7);
				\draw[red, very thick] (-135:7) to[out=north east, in=west](0,-2) to[out=east, in=north west] (-45:7);
			}
		\ }\tag{R2}\\
		\mathord{
			\ \tikz[baseline=-.6ex, scale=.1]{
				\draw[dashed] (0,0) circle [radius=7];
				\draw[red, very thick] (-135:7) -- (45:7);
				\draw[overarc] (135:7) -- (-45:7);
				\draw[overarc] (180:7) to[out=east, in=west](0,5) to[out=east, in=west] (0:7);
			}
		\ }
		&\mathord{
			\tikz[baseline=-.6ex, scale=.1]{
				\draw[<->] (0,0)--(10,0);
			}
		} 
		\mathord{
			\ \tikz[baseline=-.6ex, scale=.1]{
				\draw[dashed] (0,0) circle [radius=7];
				\draw[red, very thick] (-135:7) -- (45:7);
				\draw[overarc] (135:7) -- (-45:7);
				\draw[overarc] (180:7) to[out=east, in=west] (0,-5) to[out=east, in=west] (0:7);
			}
		\ }\tag{R3}\\
		\mathord{
			\ \tikz[baseline=-.6ex, scale=.1]{
				\draw[dashed] (0,0) circle [radius=7];
				\draw[red, very thick] (0:0) -- (90:7); 
				\draw[red, very thick] (0:0) -- (210:7); 
				\draw[red, very thick] (0:0) -- (-30:7);
				\draw[overarc] (180:7) to[out=east, in=west] (0,5) to[out=east, in=west] (0:7);
			}
		\ }
		\mathord{
			\tikz[baseline=-.6ex, scale=.1]{
				\draw[<->] (0,0)--(10,0);
			}
		} 
		\mathord{
			\ \tikz[baseline=-.6ex, scale=.1]{
				\draw[dashed] (0,0) circle [radius=7];
				\draw[red, very thick] (0:0) -- (90:7); 
				\draw[red, very thick] (0:0) -- (210:7); 
				\draw[red, very thick] (0:0) -- (-30:7);
				\draw[overarc] (180:7) to[out=east, in=west] (0,-5) to[out=east, in=west](0:7);
			}
		\ },\quad
		&\mathord{
			\ \tikz[baseline=-.6ex, scale=.1]{
				\draw[dashed] (0,0) circle [radius=7];
				\draw[overarc] (180:7) to[out=east, in=west] (0,5) to[out=east, in=west] (0:7);
				\draw[overarc] (0:0) -- (90:7); 
				\draw[red, very thick] (0:0) -- (210:7); 
				\draw[red, very thick] (0:0) -- (-30:7);
			}
		\ }
		\mathord{
			\tikz[baseline=-.6ex, scale=.1]{
				\draw [<->] (0,0)--(10,0);
			}
		} 
		\mathord{
			\ \tikz[baseline=-.6ex, scale=.1]{
				\draw[dashed] (0,0) circle [radius=7];
				\draw[overarc] (180:7) to[out=east, in=west] (0,-5) to[out=east, in=west](0:7);
				\draw[red, very thick] (0:0) -- (90:7); 
				\draw[overarc] (0:0) -- (210:7); 
				\draw[overarc] (0:0) -- (-30:7);
			}
		\ }\tag{R4}\\
		\mathord{
			\ \tikz[baseline=-.6ex, scale=.1, yshift=-2cm]{
				\coordinate (P) at (0,0);
				\draw[very thick, red, rounded corners] ($(P)+(135:2)$) -- (135:4) -- (60:8);
				\draw[very thick, red, rounded corners, overarc] (P) -- (45:4) -- (120:8);
				\draw[dashed] (8,0) arc (0:180:8cm);
				\bline{-8,0}{8,0}{2}
				\draw[fill=black] (P) circle [radius=20pt];
			\ }
		}\ 
		&\mathord{
			\tikz[baseline=-.6ex, scale=.1]{
				\draw [<->] (0,0)--(10,0);
			}
		} 
		\mathord{
			\ \tikz[baseline=-.6ex, scale=.1, yshift=-2cm]{
				\coordinate (P) at (0,0);
				\draw[very thick, red] (P) -- (120:8);
				\draw[very thick, red] ($(P)+(60:2)$) -- (60:8);
				\draw[dashed] (8,0) arc (0:180:8cm);
				\bline{-8,0}{8,0}{2}
				\draw[fill=black] (P) circle [radius=20pt];
			\ }
		}\tag{bR}
	\end{align*}
\end{dfn}

\begin{dfn}[the $\mathfrak{sl}_3$-skein algebra of an unpunctured marked surface~\cite{FrohmanSikora20}]
	The \emph{$\mathfrak{sl}_3$-skein algebra $\SK{\Sigma}$ of an unpunctured marked surface $(\Sigma,\bM)$} is defined to be the quotient module of the free $\bZ_{A}$-module spanned by tangled trivalent graphs in $(\Sigma,\bM)$ by the $\mathfrak{sl}_3$-skein relations  (\cref{def:skeinrel}), the boundary $\mathfrak{sl}_3$-skein relations (\cref{def:bskeinrel}), and isotopy in $\Sigma$ relative to $\partial \Sigma$.
	It is equipped with a multiplication defined by the superposition of tangled trivalent graphs. 
	The product $G_{1}G_{2}$ of two tangled trivalent graphs $G_{1}$ and $G_{2}$ in generic position is defined by superposing $G_{1}$ on $G_{2}$, that is, $G_{1}$ is over-passing $G_{2}$ in all intersection points. 
	We call an element in $\SK{\Sigma}$ an \emph{$\mathfrak{sl}_3$-web} or simply a \emph{web}.
\end{dfn}

\begin{rem}
	In the right-hand sides of the boundary skein relations at $p\in\bM$, 
	the sign of an exponent of $A$ only depends on the orientation from the arc with a higher elevation to the lower one with respect to the orientation of $\Sigma$.
	The absolute value of an exponent depends on whether two arcs have a parallel direction or anti-parallel.
\end{rem}

It is useful to slightly extend the definition of the $\mathfrak{sl}_3$-webs, allowing them to have univalent vertices with the same elevation.

\begin{dfn}[simultaneous crossings]\label{def:simulcrossing}
	An $\mathfrak{sl}_3$-web with the \emph{simultaneous crossing} at $p \in \bM$ is recursively defined by the following skein relations:
	\begin{align}
		A^{-l-\frac{k}{2}}\mathord{
			\ \tikz[baseline=-.6ex, scale=.1, yshift=-2cm]{
				\coordinate (P) at (0,0);
				\draw[line width=2.6, red] ($(P)+(135:2)$) -- (135:7);
				\draw[very thick, red, ->-={.7}{}] (P) -- (45:7);
				\draw[dashed] (7,0) arc (0:180:7cm);
				\bline{-7,0}{7,0}{2}
				\draw[fill=black] (P) circle [radius=20pt];
				\node at (135:7) [above, red]{\scriptsize $(k,l)$};
			\ }
		}
		&=
		\mathord{
			\ \tikz[baseline=-.6ex, scale=.1, yshift=-2cm]{
				\coordinate (P) at (0,0);
				\draw[very thick, red, ->-={.7}{}] (P) -- (45:7);
				\draw[line width=2.6, red] (P) -- (135:7);
				\draw[dashed] (7,0) arc (0:180:7cm);
				\bline{-7,0}{7,0}{2}
				\draw[fill=black] (P) circle [radius=20pt];
				\node at (135:7) [above, red]{\scriptsize $(k,l)$};
			\ }
		}
		=
		A^{l+\frac{k}{2}}\mathord{
			\ \tikz[baseline=-.6ex, scale=.1, yshift=-2cm]{
				\coordinate (P) at (0,0);
				\draw[very thick, red, ->-={.7}{}] ($(P)+(45:2)$) -- (45:7);
				\draw[line width=2.6, red] (P) -- (135:7);
				\draw[dashed] (7,0) arc (0:180:7cm);
				\bline{-7,0}{7,0}{2}
				\draw[fill=black] (P) circle [radius=20pt];
				\node at (135:7) [above, red]{\scriptsize $(k,l)$};
			\ }
		},\notag\\
		A^{-k-\frac{l}{2}}\mathord{
			\ \tikz[baseline=-.6ex, scale=.1, yshift=-2cm]{
				\coordinate (P) at (0,0);
				\draw[line width=2.6, red] ($(P)+(135:2)$) -- (135:7);
				\draw[very thick, red, -<-={.7}{}] (P) -- (45:7);
				\draw[dashed] (7,0) arc (0:180:7cm);
				\bline{-7,0}{7,0}{2}
				\draw[fill=black] (P) circle [radius=20pt];
				\node at (135:7) [above, red]{\scriptsize $(k,l)$};
			\ }
		}
		&=
		\mathord{
			\ \tikz[baseline=-.6ex, scale=.1, yshift=-2cm]{
				\coordinate (P) at (0,0);
				\draw[very thick, red, -<-={.7}{}] (P) -- (45:7);
				\draw[line width=2.6, red] (P) -- (135:7);
				\draw[dashed] (7,0) arc (0:180:7cm);
				\bline{-7,0}{7,0}{2}
				\draw[fill=black] (P) circle [radius=20pt];
				\node at (135:7) [above, red]{\scriptsize $(k,l)$};
			\ }
		}
		=
		A^{k+\frac{l}{2}}\mathord{
			\ \tikz[baseline=-.6ex, scale=.1, yshift=-2cm]{
				\coordinate (P) at (0,0) {};
				\draw[very thick, red, -<-={.7}{}] ($(P)+(45:2)$) -- (45:7);
				\draw[line width=2.6, red] (P) -- (135:7);
				\draw[dashed] (7,0) arc (0:180:7cm);
				\bline{-7,0}{7,0}{2}
				\draw[fill=black] (P) circle [radius=20pt];
				\node at (135:7) [above, red]{\scriptsize $(k,l)$};
			\ }
		},\label{rel:simultaneous}
	\end{align}
	where the thickened edge labeled by $(k,l)$ is a collection of $k+l$ half-edges with simultaneous crossing whose endpoint degree is $(k,l)$.  
	From the above skein relations, we obtain the \emph{boundary twist relations} for two adjacent half-edges with a simultaneous crossing:
	\begin{align}
		&\mathord{
			\ \tikz[baseline=-.6ex, scale=.1, yshift=-2cm]{
				\coordinate (P) at (0,0);
				\draw[very thick, red, rounded corners, ->-={.9}{red}] (P) -- (45:4) -- (120:8);
				\draw[very thick, red, rounded corners, overarc, ->-={.9}{red}] (P) -- (135:4) -- (60:8);
				\draw[dashed] (8,0) arc (0:180:8cm);
				\bline{-8,0}{8,0}{2}
				\draw[fill=black] (P) circle [radius=20pt];
			\ }
		}\ 
		=A^{2}\mathord{
			\ \tikz[baseline=-.6ex, scale=.1, yshift=-2cm]{
				\coordinate (P) at (0,0);
				\draw[very thick, red, ->-={.7}{}] (P) -- (120:8);
				\draw[very thick, red, ->-={.7}{}] (P) -- (60:8);
				\draw[dashed] (8,0) arc (0:180:8cm);
				\bline{-8,0}{8,0}{2}
				\draw[fill=black] (P) circle [radius=20pt];
			\ }
		},
		&&\mathord{
			\ \tikz[baseline=-.6ex, scale=.1, yshift=-2cm]{
				\coordinate (P) at (0,0);
				\draw[very thick, red, rounded corners, -<-={.9}{red}] (P) -- (45:4) -- (120:8);
				\draw[very thick, red, rounded corners, overarc, -<-={.9}{red}] (P) -- (135:4) -- (60:8);
				\draw[dashed] (8,0) arc (0:180:8cm);
				\bline{-8,0}{8,0}{2}
				\draw[fill=black] (P) circle [radius=20pt];
			\ }
		}\ 
		=A^{2}\mathord{
			\ \tikz[baseline=-.6ex, scale=.1, yshift=-2cm]{
				\coordinate (P) at (0,0);
				\draw[very thick, red, -<-={.7}{}] (P) -- (120:8);
				\draw[very thick, red, -<-={.7}{}] (P) -- (60:8);
				\draw[dashed] (8,0) arc (0:180:8cm);
				\bline{-8,0}{8,0}{2}
				\draw[fill=black] (P) circle [radius=20pt];
			\ }
		},\label{rel:simultwist}\\
		&\mathord{
			\ \tikz[baseline=-.6ex, scale=.1, yshift=-2cm]{
				\coordinate (P) at (0,0);
				\draw[very thick, red, rounded corners, -<-={.9}{red}] (P) -- (45:4) -- (120:8);
				\draw[very thick, red, rounded corners,overarc, ->-={.9}{red}] (P) -- (135:4) -- (60:8);
				\draw[dashed] (8,0) arc (0:180:8cm);
				\bline{-8,0}{8,0}{2}
				\draw[fill=black] (P) circle [radius=20pt];
			\ }
		}\ 
		=A\mathord{
			\ \tikz[baseline=-.6ex, scale=.1, yshift=-2cm]{
				\coordinate (P) at (0,0);
				\draw[very thick, red, -<-={.7}{}] (P) -- (120:8);
				\draw[very thick, red, ->-={.7}{}] (P) -- (60:8);
				\draw[dashed] (8,0) arc (0:180:8cm);
				\bline{-8,0}{8,0}{2}
				\draw[fill=black] (P) circle [radius=20pt];
			\ }
		},
		&&\mathord{
			\ \tikz[baseline=-.6ex, scale=.1, yshift=-2cm]{
				\coordinate (P) at (0,0);
				\draw[very thick, red, rounded corners, ->-={.9}{red}] (P) -- (45:4) -- (120:8);
				\draw[very thick, red, rounded corners,overarc, -<-={.9}{red}] (P) -- (135:4) -- (60:8);
				\draw[dashed] (8,0) arc (0:180:8cm);
				\bline{-8,0}{8,0}{2}
				\draw[fill=black] (P) circle [radius=20pt];
			\ }
		}\ 
		=A\mathord{
			\ \tikz[baseline=-.6ex, scale=.1, yshift=-2cm]{
				\coordinate (P) at (0,0);
				\draw[very thick, red, ->-={.7}{}] (P) -- (120:8);
				\draw[very thick, red, -<-={.7}{}] (P) -- (60:8);
				\draw[dashed] (8,0) arc (0:180:8cm);
				\bline{-8,0}{8,0}{2}
				\draw[fill=black] (P) circle [radius=20pt];
			\ }
		}.\label{rel:antisimultwist}
	\end{align}
	From the above relations, we can obtain more general formula:
	\begin{align}
		\mathord{
			\ \tikz[baseline=-.6ex, scale=.1, yshift=-2cm]{
				\coordinate (P) at (0,0);
				\draw[line width=2.6, red] (P) -- (30:7);
				\draw[line width=2.6, red] (P) -- (90:7);
				\draw[line width=2.6, red] (P) -- (150:7);
				\draw[dashed] (7,0) arc (0:180:7cm);
				\bline{-7,0}{7,0}{2}
				\draw[fill=black] (P) circle [radius=20pt];
				\node at (30:7) [right, red]{\scriptsize $(e,f)$};
				\node at (90:7) [above, red]{\scriptsize $(c,d)$};
				\node at (150:7) [left, red]{\scriptsize $(a,b)$};
			\ }
		}
		=A^{ac+bd-ec-fd}A^{(ad+bc-ec-fd)/2}
		\mathord{
			\ \tikz[baseline=-.6ex, scale=.1, yshift=-2cm]{
				\coordinate (P) at (0,0);
				\draw[line width=2.6, red] (P) -- (30:7);
				\draw[line width=2.6, red] ($(P)+(90:2)$) -- (90:7);
				\draw[line width=2.6, red] (P) -- (150:7);
				\draw[dashed] (7,0) arc (0:180:7cm);
				\bline{-7,0}{7,0}{2}
				\draw[fill=black] (P) circle [radius=20pt];
				\node at (30:7) [right, red]{\scriptsize $(e,f)$};
				\node at (90:7) [above, red]{\scriptsize $(c,d)$};
				\node at (150:7) [left, red]{\scriptsize $(a,b)$};
			\ }
		}.
	\end{align}

	For any tangled trivalent graph $G$ with no interior crossings, the \emph{Weyl ordering} $[G]$ is the $\mathfrak{sl}_{3}$-web obtained by replacing all the crossings on $\bM$ with the simultaneous crossings. 
	One can represent the Weyl ordering $[G]$ of $G$ by a \emph{flat trivalent graph}, \emph{i.e.}, a uni-trivalent graph such that its univalent vertices lie in $\bM$ and the other part is embedded into $\Sigma\setminus\partial\Sigma$.
\end{dfn}

\begin{dfn}[basis webs]\label{def:basisweb}
	Let $[G]$ be a flat trivalent graph.
	A polygon $P$ in $\Sigma\setminus [G]$ is an \emph{elliptic face} if $P$ is one of the shaded faces shown in \cref{fig:elliptic}.
	A flat trivalent graph $[G]$ is \emph{elliptic} if $\Sigma\setminus[G]$ has an elliptic face.
	A \emph{basis web} is an $\mathfrak{sl}_{3}$-web represented by non-elliptic flat trivalent graph $[G]$ on $(\Sigma,\bM)$.
	We call the set $\Bweb{\Sigma}$ of basis webs the \emph{graphical basis}.
\end{dfn}

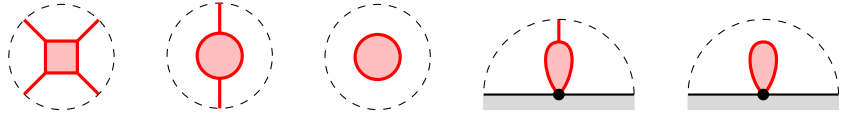
\begin{figure}
	\begin{tikzpicture}[scale=.1]
		\draw[dashed, fill=white] (0,0) circle [radius=7];
		\draw[very thick, red] (-45:7) -- (-45:3);
		\draw[very thick, red] (-135:7) -- (-135:3);
		\draw[very thick, red] (45:7) -- (45:3);
		\draw[very thick, red] (135:7) -- (135:3);
		\draw[very thick, red, fill=pink] (45:3) -- (135:3) -- (225:3) -- (315:3) -- cycle;
	\end{tikzpicture}
	\hspace{1em}
	\begin{tikzpicture}[scale=.1]
		\draw[dashed, fill=white] (0,0) circle [radius=7];
		\draw[very thick, red] (0,-7) -- (0,-3);
		\draw[very thick, red] (0,3) -- (0,7);
		\draw[very thick, red, fill=pink] (0,0) circle [radius=3];
	\end{tikzpicture}
	\hspace{1em}
	\begin{tikzpicture}[scale=.1]
		\draw[dashed, fill=white] (0,0) circle [radius=7];
		\draw[very thick, red, fill=pink] (0,0) circle [radius=3];
	\end{tikzpicture}
	\hspace{1em}
	\begin{tikzpicture}[scale=.1]
		\coordinate (P) at (0,0) {};
		\coordinate (C) at (90:7) {};
		\draw[very thick, red, fill=pink] (P) to[out=north west, in=west] (C) to[out=east, in=north east] (P);
		\draw[very thick, red] (C) -- (90:10);
		\draw[dashed] (10,0) arc (0:180:10cm);
		\bline{-10,0}{10,0}{2}
		\draw[fill=black] (P) circle [radius=20pt];
	\end{tikzpicture}
	\hspace{1em}
	\begin{tikzpicture}[scale=.1]
		\coordinate (P) at (0,0) {};
		\coordinate (C) at (90:7) {};
		\draw[very thick, red, fill=pink] (P) to[out=north west, in=west] (C) to[out=east, in=north east] (P);
		\draw[dashed] (10,0) arc (0:180:10cm);
		\bline{-10,0}{10,0}{2}
		\draw[fill=black] (P) circle [radius=20pt];
	\end{tikzpicture}
	\caption{elliptic faces}
	\label{fig:elliptic}
\end{figure}

We will see that $\Bweb{\Sigma}$ is indeed a free $\bZ_A$-basis of $\SK{\Sigma}$ soon below. 
The following notion provides a useful tool to study the webs.

\begin{dfn}[cut-paths~\cite{Kuperberg96}]\label{def:cut-path}
	Let $G$ be a flat trivalent graph on $(\Sigma,\bM)$ and $p,q\in\partial^{\times}\Sigma$ distinct points.
	\begin{enumerate}
		\item A \emph{cut-path} from $p$ to $q$ of $G$ is a properly embedded oriented interval from $p$ to $q$ which transversely intersects with edges of $G$.
		An \emph{identity move} of a cut-path $\alpha$ with respect to $G$ is a deformation of $\alpha$ into another cut-path $\alpha'$ such that $\alpha$ and $\alpha'$ bound only one biangle which cuts out a subgraph of $G$ consisting of an identity braid between two edges of the biangle. In a similar way, an \emph{$H$-move} is defined, for which the biangle cuts out an $H$-web. See \cref{fig:H-move}.
		\item The \emph{weight $\mathrm{wt}_{\alpha}(G)=(k,l)\in\bN\times\bN$} of $G$ for a cut-path $\alpha$ is defined by
			\begin{align*}
				k:=\#\{p\in G\cap\alpha\mid (G,\alpha)_{p}=1\} \quad \mbox{and} \quad l:=\#\{p\in G\cap\alpha\mid (G,\alpha)_{p}=-1\},
			\end{align*}
		where $(G,\alpha)_{p}$ is the local intersection number of $G$ and $\alpha$ at $p$: see \cref{fig:cutpath}. Let $|\mathrm{wt}_{\alpha}(G)|:=k+l$ denote the total number of intersection points of $\alpha$ and $G$.
		\item A cut-path $\alpha$ from $p$ to $q$ is said to be \emph{minimal} for $G$ if $\mathrm{wt}_{\alpha}(G)$ is minimal in the set of cut-paths homotopic to $\alpha$ rel.~to endpoints, with respect to the partial order on the weight lattice of $\mathfrak{sl}_3$ given by
			\[
				(k,l)\succeq (k+1,l-2),\quad (k,l)\succeq (k-2,l+1).
			\]
		\item A cut-path $\alpha$ from $p$ to $q$ of $G$ is \emph{non-convex to the left (resp. right) side} if $\mathrm{wt}_{\alpha}(G)\preceq\mathrm{wt}_{\beta}(G)$ for any cut-path $\beta\subset \Sigma\setminus\alpha$ from $p$ to $q$ of $G$ such that $\alpha\cup\beta$ bounds a biangle, and $\beta$ lies in the left side (resp. right side) of $\alpha$.
	\end{enumerate}
\end{dfn}

Kuperberg proved some lemmas about cut-paths described below.
\begin{lem}[Kuperberg~{\cite[Lemma~6.5, 6.6,]{Kuperberg96}}]\label{lem:Kuperberg-lemma}
	Let $\Sigma$ be an unpunctured marked surface, $p,q\in\partial^{\times}\Sigma$ distinct points, and $G$ a non-elliptic flat trivalent graph.
	\begin{enumerate}
		\item If $\alpha$ and $\beta$ are homotopic (rel.~to endpoints) cut-paths from $p$ to $q$ of $G$ and $\alpha$ is minimal, then $\mathrm{wt}_{\alpha}(G)\preceq \mathrm{wt}_{\beta}(G)$. If $\beta$ is also minimal, then $\alpha$ is related to $\beta$ by a finite sequence of $H$-moves and identity moves.
		\item If a cut-path $\alpha$ from $p$ to $q$ of $G$ is non-convex to the left side (resp. right side), there exists a unique class $\alpha_{L}$ (resp. $\alpha_{R}$) of cut-paths from $p$ to $q$ under identity moves such that any cut-path $\beta$ from $p$ to $q$ with $\mathrm{wt}_{\beta}(G)=\mathrm{wt}_{\alpha}(G)$ in the left side (resp. right side) of $\alpha$ lies between $\alpha_{L}$ (resp. $\alpha_{R}$) and $\alpha$.
	\end{enumerate}
\end{lem}

We call the above cut-path $\alpha_{L}$ (resp.~$\alpha_{R}$) the \emph{left (resp. right) core} of $\alpha$: see \cref{fig:cutpath}.

\begin{figure}
	\begin{tikzpicture}[scale=.1]
		\coordinate (P) at (0,23);
		\coordinate (A3) at (-8,16);
		\coordinate (A2) at (-8,12);
		\coordinate (A1) at (-8,4);
		\coordinate (B3) at (8,16);
		\coordinate (B2) at (8,12);
		\coordinate (B1) at (8,4);
		\coordinate (Q) at (0,-3);
		\draw[very thick, red] (A1) -- (B1);
		\draw[very thick, red] (A2) -- (B2);
		\draw[very thick, red] (A3) -- (B3);
		\draw[rounded corners, ->-] (Q) -- ($(A1)+(0,-2)$) -- ($(A3)+(0,2)$) -- (P);
		\draw[rounded corners, ->-] (Q) -- ($(B1)+(0,-2)$) -- ($(B3)+(0,2)$) -- (P);
		\draw (P) -- ($(P)+(0,3)$);
		\draw (Q) -- ($(Q)+(0,-3)$);
		\node at (0,8) [rotate=90, red]{$\cdots$};
		\node at (-8,8) [left]{$\alpha$};
		\node at (8,8) [right]{$\alpha'$};
	\end{tikzpicture}
	\begin{tikzpicture}[scale=.1]
		\coordinate (P) at (0,23);
		\coordinate (A3) at (-8,16);
		\coordinate (A1) at (-8,4);
		\coordinate (B3) at (8,16);
		\coordinate (B1) at (8,4);
		\coordinate (Q) at (0,-3);
		\draw[very thick, red] (A1) -- (B1);
		\draw[very thick, red] (A3) -- (B3);
		\draw[very thick, red] ($(A1)!.5!(B1)$) -- ($(A3)!.5!(B3)$);
		\draw[rounded corners, ->-] (Q) -- ($(A1)+(0,-2)$) -- ($(A3)+(0,2)$) -- (P);
		\draw[rounded corners, ->-] (Q) -- ($(B1)+(0,-2)$) -- ($(B3)+(0,2)$) -- (P);
		\draw (P) -- ($(P)+(0,3)$);
		\draw (Q) -- ($(Q)+(0,-3)$);
		\node at (-8,8) [left]{$\alpha$};
		\node at (8,8) [right]{$\alpha'$};
	\end{tikzpicture}
	\caption{An identity move (left) and an $H$-move (right). The other parts of $\alpha$ and $\alpha'$ are identical.}
	\label{fig:H-move}
\end{figure}
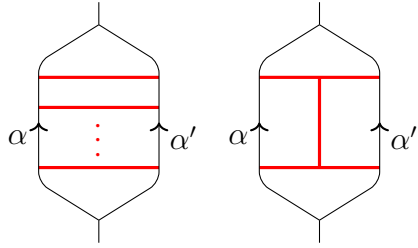

\begin{figure}
	\begin{tikzpicture}[scale=.2]
		\coordinate (A) at (60:5);
		\coordinate (B) at (180:5);
		\coordinate (C) at (300:5);
		\foreach \i in {1,2,3,4,5,6}{
			\coordinate (A\i) at ($(A)+(60*\i:5)$);
			\coordinate (B\i) at ($(B)+(60*\i:5)$);
			\coordinate (C\i) at ($(C)+(60*\i:5)$);
		}
		\foreach \i in {1,2,3,4,5,6}{
			\coordinate (L\i) at ($(-15,5*\i-15)$);
		}
		\draw[very thick, red, ->-] (A1) -- (A2);
		\draw[very thick, red, -<-] (A2) -- (A3);
		\draw[very thick, red, ->-] (A3) -- (A4);
		\draw[very thick, red, -<-] (A4) -- (A5);
		\draw[very thick, red, ->-] (A5) -- (A6);
		\draw[very thick, red, -<-] (A6) -- (A1);
		\draw[very thick, red, ->-] (B1) -- (B2);
		\draw[very thick, red, -<-] (B2) -- (B3);
		\draw[very thick, red, ->-] (B3) -- (B4);
		\draw[very thick, red, -<-] (B4) -- (B5);
		\draw[very thick, red, ->-] (B5) -- (B6);
		\draw[very thick, red, ->-] (C3) -- (C4);
		\draw[very thick, red, -<-] (C4) -- (C5);
		\draw[very thick, red, ->-] (C5) -- (C6);
		\draw[very thick, red, -<-] (C6) -- (C1);
		\draw[very thick, red, ->-] (A1) -- (15,10);
		\draw[very thick, red, -<-] (A6) -- (15,5);
		\draw[very thick, red, -<-] (C6) -- (15,-5);
		\draw[very thick, red, ->-] (C5) -- (15,-10);
		\draw[very thick, red, -<-] (A2) -- ($(A2)+(-10,3)$);
		\draw[very thick, red, -<-] (B2) -- (L4);
		\draw[very thick, red, ->-] (B3) -- (L3);
		\draw[very thick, red, -<-] (B4) -- (L2);
		\draw[very thick, red, -<-] (C4) -- (L1);
		\draw[very thick, red, ->-] ($(A2)+(-10,3)$) -- (L5);
		\draw[very thick, red, ->-] ($(A2)+(-10,3)$) -- (L6);
		\coordinate (Q) at ($(A2)+(0,7)$);
		\coordinate (P) at ($(C4)+(0,-5)$);
		\draw[->-={.9}{}, rounded corners] (P) -- ($(B)+(+1,-8)$) -- ($(B)+(+1,0)$) -- ($(A)+(-1,0)$) -- ($(A)+(-1,+5)$) -- (Q);
		\draw[->-={.9}{}, rounded corners] (P) -- ($(B)+(-1,-8)$) -- ($(B)+(-1,10)$) -- (Q);
		\draw[->-={.2}{}, rounded corners] (P) -- (C) -- ($(A)+(+1,0)$) -- ($(A)+(+1,5)$) -- (Q);
		\fill (P) circle [radius=10pt];
		\fill (Q) circle [radius=10pt];
		\node at (P) [right]{$p$};
		\node at (Q) [right]{$q$};
		\node at ($(A)+(-2,1)$) {$\alpha$};
		\node at ($(B)+(-1,8)$) [left]{$\alpha_{L}$};
		\node at ($(C)$) [right]{$\beta$};
	\end{tikzpicture}
	\caption{The curves $\alpha$, $\alpha_L$, and $\beta$ are homotopic cut-paths of $G$ such that $\mathrm{wt}_{\alpha}(G)=(2,2)$, $\mathrm{wt}_{\alpha_{L}}(G)=(2,2)$, and $\mathrm{wt}_{\beta}(G)=(0,3)$. The cut-path $\alpha$ is non-convex to the left side but convex to the right side, and $\alpha_L$ is the left core of $\alpha$. The cut-paths $\alpha$ and $\alpha_{L}$ are related by an \emph{$H$-move}. The cut-path $\beta$ is minimal.}
	\label{fig:cutpath}
\end{figure}
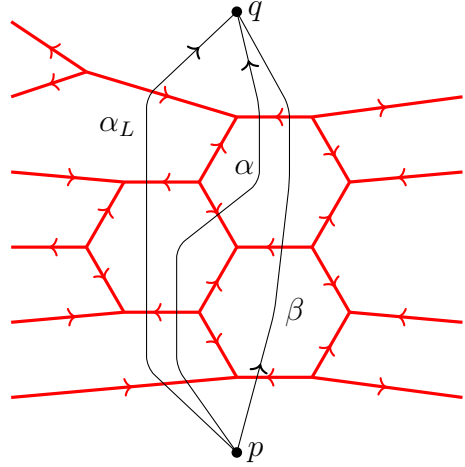

\begin{prop}
	The skein algebra $\SK{\Sigma}$ is a free $\bZ_{A}$-module generated by $\Bweb{\Sigma}$.  
\end{prop}
\begin{proof}
	Let us consider a neighborhood of a special point $p\in\bM$.
	For a given tangled trivalent graph $G$ with simultaneous crossings, we can expand univalent vertices at $p\in\bM$ as follows:  
	\begin{align*}
		\mathord{
			\ 
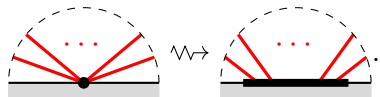
\tikz[baseline=-.6ex, scale=.1, yshift=-4cm]{
				\coordinate (P) at (0,0);
				\draw[very thick, red] (P) -- (160:10);
				\draw[very thick, red] (P) -- (140:10);
				\draw[very thick, red] (P) -- (40:10);
				\draw[very thick, red] (P) -- (20:10);
				\draw[dashed] (10,0) arc (0:180:10cm);
				\bline{-10,0}{10,0}{2}
				\draw[fill=black] (P) circle [radius=20pt];
				\node at (90:5) [red]{$\cdots$};
			\ }
		}
		\mathord{
			\ \tikz[baseline=-.6ex, scale=.05]{
			\draw[->, decorate, decoration={zigzag,segment length=1.5mm, post length=1.5mm}] (-5,0) -- (5,0);
			}
		}
		\mathord{
			\ \tikz[baseline=-.6ex, scale=.1, yshift=-4cm]{
				\coordinate (P) at (0,0) {};
				\draw[very thick, red] ($(P)+(-5,0)$) -- (160:10);
				\draw[very thick, red] ($(P)+(-3,0)$) -- (140:10);
				\draw[very thick, red] ($(P)+(3,0)$) -- (40:10);
				\draw[very thick, red] ($(P)+(5,0)$) -- (20:10);
				\draw[dashed] (10,0) arc (0:180:10cm);
				\bline{-10,0}{10,0}{2}
				\draw[line width=3pt] (-7,0) -- (7,0);
				\node at (90:5) [red]{$\cdots$};
			\ }
		}.
	\end{align*}
	The special point $p$ is replaced by an interval $I_p$ containing all expanded univalent vertices at $p$.
	For any tangled trivalent graph, one can obtain a tangled trivalent graph with simultaneous crossings by \eqref{rel:simultaneous} and expand it by applying the above deformation.
	The boundary skein relations \eqref{rel:pbigon}, \eqref{rel:degbigon}, and \eqref{rel:pcircle} are described as
	\begin{align}
		\mathord{
			\ \tikz[baseline=-.6ex, scale=.1, yshift=-4cm]{
				\draw[very thick, red, ->-] (3,0) -- (3,5);
				\draw[very thick, red, -<-] (-3,0) -- (-3,5);
				\draw[very thick, red, -<-] (3,5) -- (3,10);
				\draw[very thick, red, ->-] (-3,5) -- (-3,10);
				\draw[very thick, red, ->-] (-3,5) -- (3,5);
				\bline{-10,0}{10,0}{2}
				\draw[line width=3pt] (-7,0) -- (7,0);
			\ }
		}
		&=
		\mathord{
			\ \tikz[baseline=-.6ex, scale=.1, yshift=-4cm]{
				\draw[very thick, red, -<-] (3,0) -- (3,10);
				\draw[very thick, red, ->-] (-3,0) -- (-3,10);
				\bline{-10,0}{10,0}{2}
				\draw[line width=3pt] (-7,0) -- (7,0);
			\ }
		},&
		\mathord{
			\ \tikz[baseline=-.6ex, scale=.1, yshift=-4cm]{
				\draw[very thick, red, -<-] (3,0) -- (3,5);
				\draw[very thick, red, ->-] (-3,0) -- (-3,5);
				\draw[very thick, red, ->-] (3,5) -- (3,10);
				\draw[very thick, red, -<-] (-3,5) -- (-3,10);
				\draw[very thick, red, -<-] (-3,5) -- (3,5);
				\bline{-10,0}{10,0}{2}
				\draw[line width=3pt] (-7,0) -- (7,0);
			\ }
		}
		&=
		\mathord{
			\ \tikz[baseline=-.6ex, scale=.1, yshift=-4cm]{
				\draw[very thick, red, ->-] (3,0) -- (3,10);
				\draw[very thick, red, -<-] (-3,0) -- (-3,10);
				\bline{-10,0}{10,0}{2}
				\draw[line width=3pt] (-7,0) -- (7,0);
			\ }
		},\label{rel:H_equivalence}\\
		\mathord{
			\ \tikz[baseline=-.6ex, scale=.1, yshift=-4cm]{
				\draw[very thick, red, ->-] (3,0) -- (0,5);
				\draw[very thick, red, ->-] (-3,0) -- (0,5);
				\draw[very thick, red, -<-] (0,5) -- (0,10);
				\bline{-10,0}{10,0}{2}
				\draw[line width=3pt] (-7,0) -- (7,0);
			\ }
		}
		&=0,&
		\mathord{
			\ \tikz[baseline=-.6ex, scale=.1, yshift=-4cm]{
				\draw[very thick, red, -<-] (3,0) -- (0,5);
				\draw[very thick, red, -<-] (-3,0) -- (0,5);
				\draw[very thick, red, ->-] (0,5) -- (0,10);
				\bline{-10,0}{10,0}{2}
				\draw[line width=3pt] (-7,0) -- (7,0);
			\ }
		}
		&=0,\label{rel:Y_elimination}\\
		\mathord{
			\ \tikz[baseline=-.6ex, scale=.1, yshift=-4cm]{
				\draw[very thick, red, ->-] (3,0) to[out=north, in=east] (0,5) to[out=west,in=north] (-3,0);
				\bline{-10,0}{10,0}{2}
				\draw[line width=3pt] (-7,0) -- (7,0);
			\ }
		}
		&=0,&
		\mathord{
			\ \tikz[baseline=-.6ex, scale=.1, yshift=-4cm]{
				\draw[very thick, red, -<-] (3,0) to[out=north, in=east] (0,5) to[out=west,in=north] (-3,0);
				\bline{-10,0}{10,0}{2}
				\draw[line width=3pt] (-7,0) -- (7,0);
			\ }
		}
		&=0.\label{rel:U_elimination}
	\end{align}
	Let us give an orientation, induced from the orientation of $\Sigma$, to the union of the intervals $\{I_p\mid p\in\bM\}$.
	For a map $f\colon\bM\to\bN\times\bN$, we define $B_{\mathrm{irr}}(f)$ to be the set of boundary-fixing isotopy classes of embeddings of trivalent graphs satisfying the following conditions:
	\begin{itemize}
		\item $G\in B_{\mathrm{irr}}(f)$ has distinct univalent vertices lying in $\cup_{p\in\bM}I_p$ and $\gr(G)=f$,
		\item For each interval $I_p$, outgoing univalent vertices follow incoming univalent vertices with respect to the orientation on $I_p$,
		\item $G$ is non-elliptic, and $I_p$ is a cut-path of $G$ which is non-convex to the left side for all $p\in\bM$. 
	\end{itemize}
	One can use the confluence theory of Sikora--Westbury~\cite{SikoraWestbury07}.
	It shows that $\cup_{f}B_{\mathrm{irr}}(f)$ gives a basis of $\SK{\Sigma}$ as a $\bZ_{A}$-module.
	To apply \cite[Theorem~2.3]{SikoraWestbury07}, we consider ``reductions'' from the left-hand side of the first equation in \eqref{rel:H_equivalence}, \eqref{rel:Y_elimination}, \eqref{rel:U_elimination} to their right-hand sides, and a reduction of pairs of parallel arcs or loops that are oriented inconsistently with bounding rectangles or annuli (see ``British highways'' in \cite[Corollary~5]{FrohmanSikora20}).
	Adding these reductions to the reductions of $\mathfrak{sl}_3$-webs in \cite[Chapter~5]{SikoraWestbury07}, we conclude that $B_{\mathrm{irr}}(f)$ gives a basis of $\SK{\Sigma}$ by shrinking $I_p$ to $p$ for all $p\in\bM$.
\end{proof}

\begin{rem}[Relation to the Frohman--Sikora's skein algebra]\label{rem:relation_FS}
	From the above relations \eqref{rel:H_equivalence}, \eqref{rel:Y_elimination}, and \eqref{rel:U_elimination}, one can see that our skein algebra is identified with the one in Frohman--Sikora~\cite{FrohmanSikora20} with $A=q^{-1/3}$ and $a=1$. 
	The variable $a$ is a coefficient related to \eqref{rel:H_equivalence}. 
	We also remark that there is a difference in the basis webs. Our basis webs are a modification of Frohman--Sikora's basis webs by $A^{n/2}$. 
\end{rem}
 The above remark and the result of Frohman--Sikora say the following:
\begin{thm}[Frohman--Sikora~{\cite[Theorem~7]{FrohmanSikora20}}]
	The skein algebra $\SK{\Sigma}$ is finitely generated.
\end{thm}

\subsubsection{Basic structures on the skein algebra $\SK{\Sigma}$}\label{subsub:basic_structures}
The skein algebra $\SK{\Sigma}$ has the following basic structures, which will be compared with the corresponding structures on the quantum cluster algebra in \cref{sect:correspondance}.

\paragraph{\textbf{The mirror-reflection}}
The \emph{mirror-reflection} $G^{\dagger}$ of a tangled trivalent graph $G$ is defined by reversing the ordering of the univalent vertices on each special point and exchanging the over-/under-passing information at each internal crossing.
The mirror-reflection is extended to an anti-involution  $\dagger:\SK{\Sigma} \to \SK{\Sigma}$ by $\bZ$-linearly and by setting $(A^{\pm 1/2})^{\dagger}:=A^{\mp 1/2}$. 

\paragraph{\textbf{Some group actions}}
The group $MC(\Sigma) \times \mathrm{Out}(SL_3)$ acts on $\SK{\Sigma}$ from the right as follows.
\begin{itemize}
    \item Each mapping class $\phi$ acts on $\SK{\Sigma}$ by sending each web $[G]$ to $[\phi^{-1}(G)]$. It is well-defined since $\phi$ preserves the set $\bM$ and respects the defining relations.
    \item The Dynkin involution $\ast \in \mathrm{Out}(SL_3)$ acts on $\SK{\Sigma}$ as an $\bZ_{A}$-algebra involution by reversing the orientation of each edge of a basis web.
\end{itemize}

\paragraph{\textbf{The endpoint grading}}
The skein algebra has the following grading. 

\begin{dfn}[the endpoint grading]\label{def:egrading}
The \emph{endpoint degree} 
	\begin{align*}
	    \mathrm{gr}=(\mathrm{gr}_p)_{p \in \bM} \colon\Bweb{\Sigma}\to(\bN\times\bN)^{\bM}
	\end{align*}
is defined as follows.
For a basis web $[G]\in\Bweb{\Sigma}$ and $p\in\bM$, the first (resp. second) entry of $\mathrm{gr}_{p}(G) \in \bN \times \bN$ is the number of incoming (resp. outgoing) edges of $G$ incident to $p$. The resulting grading on the skein algebra is called the \emph{endpoint grading}. 
\end{dfn}
We define a non-negative bi-grading by the sum
\[
	\vec{\gr}(G)=\sum_{p\in\bM}\gr_{p}(G)
\]
and an augmentation map $\varepsilon_{\gr}(G)=k-l$
for $\vec{\gr}(G)=(k,l)$. 
Note that the skein relations (\cref{def:skeinrel}) and the boundary skein relations (\cref{def:bskeinrel}) are homogeneous with respect to $\mathrm{gr}$. 
Hence the skein algebra $\SK{\Sigma}=\bigoplus_{(k,l)\in\bN\times\bN}\left(\SK{\Sigma}\right)_{(k,l)}$ is a bi-graded algebra with respect to $\vec{\gr}$. 
The augmentation map $\varepsilon_{\gr}$ defines a $\bZ$-valued grading. 
\begin{rem}
	\begin{enumerate}
		\item The endpoint degree at $p\in\bM$ is the weight of a minimal cut-path surrounding $p$.
		\item We have $\varepsilon_{\gr}(G)=3(t_{-}(G)-t_{+}(G))$ for $G\in\Bweb{\Sigma}$.
	\end{enumerate}
\end{rem}
For $q \in \bN$, consider the lattice
\[
		L(q):=\ker \left((\bZ \times \bZ)^\bM \xrightarrow{\mathrm{aug}} \bZ \xrightarrow{\mathrm{mod}_q}\bZ/q\bZ\right).
\]
Here $\mathrm{aug}((k_p,l_p)):=\sum_{p \in \bM}(k_p-l_p)$. 
Then by (2) in the remark above, we have $\mathrm{gr}(G) \in L(3)$ for any $G \in \Bweb{\Sigma}$.

\subsubsection{Elementary webs and web clusters}

Let $\Bweb{\partial^{\times}\Sigma} \subset \Bweb{\Sigma}$ denote the set of \emph{boundary webs} on $\Sigma$, that is, $\mathfrak{sl}_{3}$-webs consisting of oriented boundary intervals of $\Sigma$. The following notions are expected to be skein theoretic incarnations of some concepts in the quantum cluster algebra.

\begin{dfn}[elementary webs]\label{elementary}
	A basis web $G\in\Bweb{\Sigma}$ is called an \emph{elementary web} if there are no basis webs $G_{1}$ and $G_{2}$ such that $G=A^{k}G_{1}G_{2}$ for some $k \in \bZ$.
	We denote the set of elementary webs by $\Eweb{\Sigma} \subset \Bweb{\Sigma}$.
\end{dfn}

\begin{dfn}[web clusters]\label{webcluster}
	A subset $C\subset\Eweb{\Sigma}$ is called a \emph{web cluster} if $C$ is a maximal $A$-commutative subset in $\Eweb{\Sigma}$ with cardinality $\# I_{\mathfrak{sl}_{3}}(\Delta)$.
	We denote the collection of web clusters by $\Cweb{\Sigma}$.
\end{dfn}

\begin{dfn}\label{def:web-exchange}
	For two elementary webs $G_{1},G_{2}\in\Eweb{\Sigma}$ contained in a common web cluster, define $\Pi(G_1,G_2)\in\bZ$ by
	\[
		G_{1}G_{2}=A^{\Pi(G_{1},G_{2})}G_{2}G_{1}.
	\]
\end{dfn}
\begin{rem}[Relation to the Fomin--Pylyavskyy's algebras of $SL_3$-invariants]\label{rem:FP}
Let $\Sigma=\mathbb{D}_{a+b}$ be a $(a+b)$-gon, which is a disk with $a+b$ special points $p_1,\dots,p_{a+b}$ in this clockwise ordering. Choose a \emph{signature} $\sigma_i \in \{\bullet,\circ\}$ in two alphabets $\bullet,\circ$ for $i=1,\dots,a+b$, and consider the subalgebra $\SK{\mathbb{D}_{a+b}}(\sigma) \subset \SK{\mathbb{D}_{a+b}}$ consisting of webs $G$ such that 
\begin{align*}
    \mathrm{gr}_{p_i}(G) \in \begin{cases}
    \bN \times \{0\} & \mbox{if $\sigma_i=\circ$}, \\
    \{0\} \times \bN & \mbox{if $\sigma_i=\bullet$}.
    \end{cases}
\end{align*}
Then by the specialization $A=1$, $\mathscr{S}_{\mathfrak{sl}_3,\mathbb{D}_{a+b}}^1(\sigma)$ reproduces the algebra $R_\sigma(V)$ in \cite{FP16}. For any unpunctured marked surface $\Sigma$, they also introduced a variant where each special point carries both the black and white signature in \cite[Section 12]{FP14}. This exactly corresponds to the entire algebra $\mathscr{S}_{\mathfrak{sl}_3,\Sigma}^1$. 
\end{rem}

\subsection{The \texorpdfstring{$\mathfrak{sl}_3$}{sl(3)}-skein algebra for a triangle}\label{subsec:triangle_skein}
Let us consider a triangle $T$ with special points $p_{1},p_{2},p_{3}$ and the unique triangulation $\Delta_{T}$.
The boundary web $e_{ij}$ is defined to be the simple oriented arc from $p_{i}$ to $p_{j}$.
In this case, we have 
\begin{align*}
    \Bweb{\partial^{\times}T}=\{e_{12},e_{21},e_{23},e_{32},e_{13},e_{31}\}.
\end{align*}
The $\mathfrak{sl}_3$-web $t_{{123}}^{+}$ (resp. $t_{{123}}^{-}$) is defined to be the flat trivalent graph with a trivalent sink (resp. source) vertex and univalent vertices $p_{1},p_{2},p_{3}$.
Note that $\ast(t_{{123}}^{+})=t_{{123}}^{-}$. See \cref{fig:triEweb}.

\begin{figure}
\centering
		\begin{tikzpicture}[scale=.1]
			\node [draw, fill=black, circle, inner sep=1.5] (C) at (90:12) {};
			\node [draw, fill=black, circle, inner sep=1.5] (A) at (210:12) {} edge[blue] (C);
			\node [draw, fill=black, circle, inner sep=1.5] (B) at (-30:12) {} edge[blue] (C) edge[blue] (A);
			\draw[->-, very thick, red] (A) -- (B);
			\path (A) -- (B) node [pos=1/2, below=3pt] {$e_{12}$};
			\node at (A) [left] {\scriptsize $p_{1}$};
			\node at (B) [right] {\scriptsize $p_{2}$};
			\node at (C) [above] {\scriptsize $p_{3}$};
		\end{tikzpicture}
		\hspace{1em}
		\begin{tikzpicture}[scale=.1]
			\node [draw, fill=black, circle, inner sep=1.5] (C) at (90:12) {};
			\node [draw, fill=black, circle, inner sep=1.5] (A) at (210:12) {} edge[blue] (C);
			\node [draw, fill=black, circle, inner sep=1.5] (B) at (-30:12) {} edge[blue] (C) edge[blue] (A);
			\draw[->-, very thick, red] (B) -- (A);
			\path (A) -- (B) node [pos=1/2, below=3pt] {$e_{21}$};	
			\node at (A) [left] {\scriptsize $p_{1}$};
			\node at (B) [right] {\scriptsize $p_{2}$};
			\node at (C) [above] {\scriptsize $p_{3}$};
		\end{tikzpicture}
		\hspace{1em}
		\begin{tikzpicture}[scale=.1]
			\node [draw, fill=black, circle, inner sep=1.5] (C) at (90:12) {};
			\node [draw, fill=black, circle, inner sep=1.5] (A) at (210:12) {} edge[blue] (C);
			\node [draw, fill=black, circle, inner sep=1.5] (B) at (-30:12) {} edge[blue] (C) edge[blue] (A);
			\draw[->-, very thick, red] (B) -- (C);
			\path (A) -- (B) node [pos=1/2, below=3pt] {$e_{23}$};
			\node at (A) [left] {\scriptsize $p_{1}$};
			\node at (B) [right] {\scriptsize $p_{2}$};
			\node at (C) [above] {\scriptsize $p_{3}$};
		\end{tikzpicture}
		\hspace{1em}
		\begin{tikzpicture}[scale=.1]
			\node [draw, fill=black, circle, inner sep=1.5] (C) at (90:12) {};
			\node [draw, fill=black, circle, inner sep=1.5] (A) at (210:12) {} edge[blue] (C);
			\node [draw, fill=black, circle, inner sep=1.5] (B) at (-30:12) {} edge[blue] (C) edge[blue] (A);
			\draw[->-, very thick, red] (C) -- (B);
			\path (A) -- (B) node [pos=1/2, below=3pt] {$e_{32}$};
			\node at (A) [left] {\scriptsize $p_{1}$};
			\node at (B) [right] {\scriptsize $p_{2}$};
			\node at (C) [above] {\scriptsize $p_{3}$};
		\end{tikzpicture}
		\hspace{1em}
		\begin{tikzpicture}[scale=.1]
			\node [draw, fill=black, circle, inner sep=1.5] (C) at (90:12) {};
			\node [draw, fill=black, circle, inner sep=1.5] (A) at (210:12) {} edge[blue] (C);
			\node [draw, fill=black, circle, inner sep=1.5] (B) at (-30:12) {} edge[blue] (C) edge[blue] (A);
			\draw[->-, very thick, red] (C) -- (A);
			\path (A) -- (B) node [pos=1/2, below=3pt] {$e_{31}$};
			\node at (A) [left] {\scriptsize $p_{1}$};
			\node at (B) [right] {\scriptsize $p_{2}$};
			\node at (C) [above] {\scriptsize $p_{3}$};
		\end{tikzpicture}
		\hspace{1em}
		\begin{tikzpicture}[scale=.1]
			\node [draw, fill=black, circle, inner sep=1.5] (C) at (90:12) {};
			\node [draw, fill=black, circle, inner sep=1.5] (A) at (210:12) {} edge[blue] (C);
			\node [draw, fill=black, circle, inner sep=1.5] (B) at (-30:12) {} edge[blue] (C) edge[blue] (A);
			\draw[->-, very thick, red] (A) -- (C);
			\path (A) -- (B) node [pos=1/2, below=3pt] {$e_{13}$};
			\node at (A) [left] {\scriptsize $p_{1}$};
			\node at (B) [right] {\scriptsize $p_{2}$};
			\node at (C) [above] {\scriptsize $p_{3}$};
		\end{tikzpicture}
		\hspace{1em}
		\begin{tikzpicture}[scale=.1]
			\node [draw, fill=black, circle, inner sep=1.5] (A) at (90:12) {};
			\node [draw, fill=black, circle, inner sep=1.5] (B) at (210:12) {} edge[blue] (A);
			\node [draw, fill=black, circle, inner sep=1.5] (C) at (-30:12) {} edge[blue] (A) edge[blue] (B);
			\coordinate (D) at (0,0);
			\draw[->-, very thick, red] (A) -- (D);
			\draw[->-, very thick, red] (B) -- (D);
			\draw[->-, very thick, red] (C) -- (D);
			\path (B) -- (C) node [pos=1/2, below]{$t_{{123}}^{+}$};	
			\node at (B) [left] {\scriptsize $p_{1}$};
			\node at (C) [right] {\scriptsize $p_{2}$};
			\node at (A) [above] {\scriptsize $p_{3}$};
		\end{tikzpicture}
		\hspace{1em}
		\begin{tikzpicture}[scale=.1]
			\node [draw, fill=black, circle, inner sep=1.5] (A) at (90:12) {};
			\node [draw, fill=black, circle, inner sep=1.5] (B) at (210:12) {} edge[blue] (A);
			\node [draw, fill=black, circle, inner sep=1.5] (C) at (-30:12) {} edge[blue] (A) edge[blue] (B);
			\coordinate (D) at (0,0);
			\draw[-<-, very thick, red,] (A) -- (D);
			\draw[-<-, very thick, red,] (B) -- (D);
			\draw[-<-, very thick, red,] (C) -- (D);
			\path (B) -- (C) node [pos=1/2, below]{$t_{{123}}^{-}$};	
			\node at (B) [left] {\scriptsize $p_{1}$};
			\node at (C) [right] {\scriptsize $p_{2}$};
			\node at (A) [above] {\scriptsize $p_{3}$};
		\end{tikzpicture}
\caption{Elementary webs in the triangle}\label{fig:triEweb}	
\end{figure}
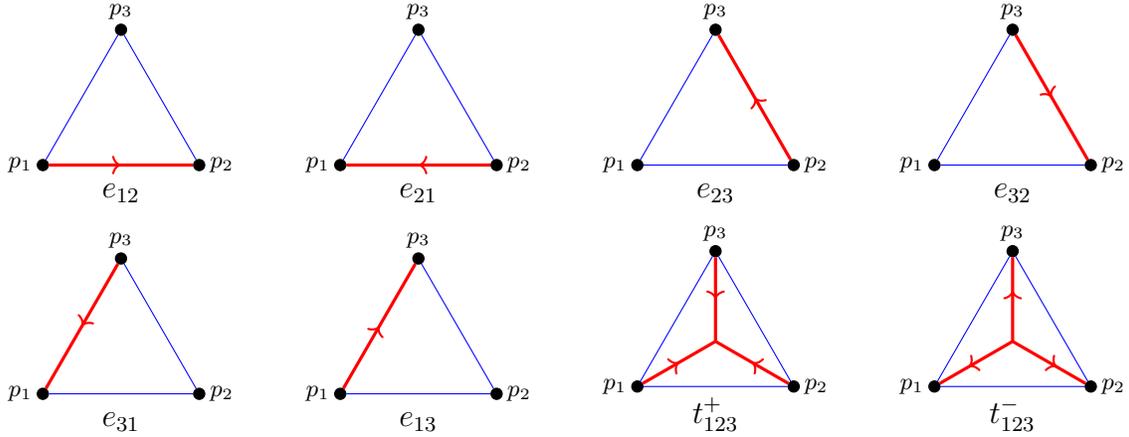

\begin{lem}\label{lem:str const T}
	The complete list of relations among $\Bweb{\partial^{\times}T}\cup\{t_{{123}}^{+},t_{{123}}^{-}\}$ in $\SK{T}$ is given as follows.
	\begin{gather*}
		e_{21}e_{12}=e_{12}e_{21},\quad e_{32}e_{23}=e_{23}e_{32},\quad e_{13}e_{31}=e_{31}e_{13},\\
		\begin{aligned}
			e_{21}e_{32}&=A^{-1/2}[e_{32}e_{21}],& e_{12}e_{32}&=A^{-1}[e_{32}e_{12}],\\
			e_{21}e_{23}&=A^{-1}[e_{23}e_{21}],& e_{12}e_{23}&=A^{-1/2}[e_{23}e_{12}],\\
			e_{21}e_{13}&=A^{1/2}[e_{13}e_{21}],& e_{12}e_{13}&=A[e_{13}e_{12}],\\
			e_{21}e_{31}&=A[e_{31}e_{21}],& e_{12}e_{13}&=A^{1/2}[e_{13}e_{12}],\\
			e_{21}t_{{123}}^{+}&=A^{-1/2}[t_{{123}}^{+}e_{21}],& e_{12}t_{{123}}^{+}&=A^{1/2}[t_{{123}}^{+}e_{12}],\\
			e_{21}t_{{123}}^{-}&=A^{1/2}[t_{{123}}^{-}e_{21}],& e_{12}t_{{123}}^{-}&=A^{-1/2}[t_{{123}}^{-}e_{12}],
		\end{aligned}\\
		t_{{123}}^{+}t_{{123}}^{-}=A^{3/2}[e_{21}e_{13}e_{32}]+A^{-3/2}[e_{12}e_{23}e_{31}],\\
		t_{{123}}^{-}t_{{123}}^{+}=A^{-3/2}[e_{21}e_{13}e_{32}]+A^{3/2}[e_{12}e_{23}e_{31}].
	\end{gather*}
\end{lem}
\begin{proof}
	Compute straightforwardly by the skein relation and the boundary skein relation.
	We remark that we do not need to confirm all relations, thanks to the symmetries given by the Dynkin involution, the mirror-reflection, and rotations of the triangle.
\end{proof}

\begin{prop}\label{prop:generators T}
	The skein algebra $\SK{T}$ is generated by $\Bweb{\partial^{\times}T}\cup\{t_{{123}}^{+},t_{{123}}^{-}\}$ as a $\bZ_{A}$-algebra.
\end{prop}
\begin{proof}
	For a flat trivalent graph $[G]$ in $T$, relations \eqref{rel:H_equivalence}--\eqref{rel:U_elimination} can be used to eliminate $4$-, $3$-, and $2$-gons in $\Sigma\setminus [G]$.
	Obviously, these eliminations do not change the representing basis web.
	A diagram of the resulting non-elliptic flat trivalent graph, described in \cref{fig:Bweb T}, was explicitly given by Kim~\cite{Kim07} and Frohman--Sikora~\cite{FrohmanSikora20}.
	Here a strand labeled by a positive integer $m$ means the $m$-parallelization of the strand;
	the white triangle with three strands labeled by $l$ is a \emph{triangle web} defined by 
	\begin{align*}
		\mathord{
			\ 
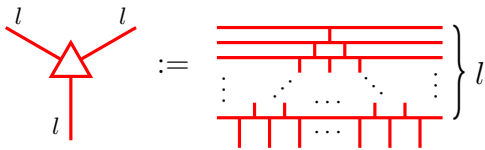
\tikz[baseline=-.6ex, scale=.1]{
				\coordinate (A) at (30:10);
				\coordinate (B) at (150:10);
				\coordinate (C) at (270:10);
				\coordinate (D) at (0,0) {};
				\draw[red, very thick] (A) -- (D);
				\draw[red, very thick] (B) -- (D);
				\draw[red, very thick] (C) -- (D);
				\draw[red, very thick, fill=white] (90:3) -- (210:3) -- (330:3) -- cycle;
				\node at (30:8) [above] {\scriptsize $l$};	
				\node at (150:8) [above] {\scriptsize $l$};	
				\node at (270:8) [left] {\scriptsize $l$};	
			}
		\ }
		:=
		\mathord{
			\ \tikz[baseline=-.6ex, scale=.1]{
				\coordinate (A) at (-15,5);
				\coordinate (B) at (15,5);
				\foreach \i in {0,1,2,6}{
					\draw[red, very thick] ($(A)+(0,-2*\i)$) -- ($(B)+(0,-2*\i)$);
				}
				\draw[red, very thick] ($(A)!0.5!(B)$) -- ($(A)!0.5!(B)+(0,-2)$);
				\foreach \i in {0,1}{
					\draw[red, very thick] ($(A)!0.5!(B)+(-2+4*\i,-2)$) -- ($(A)!0.5!(B)+(-2+4*\i,-4)$);
				}
				\foreach \i in {0,1,2}{
					\draw[red, very thick] ($(A)!0.5!(B)+(-4+4*\i,-4)$) -- ($(A)!0.5!(B)+(-4+4*\i,-6)$);
				}
				\foreach \i in {0,1,4,5}{
					\draw[red, very thick] ($(A)!0.5!(B)+(-10+4*\i,-10)$) -- ($(A)!0.5!(B)+(-10+4*\i,-12)$);
				}
				\foreach \i in {0,1,2,4,5,6}{
					\draw[red, very thick] ($(A)!0.5!(B)+(-12+4*\i,-12)$) -- ($(A)!0.5!(B)+(-12+4*\i,-16)$);
				}
				\node at ($(A)!0.5!(B)+(-14,-8)$) [rotate=90]{\scriptsize $\cdots$};
				\node at ($(A)!0.5!(B)+(14,-8)$) [rotate=-90]{\scriptsize $\cdots$};
				\node at ($(A)!0.5!(B)+(-6,-8)$) [rotate=45]{\scriptsize $\cdots$};
				\node at ($(A)!0.5!(B)+(6,-8)$) [rotate=-45]{\scriptsize $\cdots$};
				\node at ($(A)!0.5!(B)+(0,-10)$) {\scriptsize $\cdots$};
				\node at ($(A)!0.5!(B)+(0,-14)$) {\scriptsize $\cdots$};
				\node[yscale=3] at (17,-1) {$\left.\right\}$};
				\node at (20,-1) {$l$};
			}
		\ }
	\end{align*}
	It can be seen that the triangle web in the left (resp. right) in \cref{fig:Bweb T} is equal to $A^{r}(t_{{123}}^{+})^{l}$ (resp. $A^{r}(t_{{123}}^{-})^{l}$) for some $r\in\bZ$.
	Thus $\SK{T}$ is generated by $\{e_{21}, e_{12},\ldots,e_{31}, t_{{123}}^{+},t_{{123}}^{-}\}$.
\end{proof}
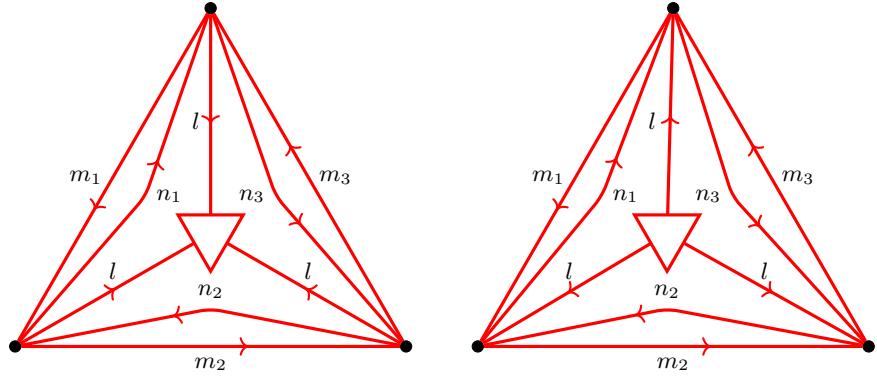
\begin{figure}
	\centering
	\begin{tikzpicture}
		\node (scope1){
			\begin{tikzpicture}[scale=.1]
				\node [draw, fill=black, circle, inner sep=1.5] (A) at (90:30) {};
				\node [draw, fill=black, circle, inner sep=1.5] (B) at (210:30) {} edge (A);
				\node [draw, fill=black, circle, inner sep=1.5] (C) at (-30:30) {} edge (A) edge (B);
				\coordinate (D) at (0,0) {};
				\draw[red, ->-={.6}{}, very thick] (A) -- (B);
				\draw[red, ->-={.6}{}, very thick] (B) -- (C);
				\draw[red, ->-={.6}{}, very thick] (C) -- (A);
				\draw[red, ->-={.6}{}, very thick, rounded corners] (B) -- (150:10) -- (A);
				\draw[red, ->-={.6}{}, very thick, rounded corners] (C) -- (270:10) -- (B);
				\draw[red, ->-={.6}{}, very thick, rounded corners] (A) -- (30:10) -- (C);
				\draw[red, ->-, very thick] (A) -- (D);
				\draw[red, ->-, very thick] (B) -- (D);
				\draw[red, ->-, very thick] (C) -- (D);
				\draw[red, very thick, fill=white] (30:5) -- (150:5) -- (270:5) -- cycle;
				\node at ($(A)!0.5!(B)$) [left] {\scriptsize $m_{1}$};	
				\node at ($(B)!0.5!(C)$) [below] {\scriptsize $m_{2}$};	
				\node at ($(C)!0.5!(A)$) [right] {\scriptsize $m_{3}$};	
				\node at (150:10) [right] {\scriptsize $n_{1}$};	
				\node at (270:10) [above] {\scriptsize $n_{2}$};	
				\node at (30:10) [left] {\scriptsize $n_{3}$};	
				\node at (90:15) [left] {\scriptsize $l$};	
				\node at (210:15) [above] {\scriptsize $l$};	
				\node at (330:15) [above] {\scriptsize $l$};	
			\end{tikzpicture}
		};
		\node[at={($(scope1.east)+(.5cm,0)$)}, anchor=west]{
			\begin{tikzpicture}[scale=.1]
				\node [draw, fill=black, circle, inner sep=1.5] (A) at (90:30) {};
				\node [draw, fill=black, circle, inner sep=1.5] (B) at (210:30) {} edge (A);
				\node [draw, fill=black, circle, inner sep=1.5] (C) at (-30:30) {} edge (A) edge (B);
				\coordinate (D) at (0,0) {};
				\draw[red, ->-={.6}{}, very thick] (A) -- (B);
				\draw[red, ->-={.6}{}, very thick] (B) -- (C);
				\draw[red, ->-={.6}{}, very thick] (C) -- (A);
				\draw[red, ->-={.6}{}, very thick, rounded corners] (B) -- (150:10) -- (A);
				\draw[red, ->-={.6}{}, very thick, rounded corners] (C) -- (270:10) -- (B);
				\draw[red, ->-={.6}{}, very thick, rounded corners] (A) -- (30:10) -- (C);
				\draw[red, -<-, very thick] (A) -- (D);
				\draw[red, -<-, very thick] (B) -- (D);
				\draw[red, -<-, very thick] (C) -- (D);
				\draw[red, very thick, fill=white] (30:5) -- (150:5) -- (270:5) -- cycle;
				\node at ($(A)!0.5!(B)$) [left] {\scriptsize $m_{1}$};	
				\node at ($(B)!0.5!(C)$) [below] {\scriptsize $m_{2}$};	
				\node at ($(C)!0.5!(A)$) [right] {\scriptsize $m_{3}$};	
				\node at (150:10) [right] {\scriptsize $n_{1}$};	
				\node at (270:10) [above] {\scriptsize $n_{2}$};	
				\node at (30:10) [left] {\scriptsize $n_{3}$};	
				\node at (90:15) [left] {\scriptsize $l$};	
				\node at (210:15) [above] {\scriptsize $l$};	
				\node at (330:15) [above] {\scriptsize $l$};	
			\end{tikzpicture}
		};
	\end{tikzpicture}
	\caption{$\mathfrak{sl}_3$ basis webs of $\scS_{\mathfrak{sl}_{3},T}^{A}$}
	\label{fig:Bweb T}	
\end{figure}

\begin{prop}\label{prop:Eweb T}
	Let $C_{(\Delta_T,{\epsilon})}:=\Bweb{\partial^{\times}T}\cup\{t_{123}^{\epsilon}\}$ for $\epsilon \in \{+,-\}$.
	Then we have $\Eweb{T}=\Bweb{\partial^{\times}T}\cup\{t_{{123}}^{+},t_{{123}}^{-}\}$ and $\Cweb{T}=\{C_{(\Delta_T,{+})},C_{(\Delta_T,{-})}\}$.
\end{prop}
\begin{proof}
	By the proof of \cref{prop:generators T}, it suffices to show that $t_{{123}}^{+}$ and $t_{{123}}^{-}$ can not be expressed as a product of basis webs.
	Otherwise, with a notice that the sum $\vec{\gr}$ of the endpoint degree is additive with respect to the multiplication in the skein algebra, $\vec{\gr}(t_{{123}}^{+})=(3,0)$ means that $t_{{123}}^{+}$ must be decomposed into a product of basis webs $G_1$ and $G_2$ with $\vec{\gr}(G_1)=(2,0)$ and $\vec{\gr}(G_2)=(1,0)$.
	However, there exist no such pair of basis webs in $\Eweb{T}$.
	Hence $\Eweb{T}=\{t_{{123}}^{+},t_{{123}}^{-}\}\cup\Bweb{\partial^{\times}T}$.
	It is easy to see that ($t_{{123}}^{+},t_{{123}}^{-})$ is the only pair which do not $A$-commute with each other from \cref{lem:str const T}.
\end{proof}

We will use the following notation.

\begin{dfn}\label{def:monomial}
	For a subset $S$ of $\Bweb{\Sigma}$, let $\langle S\rangle_{\mathrm{alg}}$ denote the subalgebra of $\SK{\Sigma}$ generated by $S$, and $\mathrm{mon}(S)$ the multiplicatively closed set generated by $S\cup\{A^{\pm 1/2}\}$ in $\SK{\Sigma}$.	
\end{dfn}

\begin{thm}[Laurent expression in web clusters for a triangle]\label{thm:Laurant of T}
	For any $x\in\SK{T}$ and $\epsilon\in\{{+},{-}\}$, there exists $(t_{123}^{\epsilon})^{k}\in\mathrm{mon}(C_{(\Delta_{T},\epsilon)})$ for some $k\in\bN$ such that $(t_{123}^{\epsilon})^{k}x\in\langle C_{(\Delta_{T},\epsilon)}\rangle_{\mathrm{alg}}$ and it has positive coefficients if $x\in\Bweb{T}$.
\end{thm}
\begin{proof}
	By \cref{prop:generators T,prop:Eweb T}, any web $x\in\SK{T}$ can be written as a polynomial on the generators $\Eweb{T}$.
	By \cref{lem:str const T}, $t_{{123}}^{+}t_{{123}}^{-}$ is expanded as a polynomial in $\Bweb{\partial^{\times}T}$ with positive coefficients.
	Moreover, $t_{{123}}^{+}$ $A$-commutes with the webs in $\Bweb{\partial^{\times}T}$.
	Hence by multiplying a sufficiently large power $(t_{{123}}^{+})^{k}$ to $x$, we can replace $t_{{123}}^{-}$'s in each monomial in $x$ with boundary webs, without changing the signs of the coefficients. The second assertion follows since each basis web is a monomial on $\Eweb{T}$.
\end{proof}

We remark that the above propositions say that $\SK{T}$ is generated by $\Eweb{T}$, and generated by ``Laurent polynomials'' in a cluster web $C_{(\Delta_T,\epsilon)}$ for $\epsilon\in\{{+}, {-}\}$.

\section{Expansions and positivity of \texorpdfstring{$\mathfrak{sl}_3$}{sl(3)}-webs}\label{sec:expansion}
Based on the expansion results in triangles and quadrilaterals obtained in the previous sections, we are going to give two expansion results in a general unpunctured marked surface.
One expands an $\mathfrak{sl}_3$-web to a Laurent polynomial in web clusters associated with a triangulation.
Moreover, we discuss the positivity of the coefficients of this expansion. 
Another expands an $\mathfrak{sl}_3$-web to elementary webs by localizing $\mathfrak{sl}_3$-webs along boundary intervals.
\subsection{Laurent expressions, positivity and localized skein algebras}
We firstly prepare a key lemma.
\begin{lem}[The cutting trick~\cite{FP14}]\label{lem:cuttingweb}
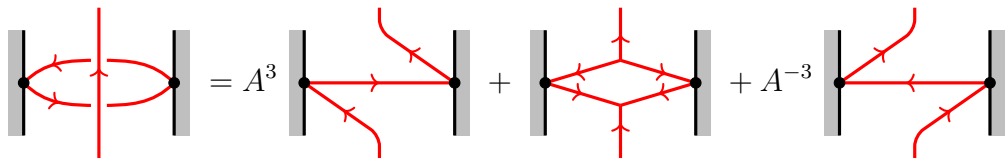
	\begin{align*}
		\mathord{
			\ \tikz[baseline=-.6ex, scale=.1]{
			\coordinate (A) at (-10,0);
			\coordinate (B) at (10,0);
			\coordinate (NW) at ($(A)+(0,7)$);
			\coordinate (SW) at ($(A)+(0,-7)$);
			\coordinate (NE) at ($(B)+(0,7)$);
			\coordinate (SE) at ($(B)+(0,-7)$);
			\coordinate (T1) at ($(A)!0.5!(B)+(0,7)$);
			\coordinate (T2) at ($(A)!0.5!(B)+(0,-7)$);
			\fill[lightgray] (NW) -- (SW) -- ($(SW)+(-2,0)$) -- ($(NW)+(-2,0)$) -- (NW) -- cycle;
			\fill[lightgray] (NE) -- (SE) -- ($(SE)+(2,0)$) -- ($(NE)+(2,0)$) -- (NE) -- cycle;
			\draw[very thick] (NW) -- (SW);
			\draw[very thick] (NE) -- (SE);
			\draw[red, very thick, ->-={.8}{}] (B) to[out=north west,in=east] (0,3) to[out=west, in=north east] (A);
			\draw[red, very thick, -<-={.8}{}] (B) to[out=south west,in=east] (0,-3) to[out=west, in=south east] (A);
			\draw[overarc, ->-={.6}{red}] (0,-10) -- (0,10);
			\draw[fill] (A) circle (20pt);
			\draw[fill] (B) circle (20pt);
			}
		\ }
		&=A^{3}\mathord{
			\ \tikz[baseline=-.6ex, scale=.1]{
				\coordinate (A) at (-10,0);
				\coordinate (B) at (10,0);
				\coordinate (NW) at ($(A)+(0,7)$);
				\coordinate (SW) at ($(A)+(0,-7)$);
				\coordinate (NE) at ($(B)+(0,7)$);
				\coordinate (SE) at ($(B)+(0,-7)$);
				\coordinate (T1) at ($(A)!0.5!(B)+(0,7)$);
				\coordinate (T2) at ($(A)!0.5!(B)+(0,-7)$);
				\fill[lightgray] (NW) -- (SW) -- ($(SW)+(-2,0)$) -- ($(NW)+(-2,0)$) -- (NW) -- cycle;
				\fill[lightgray] (NE) -- (SE) -- ($(SE)+(2,0)$) -- ($(NE)+(2,0)$) -- (NE) -- cycle;
				\draw[very thick] (NW) -- (SW);
				\draw[very thick] (NE) -- (SE);
				\draw[red, very thick, ->-] (A) -- (B);
				\draw[red, very thick, ->-, rounded corners] (B) -- (T1) -- ($(T1)+(0,3)$);
				\draw[red, very thick, -<-, rounded corners] (A) -- (T2) -- ($(T2)+(0,-3)$);
				\draw[fill] (A) circle (20pt);
				\draw[fill] (B) circle (20pt);
			}
		\ }
		+\mathord{
			\ \tikz[baseline=-.6ex, scale=.1]{
				\coordinate (A) at (-10,0);
				\coordinate (B) at (10,0);
				\coordinate (NW) at ($(A)+(0,7)$);
				\coordinate (SW) at ($(A)+(0,-7)$);
				\coordinate (NE) at ($(B)+(0,7)$);
				\coordinate (SE) at ($(B)+(0,-7)$);
				\coordinate (T1) at ($(A)!0.5!(B)+(0,3)$);
				\coordinate (T2) at ($(A)!0.5!(B)+(0,-3)$);
				\fill[lightgray] (NW) -- (SW) -- ($(SW)+(-2,0)$) -- ($(NW)+(-2,0)$) -- (NW) -- cycle;
				\fill[lightgray] (NE) -- (SE) -- ($(SE)+(2,0)$) -- ($(NE)+(2,0)$) -- (NE) -- cycle;
				\draw[very thick] (NW) -- (SW);
				\draw[very thick] (NE) -- (SE);
				\draw[very thick, red, -<-] (A) -- (T1);
				\draw[very thick, red, -<-] (B) -- (T1);
				\draw[very thick, red, ->-] (A) -- (T2);
				\draw[very thick, red, ->-] (B) -- (T2);
				\draw[very thick, red, ->-] (T1) -- ($(T1)+(0,7)$);
				\draw[very thick, red, -<-] (T2) -- ($(T2)+(0,-7)$);
				\draw[fill] (A) circle (20pt);
				\draw[fill] (B) circle (20pt);
			}
		\ }
		+A^{-3}\mathord{
			\ \tikz[baseline=-.6ex, scale=.1]{
				\coordinate (A) at (-10,0);
				\coordinate (B) at (10,0);
				\coordinate (NW) at ($(A)+(0,7)$);
				\coordinate (SW) at ($(A)+(0,-7)$);
				\coordinate (NE) at ($(B)+(0,7)$);
				\coordinate (SE) at ($(B)+(0,-7)$);
				\coordinate (T1) at ($(A)!0.5!(B)+(0,7)$);
				\coordinate (T2) at ($(A)!0.5!(B)+(0,-7)$);
				\fill[lightgray] (NW) -- (SW) -- ($(SW)+(-2,0)$) -- ($(NW)+(-2,0)$) -- (NW) -- cycle;
				\fill[lightgray] (NE) -- (SE) -- ($(SE)+(2,0)$) -- ($(NE)+(2,0)$) -- (NE) -- cycle;
				\draw[very thick] (NW) -- (SW);
				\draw[very thick] (NE) -- (SE);
				\draw[red, very thick, -<-] (A) -- (B);
				\draw[red, very thick, ->-, rounded corners] (A) -- (T1) -- ($(T1)+(0,3)$);
				\draw[red, very thick, -<-, rounded corners] (B) -- (T2) -- ($(T2)+(0,-3)$);
				\draw[fill] (A) circle (20pt);
				\draw[fill] (B) circle (20pt);
			}
		\ }
	\end{align*}
\end{lem}
\begin{proof}
	Apply skein relations to two internal crossings on the left-hand side.
\end{proof}
The above diagram pictures a neighborhood of an arc between distinguished special points.
We can apply this formula if another $\mathfrak{sl}_3$-web lies in the upper or lower layer of depicted webs.
Such decomposition formula (at $A=1$) of a web by an ideal arc also appears in \cite[Fig.~11]{FP14} to expand $\mathfrak{sl}_3$-skein algebra at $A=1$ into a ring of invariants $R_{a,b,c}(V)$.

\subsubsection{Laurent expression}\label{subsec:expansion}
Let $\bD=(\Delta,\bs_{\Delta})$ be a decorated 
triangulation of $\Sigma$, namely, $\bs_{\Delta}\colon t(\Delta)\to\{\pm\}$ is a function from the set of triangles to the signs.
For each triangle $T\in t(\Delta)$, we can naturally regard the set of elementary webs $\Eweb{T}$ as a subset of $\Bweb{\Sigma}$. Similarly we regard each web cluster $C_{(\Delta_{T},\bs_{\Delta}(T))}\in\Cweb{T}$ as a subset of $\Bweb{\Sigma}$. 

We are going to show that any $\mathfrak{sl}_3$-web $x\in \SK{\Sigma}$ can be expressed as a Laurent polynomial in the web cluster $C_\bD:=\cup_{T\in t(\Delta)}C_{(\Delta_{T},\bs_{\Delta}(T))} \subset \Bweb{\Sigma}$.
Let
\begin{align*}
    \mathrm{mon}(\Delta):=\mathrm{mon}(\Eweb{\partial^{\times}T}) \subset \SK{\Sigma}
\end{align*}
be the multiplicatively closed set generated by the elementary webs along the edges of $\Delta$ and $A^{1/2}$.

\begin{thm}[Expansions in elementary webs on triangles]\label{thm:elementary-web-expansion-T}
	For any web $x\in\SK{\Sigma}$ and a triangulation $\Delta$, there exists a monomial $J_\Delta\in\mathrm{mon}(\Delta)$ such that $xJ_\Delta\in\langle\cup_{T}\Eweb{T}\rangle_{\mathrm{alg}}$.
\end{thm}

\begin{proof}
Let $[G]$ be a flat trivalent graph in $(\Sigma,\bM)$ such that the representative has only finitely many internal (transverse) intersection points with $\Delta$.
Assume that we have $\# [G]\cap E_{ij}=n$ for an edge $E_{ij}\in e(\Delta)$ connecting $p_{i}$ and $p_{j}$, and let $\gamma_1,\gamma_2,\ldots,\gamma_n$ be the corresponding short subarcs of $[G]$ at $[G]\cap \interior E_{ij}$.
One can reduce the number of internal intersection points with $E_{ij}$ by multiplying $[G]$ by $[e_{ij}e_{ji}]$ where $e_{ij}$ and $e_{ji}:=e_{ij}^{*}$ are oriented arcs corresponding to $E_{ij}$.
In fact, the product $\gamma_1[e_{ij}e_{ji}]$ can expanded by \cref{lem:cuttingweb} and denote it by $\gamma_1[e_{ij}e_{ji}]=A^{3}G_{1}^{+}+G_{1}^{0}+A^{-3}G_{1}^{-}$.
It is easy to see that the number of internal intersection points of $G_{1}^{\epsilon}$ with $\Delta$ is one of $[G]$ minus $1$ for $\epsilon={+},0,{-}$, and the operation causes no change in $\gamma_2,\gamma_3,\dots\gamma_n$.
Moreover, one can see that $G_{1}^{\epsilon}$ are $A$-commutative with $[e_{ij}e_{ji}]$ because of the boundary $\mathfrak{sl}_3$-skein relation.
It means that, for any $\ell \geq 0$, we have $(A^{3}G_{1}^{+}+G_{1}^{0}+A^{-3}G_{1}^{-})[e_{ij}e_{ji}]^{\ell}=[e_{ij}e_{ji}]^{\ell}(A^{r_{+}}G_{1}^{+}+A^{r_{0}}G_{1}^{0}+A^{r_{-}}G_{1}^{-})$ for some $r_{+}$, $r_0$, and $r_{-}$.
Therefore one can also apply the above computation to $\gamma_k$ for $k=2,\dots,n$. 
Thus the product $G[e_{ij}e_{ji}]^{n}$ is expanded into a sum of webs without transverse intersection points with $E_{ij}$. 

Any $\mathfrak{sl}_3$-web $x \in \SK{\Sigma}$ can be written as a sum $x=\sum_{G}\lambda_{G}[G]$ of basis webs.
For each $E\in e(\Delta)$, we denote the maximum number of internal intersection points of $E$ with $G$'s by $n_{E}(x):=\max\{\# [G]\cap \interior E \mid \lambda_G\neq 0\}$.
By applying the above computation, the product
\[
	x\prod_{E_{ij}\in\Delta}(e_{ij}e_{ji})^{n_{E_{ij}}(x)}
\]
is expanded into a polynomial of webs in triangles of $\Delta$, which can be further written as a polynomial in $\cup_{T\in t(\Delta)}\Eweb{T}$ by \cref{prop:generators T}. Thus we get the assertion with $J_\Delta:=\prod_{E_{ij}\in\Delta}(e_{ij}e_{ji})^{n_{E_{ij}}(x)}$.
\end{proof}
Naively, the above theorem tells us that the web $x$ has a Laurent expression $x=f_\Delta J_\Delta^{-1}$ with $f_\Delta \in \langle\cup_{T}\Eweb{T}\rangle_{\mathrm{alg}}$. This will be made more precise in \cref{subsec:localization}.

Let us further consider the multiplicatively closed set $\mathrm{mon}(C_{\bD})$, which obviously contains $\mathrm{mon}(\Delta)$.  

\begin{cor}[Expansions in the web cluster $C_\bD$]\label{cor:web-cluster-expansion-T}
	For any web $x\in\SK{\Sigma}$ and a decorated triangulation $\bD$, there exists a monomial $J_\bD\in\mathrm{mon}(C_{\bD})$ such that $xJ_\bD \in\langle C_{\bD}\rangle_{\mathrm{alg}}$.
\end{cor}
\begin{proof}
	By \cref{thm:elementary-web-expansion-T}, there exists a monomial $J'_\Delta\in\mathrm{mon}(\Delta)$ such that $xJ'_\Delta\in\langle\cup_{T}\Eweb{T}\rangle_{\mathrm{alg}}$.
	In the same way as the proof of \cref{thm:Laurant of T}, by multiplying $(e_{v^{\bs(T)}})^{k_T}$ to $xJ'_\Delta$, we can replace $\ast(e_{v^{\bs(T)}})$ for each $v\in I^{\mathrm{tri}}(\Delta)\cap T$ with elementary webs along the edges.
	Here $k_T$ is the degree of $\ast(e_{v^{\bs(T)}})$ in the polynomial $xJ'_\Delta$, and
	note that any monomial containing no $\ast(e_{v^{\bs(T)}})$ is $A$-commutative with $e_{v^{\bs(T)}}$.
	Thus $xJ'_\Delta\prod_{T\in t(\Delta)}(e_{v^{\bs(T)}})^{k_{T}}$ is contained in $\langle C_{\bD}\rangle_{\mathrm{alg}}$, and we get the assertion with $J_\bD:=J'_\Delta\prod_{T\in t(\Delta)}(e_{v^{\bs(T)}})^{k_{T}} \in \mathrm{mon}(C_\bD)$.
\end{proof}

\subsubsection{Laurent positivity for elevation-preserving webs}

We are going to show that the Laurent expressions of webs of a certain kind, which we call the \emph{elevation-preserving $\mathfrak{sl}_3$-webs}, in $\SK{\Sigma}[\Delta^{-1}]$ have positive coefficients.
By arguing as in the proof of \cref{cor:web-cluster-expansion-T}, it implies that the Laurent expressions in the web cluster $C_\bD$ also have positive coefficients.
Elevation-preserving $\mathfrak{sl}_3$-webs include the bracelets and the bangles along an oriented simple closed curve.

For an ideal triangulation $\Delta$ of $\Sigma$, let $\Delta^{\mathrm{split}}$ be the associated \emph{splitting triangulation} obtained by replacing each edge of $\Delta$ with doubled edges as shown in \cref{fig:split triangulation}.
The set of connected components of $\Sigma\setminus\Delta^{\mathrm{split}}$ is divided into two subsets: the set $t(\Delta^{\mathrm{split}})$ of triangles and the set $b(\Delta^{\mathrm{split}})$ of biangles.
We can canonically identify $t(\Delta^{\mathrm{split}})$ with $t(\Delta)$, and $b(\Delta^{\mathrm{split}})$ with $e(\Delta)$.
We denote a triangle in $t(\Delta^{\mathrm{split}})$ by the same symbol as the corresponding triangle in $t(\Delta)$, while the biangle corresponding to an edge $E\in e(\Delta)$ is denoted by $B_{E}\in b(\Delta^{\mathrm{split}})$. 
For an edge $E\in e(\Delta)$ and a triangle $T\in t(\Delta)$ adjacent to $E$, let $E_{T}\in e(\Delta^{\mathrm{split}})$ denote the edge shared by $B_{E}$ and $T$.
\begin{dfn}
	\begin{enumerate}
		\item A \emph{fundamental piece} in $T\in t(\Delta^{\mathrm{split}})$ consists of a superposition of trivalent graphs with at most one trivalent vertex and distinct endpoints on $\partial^{\times}T$ such that endpoints of the same connected component lie in distinct connected components of $\partial^{\times}T$ each other. An \emph{elevation} of a fundamental piece of $T$ is a labeling of its connected components by positive integers. See  left and right of \cref{fig:elevation in T}.
		\item Let $E\in e(\Delta)$ be an edge shared by $T$ and $T'$ in $t(\Delta)$. An \emph{elevation-preserving braid} in $B_{E}\in b(\Delta^{\mathrm{split}})$ connecting fundamental pieces with elevations in $T$ and $T'$  is a braid between $E_{T}$ and $E_{T'}$ such that
			\begin{itemize}
				\item the braid consists of a superposition of strands connecting endpoints of fundamental pieces of $T$ and $T'$;
				\item for any strands $\alpha$ and $\beta$ of the braid, $\alpha(T)\leq\beta(T)$ if and only if $\alpha(T')\leq\beta(T')$;
				\item a strand $\alpha$ passes above another strand $\beta$ if $\alpha(T)>\beta(T)$ or $\alpha(T')>\beta(T')$;
			\end{itemize}
		where $\alpha(T)$ (resp. $\beta(T)$) denotes the elevation on the endpoint of $\alpha$ (resp. $\beta$) in $E_{T}$ induced from the fundamental piece with the elevation in $T$, and similarly for $T'$. See the middle of \cref{fig:elevation in T}.
		\item Let $E\in \mathbb{B}$ be a boundary interval and $T\in t(\Delta)$ the adjacent triangle. An \emph{elevation-preserving braid} in $B_{E}\in b(\Delta^{\mathrm{split}})$ consists of elevation-preserving arcs with no internal crossings connecting $E_{T}$ to one of the tow special points of $B_{E}$.
	\end{enumerate}
\end{dfn}
We define a certain $\mathfrak{sl}_{3}$-web which satisfies positivity by concatenating fundamental pieces with elevations by elevation-preserving braids.
\begin{dfn}[elevation-preserving $\mathfrak{sl}_{3}$-webs]\label{def:elevation-preserving web}
	A tangled trivalent graph in $\Sigma$ is said to be \emph{elevation-preserving with respect to $\Delta$} if it can be decomposed into fundamental pieces in triangles and elevation-preserving braids in biangles connecting them by cutting along $e(\Delta^{\mathrm{split}})$.
	An \emph{elevation-preserving $\mathfrak{sl}_{3}$-web} is an $\mathfrak{sl}_{3}$-web such that it is represented by an elevation-preserving graph with respect to some $\Delta$.
\end{dfn}

\begin{figure}
	\begin{tikzpicture}[scale=.1]
		\coordinate (P1) at (90:30);
		\coordinate (P2) at (180:30);
		\coordinate (P3) at (270:30);
		\coordinate (P4) at (0:30);
		\draw[blue] (P1) -- (P2) -- (P3) -- (P4) -- (P1) -- cycle;
		\draw[blue] (P1) -- (P3);
		\node at ($(P2)!.5!(P4)+(0,5)$) [right]{$E$};
		\node at ($(P2)!.3!(P4)$) {$T$};
		\node at ($(P4)!.3!(P2)$) {$T'$};
		\node at (P3) [below=10pt]{$\Delta$};
	\end{tikzpicture}
	\hspace{1em}
	\begin{tikzpicture}[scale=.1]
		\coordinate (P1) at (90:30);
		\coordinate (P2) at (180:30);
		\coordinate (P3) at (270:30);
		\coordinate (P4) at (0:30);
		\draw[blue!50] (P1) -- (P2) -- (P3) -- (P4) -- (P1) -- cycle;
		\draw[blue!50] (P1) -- (P3);
		\draw[blue, thick, rounded corners] (P1) -- ($(P1)!.5!(P3)+(3,0)$) -- (P3);
		\draw[blue, thick, rounded corners] (P1) -- ($(P1)!.5!(P3)+(-3,0)$) -- (P3);
		\draw[blue, thick, rounded corners] (P1) -- ($(P1)!.5!(P2)+(2,-2)$) -- (P2);
		\draw[blue, thick, rounded corners] (P1) -- ($(P1)!.5!(P2)+(-2,2)$) -- (P2);
		\draw[blue, thick, rounded corners] (P3) -- ($(P3)!.5!(P4)+(2,-2)$) -- (P4);
		\draw[blue, thick, rounded corners] (P3) -- ($(P3)!.5!(P4)+(-2,2)$) -- (P4);
		\draw[blue, thick, rounded corners] (P2) -- ($(P2)!.5!(P3)+(-2,-2)$) -- (P3);
		\draw[blue, thick, rounded corners] (P2) -- ($(P2)!.5!(P3)+(2,2)$) -- (P3);
		\draw[blue, thick, rounded corners] (P4) -- ($(P4)!.5!(P1)+(-2,-2)$) -- (P1);
		\draw[blue, thick, rounded corners] (P4) -- ($(P4)!.5!(P1)+(2,2)$) -- (P1);
		\node at (0,0) {\scriptsize $B_{E}$};
		\node at ($(P2)!.5!(P4)+(0,5)$) [left=5pt]{$E_{T}$};
		\node at ($(P2)!.5!(P4)+(0,5)$) [right=5pt]{$E_{T'}$};
		\node at ($(P2)!.2!(P4)$) {$T$};
		\node at ($(P4)!.2!(P2)$) {$T'$};
		\node at (P3) [below=10pt]{$\Delta^{\mathrm{split}}$};
	\end{tikzpicture}
	\caption{The split triangulation $\Delta^{\mathrm{split}}$ associated with $\Delta$.}
	\label{fig:split triangulation}
\end{figure}
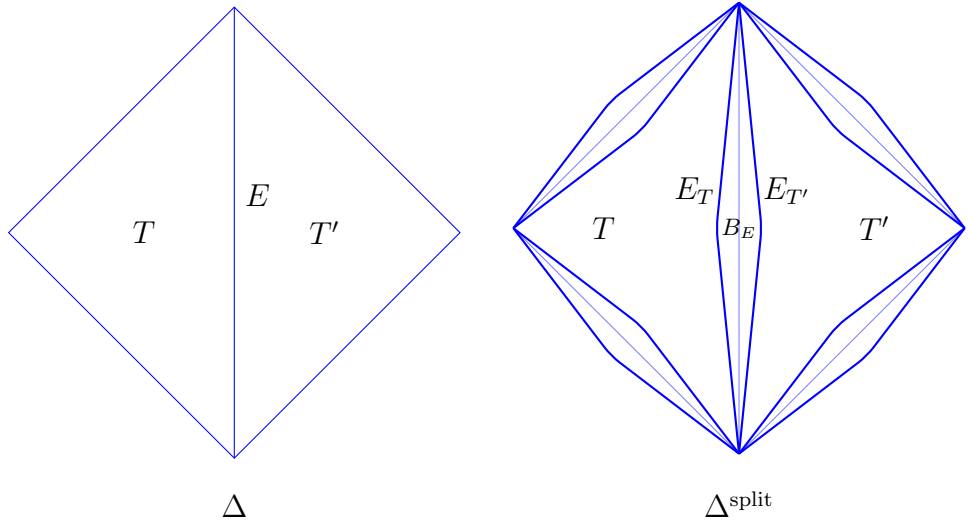
\begin{figure}
	\begin{tikzpicture}[scale=.1]
		\coordinate (P1) at (210:30);
		\coordinate (P2) at (330:30);
		\coordinate (P3) at (90:30);
		\foreach \i in {4,5,6}{
			\coordinate (T\i) at ($(P1)!.\i!(P2)+(0,5)$);
		}
		\draw[very thick, red, overarc, ->-={.5}{red}] ($(P1)!.4!(P3)$) -- (T4);
		\draw[very thick, red, ->-] ($(P1)!.4!(P2)$) -- (T4);
		\draw[very thick, red, ->-] ($(P3)!.4!(P2)$) -- (T4);
		\draw[very thick, red, overarc, -<-={.5}{red}] ($(P1)!.5!(P3)$) -- (T5);
		\draw[very thick, red, -<-] ($(P1)!.5!(P2)$) -- (T5);
		\draw[very thick, red, -<-] ($(P3)!.5!(P2)$) -- (T5);
		\draw[very thick, red, overarc, ->-={.5}{red}] ($(P1)!.6!(P3)$) -- (T6);
		\draw[very thick, red, ->-] ($(P1)!.6!(P2)$) -- (T6);
		\draw[very thick, red, ->-] ($(P3)!.6!(P2)$) -- (T6);
		\draw[very thick, red, ->-] ($(P1)!.1!(P2)$) -- ($(P1)!.1!(P3)$);
		\draw[very thick, red, ->-] ($(P1)!.2!(P2)$) -- ($(P1)!.2!(P3)$);
		\draw[very thick, red, -<-] ($(P2)!.1!(P1)$) -- ($(P2)!.1!(P3)$);
		\draw[very thick, red, ->-] ($(P2)!.2!(P1)$) -- ($(P2)!.2!(P3)$);
		\draw[very thick, red, ->-] ($(P2)!.3!(P1)$) -- ($(P2)!.3!(P3)$);
		\draw[very thick, red, ->-] ($(P3)!.2!(P1)$) -- ($(P3)!.2!(P2)$);
		\draw[very thick, red, -<-] ($(P3)!.3!(P1)$) -- ($(P3)!.3!(P2)$);
		\draw[blue] (P1) -- (P2) -- (P3) -- (P1) -- cycle;
		\node at ($(P1)!.1!(P3)$) [left]{\scriptsize $3$};
		\node at ($(P1)!.2!(P3)$) [left]{\scriptsize $5$};
		\node at ($(P2)!.8!(P3)$) [right]{\scriptsize $1$};
		\node at ($(P2)!.7!(P3)$) [right]{\scriptsize $4$};
		\node at ($(P2)!.6!(P3)$) [right]{\scriptsize $2$};
		\node at ($(P2)!.5!(P3)$) [right]{\scriptsize $6$};
		\node at ($(P2)!.4!(P3)$) [right]{\scriptsize $7$};
		\node at ($(P2)!.3!(P3)$) [right]{\scriptsize $7$};
		\node at ($(P2)!.2!(P3)$) [right]{\scriptsize $5$};
		\node at ($(P2)!.1!(P3)$) [right]{\scriptsize $5$};
		\node[draw, fill=black, circle, inner sep=1] at (P1) {};
		\node[draw, fill=black, circle, inner sep=1] at (P2) {};
		\node[draw, fill=black, circle, inner sep=1] at (P3) {};
		\node at ($(P1)!.5!(P2)$) [below=5pt] {$T$};
	\end{tikzpicture}
	\begin{tikzpicture}[scale=.1]
		\foreach \i [evaluate=\i as \x using \i*5] in {1,2,...,8}{
			\coordinate (L\i) at (-10,\x);
			\coordinate (R\i) at (10,\x);
		}
		\draw[very thick, red, overarc] (L8) to[out=east, in=west] (R7);
		\draw[very thick, red, overarc] (L6) to[out=east, in=west] ($(L6)!.5!(R6)+(0,3)$) to[out=east, in=west] (R6);
		\draw[very thick, red, overarc] (L7) -- ($(L7)+(2,0)$) to[out=east, in=west] (R3);
		\draw[very thick, red, overarc] (L2) to[out=east, in=west] (R1);
		\draw[very thick, red, overarc] (L1) to[out=east, in=west] (R8);
		\draw[very thick, red, overarc] (L5) to[out=east, in=west] ($(L5)!.5!(R5)-(0,4)$) to[out=east, in=west] (R5);
		\draw[very thick, red, overarc] (L3) to[out=east, in=west] (R2);
		\draw[very thick, red, overarc] (L4) to[out=east, in=west] ($(L4)!.5!(R4)-(0,3)$) to[out=east, in=west] (R4);
		\draw[blue, rounded corners] (0,45) -- ($(L8)+(0,2)$) -- ($(L1)-(0,2)$) -- (0,0);
		\draw[blue, rounded corners] (0,45) -- ($(R8)+(0,2)$) -- ($(R1)-(0,2)$) -- (0,0);
		\node[draw, fill=black, circle, inner sep=1] at (0,45) {};
		\node[draw, fill=black, circle, inner sep=1] at (0,0) {};
		\node at (0,0) [below=5pt]{$B_{E}$};
		\node at (L8) [left]{\scriptsize $1$};
		\node at (L7) [left]{\scriptsize $4$};
		\node at (L6) [left]{\scriptsize $2$};
		\node at (L5) [left]{\scriptsize $6$};
		\node at (L4) [left]{\scriptsize $7$};
		\node at (L3) [left]{\scriptsize $7$};
		\node at (L2) [left]{\scriptsize $5$};
		\node at (L1) [left]{\scriptsize $5$};
		\node at (R8) [right]{\scriptsize $5$};
		\node at (R7) [right]{\scriptsize $1$};
		\node at (R6) [right]{\scriptsize $2$};
		\node at (R5) [right]{\scriptsize $5$};
		\node at (R4) [right]{\scriptsize $8$};
		\node at (R3) [right]{\scriptsize $3$};
		\node at (R2) [right]{\scriptsize $6$};
		\node at (R1) [right]{\scriptsize $4$};
	\end{tikzpicture}
	\begin{tikzpicture}[scale=.1]
		\coordinate (P1) at (210:30);
		\coordinate (P2) at (330:30);
		\coordinate (P3) at (90:30);
		\foreach \i in {4,5,6,7}{
			\coordinate (T\i) at ($(P1)!.\i!(P2)+(0,5)$);
		}
		\draw[very thick, red, overarc, ->-={.2}{red}] ($(P3)!.7!(P2)$) -- (T7);
		\draw[very thick, red, ->-] ($(P1)!.7!(P3)$) -- (T7);
		\draw[very thick, red, ->-] ($(P1)!.7!(P2)$) -- (T7);
		\draw[very thick, red, overarc, -<-={.2}{red}] ($(P3)!.6!(P2)$) -- (T6);
		\draw[very thick, red, -<-] ($(P1)!.6!(P3)$) -- (T6);
		\draw[very thick, red, -<-] ($(P1)!.6!(P2)$) -- (T6);
		\draw[very thick, red, overarc, ->-={.2}{red}] ($(P3)!.5!(P2)$) -- (T5);
		\draw[very thick, red, ->-] ($(P1)!.5!(P3)$) -- (T5);
		\draw[very thick, red, -<-] ($(P1)!.5!(P2)$) -- (T5);
		\draw[very thick, red, overarc, -<-={.2}{red}] ($(P3)!.4!(P2)$) -- (T4);
		\draw[very thick, red, -<-] ($(P1)!.4!(P3)$) -- (T4);
		\draw[very thick, red, -<-] ($(P1)!.4!(P2)$) -- (T4);
		\draw[very thick, red, -<-] ($(P1)!.1!(P2)$) -- ($(P1)!.1!(P3)$);
		\draw[very thick, red, -<-] ($(P1)!.2!(P2)$) -- ($(P1)!.2!(P3)$);
		\draw[very thick, red, ->-] ($(P1)!.3!(P2)$) -- ($(P1)!.3!(P3)$);
		\draw[very thick, red, ->-] ($(P2)!.2!(P1)$) -- ($(P2)!.2!(P3)$);
		\draw[very thick, red, -<-] ($(P3)!.2!(P1)$) -- ($(P3)!.2!(P2)$);
		\draw[blue] (P1) -- (P2) -- (P3) -- (P1) -- cycle;
		\node at ($(P2)!.2!(P3)$) [right]{\scriptsize $7$};
		\node at ($(P1)!.8!(P3)$) [left]{\scriptsize $5$};
		\node at ($(P1)!.7!(P3)$) [left]{\scriptsize $1$};
		\node at ($(P1)!.6!(P3)$) [left]{\scriptsize $2$};
		\node at ($(P1)!.5!(P3)$) [left]{\scriptsize $5$};
		\node at ($(P1)!.4!(P3)$) [left]{\scriptsize $8$};
		\node at ($(P1)!.3!(P3)$) [left]{\scriptsize $3$};
		\node at ($(P1)!.2!(P3)$) [left]{\scriptsize $6$};
		\node at ($(P1)!.1!(P3)$) [left]{\scriptsize $4$};
		\node[draw, fill=black, circle, inner sep=1] at (P1) {};
		\node[draw, fill=black, circle, inner sep=1] at (P2) {};
		\node[draw, fill=black, circle, inner sep=1] at (P3) {};
		\node at ($(P1)!.5!(P2)$) [below=5pt] {$T'$};
	\end{tikzpicture}
	\caption{Fundamental pieces in triangles $T$ and $T'$, and an elevation-preserving braid in the biangle $B_E$ connecting them. Elevations are presented by positive integers.}
	\label{fig:elevation in T}
\end{figure}
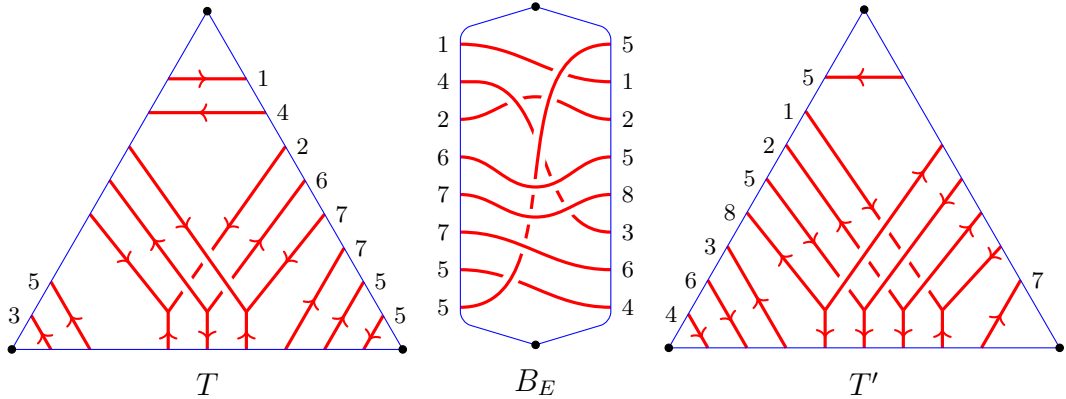

\begin{ex}
	For a triangulation $\Delta$, a simple trivalent graph obtained by attaching fundamental pieces to the triangles $t(\Delta^{\mathrm{split}})$ with no internal crossings and connecting them by identity braids in biangles in $b(\Delta^{\mathrm{split}})$ gives an elevation-preserving trivalent tangle with respect to $\Delta$.
	In particular, oriented simple loops and oriented simple arcs are elevation-preserving for any $\Delta$.
	For any triangulation $\Delta$, the $n$-bracelet along a (non-null homotopic) simple loop $\gamma$ ( \cref{fig:bracelet}) is obtained from the $n$-bangle of $\gamma$ by replacing the identity $n$-braid in some biangle by a braid corresponding to a cyclic permutation $(12\cdots n)$. See \cref{fig:bracelet-example}.
\end{ex}
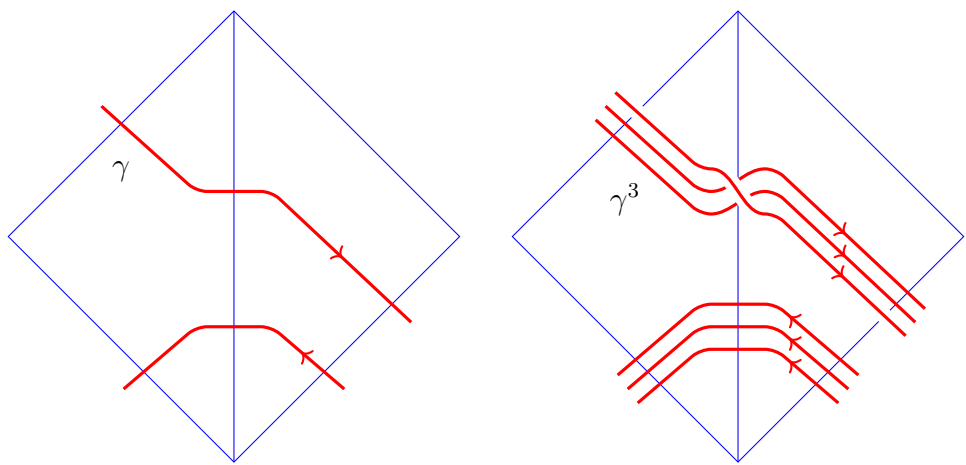
\begin{figure}
	\begin{tikzpicture}[scale=.1]
		\coordinate (P1) at (90:30);
		\coordinate (P2) at (180:30);
		\coordinate (P3) at (270:30);
		\coordinate (P4) at (0:30);
		\draw[blue] (P1) -- (P2) -- (P3) -- (P4) -- (P1) -- cycle;
		\draw[blue] (P1) -- (P3);
		\draw[very thick, red, ->-={.8}{}, rounded corners, shorten <=-10, shorten >=-10] ($(P1)!.5!(P2)$) -- ($(P1)!.4!(P3)-(5,0)$) -- ($(P1)!.4!(P3)+(5,0)$) -- ($(P4)!.3!(P3)$);
		\draw[very thick, red, -<-={.9}{}, rounded corners, shorten <=-10, shorten >=-10] ($(P2)!.6!(P3)$) -- ($(P1)!.7!(P3)-(5,0)$) -- ($(P1)!.7!(P3)+(5,0)$) -- ($(P4)!.6!(P3)$);
		\node at ($(P1)!.5!(P2)$) [below=10pt]{$\gamma$};
	\end{tikzpicture}
	\hspace{1em}
	\begin{tikzpicture}[scale=.1]
		\coordinate (P1) at (90:30);
		\coordinate (P2) at (180:30);
		\coordinate (P3) at (270:30);
		\coordinate (P4) at (0:30);
		\draw[blue] (P1) -- (P2) -- (P3) -- (P4) -- (P1) -- cycle;
		\draw[blue] (P1) -- (P3);
		\draw[very thick, red, ->-={.8}{red}, rounded corners, shorten <=-10, shorten >=-15] ($(P1)!.5!(P2)$) -- ($(P1)!.4!(P3)+(-5,0)$) to[out=east, in=west] ($(P1)!.4!(P3)+(5,3)$) -- ($(P4)!.3!(P3)+(0,3)$);
		\draw[very thick, red, ->-={.8}{red}, rounded corners, shorten <=-15, shorten >=-10] ($(P1)!.5!(P2)+(0,-3)$) -- ($(P1)!.4!(P3)+(-5,-3)$) to[out=east, in=west] ($(P1)!.4!(P3)+(5,0)$) -- ($(P4)!.3!(P3)$);
		\draw[very thick, red, ->-={.8}{red}, rounded corners, shorten <=-5, shorten >=-5, overarc] ($(P1)!.5!(P2)+(0,3)$) -- ($(P1)!.4!(P3)+(-5,3)$) to[out=east, in=west] ($(P1)!.4!(P3)+(5,-3)$) -- ($(P4)!.3!(P3)+(0,-3)$);
		\draw[very thick, red, -<-={.8}{}, rounded corners, shorten <=-10, shorten >=-10] ($(P2)!.6!(P3)$) -- ($(P1)!.7!(P3)+(-5,0)$) -- ($(P1)!.7!(P3)+(5,0)$) -- ($(P4)!.6!(P3)$);
		\draw[very thick, red, -<-={.8}{}, rounded corners, shorten <=-15, shorten >=-15] ($(P2)!.6!(P3)+(0,3)$) -- ($(P1)!.7!(P3)+(-5,3)$) -- ($(P1)!.7!(P3)+(5,3)$) -- ($(P4)!.6!(P3)+(0,3)$);
		\draw[very thick, red, -<-={.8}{}, rounded corners, shorten <=-5, shorten >=-5] ($(P2)!.6!(P3)+(0,-3)$) -- ($(P1)!.7!(P3)+(-5,-3)$) -- ($(P1)!.7!(P3)+(5,-3)$) -- ($(P4)!.6!(P3)+(0,-3)$);
		\node at ($(P1)!.5!(P2)+(0,-3)$) [below=10pt]{$\gamma^{3}$};
	\end{tikzpicture}
	\caption{The left-hand side shows a portion of the diagram of an oriented simple loop $\gamma$ in a quadrilateral in of a triangulation $\Delta$ of $\Sigma$. The right-hand side shows the associated $3$-bracelet.}
	\label{fig:bracelet-example}
\end{figure}

\begin{thm}\label{thm:elevation-preserving web}
	Let $\Sigma$ be any unpunctured marked surface and $\Delta$ its triangulation. 
	For any elevation-preserving web $x\in\SK{\Sigma}$ with respect to a triangulation $\Delta$, there exists $J_\Delta\in \mathrm{mon}(\Delta)$ such that the expansion of $xJ_\Delta$ in $\langle\cup_{T}\Eweb{T}\rangle_{\mathrm{alg}}$ has positive coefficients.
\end{thm}
\begin{proof}
	Let $x$ be an elevation-preserving web with respect to $\Delta$.
	Our strategy for the proof is the following. 
	Firstly, we decompose $x$ into webs in triangles of $\Delta^{\mathrm{split}}$ as in \cref{thm:elementary-web-expansion-T}, in order of increasing elevation.
	Next, expand the remaining part in biangles. Notice that the right-hand sides of the $\mathfrak{sl}_3$-skein relations have positive coefficients, except for \eqref{rel:bigon}, and we can avoid using this relation in the above process. 
	Therefore we can observe that the coefficients in these expansions are positive.
	
	Let us describe the details of the proof.
	We focus on a piece $G$ of $x$ in $T\in t(\Delta^{\mathrm{split}})$ and represent it by $G=G_{n}\dots G_{2}G_{1}$ as a superposition of connected components, where subscripts indicate their elevations. Here $G_{1}$ is the connected component of $G$ of the lowest elevation, and each $G_{i}$ is an arc or a trivalent graph with a single vertex.
	Let $\{p_1,p_2,p_3\}$ be the three special points of $T$, $E_{ij}$ the edge between $p_{i}$ and $p_{j}$, and $B_{ij}:=B_{E_{ij}}$. We will use the notation in \cref{subsec:triangle_skein} for the elementary webs in $T$.
	
	Firstly, expand $G_{1}$ by multiplying  $[e_{12}e_{21}e_{23}e_{32}e_{31}e_{13}]\in \mathrm{mon}(\Delta)$ and by using 
	\cref{lem:cuttingweb}.
	We remark that one can omit to multiply one of $e_{ij}e_{ji}$ if $G_{1}$ is an arc.
	In a neighborhood of each edge, the resulting diagrams in the expansion are decomposed into diagrams in the three parts: in the \emph{biangle part} (shown as a shaded region), on the \emph{edge part}, and the \emph{triangle part} (\emph{i.e.}, the interior of $T$) as follows.
	For a strand incoming to $T$,
	\begin{align*}
		\mathord{
			\ \tikz[baseline=-.6ex, scale=.1]{
			\node [blue, draw, fill=black, circle, inner sep=1.5] (A) at (-10,0) {};
			\node [blue, draw, fill=black, circle, inner sep=1.5] (B) at (10,0) {} edge[blue] (A);
			\coordinate (T1) at ($(A)!0.5!(B)+(0,7)$);
			\coordinate (T2) at ($(A)!0.5!(B)+(0,-7)$);
			\fill[blue!20] (A) -- (B) to[out=south, in=east] ($(T2)+(0,-3)$) to[out=west, in=south] (A) -- cycle;
			\draw[red, very thick, ->-={.8}{}] (B) to[out=north west,in=east] (0,3) to[out=west, in=north east] (A);
			\draw[red, very thick, -<-={.8}{}] (B) to[out=south west,in=east] (0,-3) to[out=west, in=south east] (A);
			\draw[overarc, ->-={.6}{red}] (0,-10) -- (0,10);
			\node at (T1) [left=5pt]{$T$};
			}
		\ }
		&=A^{3}\mathord{
			\ \tikz[baseline=-.6ex, scale=.1]{
			\node [draw, fill=black, circle, inner sep=1.5] (A) at (-10,0) {};
			\node [draw, fill=black, circle, inner sep=1.5] (B) at (10,0) {} edge[blue] (A);
			\coordinate (T1) at ($(A)!0.5!(B)+(0,7)$);
			\coordinate (T2) at ($(A)!0.5!(B)+(0,-7)$);
			\fill[blue!20] (A) -- (B) to[out=south, in=east] ($(T2)+(0,-3)$) to[out=west, in=south] (A) -- cycle;
			\draw[red, very thick, ->-] (A) -- (B);
			\draw[red, very thick, ->-, rounded corners] (B) -- (T1) -- ($(T1)+(0,3)$);
			\draw[red, very thick, -<-, rounded corners] (A) -- (T2) -- ($(T2)+(0,-3)$);
			\node at (T1) [left=5pt]{$T$};
			}
		\ }
		+\mathord{
			\ \tikz[baseline=-.6ex, scale=.1]{
			\node [draw, fill=black, circle, inner sep=1.5] (A) at (-10,0) {};
			\node [draw, fill=black, circle, inner sep=1.5] (B) at (10,0) {} edge[blue] (A);
			\coordinate (T1) at ($(A)!0.5!(B)+(0,7)$);
			\coordinate (T2) at ($(A)!0.5!(B)+(0,-7)$);
			\fill[blue!20] (A) -- (B) to[out=south, in=east] ($(T2)+(0,-3)$) to[out=west, in=south] (A) -- cycle;
			\draw[very thick, red, -<-] (A) -- (T1);
			\draw[very thick, red, -<-] (B) -- (T1);
			\draw[very thick, red, ->-] (A) -- (T2);
			\draw[very thick, red, ->-] (B) -- (T2);
			\draw[very thick, red, ->-] (T1) -- ($(T1)+(0,3)$);
			\draw[very thick, red, -<-] (T2) -- ($(T2)+(0,-3)$);
			\node at (T1) [left=5pt]{$T$};
			}
		\ }
		+A^{-3}\mathord{
			\ \tikz[baseline=-.6ex, scale=.1]{
			\node [draw, fill=black, circle, inner sep=1.5] (A) at (-10,0) {};
			\node [draw, fill=black, circle, inner sep=1.5] (B) at (10,0) {} edge[blue] (A);
			\coordinate (T1) at ($(A)!0.5!(B)+(0,7)$);
			\coordinate (T2) at ($(A)!0.5!(B)+(0,-7)$);
			\fill[blue!20] (A) -- (B) to[out=south, in=east] ($(T2)+(0,-3)$) to[out=west, in=south] (A) -- cycle;
			\draw[red, very thick, -<-] (A) -- (B);
			\draw[red, very thick, ->-, rounded corners] (A) -- (T1) -- ($(T1)+(0,3)$);
			\draw[red, very thick, -<-, rounded corners] (B) -- (T2) -- ($(T2)+(0,-3)$);
			\node at (T1) [left=5pt]{$T$};
			}
		\ },
	\end{align*}
	where the bottom-half belongs to a biangle part and the top-half does to one of the triangle parts covering $T$. The oriented edges between the special points belong to the edge part. 
	For a strand outgoing from $T$ (obtained by applying the Dynkin involution to the incoming case),
	\begin{align*}
		\mathord{
			\ \tikz[baseline=-.6ex, scale=.1]{
			\node [blue, draw, fill=black, circle, inner sep=1.5] (A) at (-10,0) {};
			\node [blue, draw, fill=black, circle, inner sep=1.5] (B) at (10,0) {} edge[blue] (A);
			\coordinate (T1) at ($(A)!0.5!(B)+(0,7)$);
			\coordinate (T2) at ($(A)!0.5!(B)+(0,-7)$);
			\fill[blue!20] (A) -- (B) to[out=south, in=east] ($(T2)+(0,-3)$) to[out=west, in=south] (A) -- cycle;
			\draw[red, very thick, ->-={.8}{}] (B) to[out=north west,in=east] (0,3) to[out=west, in=north east] (A);
			\draw[red, very thick, -<-={.8}{}] (B) to[out=south west,in=east] (0,-3) to[out=west, in=south east] (A);
			\draw[overarc, -<-={.6}{red}] (0,-10) -- (0,10);
			\node at (T1) [left=5pt]{$T$};
			}
		\ }
		&=A^{3}\mathord{
			\ \tikz[baseline=-.6ex, scale=.1]{
			\node [draw, fill=black, circle, inner sep=1.5] (A) at (-10,0) {};
			\node [draw, fill=black, circle, inner sep=1.5] (B) at (10,0) {} edge[blue] (A);
			\coordinate (T1) at ($(A)!0.5!(B)+(0,7)$);
			\coordinate (T2) at ($(A)!0.5!(B)+(0,-7)$);
			\fill[blue!20] (A) -- (B) to[out=south, in=east] ($(T2)+(0,-3)$) to[out=west, in=south] (A) -- cycle;
			\draw[red, very thick, -<-] (A) -- (B);
			\draw[red, very thick, -<-, rounded corners] (B) -- (T1) -- ($(T1)+(0,3)$);
			\draw[red, very thick, ->-, rounded corners] (A) -- (T2) -- ($(T2)+(0,-3)$);
			\node at (T1) [left=5pt]{$T$};
			}
		\ }
		+\mathord{
			\ \tikz[baseline=-.6ex, scale=.1]{
			\node [draw, fill=black, circle, inner sep=1.5] (A) at (-10,0) {};
			\node [draw, fill=black, circle, inner sep=1.5] (B) at (10,0) {} edge[blue] (A);
			\coordinate (T1) at ($(A)!0.5!(B)+(0,7)$);
			\coordinate (T2) at ($(A)!0.5!(B)+(0,-7)$);
			\fill[blue!20] (A) -- (B) to[out=south, in=east] ($(T2)+(0,-3)$) to[out=west, in=south] (A) -- cycle;
			\draw[very thick, red, ->-] (A) -- (T1);
			\draw[very thick, red, ->-] (B) -- (T1);
			\draw[very thick, red, -<-] (A) -- (T2);
			\draw[very thick, red, -<-] (B) -- (T2);
			\draw[very thick, red, -<-] (T1) -- ($(T1)+(0,3)$);
			\draw[very thick, red, ->-] (T2) -- ($(T2)+(0,-3)$);
			\node at (T1) [left=5pt]{$T$};
			}
		\ }
		+A^{-3}\mathord{
			\ \tikz[baseline=-.6ex, scale=.1]{
			\node [draw, fill=black, circle, inner sep=1.5] (A) at (-10,0) {};
			\node [draw, fill=black, circle, inner sep=1.5] (B) at (10,0) {} edge[blue] (A);
			\coordinate (T1) at ($(A)!0.5!(B)+(0,7)$);
			\coordinate (T2) at ($(A)!0.5!(B)+(0,-7)$);
			\fill[blue!20] (A) -- (B) to[out=south, in=east] ($(T2)+(0,-3)$) to[out=west, in=south] (A) -- cycle;
			\draw[red, very thick, ->-] (A) -- (B);
			\draw[red, very thick, -<-, rounded corners] (A) -- (T1) -- ($(T1)+(0,3)$);
			\draw[red, very thick, ->-, rounded corners] (B) -- (T2) -- ($(T2)+(0,-3)$);
			\node at (T1) [left=5pt]{$T$};
			}
		\ }.
	\end{align*}
	Then the webs in the expansion of $G_{1}[e_{12}e_{21}e_{23}e_{32}e_{31}e_{13}]$ in $T$ are obtained by concatenating the pieces in the three sectors $T_{12},T_{23},T_{31}$ shown in \cref{fig:triangle part}, and their coefficients are one of $\{1, A^{\pm 3}, A^{\pm 6}\}$.
	\begin{figure}
		\begin{tikzpicture}[scale=.1]
			\node [draw, fill=black, circle, inner sep=1.5] (A) at (210:20) {};
			\node [draw, fill=black, circle, inner sep=1.5] (B) at (330:20) {};
			\node [draw, fill=black, circle, inner sep=1.5] (C) at (90:20) {};
			\draw[dashed] (A) -- (0,0);
			\draw[dashed] (B) -- (0,0);
			\draw[dashed] (C) -- (0,0);
			\draw[blue] (A) -- (B) -- (C) -- (A) -- cycle;
			\node at (270:6) {$T_{12}$};
			\node at (30:6) {$T_{23}$};
			\node at (150:6) {$T_{31}$};
			\node at (A) [left]{$p_{1}$};
			\node at (B) [right]{$p_{2}$};
			\node at (C) [above]{$p_{3}$};
		\end{tikzpicture}
		\caption{Concatenation of webs in three region $X,Y,Z$ gives webs appearing in the expansion on $T$.}
		\label{fig:triangle part}
	\end{figure}
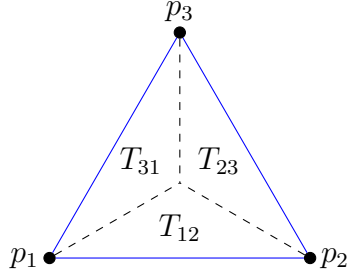
	In order to list the concatenation patterns, let us denote the resulting webs in each sector by
	\begin{align*}
		X_{ij}^{+}
		=\mathord{
			\ \tikz[baseline=-.6ex, scale=.1]{
			\node [draw, fill=black, circle, inner sep=1.5] (A) at (-10,0) {};
			\node [draw, fill=black, circle, inner sep=1.5] (B) at (10,0) {} edge[blue] (A);
			\coordinate (T1) at ($(A)!0.5!(B)+(0,7)$);
			\draw[red, very thick, ->-, rounded corners] (B) -- (T1) -- ($(T1)+(0,3)$);
			\node at (A) [below]{$p_{i}$};
			\node at (B) [below]{$p_{j}$};
			\node at (T1) [left=5pt]{$T_{ij}$};
			}
		\ },\quad
		X_{ij}^{0}
		=\mathord{
			\ \tikz[baseline=-.6ex, scale=.1]{
			\node [draw, fill=black, circle, inner sep=1.5] (A) at (-10,0) {};
			\node [draw, fill=black, circle, inner sep=1.5] (B) at (10,0) {} edge[blue] (A);
			\coordinate (T1) at ($(A)!0.5!(B)+(0,7)$);
			\draw[very thick, red, -<-] (A) -- (T1);
			\draw[very thick, red, -<-] (B) -- (T1);
			\draw[very thick, red, ->-] (T1) -- ($(T1)+(0,3)$);
			\node at (A) [below]{$p_{i}$};
			\node at (B) [below]{$p_{j}$};
			\node at (T1) [left=5pt]{$T_{ij}$};
			}
		\ },\quad
		X_{ij}^{-}
		+\mathord{
			\ \tikz[baseline=-.6ex, scale=.1]{
			\node [draw, fill=black, circle, inner sep=1.5] (A) at (-10,0) {};
			\node [draw, fill=black, circle, inner sep=1.5] (B) at (10,0) {} edge[blue] (A);
			\coordinate (T1) at ($(A)!0.5!(B)+(0,7)$);
			\draw[red, very thick, ->-, rounded corners] (A) -- (T1) -- ($(T1)+(0,3)$);
			\node at (T1) [left=5pt]{$T_{ij}$};
			}
		\ }.
	\end{align*}
	If $G_{1}$ is an arc connecting the edges $E_{ij}$ and $E_{jk}$, we just concatenate $X_{ij}^{\epsilon}$ and $\ast X_{jk}^{\epsilon'}$.
	If $G_{1}$ is a trivalent graph, the concatenation of the three pieces $X_{12}^{\epsilon}, X_{23}^{\epsilon'}, X_{31}^{\epsilon''}$ produce one new sink or source vertex at the center of $T$. 
	It is easy to confirm that these concatenations produce the following webs. 
	In the case that $G_{1}$ is an arc, we get
	\begin{align*}
		X_{12}^{+}\star(\ast X_{23}^{+})&=\Triweb{\tbwebneg{0}}{$e_{21}$},&
		X_{12}^{+}\star(\ast X_{23}^{0})&=\Triweb{\ttwebplus}{$t_{123}^{+}$},&
		X_{12}^{+}\star(\ast X_{23}^{-})&=\Triweb{\tbwebpos{120}}{$e_{23}$},\\
		X_{12}^{0}\star(\ast X_{23}^{+})&=0,&
		X_{12}^{0}\star(\ast X_{23}^{0})&=\Triweb{\tbwebpos{0}\tbwebpos{-120}}{$[e_{12}e_{31}]$},&
		X_{12}^{0}\star(\ast X_{23}^{-})&=\Triweb{\ttwebminus}{$t_{123}^{-}$},\\
		X_{12}^{-}\star(\ast X_{23}^{+})&=0,&
		X_{12}^{-}\star(\ast X_{23}^{0})&=0,&
		X_{12}^{-}\star(\ast X_{23}^{-})&=\Triweb{\tbwebneg{-120}}{$e_{13}$},
	\end{align*}
	\smallskip
	\begin{align*}
		X_{12}^{+}\star(\ast X_{31}^{+})&=\Triweb{\tbwebpos{120}}{$e_{23}$},&
		X_{12}^{+}\star(\ast X_{31}^{0})&=0,&
		X_{12}^{+}\star(\ast X_{31}^{-})&=0,\\
		X_{12}^{0}\star(\ast X_{31}^{+})&=\Triweb{\ttwebplus}{$t_{123}^{+}$},&
		X_{12}^{0}\star(\ast X_{31}^{0})&=\Triweb{\tbwebneg{120}\tbwebneg{0}}{$[e_{32}e_{21}]$},&
		X_{12}^{0}\star(\ast X_{31}^{-})&=0,\\
		X_{12}^{-}\star(\ast X_{31}^{+})&=\Triweb{\tbwebneg{-120}}{$e_{13}$},&
		X_{12}^{-}\star(\ast X_{31}^{0})&=\Triweb{\ttwebminus}{$t_{123}^{-}$},&
		X_{12}^{-}\star(\ast X_{31}^{-})&=\Triweb{\tbwebpos{0}}{$e_{12}$},
	\end{align*}
	where $X\star Y$ means the concatenation of $X$ and $Y$.
	In the case that $G_{1}$ is a trivalent graph, we get
	\begin{align*}
		X_{12}^{+}\star X_{23}^{+} \star X_{31}^{+} &=\Triweb{\ttwebplus}{$t_{123}^{+}$},&
		X_{12}^{+}\star X_{23}^{+} \star X_{31}^{0} &=\Triweb{\tbwebpos{120}\tbwebpos{-120}}{$[e_{23}e_{31}]$},\\
		X_{12}^{+}\star X_{23}^{+} \star X_{31}^{-} &=0,&
		X_{12}^{+}\star X_{23}^{0} \star X_{31}^{0} &=\Triweb{\tbwebpos{120}\ttwebminus}{$[e_{23}t_{123}^{-}]$},\\
		X_{12}^{+}\star X_{23}^{0} \star X_{31}^{-} &=\Triweb{\tbwebpos{120}\tbwebneg{120}}{$[e_{23}e_{32}]$},&
		X_{12}^{+}\star X_{23}^{-} \star X_{31}^{-} &=0,\\
		X_{12}^{0}\star X_{23}^{0} \star X_{31}^{0} &=\Triweb{\ttwebminus\ttwebminusdown}{$[t_{123}^{+}t_{123}^{+}]$},&
		X_{12}^{0}\star X_{23}^{0} \star X_{31}^{-} &=\Triweb{\ttwebminus\tbwebneg{120}}{$[t_{123}^{+}e_{21}]$},\\
		X_{12}^{0}\star X_{23}^{-} \star X_{31}^{-} &=\Triweb{\tbwebneg{120}\tbwebneg{0}}{$[e_{32}e_{21}]$},&
		X_{12}^{-}\star X_{23}^{-} \star X_{31}^{-} &=\Triweb{\ttwebplus}{$t_{123}^{+}$},
	\end{align*}
	where $X\star Y\star Z$ means the concatenation of $X$, $Y$ and $Z$. Here we have applied some skein relations. 
	For example,
	\begin{align*}
		X_{12}^{+}\star X_{23}^{0}\star X_{31}^{0}
		=\mathord{\ 
			\tikz[baseline=-.6ex, scale=.1]{
				\coordinate (P1) at (210:10);
				\coordinate (P2) at (330:10);
				\coordinate (P3) at (90:10);
				\coordinate (T1) at (270:2);
				\coordinate (T2) at (30:2);
				\coordinate (T3) at (150:2);
				\draw[gray] (P1) -- (P2) -- (P3) -- (P1) -- cycle;
				\draw[very thick, red] (0,0) -- (T1);
				\draw[very thick, red] (0,0) -- (T2);
				\draw[very thick, red] (0,0) -- (T3);
				\draw[very thick, red, ->-] (P2) to[out=west, in=south] (T1);
				\draw[very thick, red, -<-] (P2) -- (T2);
				\draw[very thick, red, -<-] (P3) -- (T2);
				\draw[very thick, red, -<-] (P3) -- (T3);
				\draw[very thick, red, -<-] (P1) -- (T3);
				\node[draw, fill=black, circle, inner sep=1] at (210:10) {};
				\node[draw, fill=black, circle, inner sep=1] at (330:10) {};
				\node[draw, fill=black, circle, inner sep=1] at (90:10) {};
			}
		\ }
		=\mathord{\ 
			\tikz[baseline=-.6ex, scale=.1]{
				\coordinate (P1) at (210:10);
				\coordinate (P2) at (330:10);
				\coordinate (P3) at (90:10);
				\coordinate (T1) at (270:2);
				\coordinate (T2) at (30:2);
				\coordinate (T3) at (150:2);
				\draw[gray] (P1) -- (P2) -- (P3) -- (P1) -- cycle;
				\draw[very thick, red, ->-] (0,0) -- (P1);
				\draw[very thick, red, ->-] (0,0) -- (P2);
				\draw[very thick, red, ->-] (0,0) -- (P3);
				\draw[very thick, red, ->-] (P2) -- (P3);
				\node[draw, fill=black, circle, inner sep=1] at (210:10) {};
				\node[draw, fill=black, circle, inner sep=1] at (330:10) {};
				\node[draw, fill=black, circle, inner sep=1] at (90:10) {};
			}
		\ }.
	\end{align*}
	The above calculation shows that the triangle part of $G_{1}[e_{12}e_{21}e_{23}e_{32}e_{31}e_{13}]$ in $T\in t(\Delta^{\mathrm{split}})$ is expanded as a polynomial in $\Eweb{T}$ with positive coefficients.
	The webs appearing in this expansion $A$-commute with webs along the edges of $T$, since the biangle part, edge apart, and triangle part $A$-commute with each other.
	Therefore in the product 
	\begin{align*}
	    G[e_{12}e_{21}e_{23}e_{32}e_{31}e_{13}]^{2}&=G_{n}\cdots G_{2}G_{1}[e_{12}e_{21}e_{23}e_{32}e_{31}e_{13}]^{2} \\
	    &=G_{n}\cdots G_{2}(G_{1}[e_{12}e_{21}e_{23}e_{32}e_{31}e_{13}])[e_{12}e_{21}e_{23}e_{32}e_{31}e_{13}],
	\end{align*}
	we can move $[e_{12}e_{21}e_{23}e_{32}e_{31}e_{13}]$ to the left beyond $G_{1}[e_{12}e_{21}e_{23}e_{32}e_{31}e_{13}]$ preserving the positivity of coefficients in the expansion, and hence the component $G_{2}$ can be expanded in the same way. 
	Proceeding in this way, one can decompose $G[e_{12}e_{21}e_{23}e_{32}e_{31}e_{13}]^{n}$ into webs in $T\in\Delta^{\mathrm{split}}$ and webs outside of $T$ with positive coefficients.
	
	Applying this operation to $x$ for all $T\in\Delta^{\mathrm{split}}$, we obtain a positive sum of webs such that
	\begin{itemize}
		\item their triangle parts (and edge parts) in $T\in t(\Delta^{\mathrm{split}})$ are expressed as monomials in $\Eweb{T}$,
		\item concatenations of elevation-preserving braids and biangle parts produced in the expansion procedure.
	\end{itemize}
	It remains to show that a web in each biangle becomes a polynomial in $\Eweb{T}$, especially $\Eweb{\partial^{\times}}$, with positive coefficients.
	In the expansion of $G$, webs in a biangle part adjacent to $T$ inherit the elevation from the fundamental piece in $T$.
	Hence from the elevation-preserving assumption, for adjacent triangles $T, T'\in t(\Delta^{\mathrm{split}})$, the biangle parts of $T$ and $T'$ can be connected by strands with preserving their elevations.
	Hence, the concatenation of biangle parts and the elevation-preserving braid is presented as a superposition of the concatenation of biangle parts, as listed below:
	\begin{align*}
		\mathord{\
			\tikz[baseline=-.6ex, scale=.1]{
				\coordinate (P1) at (-10,0);
				\coordinate (P2) at (10,0);
				\coordinate (U) at (0,7);
				\coordinate (D) at (0,-7);
				\coordinate (T1) at (0,5);
				\coordinate (T2) at (0,-5);
				\fill[blue!20] (P1) to[out=north, in=west] (U) to[out=east, in=north] (P2) to[out=south, in=east] (D) to[out=west, in=south] (P1) -- cycle;
				\draw[very thick, red, -<-, rounded corners] (P1) --(T1) -- (0,0);
				\draw[very thick, red, ->-, rounded corners] (P2) -- (T2) -- (0,0);
				\draw[blue] (P1) to[out=north, in=west] (U) to[out=east, in=north] (P2);
				\draw[blue] (P1) to[out=south, in=west] (D) to[out=east, in=south] (P2);
				\node[draw, fill=black, circle, inner sep=1] at (P1) {};
				\node[draw, fill=black, circle, inner sep=1] at (P2) {};
				\node at (P1) [left]{\scriptsize $p_{1}$};
				\node at (P2) [right]{\scriptsize $p_{2}$};
			}
		}&=e_{21},&
		\mathord{\
			\tikz[baseline=-.6ex, scale=.1]{
				\coordinate (P1) at (-10,0);
				\coordinate (P2) at (10,0);
				\coordinate (U) at (0,7);
				\coordinate (D) at (0,-7);
				\coordinate (T1) at (0,5);
				\coordinate (T2) at (0,-5);
				\fill[blue!20] (P1) to[out=north, in=west] (U) to[out=east, in=north] (P2) to[out=south, in=east] (D) to[out=west, in=south] (P1) -- cycle;
				\draw[very thick, red, -<-, rounded corners] (P1) --(T1) -- (0,0);
				\draw[very thick, red, -<-] (P2) -- (T2);
				\draw[very thick, red, -<-] (P1) -- (T2);
				\draw[very thick, red, ->-] (T2) -- (0,0);
				\draw[blue] (P1) to[out=north, in=west] (U) to[out=east, in=north] (P2);
				\draw[blue] (P1) to[out=south, in=west] (D) to[out=east, in=south] (P2);
				\node[draw, fill=black, circle, inner sep=1] at (P1) {};
				\node[draw, fill=black, circle, inner sep=1] at (P2) {};
				\node at (P1) [left]{\scriptsize $p_{1}$};
				\node at (P2) [right]{\scriptsize $p_{2}$};
			}
		\ }&=0,&
		\mathord{\
			\tikz[baseline=-.6ex, scale=.1]{
				\coordinate (P1) at (-10,0);
				\coordinate (P2) at (10,0);
				\coordinate (U) at (0,7);
				\coordinate (D) at (0,-7);
				\coordinate (T1) at (0,5);
				\coordinate (T2) at (0,-5);
				\fill[blue!20] (P1) to[out=north, in=west] (U) to[out=east, in=north] (P2) to[out=south, in=east] (D) to[out=west, in=south] (P1) -- cycle;
				\draw[very thick, red, -<-, rounded corners] (P1) --(T1) -- (0,0);
				\draw[very thick, red, ->-, rounded corners] (P1) -- (T2) -- (0,0);
				\draw[blue] (P1) to[out=north, in=west] (U) to[out=east, in=north] (P2);
				\draw[blue] (P1) to[out=south, in=west] (D) to[out=east, in=south] (P2);
				\node[draw, fill=black, circle, inner sep=1] at (P1) {};
				\node[draw, fill=black, circle, inner sep=1] at (P2) {};
				\node at (P1) [left]{\scriptsize $p_{1}$};
				\node at (P2) [right]{\scriptsize $p_{2}$};
			}
		}&=0,\\
		\mathord{\
			\tikz[baseline=-.6ex, scale=.1]{
				\coordinate (P1) at (-10,0);
				\coordinate (P2) at (10,0);
				\coordinate (U) at (0,7);
				\coordinate (D) at (0,-7);
				\coordinate (T1) at (0,5);
				\coordinate (T2) at (0,-5);
				\fill[blue!20] (P1) to[out=north, in=west] (U) to[out=east, in=north] (P2) to[out=south, in=east] (D) to[out=west, in=south] (P1) -- cycle;
				\draw[very thick, red, ->-] (P1) --(T1);
				\draw[very thick, red, ->-] (P2) --(T1);
				\draw[very thick, red, ->-] (0,0) --(T1);
				\draw[very thick, red, ->-, rounded corners] (P2) -- (T2) -- (0,0);
				\draw[blue] (P1) to[out=north, in=west] (U) to[out=east, in=north] (P2);
				\draw[blue] (P1) to[out=south, in=west] (D) to[out=east, in=south] (P2);
				\node[draw, fill=black, circle, inner sep=1] at (P1) {};
				\node[draw, fill=black, circle, inner sep=1] at (P2) {};
				\node at (P1) [left]{\scriptsize $p_{1}$};
				\node at (P2) [right]{\scriptsize $p_{2}$};
			}
		}&=0,&
		\mathord{\
			\tikz[baseline=-.6ex, scale=.1]{
				\coordinate (P1) at (-10,0);
				\coordinate (P2) at (10,0);
				\coordinate (U) at (0,7);
				\coordinate (D) at (0,-7);
				\coordinate (T1) at (0,5);
				\coordinate (T2) at (0,-5);
				\fill[blue!20] (P1) to[out=north, in=west] (U) to[out=east, in=north] (P2) to[out=south, in=east] (D) to[out=west, in=south] (P1) -- cycle;
				\draw[very thick, red, ->-] (P1) --(T1);
				\draw[very thick, red, ->-] (P2) --(T1);
				\draw[very thick, red, ->-] (0,0) --(T1);
				\draw[very thick, red, -<-] (P2) -- (T2);
				\draw[very thick, red, -<-] (P1) -- (T2);
				\draw[very thick, red, ->-] (T2) -- (0,0);
				\draw[blue] (P1) to[out=north, in=west] (U) to[out=east, in=north] (P2);
				\draw[blue] (P1) to[out=south, in=west] (D) to[out=east, in=south] (P2);
				\node[draw, fill=black, circle, inner sep=1] at (P1) {};
				\node[draw, fill=black, circle, inner sep=1] at (P2) {};
				\node at (P1) [left]{\scriptsize $p_{1}$};
				\node at (P2) [right]{\scriptsize $p_{2}$};
			}
		\ }&=[e_{12}e_{21}],&
		\mathord{\
			\tikz[baseline=-.6ex, scale=.1]{
				\coordinate (P1) at (-10,0);
				\coordinate (P2) at (10,0);
				\coordinate (U) at (0,7);
				\coordinate (D) at (0,-7);
				\coordinate (T1) at (0,5);
				\coordinate (T2) at (0,-5);
				\fill[blue!20] (P1) to[out=north, in=west] (U) to[out=east, in=north] (P2) to[out=south, in=east] (D) to[out=west, in=south] (P1) -- cycle;
				\draw[very thick, red, ->-] (P1) --(T1);
				\draw[very thick, red, ->-] (P2) --(T1);
				\draw[very thick, red, ->-] (0,0) --(T1);
				\draw[very thick, red, ->-, rounded corners] (P1) -- (T2) -- (0,0);
				\draw[blue] (P1) to[out=north, in=west] (U) to[out=east, in=north] (P2);
				\draw[blue] (P1) to[out=south, in=west] (D) to[out=east, in=south] (P2);
				\node[draw, fill=black, circle, inner sep=1] at (P1) {};
				\node[draw, fill=black, circle, inner sep=1] at (P2) {};
				\node at (P1) [left]{\scriptsize $p_{1}$};
				\node at (P2) [right]{\scriptsize $p_{2}$};
			}
		}&=0,\\
		\mathord{\
			\tikz[baseline=-.6ex, scale=.1]{
				\coordinate (P1) at (-10,0);
				\coordinate (P2) at (10,0);
				\coordinate (U) at (0,7);
				\coordinate (D) at (0,-7);
				\coordinate (T1) at (0,5);
				\coordinate (T2) at (0,-5);
				\fill[blue!20] (P1) to[out=north, in=west] (U) to[out=east, in=north] (P2) to[out=south, in=east] (D) to[out=west, in=south] (P1) -- cycle;
				\draw[very thick, red, -<-, rounded corners] (P2) --(T1) -- (0,0);
				\draw[very thick, red, ->-, rounded corners] (P2) -- (T2) -- (0,0);
				\draw[blue] (P1) to[out=north, in=west] (U) to[out=east, in=north] (P2);
				\draw[blue] (P1) to[out=south, in=west] (D) to[out=east, in=south] (P2);
				\node[draw, fill=black, circle, inner sep=1] at (P1) {};
				\node[draw, fill=black, circle, inner sep=1] at (P2) {};
				\node at (P1) [left]{\scriptsize $p_{1}$};
				\node at (P2) [right]{\scriptsize $p_{2}$};
			}
		}&=0,&
		\mathord{\
			\tikz[baseline=-.6ex, scale=.1]{
				\coordinate (P1) at (-10,0);
				\coordinate (P2) at (10,0);
				\coordinate (U) at (0,7);
				\coordinate (D) at (0,-7);
				\coordinate (T1) at (0,5);
				\coordinate (T2) at (0,-5);
				\fill[blue!20] (P1) to[out=north, in=west] (U) to[out=east, in=north] (P2) to[out=south, in=east] (D) to[out=west, in=south] (P1) -- cycle;
				\draw[very thick, red, -<-, rounded corners] (P2) --(T1) -- (0,0);
				\draw[very thick, red, -<-] (P2) -- (T2);
				\draw[very thick, red, -<-] (P1) -- (T2);
				\draw[very thick, red, ->-] (T2) -- (0,0);
				\draw[blue] (P1) to[out=north, in=west] (U) to[out=east, in=north] (P2);
				\draw[blue] (P1) to[out=south, in=west] (D) to[out=east, in=south] (P2);
				\node[draw, fill=black, circle, inner sep=1] at (P1) {};
				\node[draw, fill=black, circle, inner sep=1] at (P2) {};
				\node at (P1) [left]{\scriptsize $p_{1}$};
				\node at (P2) [right]{\scriptsize $p_{2}$};
			}
		\ }&=0,&
		\mathord{\
			\tikz[baseline=-.6ex, scale=.1]{
				\coordinate (P1) at (-10,0);
				\coordinate (P2) at (10,0);
				\coordinate (U) at (0,7);
				\coordinate (D) at (0,-7);
				\coordinate (T1) at (0,5);
				\coordinate (T2) at (0,-5);
				\fill[blue!20] (P1) to[out=north, in=west] (U) to[out=east, in=north] (P2) to[out=south, in=east] (D) to[out=west, in=south] (P1) -- cycle;
				\draw[very thick, red, -<-, rounded corners] (P2) --(T1) -- (0,0);
				\draw[very thick, red, ->-, rounded corners] (P1) -- (T2) -- (0,0);
				\draw[blue] (P1) to[out=north, in=west] (U) to[out=east, in=north] (P2);
				\draw[blue] (P1) to[out=south, in=west] (D) to[out=east, in=south] (P2);
				\node[draw, fill=black, circle, inner sep=1] at (P1) {};
				\node[draw, fill=black, circle, inner sep=1] at (P2) {};
				\node at (P1) [left]{\scriptsize $p_{1}$};
				\node at (P2) [right]{\scriptsize $p_{2}$};
			}
		}&=e_{12}.\\
	\end{align*}
	Consequently, $x$ is decomposed into a sum of monomials in $\cup_{T\in t(\Delta)}\Eweb{T}$ such that its coefficients are positive Laurent polynomial in $\bZ_{A}$.
\end{proof}

\begin{cor}\label{cor:elevation-preserving web in cluster}
	Let $\bD=(\Delta,\bs_{\Delta})$ be a decorated triangulation of $\Sigma$.
	Then, for any elevation-preserving web $x\in\SK{\Sigma}$ with respect to $\Delta$, there exists $J_{\bD}\in \mathrm{mon}(C_{\bD})$ such that the expansion of $xJ_{\bD}$ in $\langle C_{\bD}\rangle_{\mathrm{alg}}$ has positive coefficients.
\end{cor}
\begin{proof}
	By \cref{thm:elevation-preserving web}, an elevation-preserving web is expanded as a positive polynomial in $\cup_{T\in t(\Delta)}\Eweb{T}$ by multiplying an appropriate element in $\mathrm{mon}(\Delta)$.
	In a similar way to \cref{cor:web-cluster-expansion-T}, we can further expand it as a polynomial in a web cluster $C_{\bD}$ with positive coefficients by multiplying appropriate elementary webs in triangles.
\end{proof}

\begin{cor}\label{cor:simple web}
	Let $\gamma$ be an oriented simple loop in $\Sigma$.
	Then for any $\bD$ and $n\in\bN$, the $n$-bracelet and $n$-bangle are expressed as polynomial with positive coefficients in $\langle C_{\bD}\rangle_{\mathrm{alg}}$ by multiplying some monomial in $\mathrm{mon}(C_{\bD})$.
\end{cor}

\subsubsection{The $\Delta$-lacalization of $\SK{\Sigma}$}
\label{subsec:localization}
In \cref{thm:elementary-web-expansion-T,thm:elevation-preserving web}, we expanded any $\mathfrak{sl}_{3}$-webs in $\SK{\Sigma}$ by multiplying some monomials.
We are going to see that these expansions give rise to expressions of $\mathfrak{sl}_{3}$-webs as Laurent polynomials in suitable localizations of $\SK{\Sigma}$.

\begin{lem}
	The multiplicatively closed set $\mathrm{mon}(\Delta)$ in $\SK{\Sigma}$ satisfies the Ore condition. 
\end{lem}
\begin{proof}
	We first show the right Ore condition that for any web $x\in\SK{\Sigma}$ and monomial $J\in\mathrm{mon}(\Delta)$, there exist a web $x'\in\SK{\Sigma}$ and a monomial $J'\in\mathrm{mon}(\Delta)$ such that $xJ'=Jx'$.
	By \cref{thm:elementary-web-expansion-T}, there exist a monomial $J''\in\mathrm{mon}(\Delta)$ such that
	\[
		xJ''=\sum_{f}\lambda_{f}\bm{e}^{f}\in\langle\cup_{T\in t(\Delta)}\Eweb{T}\rangle_{\mathrm{alg}}
	\]
	where $f\colon\cup_{T\in t(\Delta)}\Eweb{T}\to\bN$ and $\bm{e}^{f}:=\prod_{i\in I(\Delta)}e_{i}^{f(i)}$.
	Since any monomial in $\mathrm{mon}(\Delta)$ is $A$-commutative with $\bm{e}^{f}$ by \cref{lem:str const T}, we obtain the following:
	\begin{align*}
		xJ''J=\left(\sum_{f}\lambda_{f}\bm{e}^{f}\right)J
		=\sum_{f}\lambda_{f}\bm{e}^{f}J
		=\sum_{f}\lambda_{f}A^{n(f,J)}J\bm{e}^{f}
		=J\left(\sum_{f}A^{n(f,J)}\lambda_{f}\bm{e}^{f}\right),
	\end{align*}
	where $n(f,J)$ is some half integer. In other words,  $xJ'=Jx'$ holds with $J':=J''J$ and $x':=\sum_{f}A^{n(f,J)}\lambda_{f}\bm{e}^{f}$. 
	By applying the mirror-reflection $\dagger$, we see that the left Ore condition also holds.
\end{proof}

\begin{dfn}[the localized $\mathfrak{sl}_3$-skein algebras for $(\Sigma,\bM)$]\label{def:localized_skein_algebras}
    The \emph{$\Delta$-localized skein algebra $\SK{\Sigma}[\Delta^{-1}]$} is the Ore localization of $\SK{\Sigma}$ by the Ore set $\mathrm{mon}(\Delta)$.  Similarly, the \emph{$\partial$-localized skein algebra $\SK{\Sigma}[\partial^{-1}]$} is the Ore localization by $\mathrm{mon}(\Eweb{\partial^{\times}\Sigma})$.
\end{dfn}

The following theorem guarantees the existence of the skew-field of fractions $\mathrm{Frac}\SK{\Sigma}$ and embeddings of the above localizations of $\SK{\Sigma}$.

\begin{thm}[\cite{IY-stateclasp}]\label{thm:localization}
	$\SK{\Sigma}$ is an Ore domain.
\end{thm}
\begin{proof}[Sketch of Proof]
    We use an isomorphism between the $\partial$-localized skein algebra $\SK{\Sigma}[\partial^{-1}]$ and the \emph{reduced stated skein algebra $\mathscr{S}^{\mathsf{st}}_{\mathfrak{sl}_3}(\Sigma)_{\mathrm{red}}$} and the splitting property of the (reduced) stated skein algebra.
    The stated skein algebra $\mathscr{S}^{\mathsf{st}}_{\mathfrak{sl}_3}(\Sigma)$ is defined by Higgins~\cite{Higgins20}.
    He also showed injectivity of the \emph{splitting homomorphism} $\theta_{\alpha}\colon\mathscr{S}^{\mathsf{st}}_{\mathfrak{sl}_3}(\Sigma)\to\mathscr{S}^{\mathsf{st}}_{\mathfrak{sl}_3}(\Sigma')$ where $\Sigma'$ is an unpunctured marked surface obtained by cutting $\Sigma$ along an ideal arc $\alpha$.
    The authors defined the reduced version of the stated skein algebra $\mathscr{S}^{\mathsf{st}}_{\mathfrak{sl}_3}(\Sigma)_{\mathrm{red}}$ and an isomorphism $\Phi_{\Sigma}\colon\mathscr{S}^{\mathsf{st}}_{\mathfrak{sl}_3}(\Sigma)_{\mathrm{red}}\to\SK{\Sigma}[\partial^{-1}]$ in \cite{IY-stateclasp}.
    Moreover, the reduced stated skein algebra inherits the splitting property and its injectivity.
    For a given ideal triangulation $\Delta=\{\alpha_1,\alpha_2,\dots,\alpha_n\}$ of $\Sigma$, one can consider the composition of the above maps:
    \[
    \SK{\Sigma}\xhookrightarrow[]{}\SK{\Sigma}[\partial^{-1}]\xrightarrow[\Phi^{-1}_{\Sigma}]{\cong}\mathscr{S}^{\mathsf{st}}_{\mathfrak{sl}_3}(\Sigma)_{\mathrm{red}}\xhookrightarrow[\theta_{\Delta}^{\mathrm{red}}]{}\bigotimes_{T\in t(\Delta)}\mathscr{S}^{\mathsf{st}}_{\mathfrak{sl}_3}(T)_{\mathrm{red}}\xrightarrow[\bigotimes_{T}\Phi_{T}]{\cong}\bigotimes_{T\in t(\Delta)}\SK{T}[\partial^{-1}]
    \]
    where $\theta_{\Delta}^{\mathrm{red}}:=\theta_{\alpha_n}^{\mathrm{red}}\circ\cdots\circ\theta_{\alpha_2}^{\mathrm{red}}\circ\theta_{\alpha_1}^{\mathrm{red}}$ is a composition of splitting homomorphisms for reduced stated skein algebras.
    Note that we use $\mathscr{S}^{\mathsf{st}}_{\mathfrak{sl}_3}(\sqcup_{T\in t(\Delta)}T)_{\mathrm{red}}\cong\bigotimes_{T\in t(\Delta)}\mathscr{S}^{\mathsf{st}}_{\mathfrak{sl}_3}(T)_{\mathrm{red}}$.
    The right-most algebra is an Ore domain.
    Then we have an embedding of $\SK{\Sigma}$ into an Ore domain, which implies that $\Skein{\Sigma}$ is an Ore domain.
\end{proof}.


\begin{cor}\label{cor:localization}We have inclusions
	\[
		\SK{\Sigma}\subset\SK{\Sigma}[\partial^{-1}]\subset\SK{\Sigma}[\Delta^{-1}]\subset\mathrm{Frac}\SK{\Sigma}.
	\]
\end{cor}

\subsection{Expansions on boundary intervals}

As we have seen in the previous subsection, we can obtain a Laurent expression of a given $\mathfrak{sl}_3$-web in $\SK{\Sigma}[\Delta^{-1}]$ by cutting it along the triangulation $\Delta$.
In this subsection, we are going to give a way to obtain a Laurent expression of an $\mathfrak{sl}_3$-web in $\SK{\Sigma}[\partial^{-1}]$. 
To obtain such an expansion, we make an $\mathfrak{sl}_3$-web stick to boundary intervals by the following lemma.
\begin{lem}[The sticking trick]\label{lem:stickto}
	\begin{align}
		\mathord{
			\ \tikz[baseline=-.6ex, scale=.1, yshift=-5cm]{
				\coordinate (A) at (-10,0);
				\coordinate (B) at (10,0);
				\fill[lightgray] (A) -- (B) -- ($(B)+(0,-3)$) -- ($(A)+(0,-3)$) -- (A) -- cycle;
				\coordinate (T11) at ($(A)!0.2!(B)+(0,15)$);
				\coordinate (T21) at ($(A)!0.2!(B)+(0,10)$);
				\coordinate (T12) at ($(A)!0.8!(B)+(0,15)$);
				\coordinate (T22) at ($(A)!0.8!(B)+(0,10)$);
				\draw[red, very thick, ->-={.5}{}] (B) to[out=north,in=east] (0,5) to[out=west, in=north] (A);
				\draw[red, very thick, -<-={.5}{}] (B) to[out=north west,in=east] (0,3) to[out=west, in=north east] (A);
				\draw[overarc, ->-={.5}{red}, rounded corners] (T11) -- (T21) -- (T22) -- (T12);
				\draw[fill] (A) circle (20pt);
				\draw[fill] (B) circle (20pt);
				\draw[very thick] (A) -- (B);
			}
		\ }
		&=A^{6}
		\mathord{
			\ \tikz[baseline=-.6ex, scale=.1, yshift=-5cm]{
				\coordinate (A) at (-10,0);
				\coordinate (B) at (10,0);
				\fill[lightgray] (A) -- (B) -- ($(B)+(0,-3)$) -- ($(A)+(0,-3)$) -- (A) -- cycle;
				\coordinate (T11) at ($(A)!0.2!(B)+(0,15)$);
				\coordinate (T21) at ($(A)!0.2!(B)+(0,10)$);
				\coordinate (T12) at ($(A)!0.8!(B)+(0,15)$);
				\coordinate (T22) at ($(A)!0.8!(B)+(0,10)$);
				\draw[red, very thick, ->-] (T11) -- (A);
				\draw[red, very thick, -<-] (T12) -- (B);
				\draw[red, very thick, -<-={.5}{}] (B) to[out=north west,in=east] (0,3) to[out=west, in=north east] (A);
				\draw[fill] (A) circle (20pt);
				\draw[fill] (B) circle (20pt);
				\draw[very thick] (A) -- (B);
		}
		\ }
		-A^{5}\mathord{
			\ \tikz[baseline=-.6ex, scale=.1, yshift=-5cm]{
				\coordinate (A) at (-10,0);
				\coordinate (B) at (10,0);
				\fill[lightgray] (A) -- (B) -- ($(B)+(0,-3)$) -- ($(A)+(0,-3)$) -- (A) -- cycle;
				\coordinate (T11) at ($(A)!0.2!(B)+(0,15)$);
				\coordinate (T21) at ($(A)!0.2!(B)+(0,10)$);
				\coordinate (T12) at ($(A)!0.8!(B)+(0,15)$);
				\coordinate (T22) at ($(A)!0.8!(B)+(0,10)$);
			\draw[very thick, red, ->-] (T11) -- (T21);
			\draw[very thick, red, -<-] (T12) -- (T22);
			\draw[very thick, red, ->-={.5}{red}, overarc] (B) -- (T21);
			\draw[very thick, red, -<-={.5}{red}, overarc] (A) -- (T22);
			\draw[very thick, red, ->-={.5}{red}] (A) -- (T21);
			\draw[very thick, red, -<-] (B) -- (T22);
			\draw[fill] (A) circle (20pt);
			\draw[fill] (B) circle (20pt);
			\draw[very thick] (A) -- (B);
		}
		\ }
		+A^{2}\mathord{
			\ \tikz[baseline=-.6ex, scale=.1, yshift=-5cm]{
				\coordinate (A) at (-10,0);
				\coordinate (B) at (10,0);
				\fill[lightgray] (A) -- (B) -- ($(B)+(0,-3)$) -- ($(A)+(0,-3)$) -- (A) -- cycle;
				\coordinate (T11) at ($(A)!0.2!(B)+(0,15)$);
				\coordinate (T21) at ($(A)!0.2!(B)+(0,10)$);
				\coordinate (T12) at ($(A)!0.8!(B)+(0,15)$);
				\coordinate (T22) at ($(A)!0.8!(B)+(0,10)$);
			\draw[red, very thick, ->-={.6}{red}, overarc] (T11) -- (B);
			\draw[red, very thick, -<-={.6}{red}, overarc] (T12) -- (A);
			\draw[red, very thick, ->-={.5}{}] (B) to[out=north west,in=east] (0,3) to[out=west, in=north east] (A);
			\draw[fill] (A) circle (20pt);
			\draw[fill] (B) circle (20pt);
			\draw[very thick] (A) -- (B);
		}
		\ }
	\end{align}
\end{lem}
\begin{proof}
	We can deform the left-hand side as
	\[ 
		\mathord{
			\ \tikz[baseline=-.6ex, scale=.1, yshift=-5cm]{
				\coordinate (A) at (-10,0);
				\coordinate (B) at (10,0);
				\fill[lightgray] (A) -- (B) -- ($(B)+(0,-3)$) -- ($(A)+(0,-3)$) -- (A) -- 	cycle;
				\coordinate (T11) at ($(A)!0.2!(B)+(0,15)$);
				\coordinate (T21) at ($(A)!0.2!(B)+(0,10)$);
				\coordinate (T12) at ($(A)!0.8!(B)+(0,15)$);
				\coordinate (T22) at ($(A)!0.8!(B)+(0,10)$);
				\draw[red, very thick, ->-={.5}{}] (B) to[out=north,in=east] (0,5) to	[out=west, in=north] (A);
				\draw[red, very thick, -<-={.5}{}] (B) to[out=north west,in=east] (0,3) to	[out=west, in=north east] (A);
				\draw[overarc, ->-={.5}{red}, rounded corners] (T11) -- (T21) -- (T22) -- 	(T12);
				\draw[fill] (A) circle (20pt);
				\draw[fill] (B) circle (20pt);
				\draw[very thick] (A) -- (B);
			}
		\ }
		=
		\mathord{
			\ \tikz[baseline=-.6ex, scale=.1, yshift=-5cm]{
				\coordinate (A) at (-10,0);
				\coordinate (B) at (10,0);
				\fill[lightgray] (A) -- (B) -- ($(B)+(0,-3)$) -- ($(A)+(0,-3)$) -- (A) -- 	cycle;
				\coordinate (T11) at ($(A)!0.2!(B)+(0,15)$);
				\coordinate (T21) at ($(A)!0.2!(B)+(0,10)$);
				\coordinate (T12) at ($(A)!0.8!(B)+(0,15)$);
				\coordinate (T22) at ($(A)!0.8!(B)+(0,10)$);
				\draw[red, very thick, ->-={.5}{}] (B) to[out=north,in=east] (0,13) to	[out=west, in=north] (A);
				\draw[red, very thick, -<-={.5}{}] (B) to[out=north,in=east] (0,10) to	[out=west, in=north] (A);
				\draw[overarc, ->-={.5}{red}, rounded corners] (T11) -- ($(T21)-(0,5)$) -- ($(T22)-(0,5)$) -- 	(T12);
				\draw[fill] (A) circle (20pt);
				\draw[fill] (B) circle (20pt);
				\draw[very thick] (A) -- (B);
			}
		\ }
	\]
	by the second Reidemeister move, and apply \cref{lem:cuttingweb} to the two crossings in the right-half of the diagram.
	Then the desired equation is the mirror reflection of the result.
\end{proof}

Let us recall the generating set of $\SK{\Sigma}$ given in Frohman--Sikora~\cite{FrohmanSikora20}. 
A \emph{triad} is a connected tangled trivalent graph with a single sink or source and three edges. 
An oriented arc is said to be \emph{descending} if one passes every self-crossing point through an over-pass first, following its orientation. 
An oriented knot is said to be descending if it admits a diagram with a basepoint such that the oriented arc starting from the basepoint satisfies the descending property. 
A triad is said to be descending if each of its three edges are descending and there exists a linear ordering among them such that the $i$-th edge is always over-passing the $j$-th edge for $j<i$.

Frohman and Sikora proved the following by induction on the ``size'' of a tangled trivalent graph.
\begin{thm}[{\cite[Theorem~6]{FrohmanSikora20}}]\label{thm:FSgenerator}
	$\SK{\Sigma}$ is generated by descending knots, arcs, and triads.
\end{thm}

We give another proof of the above theorem with a refinement of the above generating set, which will be useful for our argument on a generating set of $\SK{\Sigma}[\partial^{-1}]$.

\begin{lem}\label{lem:refineFS}
    Fix a boundary interval $E$. Then we have a generating set of $\SK{\Sigma}$ consisting of the following $\mathfrak{sl}_3$-webs:
	\begin{itemize}
		\item Descending knots $\alpha$ such that 
		there is a path from a point on $\alpha$ to $E$ without crossing $\alpha$ except for its initial point. 
		\item Descending arcs.
		\item Descending triads $\tau$ such that the diagram obtained from $\tau$ by removing the first edge has no internal crossings.
	\end{itemize}
\end{lem}
\begin{proof}
	We show that any basis web $G$ can be described as a product of the above $\mathfrak{sl}_3$-webs by induction on the number of trivalent vertices $V(G)$ of $G$. 
	First, any basis web $G$ with $\# V(G)\leq 1$ is a product of simple loops, arcs and a single triad. This establishes the basis step. 
	
	For a basis web $G$ with $\#V(G) \geq 2$, 
	let $G_{e,{+}}$ (resp.~$G_{e,-}$) be the tangled trivalent graph obtained by replacing an internal edge $e$ of $G$ with a positive (resp. negative) crossing, and $G_{e,{0}}$ the one obtained by replacing $e$ with parallel arcs. 
	Then we know 
	\begin{align*}
	    G=AG_{e,{+}}-A^3G_{e,0}\quad \mbox{and} \quad G=A^{-1}G_{e,{-}}-A^{-3}G_{e,0}
	\end{align*}
	from the skein relations \eqref{rel:plus}, \eqref{rel:minus}. This is the basic operation that we use in the following argument. 
	
	Let us first consider the case where $G$ has no endpoints on $\mathbb{M}$. In this case, our claim is that such a basis web can be written as a polynomial of descending loops. We give an algorithm to obtain such a polynomial expression of $G$. 
	Fix a basepoint $x_0$ on an arbitrary edge of $G$, and we start to move from $x_0$ along the edge following its orientation.
	If we arrive at a trivalent vertex, then choose one of the three incident edges other than what we passed, and call it $e$.
	Then we replace $G$ with $G_{e,{+}}$ or $G_{e,{-}}$ by using one of the above skein relations: see \cref{fig:edge-replacement}. Then we move to a position on one of the new internal edges, as shown there. 
	Here we make a choice of $G_{e,{\pm}}$ in such a way that our chosen path becomes an over-passing arc.
	We repeatedly apply this procedure to each of the resulting terms. 
	
	Then we eventually come back to the original point $x_0$, since the number of trivalent vertices is finite.
	We remark that in some step of the above procedure, the chosen edge $e$ may have crossings. (For example, you may choose the under-passing edge of the first web in the right-hand side of \cref{fig:edge-replacement} in a latter step.) 
	However, since the original basis web $G$ has no internal crossings, such an edge $e$ must be under-passing, and over-passed by an edge that we already passed in a former step. 
	This remark and our way to choose $G_{e,{\pm}}$ ensure that the resulting loop is descending and lies in a higher elevation than the other connected components.
	Hence the basis web $G$ is written as the product of a descending loop and a basis web with a smaller number of trivalent vertices. By the induction assumption, it follows that $G$ is written as a polynomial of descending loops.
	To see that the condition in the statement is satisfied, note that we can choose a point $x$ on an edge of the original basis web $G$ and a path from $x$ to the boundary interval $E$ so that it does not cross $G$ except for its initial point. Since the above procedure preserves a neighborhood of the path, all the resulting descending knot satisfies the desired property. 
	
	For a basis web $G$ having endpoints on $\mathbb{M}$, we choose the basepoint to be one of its endpoint and apply the same procedure.
	It will terminate when we arrive at another endpoint of $G$ or a trivalent vertex with two external edges.
	These external edges have no internal crossings, since $G$ is a basis web.
	In these cases, the result is a descending arc in the first case, or a trivalent graph satisfying the required condition in the second case. Then $G$ is written as a product of one of these graphs and another web with a smaller number of trivalent vertices. Thus the assertion follows from the induction assumption. 
	\begin{figure}
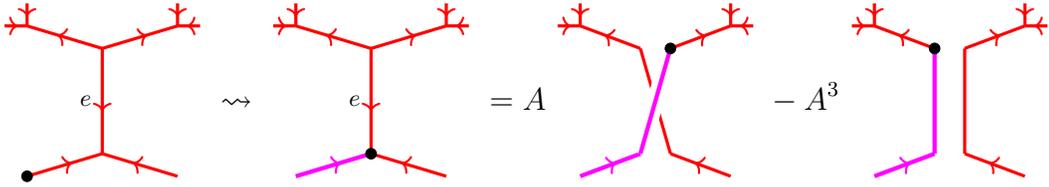

		\begin{align*}
			\mathord{
				\ \tikz[baseline=-.6ex, scale=.1]{
					\coordinate (AL) at (-10,-10);
					\coordinate (AR) at (10,-10);
					\coordinate (BL) at (-10,10);
					\coordinate (BR) at (10,10);
					\coordinate (CD) at (0,-7);
					\coordinate (CU) at (0,7);
					\coordinate (BLL) at ($(-10,10)+(-3,0)$);
					\coordinate (BLR) at ($(-10,10)+(0,3)$);
					\coordinate (BRR) at ($(10,10)+(3,0)$);
					\coordinate (BRL) at ($(10,10)+(0,3)$);
					\draw[red, very thick, ->-={.6}{}] (AL) -- (CD);
					\draw[red, very thick, ->-={.6}{}] (AR) -- (CD);
					\draw[red, very thick, ->-={.6}{}] (CU) -- (CD);
					\draw[red, very thick, ->-={.6}{}] (CU) -- (BL);
					\draw[red, very thick, ->-={.6}{}] (CU) -- (BR);
					\draw[red, very thick, ->-={.6}{}] (BRR) -- (BR);
					\draw[red, very thick, ->-={.6}{}] (BRL) -- (BR);
					\draw[red, very thick, ->-={.6}{}] (BLL) -- (BL);
					\draw[red, very thick, ->-={.6}{}] (BLR) -- (BL);
					\draw[fill] (AL) circle (20pt);
					\node at (0,0) [left]{\scriptsize$e$};
				}
			\ }
			\leadsto
			\mathord{
				\ \tikz[baseline=-.6ex, scale=.1]{
					\coordinate (AL) at (-10,-10);
					\coordinate (AR) at (10,-10);
					\coordinate (BL) at (-10,10);
					\coordinate (BR) at (10,10);
					\coordinate (CD) at (0,-7);
					\coordinate (CU) at (0,7);
					\coordinate (BLL) at ($(-10,10)+(-3,0)$);
					\coordinate (BLR) at ($(-10,10)+(0,3)$);
					\coordinate (BRR) at ($(10,10)+(3,0)$);
					\coordinate (BRL) at ($(10,10)+(0,3)$);
					\draw[magenta, ultra thick, ->-={.6}{}] (AL) -- (CD);
					\draw[red, very thick, ->-={.6}{}] (AR) -- (CD);
					\draw[red, very thick, ->-={.6}{}] (CU) -- (CD);
					\draw[red, very thick, ->-={.6}{}] (CU) -- (BL);
					\draw[red, very thick, ->-={.6}{}] (CU) -- (BR);
					\draw[red, very thick, ->-={.6}{}] (BRR) -- (BR);
					\draw[red, very thick, ->-={.6}{}] (BRL) -- (BR);
					\draw[red, very thick, ->-={.6}{}] (BLL) -- (BL);
					\draw[red, very thick, ->-={.6}{}] (BLR) -- (BL);
					\draw[fill] (CD) circle (20pt);
					\node at (0,0) [left]{\scriptsize$e$};
				}
			\ }
			&=A
			\mathord{
				\ \tikz[baseline=-.6ex, scale=.1]{
					\coordinate (AL) at (-10,-10);
					\coordinate (AR) at (10,-10);
					\coordinate (BL) at (-10,10);
					\coordinate (BR) at (10,10);
					\coordinate (CD) at (0,-7);
					\coordinate (CU) at (0,7);
					\coordinate (CDL) at ($(CD)+(-2,0)$);
					\coordinate (CDR) at ($(CD)+(2,0)$);
					\coordinate (CUL) at ($(CU)+(-2,0)$);
					\coordinate (CUR) at ($(CU)+(2,0)$);
					\coordinate (BLL) at ($(-10,10)+(-3,0)$);
					\coordinate (BLR) at ($(-10,10)+(0,3)$);
					\coordinate (BRR) at ($(10,10)+(3,0)$);
					\coordinate (BRL) at ($(10,10)+(0,3)$);
					\draw[magenta, ultra thick, ->-={.6}{}] (AL) -- (CDL);
					\draw[red, very thick, ->-={.6}{}] (AR) -- (CDR);
					\draw[red, very thick] (CUL) -- (CDR);
					\draw[white, double=magenta, double distance=1.6pt, line width=2.4pt] (CUR) -- (CDL);
					\draw[red, very thick, ->-={.6}{}] (CUL) -- (BL);
					\draw[red, very thick, ->-={.6}{}] (CUR) -- (BR);
					\draw[red, very thick, ->-={.6}{}] (BRR) -- (BR);
					\draw[red, very thick, ->-={.6}{}] (BRL) -- (BR);
					\draw[red, very thick, ->-={.6}{}] (BLL) -- (BL);
					\draw[red, very thick, ->-={.6}{}] (BLR) -- (BL);
					\draw[fill] (CUR) circle (20pt);
				}
			\ }
			-A^3
			\mathord{
				\ \tikz[baseline=-.6ex, scale=.1]{
					\coordinate (AL) at (-10,-10);
					\coordinate (AR) at (10,-10);
					\coordinate (BL) at (-10,10);
					\coordinate (BR) at (10,10);
					\coordinate (CD) at (0,-7);
					\coordinate (CU) at (0,7);
					\coordinate (CDL) at ($(CD)+(-2,0)$);
					\coordinate (CDR) at ($(CD)+(2,0)$);
					\coordinate (CUL) at ($(CU)+(-2,0)$);
					\coordinate (CUR) at ($(CU)+(2,0)$);
					\coordinate (BLL) at ($(-10,10)+(-3,0)$);
					\coordinate (BLR) at ($(-10,10)+(0,3)$);
					\coordinate (BRR) at ($(10,10)+(3,0)$);
					\coordinate (BRL) at ($(10,10)+(0,3)$);
					\draw[magenta, ultra thick, ->-={.6}{}] (AL) -- (CDL);
					\draw[red, very thick, ->-={.6}{}] (AR) -- (CDR);
					\draw[red, very thick] (CUR) -- (CDR);
					\draw[magenta, ultra thick] (CUL) -- (CDL);
					\draw[red, very thick, ->-={.6}{}] (CUL) -- (BL);
					\draw[red, very thick, ->-={.6}{}] (CUR) -- (BR);
					\draw[red, very thick, ->-={.6}{}] (BRR) -- (BR);
					\draw[red, very thick, ->-={.6}{}] (BRL) -- (BR);
					\draw[red, very thick, ->-={.6}{}] (BLL) -- (BL);
					\draw[red, very thick, ->-={.6}{}] (BLR) -- (BL);
					\draw[fill] (CUL) circle (20pt);
				}
			\ }
		\end{align*}
			\caption{Decomposition of a basis web $G$ into the sum of $G_{e,+}$ and $G_{e,0}$. The black dot shows our location before and after the step.}
		\label{fig:edge-replacement}
	\end{figure}
\end{proof}

We next show that $\SK{\Sigma}[\partial^{-1}]$ is generated by elementary webs coming from triangles in $\Sigma$.
It is easy to see that any flat arcs and triads are elementary webs, since the endpoint grading of such an $\mathfrak{sl}_3$-web can not be represented as a sums of two endpoint gradings of $\mathfrak{sl}_3$-webs in $\SK{\Sigma}$. 
\begin{thm}\label{thm:localized-generators}
	For a connected unpunctured marked surface $\Sigma$ with at least two special points, $\SK{\Sigma}[\partial^{-1}]$ is generated by oriented simple arcs and triads.
\end{thm}

\begin{proof}
	We are going to replace the generators in \cref{lem:refineFS} with simple arcs and triads by using \cref{lem:stickto}.
	The condition on the number of special points of $\Sigma$ ensures that there exist two distinct boundary intervals $E$ and $E'$.
	First, we are going to replace a descending knot with a product of descending arcs, triads, and H-webs.
	By the condition on the descending knots $\alpha$ in \cref{lem:refineFS}, one can take an auxiliary path (a blue dotted arc on the left-hand side of the equation below) from the boundary interval $E$ to the basepoint of $\alpha$ without crossing the knot.
	Then we apply \cref{lem:stickto} to a neighborhood of the path, as follows:
	\begin{align*}
		\mathord{
			\ 
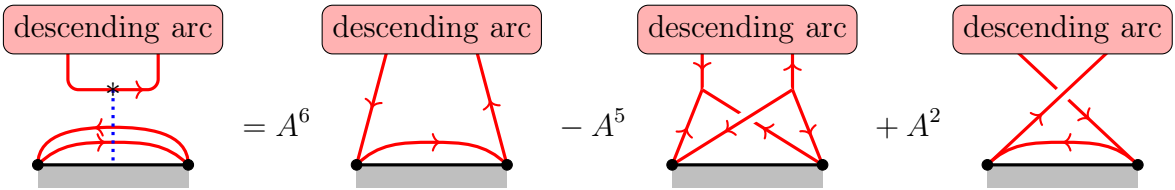
\tikz[baseline=-.6ex, scale=.1, yshift=-5cm]{
				\coordinate (A) at (-10,0);
				\coordinate (B) at (10,0);
				\coordinate (T11) at ($(A)!0.2!(B)+(0,15)$);
				\coordinate (T21) at ($(A)!0.2!(B)+(0,10)$);
				\coordinate (T12) at ($(A)!0.8!(B)+(0,15)$);
				\coordinate (T22) at ($(A)!0.8!(B)+(0,10)$);
				\draw[red, very thick, ->-={.6}{}] (B) to[out=north,in=east] (0,5) to[out=west, in=north] (A);
				\draw[red, very thick, -<-={.6}{}] (B) to[out=north west,in=east] (0,3) to[out=west, in=north east] (A);
				\draw[overarc, ->-={.7}{red}, rounded corners] (T11) -- (T21) -- (T22) -- (T12);
				\node[draw, fill=red!30, rounded corners] at (0,18) {descending arc};
				\draw[blue, very thick, dotted] ($(T21)!0.5!(T22)$) -- (0,0);
				\node at ($(T21)!0.5!(T22)$) {$\ast$};
				\fill[lightgray] (A) -- (B) -- ($(B)+(0,-3)$) -- ($(A)+(0,-3)$) -- (A) -- cycle;
				\draw[fill] (A) circle (20pt);
				\draw[fill] (B) circle (20pt);
				\draw[very thick] (A) -- (B);
			}
		\ }
		&=A^{6}
		\mathord{
			\ \tikz[baseline=-.6ex, scale=.1, yshift=-5cm]{
				\coordinate (A) at (-10,0);
				\coordinate (B) at (10,0);
				\coordinate (T11) at ($(A)!0.2!(B)+(0,15)$);
				\coordinate (T21) at ($(A)!0.2!(B)+(0,10)$);
				\coordinate (T12) at ($(A)!0.8!(B)+(0,15)$);
				\coordinate (T22) at ($(A)!0.8!(B)+(0,10)$);
				\draw[red, very thick, ->-] (T11) -- (A);
				\draw[red, very thick, -<-] (T12) -- (B);
				\draw[red, very thick, -<-={.5}{}] (B) to[out=north west,in=east] (0,3) to[out=west, in=north east] (A);
				\node[draw, fill=red!30, rounded corners] at (0,18) {descending arc};
				\fill[lightgray] (A) -- (B) -- ($(B)+(0,-3)$) -- ($(A)+(0,-3)$) -- (A) -- cycle;
				\draw[fill] (A) circle (20pt);
				\draw[fill] (B) circle (20pt);
				\draw[very thick] (A) -- (B);
		}
		\ }
		-A^{5}\mathord{
			\ \tikz[baseline=-.6ex, scale=.1, yshift=-5cm]{
				\coordinate (A) at (-10,0);
				\coordinate (B) at (10,0);
				\coordinate (T11) at ($(A)!0.2!(B)+(0,15)$);
				\coordinate (T21) at ($(A)!0.2!(B)+(0,10)$);
				\coordinate (T12) at ($(A)!0.8!(B)+(0,15)$);
				\coordinate (T22) at ($(A)!0.8!(B)+(0,10)$);
				\draw[very thick, red, ->-] (T11) -- (T21);
				\draw[very thick, red, -<-] (T12) -- (T22);
				\draw[very thick, red, ->-={.5}{red}, overarc] (B) -- (T21);
				\draw[very thick, red, -<-={.5}{red}, overarc] (A) -- (T22);
				\draw[very thick, red, ->-={.5}{red}] (A) -- (T21);
				\draw[very thick, red, -<-] (B) -- (T22);
				\node[draw, fill=red!30, rounded corners] at (0,18) {descending arc};
				\fill[lightgray] (A) -- (B) -- ($(B)+(0,-3)$) -- ($(A)+(0,-3)$) -- (A) -- cycle;
				\draw[fill] (A) circle (20pt);
				\draw[fill] (B) circle (20pt);
				\draw[very thick] (A) -- (B);
			}
		\ }
		+A^{2}\mathord{
			\ \tikz[baseline=-.6ex, scale=.1, yshift=-5cm]{
				\coordinate (A) at (-10,0);
				\coordinate (B) at (10,0);
				\coordinate (T11) at ($(A)!0.2!(B)+(0,15)$);
				\coordinate (T21) at ($(A)!0.2!(B)+(0,10)$);
				\coordinate (T12) at ($(A)!0.8!(B)+(0,15)$);
				\coordinate (T22) at ($(A)!0.8!(B)+(0,10)$);
				\draw[red, very thick, ->-={.6}{red}, overarc] (T11) -- (B);
				\draw[red, very thick, -<-={.6}{red}, overarc] (T12) -- (A);
				\draw[red, very thick, ->-={.5}{}] (B) to[out=north west,in=east] (0,3) to[out=west, in=north east] 	(A);
				\node[draw, fill=red!30, rounded corners] at (0,18) {descending arc};
				\fill[lightgray] (A) -- (B) -- ($(B)+(0,-3)$) -- ($(A)+(0,-3)$) -- (A) -- cycle;
				\draw[fill] (A) circle (20pt);
				\draw[fill] (B) circle (20pt);
				\draw[very thick] (A) -- (B);
			}
		\ }
	\end{align*}
	This procedure expands $\alpha$ into a sum of two descending arcs and an H-web with a descending internal edge, up to the multiplication of boundary webs. 
	For each of such descending arcs or H-webs, its portion from $E$ to its first self-intersection point has one of the diagrams in \cref{fig:descending endpoint}.
	\begin{figure}
		\begin{tikzpicture}[scale=.1]
			\coordinate (A) at (-8,0);
			\coordinate (B) at (8,0);
			\coordinate (C) at (0,0);
			\fill[lightgray] (A) -- (B) -- ($(B)+(0,-3)$) -- ($(A)+(0,-3)$) -- (A) -- cycle;
			\draw[blue,very thick,dotted] (-4,0) to[out=north, in=south] (4,7);
			\draw[red, very thick,->-={.8}{red}] (8,10) to[out=south, in=east] (4,7) to[out=west, in=south] (-3,20);
			\draw[red, very thick, ->-={.2}{red}, overarc] (C) -- (0,10) to[out=north, in=west] (4,13) to[out=east, in=north] (8,10);
			\draw[fill=lightgray] (4,10) circle (40pt);
			\draw[fill] (A) circle (20pt);
			\draw[fill] (C) circle (20pt);
			\draw[very thick] (A) -- (B);
			\node at (-4,0) [below]{$E$};
		\end{tikzpicture}
		\begin{tikzpicture}[scale=.1]
			\coordinate (A) at (-8,0);
			\coordinate (B) at (8,0);
			\coordinate (C) at (0,0);
			\draw[blue,very thick,dotted] (4,0) to[out=north, in=south] (4,7);
			\fill[lightgray] (A) -- (B) -- ($(B)+(0,-3)$) -- ($(A)+(0,-3)$) -- (A) -- cycle;
			\draw[red, very thick,->-={.8}{red}] (8,10) to[out=south, in=east] (4,7) to[out=west, in=south] (-3,20);
			\draw[red, very thick, ->-={.2}{red}, overarc] (C) -- (0,10) to[out=north, in=west] (4,13) to[out=east, in=north] (8,10);
			\draw[fill=lightgray] (4,10) circle (40pt);
			\draw[fill] (B) circle (20pt);
			\draw[fill] (C) circle (20pt);
			\draw[very thick] (A) -- (B);
			\node at (4,0) [below]{$E$};
		\end{tikzpicture}
		\begin{tikzpicture}[scale=.1]
			\coordinate (A) at (-8,0);
			\coordinate (B) at (8,0);
			\coordinate (C) at (0,3);
			\draw[blue,very thick,dotted] (0,0) to[out=north, in=south] (4,7);
			\fill[lightgray] (A) -- (B) -- ($(B)+(0,-3)$) -- ($(A)+(0,-3)$) -- (A) -- cycle;
			\draw[red, very thick,->-={.8}{red}] (8,10) to[out=south, in=east] (4,7) to[out=west, in=south] (-3,20);
			\draw[red, very thick, ->-={.2}{red}, overarc] (C) -- (0,10) to[out=north, in=west] (4,13) to[out=east, in=north] (8,10);
			\draw[fill=lightgray] (4,10) circle (40pt);
			\draw[red, very thick, ->-={.5}{}] (C) -- (-5,0);
			\draw[red, very thick, ->-={.5}{red}, overarc] (C) -- (5,0);
			\draw[fill] (-5,0) circle (20pt);
			\draw[fill] (5,0) circle (20pt);
			\draw[very thick] (A) -- (B);
			\node at (0,0) [below]{$E$};
		\end{tikzpicture}
		\begin{tikzpicture}[scale=.1]
			\coordinate (A) at (-8,0);
			\coordinate (B) at (8,0);
			\coordinate (C) at (0,0);
			\draw[blue,very thick,dotted] (-4,0) to[out=north, in=south] (-4,7);
			\fill[lightgray] (A) -- (B) -- ($(B)+(0,-3)$) -- ($(A)+(0,-3)$) -- (A) -- cycle;
			\draw[red, very thick,->-={.8}{red}] (-8,10) to[out=south, in=west] (-4,7) to[out=east, in=south] (3,20);
			\draw[red, very thick, ->-={.2}{red}, overarc] (C) -- (0,10) to[out=north, in=east] (-4,13) to[out=west, in=north] (-8,10);
			\draw[fill=lightgray] (-4,10) circle (40pt);
			\draw[fill] (A) circle (20pt);
			\draw[fill] (C) circle (20pt);
			\draw[very thick] (A) -- (B);
			\node at (-4,0) [below]{$E$};
		\end{tikzpicture}
		\begin{tikzpicture}[scale=.1]
			\coordinate (A) at (-8,0);
			\coordinate (B) at (8,0);
			\coordinate (C) at (0,0);
			\draw[blue,very thick,dotted] (4,0) to[out=north, in=south] (-4,7);
			\fill[lightgray] (A) -- (B) -- ($(B)+(0,-3)$) -- ($(A)+(0,-3)$) -- (A) -- cycle;
			\draw[red, very thick,->-={.8}{red}] (-8,10) to[out=south, in=west] (-4,7) to[out=east, in=south] (3,20);
			\draw[red, very thick, ->-={.2}{red}, overarc] (C) -- (0,10) to[out=north, in=east] (-4,13) to[out=west, in=north] (-8,10);
			\draw[fill=lightgray] (-4,10) circle (40pt);
			\draw[fill] (B) circle (20pt);
			\draw[fill] (C) circle (20pt);
			\draw[very thick] (A) -- (B);
			\node at (4,0) [below]{$E$};
		\end{tikzpicture}
		\begin{tikzpicture}[scale=.1]
			\coordinate (A) at (-8,0);
			\coordinate (B) at (8,0);
			\coordinate (C) at (0,3);
			\draw[blue,very thick,dotted] (0,0) to[out=north, in=south] (-4,7);
			\fill[lightgray] (A) -- (B) -- ($(B)+(0,-3)$) -- ($(A)+(0,-3)$) -- (A) -- cycle;
			\draw[red, very thick,->-={.8}{red}] (-8,10) to[out=south, in=west] (-4,7) to[out=east, in=south] (3,20);
			\draw[red, very thick, ->-={.2}{red}, overarc] (C) -- (0,10) to[out=north, in=east] (-4,13) to[out=west, in=north] (-8,10);
			\draw[fill=lightgray] (-4,10) circle (40pt);
			\draw[red, very thick, ->-={.5}{red}, overarc] (C) -- (-5,0);
			\draw[red, very thick, ->-={.5}{}] (C) -- (5,0);
			\draw[fill] (-5,0) circle (20pt);
			\draw[fill] (5,0) circle (20pt);
			\draw[very thick] (A) -- (B);
			\node at (0,0) [below]{$E$};
		\end{tikzpicture}
	\caption{Six patterns of a portion from $E$ to the first self-intersection point. These diagrams show a tubular neighborhood of an $\mathfrak{sl}_3$-web, which may have other under-crossing arcs and endpoints of lower elevation.}
		\label{fig:descending endpoint}
	\end{figure}
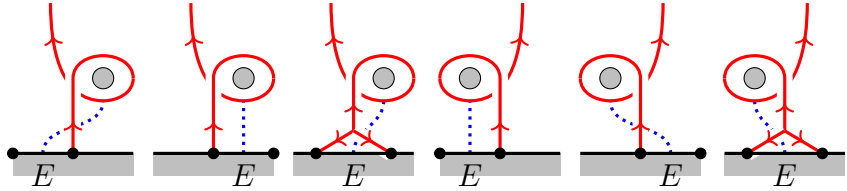
	We apply \cref{lem:stickto} along the dotted path shown in \cref{fig:descending endpoint}. 
	Then, the first $\mathfrak{sl}_3$-web in \cref{fig:descending endpoint} is expanded into a linear combination of
	\begin{align*}
		\mathord{\ 
			
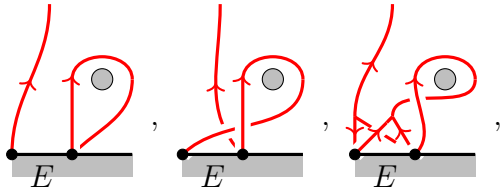
\tikz[baseline=-.6ex, scale=.1, yshift=-5cm]{
				\coordinate (A) at (-8,0);
				\coordinate (B) at (8,0);
				\coordinate (C) at (0,0);
				\fill[lightgray] (A) -- (B) -- ($(B)+(0,-3)$) -- ($(A)+(0,-3)$) -- (A) -- cycle;
				\draw[red, very thick, ->-={.5}{}] (A) to[out=north, in=south] (-3,20);
				\draw[red, very thick] (8,10) to[out=south, in=north east] (C);
				\draw[red, very thick, ->-={.5}{red}, overarc] (C) -- (0,10) to[out=north, in=west] (4,13) to	[out=east, in=north] (8,10);
				\draw[fill=lightgray] (4,10) circle (40pt);
				\draw[fill] (A) circle (20pt);
				\draw[fill] (C) circle (20pt);
				\draw[very thick] (A) -- (B);
				\node at (-4,0) [below]{$E$};
			}
		\ },
		\mathord{\ 
			\tikz[baseline=-.6ex, scale=.1, yshift=-5cm]{
				\coordinate (A) at (-8,0);
				\coordinate (B) at (8,0);
				\coordinate (C) at (0,0);
				\fill[lightgray] (A) -- (B) -- ($(B)+(0,-3)$) -- ($(A)+(0,-3)$) -- (A) -- cycle;
				\draw[red, very thick, ->-={.5}{}] (C) to[out=north west, in=south] (-3,20);
				\draw[red, very thick, overarc] (8,10) to[out=south, in=north east] (A);
				\draw[red, very thick, ->-={.5}{red}, overarc] (C) -- (0,10) to[out=north, in=west] (4,13) to	[out=east, in=north] (8,10);
				\draw[fill=lightgray] (4,10) circle (40pt);
				\draw[fill] (A) circle (20pt);
				\draw[fill] (C) circle (20pt);
				\draw[very thick] (A) -- (B);
				\node at (-4,0) [below]{$E$};
			}
		\ },
		\mathord{\ 
			\tikz[baseline=-.6ex, scale=.1, yshift=-5cm]{
				\coordinate (A) at (-8,0);
				\coordinate (B) at (8,0);
				\coordinate (C) at (0,0);
				\coordinate (ACl) at (-8,5);
				\coordinate (ACr) at (-3,5);
				\fill[lightgray] (A) -- (B) -- ($(B)+(0,-3)$) -- ($(A)+(0,-3)$) -- (A) -- cycle;
				\draw[red, very thick, ->-={.5}{}] (ACl) to[out=north, in=south] (-3,20);
				\draw[red, very thick, ->-={.5}{}] (ACl) -- (A);
				\draw[red, very thick, ->-={.5}{}, shorten >=.3cm] (ACl) -- (C);
				\draw[red, very thick] (ACr) to[out=north, in=south] (8,10);
				\draw[red, very thick, -<-={.5}{red}, overarc] (ACr) -- (A);
				\draw[red, very thick, -<-={.5}{red}, overarc, shorten >=.15cm] (ACr) -- (C);
				\draw[red, very thick, ->-={.5}{red}, overarc] (C) to[out=north east, in=south] (0,10) to[out=north, in=west] (4,13) to	[out=east, in=north] (8,10);
				\draw[red,very thick, shorten >=.6cm] (ACr) -- (A);
				\draw[fill=lightgray] (4,10) circle (40pt);
				\draw[fill] (A) circle (20pt);
				\draw[fill] (C) circle (20pt);
				\draw[very thick] (A) -- (B);
				\node at (-4,0) [below]{$E$};
			}
		\ },
	\end{align*}
	where we omitted boundary webs.
	Similarly for the second diagram
	\begin{align*}
		\mathord{\ 
			
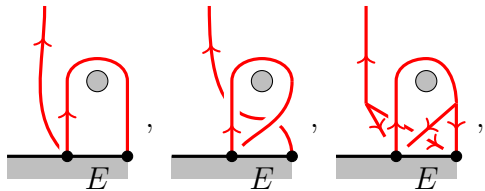
\tikz[baseline=-.6ex, scale=.1, yshift=-5cm]{
				\coordinate (A) at (-8,0);
				\coordinate (B) at (8,0);
				\coordinate (C) at (0,0);
				\fill[lightgray] (A) -- (B) -- ($(B)+(0,-3)$) -- ($(A)+(0,-3)$) -- (A) -- cycle;
				\draw[red, very thick,->-={.8}{red}] (C) to[out=north west, in=south] (-3,20);
				\draw[red, very thick, ->-={.2}{red}, overarc] (C) -- (0,10) to[out=north, in=west] (4,13) to[out=east, in=north] (8,10) -- (B);
				\draw[fill=lightgray] (4,10) circle (40pt);
				\draw[fill] (B) circle (20pt);
				\draw[fill] (C) circle (20pt);
				\draw[very thick] (A) -- (B);
				\node at (4,0) [below]{$E$};
			}
		\ },
		\mathord{\ 
			\tikz[baseline=-.6ex, scale=.1, yshift=-5cm]{
				\coordinate (A) at (-8,0);
				\coordinate (B) at (8,0);
				\coordinate (C) at (0,0);
				\fill[lightgray] (A) -- (B) -- ($(B)+(0,-3)$) -- ($(A)+(0,-3)$) -- (A) -- cycle;
				\draw[red, very thick,->-={.8}{red}] (B) to[out=north, in=east] (4,5) to[out=west, in=south] (-3,20);
				\draw[red, very thick, ->-={.2}{red}, overarc] (C) -- (0,10) to[out=north, in=west] (4,13) to[out=east, in=north] (8,10);
				\draw[red, very thick, overarc, shorten >=.2cm] (8,10) to[out=south, in=north east] (C);
				\draw[fill=lightgray] (4,10) circle (40pt);
				\draw[fill] (B) circle (20pt);
				\draw[fill] (C) circle (20pt);
				\draw[very thick] (A) -- (B);
				\node at (4,0) [below]{$E$};
			}
		\ },
		\mathord{\ 
			\tikz[baseline=-.6ex, scale=.1, yshift=-5cm]{
				\coordinate (A) at (-8,0);
				\coordinate (B) at (8,0);
				\coordinate (C) at (0,0);
				\fill[lightgray] (A) -- (B) -- ($(B)+(0,-3)$) -- ($(A)+(0,-3)$) -- (A) -- cycle;
				\draw[red, very thick,->-={.5}{red}] (-4,7) -- (-4,20);
				\draw[red, very thick,->-={.8}{red}, shorten >=.2cm] (-4,7) -- (B);
				\draw[red, very thick,->-={.5}{red}, shorten >=.3cm] (-4,7) -- (C);
				\draw[red, very thick, ->-={.2}{red}, overarc] (C) -- (0,10) to[out=north, in=west] (4,13) to[out=east, in=north] (8,7);
				\draw[red, very thick,->-={.5}{red}, overarc, shorten >=.2cm] (8,7) -- (C);
				\draw[red, very thick,->-={.5}{red}] (8,7) -- (B);
				\draw[fill=lightgray] (4,10) circle (40pt);
				\draw[fill] (B) circle (20pt);
				\draw[fill] (C) circle (20pt);
				\draw[very thick] (A) -- (B);
				\node at (4,0) [below]{$E$};
			}
		\ },
	\end{align*}
	and for the third diagram
	\begin{align*}
		\mathord{\ 
			
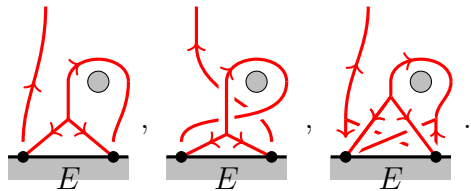
\tikz[baseline=-.6ex, scale=.1, yshift=-5cm]{
				\coordinate (A) at (-8,0);
				\coordinate (B) at (8,0);
				\coordinate (AA) at (-6,0);
				\coordinate (BB) at (6,0);
				\coordinate (C) at (0,0);
				\coordinate (Cl) at (-6,5);
				\coordinate (Cr) at (6,5);
				\draw[red, very thick, ->-={.5}{}, shorten <=.2cm] (AA) to[out=north, in=south] (-3,20);
				\draw[red, very thick, shorten <=.2cm] (BB) to[out=north, in=south] (8,10);
				\draw[red, very thick, ->-={.5}{red}] (0,5) -- (0,10) to[out=north, in=west] (4,13) to	[out=east, in=north] (8,10);
				\draw[red,very thick, ->-={.4}{red}] (0,5) -- (AA);
				\draw[red,very thick, ->-={.4}{red}] (0,5) -- (BB);
				\draw[fill=lightgray] (4,10) circle (40pt);
				\fill[lightgray] (A) -- (B) -- ($(B)+(0,-3)$) -- ($(A)+(0,-3)$) -- (A) -- cycle;
				\draw[fill] (AA) circle (20pt);
				\draw[fill] (BB) circle (20pt);
				\draw[very thick] (A) -- (B);
				\node at (0,0) [below]{$E$};
			}
		\ },
		\mathord{\ 
			\tikz[baseline=-.6ex, scale=.1, yshift=-5cm]{
				\coordinate (A) at (-8,0);
				\coordinate (B) at (8,0);
				\coordinate (AA) at (-6,0);
				\coordinate (BB) at (6,0);
				\coordinate (C) at (0,0);
				\coordinate (Cl) at (-6,5);
				\coordinate (Cr) at (6,5);
				\draw[red, very thick, ->-={.8}{}, shorten <=.2cm] (BB) to[out=north, in=south] (-4,15) -- (-4,20);
				\draw[red, very thick, shorten <=.2cm, overarc] (AA) to[out=north, in=south] (8,10);
				\draw[red, very thick, ->-={.5}{red}, overarc] (0,3) -- (0,10) to[out=north, in=west] (4,13) to	[out=east, in=north] (8,10);
				\draw[red,very thick, ->-={.4}{red}] (0,3) -- (AA);
				\draw[red,very thick, ->-={.4}{red}] (0,3) -- (BB);
				\draw[fill=lightgray] (4,10) circle (40pt);
				\fill[lightgray] (A) -- (B) -- ($(B)+(0,-3)$) -- ($(A)+(0,-3)$) -- (A) -- cycle;
				\draw[fill] (AA) circle (20pt);
				\draw[fill] (BB) circle (20pt);
				\draw[very thick] (A) -- (B);
				\node at (0,0) [below]{$E$};
			}
		\ },
		\mathord{\ 
			\tikz[baseline=-.6ex, scale=.1, yshift=-5cm]{
				\coordinate (A) at (-8,0);
				\coordinate (B) at (8,0);
				\coordinate (AA) at (-6,0);
				\coordinate (BB) at (6,0);
				\coordinate (C) at (0,0);
				\coordinate (Cl) at (-6,5);
				\coordinate (Cr) at (6,5);
				\draw[red, very thick, ->-={.5}{}] (Cl) to[out=north, in=south] (-3,20);
				\draw[red, very thick, ->-={.5}{}, shorten >=.3cm] (Cl) -- (AA);
				\draw[red, very thick, shorten >=.4cm] (Cl) -- (BB);
				\draw[red, very thick] (Cr) to[out=north, in=south] (8,10);
				\draw[red, very thick, overarc] (Cr) -- (AA);
				\draw[red, very thick, -<-={.5}{red}] (Cr) -- (BB);
				\draw[red, very thick, ->-={.5}{red}] (0,8) -- (0,10) to[out=north, in=west] (4,13) to	[out=east, in=north] (8,10);
				\draw[red,very thick, ->-={.4}{red}, overarc] (0,8) -- (AA);
				\draw[red,very thick, ->-={.4}{red}, overarc] (0,8) -- (BB);
				\draw[red,very thick, ->-={.4}{red}, shorten >=.2cm] (0,8) -- (AA);
				\draw[fill=lightgray] (4,10) circle (40pt);
				\fill[lightgray] (A) -- (B) -- ($(B)+(0,-3)$) -- ($(A)+(0,-3)$) -- (A) -- cycle;
				\draw[fill] (AA) circle (20pt);
				\draw[fill] (BB) circle (20pt);
				\draw[very thick] (A) -- (B);
				\node at (0,0) [below]{$E$};
			}
		\ }.
	\end{align*}
	The treatment of the last three diagrams in \cref{fig:descending endpoint} is similar.
	In the cases where the connected component in the top layer of the resulting diagram has self-crossings, we can further expand it into simple arcs, triads, and $H$-webs by applying \cref{lem:stickto} along an auxiliary path from another boundary interval $E'$.
	Moreover, a simple $H$-web can be written as a polynomial of simple arcs by \cref{fig:edge-replacement}.
	This argument reduces the number of self-crossings, and it can be repeatedly applied.
	Consequently, we can expand each of the $\mathfrak{sl}_3$-web in \cref{fig:descending endpoint}, and hence a descending knot, into a polynomial of simple arcs and triads by multiplying sufficiently many boundary webs along $E$ and $E'$.

	For a descending arc, one of its endpoints sticks to $E$ by \cref{lem:stickto}.
	Thus, one can apply the above argument.
	
	For a descending triad with two simple edges, we take an auxiliary path from $E$ to a point near the trivalent vertex for each of its three edges as in \cref{fig:triad2edge}.
	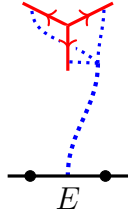
\begin{figure}
		\begin{tikzpicture}[scale=.1]
			\coordinate (A) at (-8,0);
			\coordinate (B) at (8,0);
			\coordinate (T0) at (0,20);
			\coordinate (T1) at (-6,23);
			\coordinate (T2) at (0,14);
			\coordinate (T3) at (6,23);
			\coordinate (T4) at (4,15);
			\draw[blue,very thick,dotted] (T4) to[out=north west, in=south] ($(T0)!.8!(T1)$);
			\draw[blue,very thick,dotted] (T4) to[out=north, in=south] ($(T0)!.8!(T3)$);
			\draw[blue,ultra thick,dotted] (T4) to[out=south, in=north] (0,0);
			\draw[red, very thick, ->-={.5}{}] (T0) -- (T1);
			\draw[red, very thick, ->-={.5}{red}, overarc] (T0) -- (T2);
			\draw[red, very thick, ->-={.5}{}] (T0) -- (T3);
			\draw[blue,very thick,dotted] (T4) -- ($(T0)!.8!(T2)$);
			\draw[fill] (-5,0) circle (20pt);
			\draw[fill] (5,0) circle (20pt);
			\draw[very thick] (A) -- (B);
			\node at (0,0) [below]{$E$};
		\end{tikzpicture}
		\caption{The choice of auxiliary paths from $E$ to the three edges of a triad. The three paths follow the same route from $E$ until a point near to the trivalent vertex, having distinct elevations. The auxiliary path to the first edge is arranged in an intermediate layer between the first and the second edges; the second path lies in a layer between the second and the third edges, the third path in the lowest layer.}
		\label{fig:triad2edge}
	\end{figure}
	After applying \cref{lem:stickto} along these paths, a connected component of the resulting $\mathfrak{sl}_3$-webs containing the trivalent vertex lies in a biangle.
	In particular, such connected components are represented by monomials of boundary webs along $E$.
	Other components containing three edges are descending arcs or triads with ends as in \cref{fig:descending endpoint}.
	Thus any descending triad is expanded into simple arcs and triads, up to boundary webs.
\end{proof}

\begin{rem}\label{rem:classical-proof}
    In the classical case ($A^{\frac{1}{2}}=1$ or $-1$), the proof of \cref{thm:localized-generators} becomes simple.
    We do not need to discuss generators of the (boundary-localized) $\mathfrak{sl}_3$-skein algebra because the multiplication is commutative in this case.
    The proof is the following.
    For any $\mathfrak{sl}_3$-web and every vertices of it, one can stick three edges around a vertex to a boundary interval $E$ by \cref{lem:stickto} as in \cref{fig:triad2edge}.
    The resulting $\mathfrak{sl}_3$-web is a polynomial of simple arcs and simple triads.
\end{rem}

\section{Quantum cluster algebras}\label{sect:qcluster}

\subsection{Quantum cluster algebra}
Here we recall the definition of the quantum cluster algebra following \cite{BZ} and related fundamental results. We also recall a grading on the (quantum) cluster algebra which we call the \emph{ensemble grading} from the viewpoint of the cluster variety \cite{FG09}, which has been originally investigated in \cite{GSV}.

\subsubsection{The exchange graph and the cluster algebra}\label{subsub:CA}
Fix a finite set $I=\{1,\dots,N\}$ of indices
and a field $\cF$ which is isomorphic to the field of rational functions on $N$ variables with rational coefficients. We also fix a subset $I_\uf \subset I$ and let $I_\f:=I \setminus I_\uf$. 
A \emph{(labeled) seed} in $\cF$ is a pair $(B,\mathbf{A})$, where
\begin{itemize}
    \item $B=(b_{ij})_{i,j \in I}$ is a skew-symmetric matrix with half-integral entries such that $b_{ij} \in \bZ$ unless $(i,j) \in I_\f \times I_\f$;
    \item $\mathbf{A}=(A_i)_{i \in I}$ is a tuple of algebraically independent elements in $\cF$.
\end{itemize}
We call a matrix $B$ satisfying the above conditions an \emph{exchange matrix}. The elements $A_i$ for $i \in I$ are called the \emph{cluster ($\A$-)variables}, and those for $i \in I_\f$ are called the \emph{frozen variables}. 

It is useful to represent an exchange matrix $B=(b_{ij})_{i,j \in I}$ by a quiver $Q$. 
Let us define the \emph{quiver exchange matrix}\footnote{This is identified with the Fock--Goncharov's exchange matrix \cite{FG09}. See \cref{sec:FG}.} $\varepsilon=(\varepsilon_{ij})_{i,j \in I}$ by $\varepsilon_{ij}:=b_{ji}$. Then the quiver $Q$ corresponding to $B$ has vertices parametrized by the set $I$ and $|\varepsilon_{ij}|$ arrows from $i$ to $j$ (resp. $j$ to $i$) if $\varepsilon_{ij} >0$ (resp. $\varepsilon_{ji} >0$). In figures, we draw $n$ dashed arrows from $i$ to $j$ if $\varepsilon_{ij}=n/2$ for $n \in \bZ$, where a pair of dashed arrows is replaced with a solid arrow. 

For an unfrozen index $k \in I_\uf$, the \emph{seed mutation} produces a new seed $(B',\mathbf{A}')=\mu_k(B,\mathbf{A})$ according to the following rule:
\begin{align}
    b'_{ij} &= 
    \begin{cases}
    -b_{ij} & \mbox{if $i=k$ or $j=k$}, \\
    b_{ij} + [b_{ik}]_+ [b_{kj}]_+ - [-b_{ik}]_+ [-b_{kj}]_+ & \mbox{otherwise},
    \end{cases} \label{eq:matrix mutation}\\
    A'&=\begin{cases}
    \displaystyle{ A_k^{-1} \left(\prod_{j \in I}A_j^{[ b_{jk}]_+} + \prod_{j \in I}A_j^{[-b_{jk}]_+}\right)} & \mbox{if $i=k$}, \\
    A_i & \mbox{if $i \neq k$}. 
    \end{cases} \label{eq:A-transf}
\end{align}
Here $[a]_+:=\max\{a,0\}$ for $a \in \bR$. The relation \eqref{eq:matrix mutation} is called the \emph{matrix mutation}, and \eqref{eq:A-transf} is called the \emph{exchange relation}. 
It is not hard to check that the seed mutation is involutive: $(B,\mathbf{A})=\mu_k\mu_k(B,\mathbf{A})$. 


For a permutation $\sigma\in \mathfrak{S}_{I_\uf} \times \mathfrak{S}_{I_\f}$ that do not mix the unfrozen/frozen indices, a new seed $(B',\mathbf{A}')=\sigma(B,\mathbf{A})$ is defined by 
\begin{align}\label{eq:classical_permutation}
    b'_{ij} = b_{\sigma^{-1}(i),\sigma^{-1}(j)}, \quad A'_i=A_{\sigma^{-1}(i)}.
\end{align}
An $\mathfrak{S}_{I_\uf} \times \mathfrak{S}_{I_\f}$-orbit of seeds is called an \emph{unlabeled seed}. 
Two seeds in $\cF$ are said to be \emph{mutation-equivalent} if they are transformed to each other by a finite sequence of seed mutations and permutations. An equivalence class of seeds is called a \emph{mutation class}. The relations among the seeds in a given mutation class $\sfs$ can be encoded in the \emph{(labeled) exchange graphs}:
\begin{dfn}
The \emph{labeled exchange graph} is a graph $\bExch_\sfs$ with vertices $v$ corresponding to the seeds $\sfs^{(v)}$ in $\sfs$, together with labeled edges of the following two types:
\begin{itemize}
    \item edges of the form $v \overbar{k} v'$ whenever the seeds $\sfs^{(v)}$ and $\sfs^{(v')}$ are related by the mutation $\mu_k$ for $k \in I_\uf$;
    \item edges of the form $v \overbar{\sigma} v'$ whenever the seeds $\sfs^{(v)}$ and $\sfs^{(v')}$ are related by the transposition $\sigma=(j\ k)$ for $(j,k) \in I_\uf \times I_\uf$ or $I_\f \times I_\f$.
\end{itemize}
The \emph{exchange graph} is a graph $\Exch_\sfs$ with vertices $\omega$ corresponding to the unlabeled seeds $\sfs^{(\omega)}$ in $\sfs$, together with (unlabeled) edges corresponding to the mutations. There is a graph projection $\pi_\sfs: \bExch_\sfs \to \Exch_\sfs$. 
\end{dfn}
When no confusion can occur, we simply denote a vertex of the labeled exchange graph by $v \in \bExch_\sfs$ instead of $v \in V(\bExch_\sfs)$, and similarly for the exchange graph. We remark that the (labeled) exchange graph depends only on the mutation class of the underlying exchange matrices. 

To each vertex $\omega \in \Exch_\sfs$, associated is the unordered collection $\mathbf{A}_{(\omega)}=\{A_i^{(v)}\}_{i \in I}$ of cluster variables called a \emph{cluster},
where $v \in \pi_\sfs^{-1}(\omega)$. Let $\bZ[\mathbf{A}_{(\omega)}^{\pm 1}]:=\bZ[(A_i^{(v)})^{\pm 1} \mid i \in I]$ denote the ring of Laurent polynomials.

\begin{dfn}
The \emph{cluster algebra} associated with a mutation class $\sfs$ of seeds is the subring $\CA_{\sfs} \subset \cF$ generated by the union of the clusters $\mathbf{A}_{(\omega)}$ for $\omega \in \Exch_\sfs$ and the inverses of the frozen variables.
The \emph{upper cluster algebra} is defined to be the subring
\begin{align*}
    \UCA_\sfs:= \bigcap_{\omega \in \Exch_\sfs} \bZ[\mathbf{A}_{(\omega)}^{\pm 1}] \subset \cF.
\end{align*}
\end{dfn}
The \emph{Laurent phenomenon theorem} \cite[Theorem 3.1]{FZ-CA1} tells us that each cluster variable can be expressed as a Laurent polynomial in any cluster, and hence $\CA_\sfs \subset \UCA_\sfs$ holds. We remark that the (upper) cluster algebra depends only on the mutation class of exchange matrices, up to  automorphisms of the ambient field. 


\subsubsection{The quantum cluster algebra}\label{subsub:q-CA}
We basically follow \cite{BZ}, partially employing the notation in \cite[Section 13.3]{GS19}. Recall that for a skew-symmetric form $\Pi$ on a lattice $L$, the associated \emph{based quantum torus} is the associative $\bZ_q$-algebra $T_\Pi$ such that
\begin{itemize}
    \item $T_\Pi$ has a free $\bZ_q$-basis $M^\alpha$ parametrized by $\alpha \in L$, and 
    \item the product of these basis elements is given by $M^\alpha\cdot M^\beta = q^{\Pi(\alpha,\beta)/2} M^{\alpha+\beta}$.
\end{itemize}
Let $\cF$ be a skew-field. 
A \emph{quantum seed} in $\cF$ is a quadruple $(B,\Pi,\accentset{\circ}{\Lambda},M)$, where
\begin{itemize}
    \item $B$ is an exchange matrix;
    \item $\Pi=(\pi_{ij})_{i,j \in I}$ is a skew-symmetric matrix with integral entries satisfying the \emph{compatibility relation}
    \begin{align*}
        \sum_{k \in I} b_{ki}\pi_{kj} =\delta_{ij}d_j
    \end{align*}
    for all $i \in I_\uf$ and $j \in I$, where $d_i$ is a positive integer for $i \in I_\uf$. 
    \item $\accentset{\circ}{\Lambda}=\bigoplus_{i \in I}\sff_i$ is a lattice, on which the matrix $\Pi$ defines a skew-symmetric form by $\Pi(\sff_i,\sff_j):=\pi_{ij}$ \footnote{Our notation for the lattice is motivated by the connection to the theory of cluster varieties: see \cref{sec:FG}.};
    \item $M: \accentset{\circ}{\Lambda} \to \cF\setminus \{0\}$ is a function such that 
    \begin{align*}
        M(\alpha)M(\beta) = q^{\Pi(\alpha,\beta)/2} M(\alpha+\beta)
    \end{align*}
    for $\alpha,\beta \in \accentset{\circ}{\Lambda}$, 
    and the $\bZ_q$-span of $M(\accentset{\circ}{\Lambda}) \subset \cF$ is the based quantum torus of the form $\Pi$ whose skew-field of fractions coincides with $\cF$. 
\end{itemize}
We call $\Pi$ the \emph{compatibility matrix}, and $M$ the \emph{toric frame} of the quantum seed. When no confusion can occur, we omit the lattice $\accentset{\circ}{\Lambda}$ from the notation and call the triple $(B,\Pi,M)$ a quantum seed. 
The compatibility relation can be written as
\begin{align}\label{eq:compatibility_definition}
    \varepsilon\Pi = B^{\mathsf{T}} \Pi = (D,0),
\end{align}
where $D:=\mathrm{diag}(d_i\mid i \in I_\uf)$ and $0$ denotes the $I_\uf \times I_\f$-zero matrix. By \cite[Lemma 4.4]{BZ}, a toric frame $M$ is uniquely determined by the values $A_i:=M(\sff_i)$, which we call the \emph{(quantum) cluster variable}, on the basis vectors $\sff_i$ for $i \in I$. Indeed, 
we have
\begin{align}\label{eq:extension_toric_frame}
    M\left(\sum_{i \in I} x_i\sff_i\right) = q^{\frac{1}{2}\sum_{l < k}x_k x_l \pi_{kl} } A_1^{x_1}\dots A_N^{x_N}
\end{align}
for all $(x_1,\dots,x_N) \in \bZ^N$.
Note that both sides are invariant under permutations of indices. 
Elements of the form $M(\alpha)$ for $\alpha \in \accentset{\circ}{\Lambda}$ are called \emph{cluster monomials}. 
In order to motivate the quantum mutations, we recall the following lemma:

\begin{lem}[{\cite[(3.1)]{BZ}}]
The matrix mutation \eqref{eq:matrix mutation} can be written as
\begin{align*}
    B' = E_{k,\epsilon} B F_{k,\epsilon},
\end{align*}
where $E_{k,\epsilon}= (e_{ij})_{i,j \in I}$ and $F_{k,\epsilon}= (f_{ij})_{i,j \in I}$ are defined by
\begin{align*}
    e_{ij}:=\begin{cases}
    \delta_{ij} & \mbox{if $j \neq k$},\\
    -1 & \mbox{if $i=k=j$},\\
    [-\epsilon b_{ik}]_+ & \mbox{if $i\neq k=j$},
    \end{cases}
\end{align*}
and 
\begin{align*}
    f_{ij}:=\begin{cases}
    \delta_{ij} & \mbox{if $i \neq k$},\\
    -1 & \mbox{if $i=k=j$},\\
    [\epsilon b_{kj}]_+ & \mbox{if $i=k \neq j$},
    \end{cases}
\end{align*}
respectively for $k \in I_\uf$ and $\epsilon \in \{+,-\}$. 
\end{lem}
Given a quantum seed $(B,\Pi,M)$ in $\cF$ and an unfrozen index $k \in I_\uf$, the \emph{quantum seed mutation} produces a new quantum seed $(B',\Pi',M')=\mu_k(B,\Pi,M)$ according to the following rule: 
\begin{align}
    B' &= E_{k,\epsilon} B F_{k,\epsilon}, \nonumber\\
    \Pi' &= E_{k,\epsilon}^{\mathsf{T}} \Pi E_{k,\epsilon}, \nonumber\\
    M'(\sff_i') &= \begin{cases}
    M(\sff_i) & \mbox{if $i \neq k$}, \\
    M(-\sff_k + \sum_{j \in I} [b_{jk}]_+ \sff_j) + M(-\sff_k + \sum_{j \in I} [-b_{jk}]_+ \sff_j) & \mbox{if $i=k$}.
    \end{cases}\label{eq:q-mutation}
\end{align}
Here $(\sff_i)_{i \in I}$ and $(\sff'_i)_{i \in I}$ denote the basis vectors of the underlying lattices. The relation \eqref{eq:q-mutation} is called the \emph{quantum exchange relation}. 
The verification of the following lemma is straightforward:

\begin{lem}\label{lem:compatibility_check}
Let $(B,\Pi,M)$ be a quantum seed in $\cF$, $k \in I_\uf$, and consider the exchange matrix $B' := E_{k,\epsilon} B F_{k,\epsilon}$ and the toric frame $M'$ determined by \eqref{eq:q-mutation}. Let $\Pi'=(\pi'_{ij})_{i,j \in I}$ be the skew-symmetric matrix associated with $M'$, which is uniquely determined by the condition 
\begin{align*}
    A'_i A'_j = q^{\pi'_{ij}} A'_j A'_i
\end{align*}
for $i,j \in I$ with $A'_i:=M'(\sff'_i)$. Then the pair $(B',\Pi')$ satisfies the compatibility relation.
\end{lem}


For a permutation $\sigma\in \mathfrak{S}_{I_\uf} \times \mathfrak{S}_{I_\f}$, a quantum seed $(B',\Pi',M')=\sigma(B,\Pi,M)$ is defined by 
\begin{align*}
    b'_{ij} = b_{\sigma^{-1}(i),\sigma^{-1}(j)}, \quad \pi'_{ij} = \pi_{\sigma^{-1}(i),\sigma^{-1}(j)}, \quad A'_i = A_{\sigma^{-1}(i)}.
\end{align*}
Two quantum seeds in $\cF$ are said to be \emph{mutation-equivalent} if they are transformed to each other by a finite sequence of quantum seed mutations and permutations. An equivalence class of quantum seeds is again called a \emph{mutation class}. The (labeled) exchange graphs $\bExch_{\sfs_q}$, $\Exch_{\sfs_q}$ of quantum seeds can be introduced just in the same way as the classical case. However, we do not need these graphs by the following reason. 

Given a mutation class $\sfs_q$ of quantum seeds in $\cF$, a mutation class $\sfs$ of seeds in some field is called a \emph{classical counterpart} of $\sfs_q$ if they share the collection of the underlying exchange matrices. Then it is known that the collection of quantum seeds in $\sfs_q$ are in a one-to-one correspondence with the seeds in $\sfs$, and the natural covering $\bExch_{\sfs_q} \to \bExch_\sfs$ and its unlabeled version are in fact isomorphisms \cite[Theorem 6.1]{BZ}. Therefore, to each vertex $v \in \bExch_{\sfs}$, we can associate a based quantum torus 
\begin{align*}
    T_{(v)}=\mathrm{span}_{\bZ_q} M^{(v)}(\accentset{\circ}{\Lambda}^{(v)}) \subset \cF.
\end{align*}
We also have the unlabeled version $T_{(\omega)}=\mathrm{span}_{\bZ_q} M^{(\omega)}(\accentset{\circ}{\Lambda}^{(\omega)})$ for $\omega \in \Exch_\sfs$, where the basis of $\Lambda^{(\omega)}$ is given up to permutations. The unordered collection $\mathbf{A}_{(\omega)}:=
\{A_i^{(v)}\}_{i \in I}$ is called a \emph{quantum cluster}, where $v \in \pi_\sfs^{-1}(\omega)$. 

\begin{dfn}
The \emph{quantum cluster algebra} associated with a mutation class $\sfs_q$ of quantum seeds is the $\bZ_q$-subalgebra $\CA_{\sfs_q} \subset \cF$ generated by the union of the quantum clusters $\mathbf{A}_{(\omega)}$ for $\omega \in \Exch_{\sfs}$ and the inverses of frozen variables. 
The \emph{quantum upper cluster algebra} is defined to be
\begin{align*}
    \UCA_{\sfs_q} :=\bigcap_{\omega \in \Exch_{\sfs}} T_{(\omega)} \subset \cF.
\end{align*}
\end{dfn}
For each vertex $\omega \in \Exch_\sfs$, the \emph{upper bound} at $\omega$ is defined to be 
\begin{align*}
    \UCA_{\sfs_q}(\omega):=T_{(\omega)}\cap \bigcap_{\omega'} T_{(\omega')},
\end{align*}
where $\omega' \in \Exch_\bs$ runs over the vertices adjacent to $\omega$. 

\begin{thm}[Quantum upper bound theorem {\cite[Theorem 5.1]{BZ}}]\label{thm:q-Laurent}
For any vertices $\omega,\omega' \in \Exch_\sfs$, we have $\UCA_{\sfs_q}(\omega) = \UCA_{\sfs_q}(\omega')$. In particular, we have
\begin{align*}
    \UCA_{\sfs_q} = \UCA_{\sfs_q}(\omega)
\end{align*}
for any $\omega \in \Exch_\sfs$.\footnote{We remark here that the coprimality condition required in the classical setting (\cite[Corollary 1.7]{BFZ}) is automatically satisfied in the quantum setting, since the existence of the compatibility matrix forces the exchange matrix to be full-rank.}
\end{thm}
It in particular implies the inclusion $\CA_{\sfs_q} \subset \UCA_{\sfs_q}$, which is called the \emph{quantum Laurent phenomenon}. Again we remark that the quantum (upper) cluster algebra depends only on the mutation class of the compatibility pairs $(B,\Pi)$, up to automorphisms of the ambient skew-field. In other words, the choice of toric frames determines the way of realization of these algebras in some skew-field.

\paragraph{\textbf{Bar-involution}}
For each $\omega \in \Exch_{\sfs}$, define a $\bZ$-linear
involution $\dagger: T_{(\omega)} \to T_{(\omega)}$ by
\begin{align*}
    (q^{r/2}M^{(\omega)}(\alpha))^\dagger := q^{-r/2}M^{(\omega)}(\alpha)
\end{align*}
for $r \in \bZ$ and $\alpha \in \accentset{\circ}{\Lambda}{}^{(\omega)}$. Then $\dagger$ preserves the subalgebra $\UCA_{\sfs_q} \subset T_{(\omega)}$, and the induced involution does not depend on the choice of $\omega$ \cite[Proposition 6.2]{BZ}. Following \cite{BZ}, we call this anti-involution $\dagger: \UCA_{\sfs_q} \to \UCA_{\sfs_q}$ the \emph{bar-involution}. Each quantum cluster variable is invariant under the bar-involution.

\bigskip
\paragraph{\textbf{Ensemble grading}}
We have a natural grading on the (quantum) upper cluster algebra, which we call the \emph{ensemble grading} (a.k.a. \emph{universal grading} \cite{Muller16}). In order to motivate its definition from the algebro-geometric viewpoint, we borrow some notations from \cite{FG09}, for which the reader is referred to \cref{sec:FG}.

\begin{lemdef}[cf. {\cite[Lemma 5.3]{GSV}}]\label{lemdef:ensemble_grading}
For each $v \in \bExch_\sfs$, define $\mathbf{gr}(A_i^{(v)}) \in \coker p^*$ to be the image of the basis vector $\sff_i^{(v)} \in \accentset{\circ}{\Lambda}^{(v)}$ under the natural projection $\alpha^*_{(v)}: \accentset{\circ}{\Lambda}^{(v)} \to \coker p_{(v)}^*$. Then $\mathrm{\mathbf{gr}}$ defines a grading on the ring $\UCA_\sfs$, which we call the \emph{ensemble grading}. 
The ensemble grading on the quantum upper cluster algebra $\UCA_{\sfs_q}$ is defined by the same manner, which makes the latter a graded $\bZ_q$-algebra. 
\end{lemdef}
Indeed, the grading on the upper cluster algebra $\UCA_\sfs=\cO(\A_\sfs)$ is the same as the one explained in \cref{sec:FG}. The quantum version is similarly seen to be well-defined by using the \lq\lq decomposition of mutations'' formula in the quantum setting \cite[(4.22)]{BZ}, where the automorphism part does not affect on the grading. The grading $\mathbf{gr}$ is the universal grading in the sense of \cite{GSV}, while we choose to call it the ensemble grading in order to emphasize its relation to the cluster ensemble structure.

\subsection{The cluster algebra related to the moduli space \texorpdfstring{$\A_{SL_3,\Sigma}$}{A(G,Sigma)}}\label{subsec:classical_Teich}

Let $\Sigma$ be an unpunctured marked surface as in \cref{sec:intro}. 
Recall that a decorated triangulation $\bD=(\Delta,\bs_\Delta)$ consists of an ideal triangulation $\Delta$ of $\Sigma$, together with a function $\bs_\Delta: t(\Delta) \to \{+,-\}$. 
Given a decorated triangulation $\bD$, we define a quiver $Q^{\bD}$ with the vertex set $I(\Delta)=I_{\mathfrak{sl}_3}(\Delta)$ as follows. Let $Q_+$ and $Q_-$ be the quivers shown in the left and right of \cref{fig:triangle_quivers}, respectively. These quivers are related by the mutation at the central vertex $k$. 
For each triangle $T \in t(\Delta)$, we draw the quiver $Q_{\bs_\Delta(T)}$, and glue them via the \emph{amalgamation} procedure \cite{FG06} to get a quiver $Q^{\bD}$ drawn on $\Sigma$. In our situation, opposite half-arrows cancel together, and parallel half-arrows combine to give a usual arrow. Some examples are shown in \cref{fig:amalgamation}.

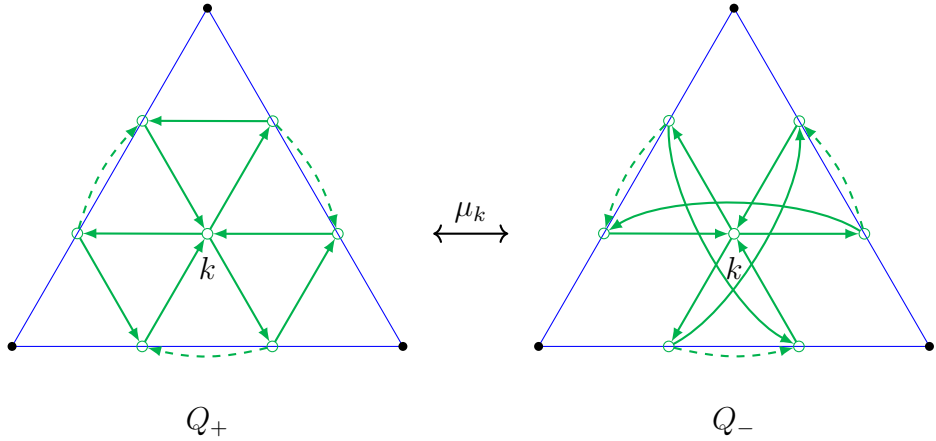
\begin{figure}[h]
\begin{tikzpicture}

\draw[blue] (210:3) -- (-30:3) -- (90:3) --cycle;
\foreach \x in {-30,90,210}
\path(\x:3) node [fill, circle, inner sep=1.2pt]{};

\begin{scope}[color=mygreen,>=latex]
\quiverplus{210:3}{-30:3}{90:3};
\node[black] at (0,-2.5) {$Q_+$};
\qdlarrow{x122}{x121}
\qdlarrow{x232}{x231}
\qdlarrow{x312}{x311}
\draw[black](G) node[below=0.4em]{$k$};
\end{scope}
\draw[<->,thick] (3,0) --node[midway,above]{$\mu_{k}$} (4,0);
\begin{scope}[xshift=7cm]
    \draw[blue] (210:3) -- (-30:3) -- (90:3) --cycle;
    \foreach \x in {-30,90,210}
    \path(\x:3) node [fill, circle, inner sep=1.2pt]{};
    {\begin{scope}[color=mygreen,>=latex]
    \quiverminus{210:3}{-30:3}{90:3};
    \node[black] at (0,-2.5) {$Q_-$};
    \qdrarrow{x121}{x122}
    \qdrarrow{x231}{x232}
    \qdrarrow{x311}{x312}
    \draw[black](G) node[below=0.4em]{$k$};
    \end{scope}}
\end{scope}
\end{tikzpicture}
    \caption{Two quivers on a triangle.}
    \label{fig:triangle_quivers}
\end{figure}

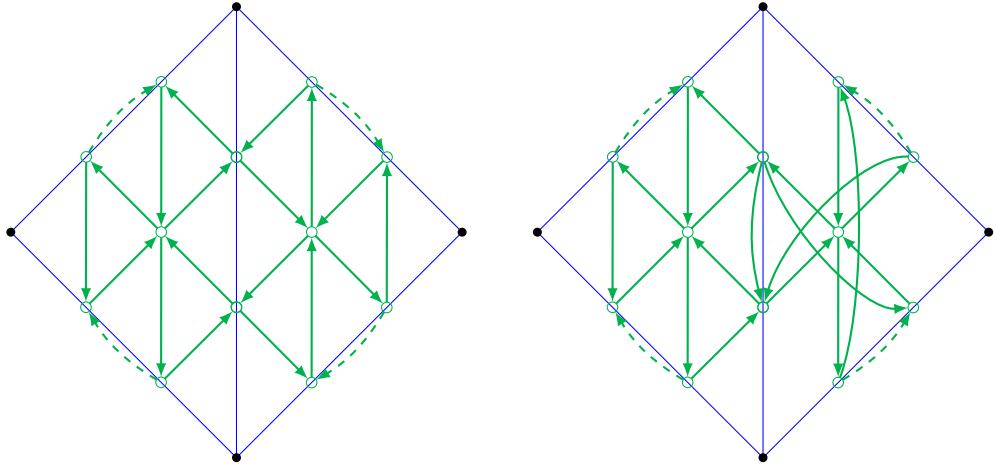
\begin{figure}
\begin{tikzpicture}
\draw[blue] (0:3) -- (90:3) -- (180:3) -- (270:3) --cycle; 
\draw[blue] (90:3) -- (270:3);
\foreach \x in {0,90,180,270}
\path(\x:3) node [fill, circle, inner sep=1.2pt]{};
\begin{scope}[color=mygreen,>=latex]
\quiverplus{0:3}{90:3}{270:3}
    \qdlarrow{x122}{x121}
    \qdlarrow{x312}{x311}
\quiverplus{90:3}{180:3}{270:3}
    \qdlarrow{x122}{x121}
    \qdlarrow{x232}{x231}
\end{scope}
    
\begin{scope}[xshift=7cm]
\draw[blue] (0:3) -- (90:3) -- (180:3) -- (270:3) --cycle; 
\draw[blue] (90:3) -- (270:3);
\foreach \x in {0,90,180,270}
\path(\x:3) node [fill, circle, inner sep=1.2pt]{};
{\begin{scope}[color=mygreen,>=latex]
\quiverminus{0:3}{90:3}{270:3}
    \qdrarrow{x121}{x122}
    \qdrarrow{x311}{x312}
\quiverplus{90:3}{180:3}{270:3}
    \qdlarrow{x122}{x121}
    \qdlarrow{x232}{x231}
    \qarrowbr{x312}{x311}
\end{scope}}
\end{scope}
\end{tikzpicture}
    \caption{The quivers on a quadrilateral with the signs $(+,+)$ (left) and $(+,-)$ (right).}
    \label{fig:amalgamation}
\end{figure}

Let $B^{\bD}=(b_{ij}^{\bD})_{i,j \in I(\bD)}$ denote the exchange matrix determined by the quiver $Q^{\bD}$. 

\begin{thm}[Fock--Goncharov {\cite[Section 10.3]{FG03}}]\label{thm:mutation-equivalence}
The exchange matrices $B^{\bD}$ associated with decorated triangulations $\bD$ of a fixed marked surface $\Sigma$ are mutation-equivalent to each other.
\end{thm}
For later use, we reproduce the proof here.

\begin{proof}
Let $\bD$, $\bD'$ be two decorated triangulations of a marked surface $\Sigma$. 
Since the quiver $Q_-$ is transformed into $Q_+$ by a mutation and the amalgamations commute with mutations at the vertices in $I^\mathrm{tri}(\Delta)$, we can assume that both $\bD$ and $\bD'$ have the positive sign on each triangle. Moreover, since any two ideal triangulations are transformed to each other by a sequence of flips, it suffices to consider the case where the underlying triangulations of $\bD$ and $\bD'$ are related by the flip along an edge. Some sequences of mutations which realizes a flip are shown in \cref{fig:flip sequence}. The assertion is proved.
\end{proof}

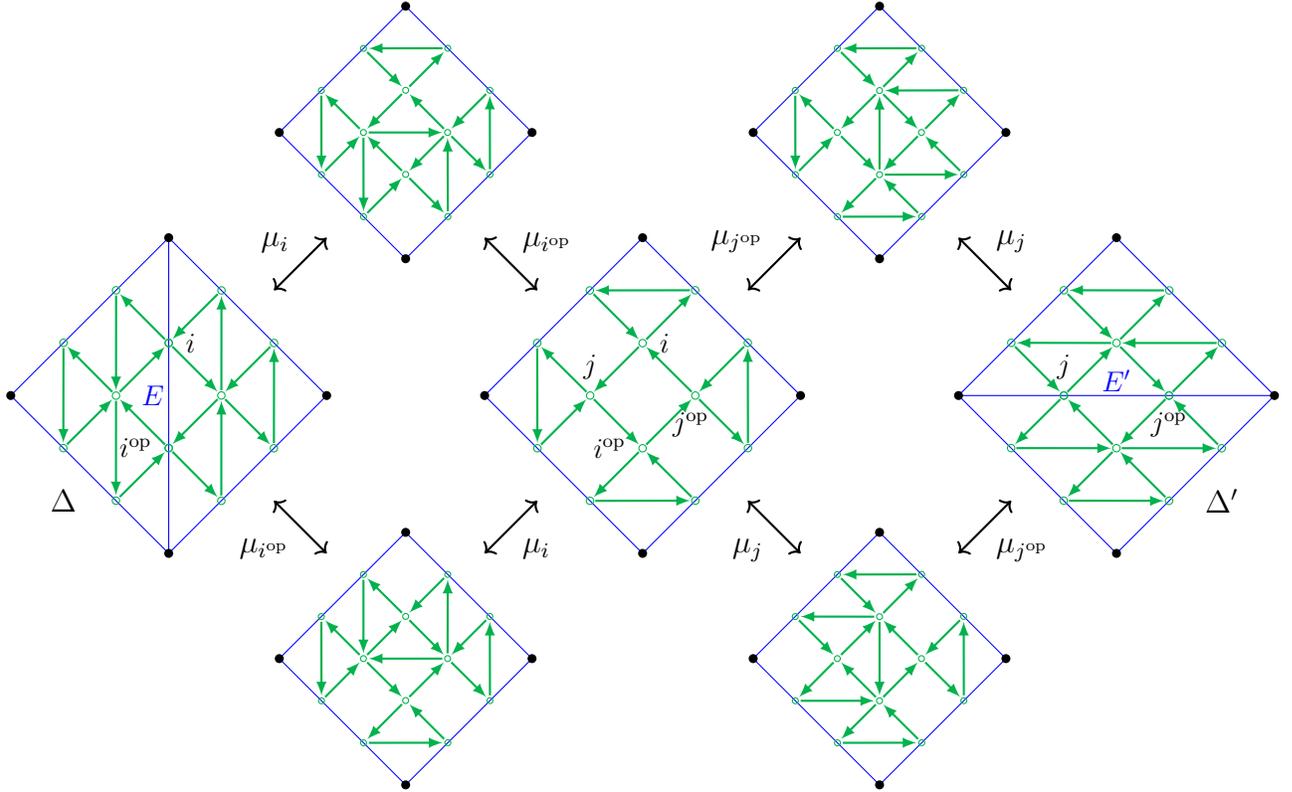
\begin{figure}
\begin{tikzpicture}[scale=0.7]
{\color{blue}
\draw (3,0) -- (0,3) -- (-3,0) -- (0,-3) --cycle;
\draw[blue] (0,-3) --node[midway,left=-0.2em]{\scalebox{0.9}{$E$}} (0,3);
}
\draw (-2,-2) node{$\Delta$};
\foreach \x in {0,90,180,270}
\path(\x:3) node [fill, circle, inner sep=1.2pt]{};
\quiverplus{3,0}{0,3}{0,-3}
\draw(x231) node[right=0.2em]{\scalebox{0.9}{$i$}};
\draw(x232) node[left=0.2em]{\scalebox{0.9}{$i^\mathrm{op}$}};
\quiverplus{-3,0}{0,-3}{0,3}

\begin{scope}[xshift=4.5cm,yshift=-5cm,scale=0.8,>=latex,yscale=-1]
\draw[blue] (3,0) -- (0,3) -- (-3,0) -- (0,-3) --cycle;
\foreach \x in {0,90,180,270}
\path(\x:3) node [fill, circle, inner sep=1.2pt]{};
{\color{mygreen}
\quiversquare{-3,0}{0,-3}{3,0}{0,3}
    \qarrow{y241}{x122}
				\qarrow{x122}{y131}
				\qarrow{y131}{x121}
				\qarrow{x121}{x412}
				\qarrow{x412}{y131}
				\qarrow{y131}{y241}
				\qarrow{y241}{y132}
				\qarrow{y132}{x341}
				\qarrow{x341}{x232}
				\qarrow{x232}{y132}
				\qarrow{y132}{x231}
				\qarrow{x231}{y241}
				
				\qarrow{x342}{y242}
				\qarrow{y242}{x411}
				\qarrow{x411}{x342}
				\qarrow{y131}{y242}
				\qarrow{y242}{y132}
				\qarrow{y132}{y131}
}

\end{scope}

\begin{scope}[xshift=4.5cm,yshift=5cm,scale=0.8,>=latex,yscale=-1]
\draw[blue] (3,0) -- (0,3) -- (-3,0) -- (0,-3) --cycle;
\foreach \x in {0,90,180,270}
\path(\x:3) node [fill, circle, inner sep=1.2pt]{};
{\color{mygreen}
\quiversquare{-3,0}{0,-3}{3,0}{0,3}
    \qarrow{y242}{x342}
				\qarrow{x342}{y132}
				\qarrow{y132}{x341}
				\qarrow{x341}{x232}
				\qarrow{x232}{y132}
				\qarrow{y132}{y242}
				\qarrow{y242}{y131}
				\qarrow{y131}{x121}
				\qarrow{x121}{x412}
				\qarrow{x412}{y131}
				\qarrow{y131}{x411}
				\qarrow{x411}{y242}
				
				\qarrow{y132}{y241}
				\qarrow{y241}{y131}
				\qarrow{y131}{y132}
				\qarrow{x122}{y241}
				\qarrow{y241}{x231}
				\qarrow{x231}{x122}
}
\end{scope}

\begin{scope}[xshift=9cm,>=latex,xscale=-1]
\draw[blue] (3,0) -- (0,3) -- (-3,0) -- (0,-3) --cycle;
\foreach \x in {0,90,180,270}
\path(\x:3) node [fill, circle, inner sep=1.2pt]{};
{\color{mygreen}
\quiversquare{-3,0}{0,-3}{3,0}{0,3}
    \qarrow{x122}{y241}
				\qarrow{y131}{x121}
				\qarrow{x121}{x412}
				\qarrow{x412}{y131}
				\qarrow{y241}{y131}
				\qarrow{y132}{y241}
				\qarrow{y132}{x341}
				\qarrow{x341}{x232}
				\qarrow{x232}{y132}
				\qarrow{y241}{x231}
				\qarrow{x231}{x122}
				
				\qarrow{x342}{y242}
				\qarrow{y242}{x411}
				\qarrow{x411}{x342}
				\qarrow{y131}{y242}
				\qarrow{y242}{y132}
}
\draw(y242) node[right=0.2em]{\scalebox{0.9}{$i$
}};
\draw(y241) node[left=0.2em]{\scalebox{0.9}{$i^\mathrm{op}$}};
\draw(y131) node[below=0.2em]{\scalebox{0.9}{$j^\mathrm{op}$
}};
\draw(y132) node[above=0.2em]{\scalebox{0.9}{$j$}};
\end{scope}

\begin{scope}[xshift=13.5cm,yshift=5cm,scale=0.8,xscale=-1,rotate=90,>=latex]
\draw[blue] (3,0) -- (0,3) -- (-3,0) -- (0,-3) --cycle;
\foreach \x in {0,90,180,270}
\path(\x:3) node [fill, circle, inner sep=1.2pt]{};
{\color{mygreen}
\quiversquare{-3,0}{0,-3}{3,0}{0,3}
    \qarrow{y242}{x342}
				\qarrow{x342}{y132}
				\qarrow{y132}{x341}
				\qarrow{x341}{x232}
				\qarrow{x232}{y132}
				\qarrow{y132}{y242}
				\qarrow{y242}{y131}
				\qarrow{y131}{x121}
				\qarrow{x121}{x412}
				\qarrow{x412}{y131}
				\qarrow{y131}{x411}
				\qarrow{x411}{y242}
				
				\qarrow{y132}{y241}
				\qarrow{y241}{y131}
				\qarrow{y131}{y132}
				\qarrow{x122}{y241}
				\qarrow{y241}{x231}
				\qarrow{x231}{x122}
}
\end{scope}

\begin{scope}[xshift=13.5cm,yshift=-5cm,scale=0.8,yscale=-1,rotate=-90,>=latex]
\draw[blue] (3,0) -- (0,3) -- (-3,0) -- (0,-3) --cycle;
\foreach \x in {0,90,180,270}
\path(\x:3) node [fill, circle, inner sep=1.2pt]{};
{\color{mygreen}
\quiversquare{-3,0}{0,-3}{3,0}{0,3}
    \qarrow{y241}{x122}
				\qarrow{x122}{y131}
				\qarrow{y131}{x121}
				\qarrow{x121}{x412}
				\qarrow{x412}{y131}
				\qarrow{y131}{y241}
				\qarrow{y241}{y132}
				\qarrow{y132}{x341}
				\qarrow{x341}{x232}
				\qarrow{x232}{y132}
				\qarrow{y132}{x231}
				\qarrow{x231}{y241}
				
				\qarrow{x342}{y242}
				\qarrow{y242}{x411}
				\qarrow{x411}{x342}
				\qarrow{y131}{y242}
				\qarrow{y242}{y132}
				\qarrow{y132}{y131}
}
\end{scope}

\begin{scope}[xshift=18cm]
{\color{blue}
\draw (3,0) -- (0,3) -- (-3,0) -- (0,-3) --cycle;
\draw (3,0) --node[midway,above=-0.2em]{\scalebox{0.9}{$E'$}} (-3,0);
}
\foreach \x in {0,90,180,270}
\path(\x:3) node [fill, circle, inner sep=1.2pt]{};
\quiverplus{-3,0}{3,0}{0,3}
\draw(x121) node[above=0.2em]{\scalebox{0.9}{$j$}};
\draw(x122) node[below=0.2em]{\scalebox{0.9}{$j^\mathrm{op}$}};
\quiverplus{-3,0}{0,-3}{3,0}
\draw (2,-2) node{$\Delta'$};
\end{scope}

\draw[thick,<->] (2,2) --node[midway,above left]{$\mu_{i}$} (3,3);
\draw[thick,<->] (2,-2) --node[midway,below left]{$\mu_{i^\mathrm{op}}$} (3,-3);
\draw[thick,<->] (6,3) --node[midway,above right]{$\mu_{i^\mathrm{op}}$} (7,2);
\draw[thick,<->] (6,-3) --node[midway,below right]{$\mu_{i}$} (7,-2);
\draw[thick,<->] (16,2) --node[midway,above right]{$\mu_{j}$} (15,3);
\draw[thick,<->] (16,-2) --node[midway,below right]{$\mu_{j^\mathrm{op}}$} (15,-3);
\draw[thick,<->] (12,3) --node[midway,above left]{$\mu_{j^\mathrm{op}}$} (11,2);
\draw[thick,<->] (12,-3) --node[midway,below left]{$\mu_{j}$} (11,-2);
\end{tikzpicture}
    \caption{Some of the sequences of mutations that realize the flip $f_E: \Delta \to \Delta'$.}
    \label{fig:flip sequence}
\end{figure}

In particular, there exists a canonical mutation class $\sfs({\mathfrak{sl}_3,\Sigma})$ containing the exchange matrices $B^{\bD}$ associated with any decorated triangulation $\bD$. For a geometric construction of an unlabeled seed $(B^{\bD},\mathbf{A}^{\bD})$, see \cref{rem:geometry of moduli} below. Let us simplify the notation as
\begin{align*}
    \CA_{\mathfrak{sl}_3,\Sigma}:=\CA_{\sfs({\mathfrak{sl}_3,\Sigma})}, \quad \UCA_{\mathfrak{sl}_3,\Sigma}:=\UCA_{\sfs({\mathfrak{sl}_3,\Sigma})}, \mbox{ and } \Exch_{\mathfrak{sl}_3,\Sigma}:=\Exch_{\sfs({\mathfrak{sl}_3,\Sigma})}. 
\end{align*}

It is typically hard to understand all the seeds in $\sfs({\mathfrak{sl}_3,\Sigma})$ in geometric terms. 
We are first going to consider those associated with the decorated triangulations and those along the flip sequences. A \emph{decorated cell decomposition} (of deficiency $\leq 1$) is an ideal cell decomposition $(\Delta;E)$ of deficiency $1$ equipped with a sign on each triangle and one of the quivers shown in \cref{fig:flip sequence} on the unique quadrilateral. In particular, a decorated triangulation is a decorated cell decomposition. 

\begin{dfn}
Define the \emph{surface subgraph} to be the subgraph $\Exch'_{\mathfrak{sl}_3,\Sigma} \subset \Exch_{{\mathfrak{sl}_3,\Sigma}}$ such that 
\begin{itemize}
    \item the vertices are the seeds corresponding to the decorated cell decompositions;
    \item the edges are mutations realizing changes of the signs (\cref{fig:triangle_quivers}) and those realizing flips (\cref{fig:flip sequence}).
\end{itemize}
The exchange matrix $B^{(\omega)}$ for any vertex $\omega \in \Exch'_{\mathfrak{sl}_3,\Sigma}$ is determined by the corresponding quiver in \cref{fig:flip sequence}. 
\end{dfn}
Here is a remark on the labeling. If we fix a labeling $\ell: I(\bD) \xrightarrow{\sim} \{1,\dots,N\}$, then the part of the exchange graph shown in \cref{fig:flip sequence} can be lifted to the labeled exchange graph $\bExch_\sfs$. We call the pair $(\bD,\ell)$ a \emph{labeled decorated triangulation}. 
In particular, we can use such a labeling for the indices of exchange matrices $B^{(\omega)}$ associated with the vertices in this part. 
In other words, for two ideal triangulations $\Delta$ and $\Delta'$ related by a single flip, the two sets $I(\Delta)$ and $I(\Delta')$ can be canonically identified, and we use this set as the common index set $I(\omega)$ for these exchange matrices.

\begin{rem}[Relation to the moduli space $\A_{SL_3,\Sigma}$]\label{rem:geometry of moduli}
Via the Goncharov--Shen's construction \cite[Section  8.2]{GS19}, for any decorated triangulation $\bD$, we get a collection $\mathbf{A}^{\bD}$ of regular functions on the moduli space $\A_{SL_3,\Sigma}$ of decorated $SL_3$-local systems on $\Sigma$. Here the sign $+$ (resp. $-$) assigned to a triangle corresponds to the reduced word $(1,2,1)$ (resp. $(2,1,2)$) of the longest element $w_0 \in W(\mathfrak{sl}_3)$ in the Weyl group of $\mathfrak{sl}_3$, and thanks to the cyclic symmetry of the cluster structure on the moduli space $\A_{SL_3,T}$, we do not need to choose a vertex of each triangle $T$ as required there. Thus we get a seed $(B^{\bD},\mathbf{A}^{\bD})$ in the field of rational functions on $\A_{SL_3,\Sigma}$. These seeds are mutation-equivalent to each other \cite[Theorem 8.7]{GS19}. 
We have the isomorphisms \cite[Theorem 4.3]{IOS}\footnote{In general, it is true that the function ring $\cO(\A^\times_{G,\Sigma})$ coincides with the upper bound $\UCA_{\sfs({\mathfrak{g},\Sigma)}}(\bD)$ at any decorated triangulation $\bD$ by an argument parallel to the proof of \cite[Theorem 1.1]{Shen20}. The first author thanks Linhui Shen for his explanation of this statement. \label{foot:moduli}
}.
\begin{align*}  \UCA_{\mathfrak{sl}_3,\Sigma}
    =\cO(\A_{\mathfrak{sl}_3,\Sigma}) = \cO(\A^\times_{SL_3,\Sigma}),
\end{align*}
where $\A^\times_{SL_3,\Sigma} \subset \A_{SL_3,\Sigma}$ is the open substack obtained by requiring the pairs of decorated flags associated with any boundary intervals to be generic. This corresponds to our localization convention of frozen variables. 

\end{rem}

\paragraph{\textbf{Some group actions}}
Recall the cluster modular group $\Gamma_\sfs$ from \cref{sec:FG}, which acts on the (upper) cluster algebra from the right by permuting the clusters. 
By \cite[Theorem 6.1]{BZ}, this action lifts to any quantization $\CA_{\sfs_q}$. 

When $\sfs=\sfs(\mathfrak{sl}_3,\Sigma)$, it is known that the cluster modular group contains the group $MC(\Sigma) \times \mathrm{Out}(SL_3)$ \cite{GS16}. Here 
\begin{itemize}
    \item $MC(\Sigma)$ denotes the \emph{mapping class group} of $\Sigma$, which consists of the isotopy classes of orientation-preserving homeomorphisms on $\Sigma$ that preserve $\partial \Sigma$ and $\bM$ set-wisely;
    \item $\mathrm{Out}(SL_3):=\mathrm{Aut}(SL_3)/\mathrm{Inn}(SL_3)$ denotes the outer automorphism group, which is generated by the \emph{Dynkin involution} $\ast:SL_3 \to SL_3$. 
\end{itemize}
The actions of these groups on $\CA_{\mathfrak{sl}_3,\Sigma}$ are described as follows. See \cite{GS16} for a detail.
\begin{itemize}
    \item Each mapping class $\phi \in MC(\Sigma)$ sends each labeled decorated triangulation $((\Delta,\bs_\Delta),\ell)$ to $((\phi^{-1}(\Delta),\phi^*\bs_\Delta),\phi^*\ell)$, where  $(\phi^*\bs_\Delta)(T):=\bs_\Delta(\phi(T))$ for $T \in t(\phi^{-1}(\Delta))$, and $\phi^*\ell:I(\phi^{-1}(\Delta)) \xrightarrow{\sim} I(\Delta) \xrightarrow{\sim} \{1,\dots,N\}$.
    The action on $\Exch_{\mathfrak{sl}_3,\Sigma}$ is uniquely interpolated by mutation-equivariance.
    
    \item The Dynkin involution $\ast$ sends each labeled decorated triangulation $((\Delta,\bs_\Delta),\ell)$ to $((\Delta,\bs^\ast_\Delta),\ell)$, where $\bs^\ast_\Delta(T):=-\bs_\Delta(T)$ for $T \in t(\Delta)$. 
    The action on $\Exch_{\mathfrak{sl}_3,\Sigma}$ is uniquely interpolated by mutation-equivariance.
\end{itemize}
Although the interpolation by mutation-equivariance is rather implicit, we will see that these actions are described as certain geometric actions on webs on $\Sigma$.

\section{Realization of the quantum cluster algebra inside \texorpdfstring{$\mathrm{Frac}\mathscr{S}^q_{\mathfrak{sl}_3,\Sigma}$}{Frac S}}\label{sect:correspondance}

In this section, we construct a mutation class $\sfs_q(\mathfrak{sl}_3,\Sigma)$ of quantum seeds in the skew-field $\mathrm{Frac}\Skein{\Sigma}^q$ of fractions of the skein algebra $\Skein{\Sigma}^q$, which quantizes the mutation class $\sfs(\mathfrak{sl}_3,\Sigma)$. It defines a quantum cluster algebra inside $\mathrm{Frac}\Skein{\Sigma}^q$.
In what follows, we identify the quantum parameters as $q=A$. 

For any vertex $\omega \in \Exch'_{\mathfrak{sl}_3,\Sigma}$ of the surface subgraph, we are going to define a quantum seed $(B^{(\omega)},\Pi^{(\omega)},M^{(\omega)})$ in $\mathrm{Frac}\Skein{\Sigma}^q$. The exchange matrix $B^{(\omega)}$ is the one already defined in \cref{subsec:classical_Teich}. 
In order to define the remaining data, we consider a web cluster (\cref{webcluster}) $C_{(\omega)}=\{e^{(\omega)}_i \mid i \in I(\omega)\}$ defined as follows. See \cref{fig:flip sequence_web}.
\begin{itemize}
    \item Suppose $\omega=\bD=(\Delta,\bs_\Delta)$ is a decorated triangulation. 
    If $i \in I^{\mathrm{tri}}(\Delta)$, then $e_i^{\bD}$ is one of the elementary webs on the corresponding triangle $T \in t(\Delta)$. If $\bs_\Delta(T)=+$ (resp. $\bs_\Delta(T)=-$), then it is defined to be the one with the unique trivalent sink (resp. source). If $i \in I^\mathrm{edge}(\Delta)$, then $e_i^{\bD}$ is one of the elementary webs given by assigning an orientation to the edge on which $i$ is located. The orientation is determined so that the terminal point is closer to the vertex $i$.
    \item For a decorated cell decomposition $\omega$ obtained by the mutation $\mu_i$ for $i \in I^\mathrm{edge}(\Delta)$ from a decorated triangulation $\bD$, we set $e_j^{(\omega)}:=e_j^\bD$ for $j \neq i$. Define $e_i^{(\omega)}$ to be the trivalent sink with endpoints three of the special points on the unique quadrilateral, which span a triangle that contains $i$ in its interior. 
    \item For a decorated cell decomposition $\omega$ obtained by the mutation $\mu_{\mathrm{op}}$ for $i \in I^\mathrm{edge}(\Delta)$ from the decorated cell decomposition $\omega':=\mu_i(\bD)$, we set $e_j^{(\omega)}:=e_j^{(\omega')}$ for $j \neq i^\mathrm{op}$. Define $e_{i^\mathrm{op}}^{(\omega)}$ to be the trivalent sink with endpoints three of the special points on the unique quadrilateral, which span a triangle that contains $i^\mathrm{op}$ in its interior. 
\end{itemize}
Then define the compatibility matrix  $\Pi^{(\omega)}=(\pi^{(\omega)}_{ij})_{i,j \in I(\omega)}$ by
\begin{align*}
    \pi^{(\omega)}_{ij}:= \Pi( e^{(\omega)}_i,e^{(\omega)}_j ).
\end{align*}
Here recall \cref{def:web-exchange}. 
Then $\Pi^{(\omega)}$ is evidently skew-symmetric.

\begin{prop}\label{prop:compatibility}
For any decorated triangulation $\bD=(\Delta,\bs_\Delta)$ with $\bs_\Delta(T)={+}$ for all $T \in t(\Delta)$, the pair $(B^{\bD},\Pi^{\bD})$ satisfies the compatibility relation
\begin{align*}
    (B^{\bD})^{\mathsf{T}}\Pi^{\bD} = (6\cdot\mathrm{Id}, 0).
\end{align*}

\end{prop}

\begin{figure}[t]
\begin{tikzpicture}
\draw[blue] (2,2) -- (-2,2) -- (-2,-2) -- (2,-2) --cycle;
\draw[blue] (2,2) -- (-2,-2);
\foreach \i in {0,90,180,270}
{
\begin{scope}[rotate=\i]
\filldraw[gray!30] (2-0.292*0.5,2+0.707*0.5) arc(180:270:0.5) -- (2+0.707*0.5,2+0.707*0.5);
\draw[thick] (2-0.292*0.5,2+0.707*0.5) arc(180:270:0.5);
\node[fill, circle, inner sep=1.2pt] at (2,2) {}; 
\end{scope}
}

\quiverplus{-2,-2}{2,-2}{2,2};
\draw(x232) node[right]{$k_1$};
\draw(G) node[above left]{$k_4$};
\quiverplus{-2,-2}{2,2}{-2,2};
\draw(x122) node[above left]{$i$};
\draw(x231) node[above]{$k_2$};
\draw(G) node[above left]{$k_3$};

\begin{scope}[xshift=7cm]
\draw[blue] (2,2) -- (-2,2) -- (-2,-2) -- (2,-2) --cycle;
\draw[blue] (2,2) -- (-2,-2);
\foreach \i in {0,90,180,270}
{
\begin{scope}[rotate=\i]
\filldraw[gray!30] (2-0.292*0.5,2+0.707*0.5) arc(180:270:0.5) -- (2+0.707*0.5,2+0.707*0.5);
\draw[thick] (2-0.292*0.5,2+0.707*0.5) arc(180:270:0.5);
\node[fill, circle, inner sep=1.2pt] at (2,2) {}; 
\end{scope}
}

\draw[->-,red,thick] (-2,2) --node[midway,above]{$e_{k_2}$} (2,2);
\draw[->-,red,thick] (2,-2) --node[midway,right]{$e_{k_1}$} (2,2);
\CoG{2,2}{-2,2}{-2,-2}
\draw[->-,red,thick] (2,2) -- (G) node[left] {$e_{k_3}$};
\draw[->-,red,thick] (-2,-2) -- (G);
\draw[->-,red,thick] (-2,2) -- (G);
\CoG{2,2}{2,-2}{-2,-2}
\draw[->-,red,thick] (2,2) -- (G)  node[right] {$e_{k_4}$};
\draw[->-,red,thick] (-2,-2) -- (G);
\draw[->-,red,thick] (2,-2) -- (G);
\draw[->-,purple,thick] (-3,0) to[out=60,in=260] node[midway,left]{$e_j$} (-2,2);
\end{scope}
\end{tikzpicture}
    \caption{The neighboring vertices to $i \in I^\mathrm{edge}(\bD)$ (left) and the corresponding collection of elementary webs (right). An additional elementary web $e_j$ is also shown in purple.}
    \label{fig:compatibility_figure1}
\end{figure}

\begin{figure}
\begin{tikzpicture}
\draw[blue] (-30:3) -- (90:3) -- (210:3) --cycle;
\foreach \i in {0,120,240}
{
\begin{scope}[rotate=\i]
\filldraw[gray!30] \centerarc(90:3.5)(-30:-150:0.5) --cycle;
\draw[thick] \centerarc(90:3.5)(-30:-150:0.5);
\end{scope}
}
\foreach \i in {-30,90,210}
\node[fill,circle, inner sep=1.2pt] at (\i:3) {}; 
\quiverplus{210:3}{-30:3}{90:3};
\draw(G) node[above=0.4em]{$i$};
\draw(x121) node[below]{$k_1$};
\draw(x122) node[below]{$k_2$};
\draw(x231) node[right]{$k_3$};
\draw(x232) node[right]{$k_4$};
\draw(x311) node[left]{$k_5$};
\draw(x312) node[left]{$k_6$};

\begin{scope}[xshift=7cm]
\draw[blue] (-30:3) -- (90:3) -- (210:3) --cycle;
\foreach \i in {0,120,240}
{
\begin{scope}[rotate=\i]
\filldraw[gray!30] \centerarc(90:3.5)(-30:-150:0.5) --cycle;
\draw[thick] \centerarc(90:3.5)(-30:-150:0.5);
\end{scope}
}
\foreach \i in {-30,90,210}
\node[fill,circle, inner sep=1.2pt] at (\i:3) {}; 

\draw[->-,red,thick] (-30:3) to[bend right=10] node[above right]{$e_{k_4}$} (90:3);
\draw[->-,red,thick] (90:3) to[bend right=10] node[midway,above left]{$e_{k_6}$} (210:3);
\draw[->-,red,thick] (210:3) to[bend right=10] node[below]{$e_{k_2}$} (-30:3);
\draw[->-,red,thick] (-30:3) to[bend right=10] node[midway,above left]{$e_{k_1}$}  (210:3);
\draw[->-,red,thick] (210:3) to[bend right=10] node[midway,right]{$e_{k_5}$} (90:3);
\draw[->-,red,thick] (90:3) to[bend right=10] node[midway,left]{$e_{k_3}$}  (-30:3);
\CoG{-30:3}{90:3}{210:3}
\draw[red](G) node[below]{$e_i$};
\draw[->-,red,thick] (-30:3) -- (G);
\draw[->-,red,thick] (90:3) -- (G);
\draw[->-,red,thick] (210:3) -- (G);
\draw[->-,purple,thick] (-3,1) to[out=-60,in=90] node[midway,left]{$e_j$} (210:3);
\end{scope}
\end{tikzpicture}
    \caption{The neighboring vertices to $i \in I^\mathrm{tri}(\bD)$ (left) and the corresponding collection of elementary webs (right). An additional elementary web $e_j$ is also shown in purple.}
    \label{fig:compatibility_figure2}
\end{figure}

\begin{proof}
During the proof, we fix a decorated triangulation $\bD$ and omit the superscript $\bD$. 
Let $\varepsilon:=B^{\mathsf{T}}$ denote the quiver exchange matrix associated with $\bD$. 
For $i \in I(\bD)_\uf$ and $j \in I(\bD)$, we are going to compute $(\varepsilon\Pi)_{ij}=\sum_{k \in I(\bD)} \varepsilon_{ik}\pi_{kj}$. Let us divide into the cases $i \in I^\mathrm{edge}(\Delta)$ and $i \in I^\mathrm{tri}(\Delta)$.

\paragraph{\textbf{The case $i \in I^\mathrm{edge}(\Delta)$:}}
Let $Q$ be the quadrilateral having $E$ as its diagonal. 
Label the neighboring vertices of the quiver as in \cref{fig:compatibility_figure1}. 
In this case, we have
\begin{align*}
    (\varepsilon \Pi )_{ij}=\sum_{\nu=1}^4 (-1)^\nu \pi_{k_\nu j}.
\end{align*}
If $j$ lies on an edge outside of $Q$, then one can easily see that $(\varepsilon \Pi )_{ij}=0$. One example of the elementary web $e_j$ corresponding to such a vertex is shown in the right of \cref{fig:compatibility_figure1}. For this example, we have $\pi_{k_2j}=\pi_{k_3j}=+1$ and thus $(\varepsilon \Pi )_{ij}=(-1)^2+(-1)^3 =0$. This is also the case for the vertices lying on the left and the bottom edges in \cref{fig:compatibility_figure1}, since each entry of the compatibility matrix is defined as the sum of the contribution from each end. If $j$ lies on the face of a triangle outside of $Q$, then we get $(\varepsilon \Pi )_{ij}=0$ by a similar consideration. The remaining entries are computed as follows:
\begin{align*}
    (\varepsilon \Pi )_{ik_1}= (-1)^2\cdot (+2) + (-1)^3 \cdot (+1) + (-1)^4 \cdot (+1-2) = 0
\end{align*}
and similarly for $j=k_1^{\mathrm{op}}, k_2, k_2^{\mathrm{op}}$;
\begin{align*}
    (\varepsilon \Pi )_{ik_3}= (-1)^1\cdot (-1) + (-1)^2\cdot (-2+1) + (-1)^4\cdot(-2+2) =0
\end{align*}
and similarly for $j=k_4$;
\begin{align*}
    (\varepsilon \Pi )_{ii^{\mathrm{op}}} = (-1)^1\cdot (-1) + (-1)^2\cdot (+1) + (-1)^3\cdot (-1+2) + (-1)^4 \cdot (-2+1) = 0,
\end{align*}
and finally
\begin{align*}
    (\varepsilon \Pi )_{ii} =  (-1)^1\cdot (-2) + (-1)^2\cdot (+2) + (-1)^3\cdot (+1-2) + (-1)^4 \cdot (-1+2) = 6.
\end{align*}
\paragraph{\textbf{The case $i \in I^\mathrm{tri}(\Delta)$:}}
Label the neighboring vertices of the quiver as in \cref{fig:compatibility_figure2}. In this case, we have
\begin{align*}
    (\varepsilon \Pi )_{ij} = \sum_{\nu=1}^6 (-1)^\nu \pi_{k_\nu j}.
\end{align*}
Then a similar computation shows that the matrix entries $(\varepsilon \Pi )_{ij}$ vanishes except for
\begin{align*}
    (\varepsilon \Pi )_{ii} = -3\cdot (+1-2) +3\cdot (+2-1) =6.
\end{align*}
Thus $(B^{(\omega)},\Pi^{(\omega)})$ is a compatible pair when $\omega=(\Delta,\bs_\Delta)$ is a decorated triangulation with $\bs_T=+$ for all $T \in t(\Delta)$. 
\end{proof}
The check of the compatibility relation for a general  $\omega \in \Exch'_{\mathfrak{sl}_3,\Sigma}$ is postponed until the proof of \cref{thm:q-mutation-equivalence} below. 
For any vertex $\omega \in \Exch'_{\mathfrak{sl}_3,\Sigma}$, define a toric frame
\begin{align*}
    M^{(\omega)}: \accentset{\circ}{\Lambda}{}^{(\omega)} \to \mathrm{Frac}\Skein{\Sigma}^q
\end{align*}
by sending the basis vector $\sff_i^{(\omega)}$ to the corresponding elementary web $e^{(\omega)}_i$, and extending by
\begin{align*}
    M^{(\omega)}\left(\sum_{i=1}^N x_i\sff^{(\omega)}_i\right):=[(e^{(\omega)}_{1})^{x_1}\dots (e^{(\omega)}_{N})^{x_N}]
\end{align*}
by using the Weyl ordering (\cref{def:simulcrossing}) for an auxiliary labeling $I(\omega) \cong \{1,\dots,N\}$. Note that this is the same extension rule as \eqref{eq:extension_toric_frame}, and hence we get: 

\begin{lem}\label{prop:q-seed condition}
For any vertex $\omega \in \Exch'_{\mathfrak{sl}_3,\Sigma}$, the pair $(\Pi^{(\omega)},M^{(\omega)})$ satisfies
\begin{align*}
    M^{(\omega)}(\alpha)M^{(\omega)}(\beta) = q^{\Pi^{(\omega)}(\alpha,\beta)/2} M^{(\omega)}(\alpha+\beta)
\end{align*}
for $\alpha,\beta \in \accentset{\circ}{\Lambda}{}^{(\omega)}$. 
\end{lem}

\begin{thm}\label{thm:q-mutation-equivalence}
For any vertex $\omega \in \Exch'_{\mathfrak{sl}_3,\Sigma}$, the triple  $(B^{(\omega)},\Pi^{(\omega)},M^{(\omega)})$ is a quantum seed in $\mathrm{Frac}\Skein{\Sigma}^q$. These quantum seeds are mutation-equivalent to each other. 
\end{thm}

\begin{proof}
By \cref{prop:compatibility} and \cref{prop:q-seed condition}, the triple $(B^{\bD},\Pi^{\bD},M^{\bD})$ associated with a decorated triangulation $\omega=\bD$ is a quantum seed. Here the condition $\mathrm{Frac}T_\bD=\mathrm{Frac} \Skein{\Sigma}^q$ follows from \cref{cor:web-cluster-expansion-T}. 
We have also seen that the exchange matrices $B^{(\omega)}$ are related to each other by matrix mutations. 

We are going to first show that the toric frames $M^{(\omega)}$ are related to each other by the quantum exchange relations \eqref{eq:q-mutation}. 
By the connectivity of the surface subgraph $\Exch'_{\mathfrak{sl}_3,\Sigma}$ and symmetry, it suffices to consider the toric frames associated with two vertices $\omega$ and $\omega'$ connected by an edge of the following three types.
\begin{enumerate}
    \item The first mutation from a decorated triangulation, where $\omega=(\Delta,\bs_\Delta)$ is a decorated triangulation and $\omega'=\mu_i(\omega)$ for a vertex $i\in I^\mathrm{edge}(\Delta)$.
    \item The second mutation from a decorated triangulation, where $\omega=\mu_{i^\mathrm{op}}(\bD)$ and $\omega'=\mu_{i}\mu_{i^\mathrm{op}}(\bD)$ for some decorated triangulation $\bD=(\Delta,\bs_\Delta)$ and a vertex $i\in I^\mathrm{edge}(\Delta)$.
    \item A change of a sign at a triangle $T$, where $\omega=(\Delta,\bs_\Delta)$ and $\omega'=(\Delta,\bs_{\Delta'})$ are both decorated triangulations with the same underlying triangulation but with $\bs_\Delta(T) = +$, $\bs_{\Delta'}(T)=-$ and $\bs_\Delta(T')=\bs_{\Delta'}(T')$ for $T' \in t(\Delta) \setminus \{T\}$.
\end{enumerate}
During the proof, we simply denote the elementary webs by $e_j:=e_j^{(\omega)}$ and $e'_j:=e_j^{(\omega')}$ in each case. 

In the first case, label the vertices of the quiver as in the left in \cref{fig:compatibility_figure1}. Then we need to check the quantum exchange relation 
\begin{align}
    e'_i &= M^{(\omega)}(-\sff_i + \sff_{k_2}+\sff_{k_4}) + M^{(\omega)}(-\sff_i + \sff_{k_1}+\sff_{k_3}) 
\end{align}
holds in $\mathrm{Frac}\Skein{\Sigma}^q$. Using the relation $M^{(\omega)}(\alpha+\beta)=q^{\Pi^{(\omega)}(\alpha,\beta)/2} M^{(\omega)}(\beta)M^{(\omega)}(\alpha)$ which follows from the definition of the toric frame, this is equivalent to
\begin{align*}
    e_ie'_i &= q^{-3/2}M^{(\omega)}(\sff_{k_2}+\sff_{k_4}) + q^{3/2} M^{(\omega)}(\sff_{k_1}+\sff_{k_3}) \\
    &=q^{-3/2}[e_{k_2}e_{k_4}] + q^{3/2} [e_{k_1}e_{k_3}],
\end{align*}
This is nothing but the \eqref{eq:skein_1st_mutation} (with a suitable change of labelings).

In the second case, label the vertices of the quiver as in the left of \cref{fig:compatibility_figure3}. The expected quantum exchange relation is:
\begin{align*}
    e'_i &= M^{(\omega)}(-\sff_i + \sff_{k_2}+\sff_{k_4}) + M^{(\omega)}(-\sff_i + \sff_{k_1}+\sff_{k_3}) \\
    &=q^{-3/2}e_i^{-1} M^{(\omega)}(\sff_{k_2}+\sff_{k_4}) + q^{3/2}e_i^{-1} M^{(\omega)}(\sff_{k_1}+\sff_{k_3}) \\
    &=e_i^{-1} (q^{-3/2}[e_{k_2}e_{k_4}] + q^{3/2}[e_{k_1}e_{k_3}]).
\end{align*}
This is again the relation \eqref{eq:skein_1st_mutation}. 

In the third case, label the vertices of the quiver on $T$ as in the left of \cref{fig:compatibility_figure2}. The expected quantum exchange relation is
\begin{align*}
    e'_i &= M^{(\omega)}(-\sff_i + \sff_{k_2}+\sff_{k_4}+\sff_{k_6}) + M^{(\omega)}(-\sff_i + \sff_{k_1}+\sff_{k_3}+\sff_{k_5}) \\
    &=e_i^{-1} (q^{-3/2}  M(\sff_{k_2}+\sff_{k_4}+\sff_{k_6}) + q^{3/2} M(\sff_{k_1}+\sff_{k_3}+\sff_{k_5})) \\
    &=e_i^{-1} (q^{-3/2}  [e_{k_2}e_{k_4}e_{k_6}] + q^{3/2} [e_{k_1}e_{k_3}e_{k_5}])
\end{align*}
This is the relation \eqref{eq:skein_sign_change}. Thus the toric frames $M^{(\omega)}$ are related to each other by the quantum exchange relations. 

Then it follows from \cref{prop:q-seed condition} and \cref{lem:compatibility_check} that the pair $(B^{(\omega)},\Pi^{(\omega)})$ satisfies the compatibility relation for all $\omega \in \Exch'_{\mathfrak{sl}_3,\Sigma}$. The assertion is proved.
\end{proof}

\begin{figure}
\begin{tikzpicture}
\begin{scope}[>=latex,yscale=-1,rotate=45]
\draw[blue] (2.8284,0) -- (0,2.8284) -- (-2.8284,0) -- (0,-2.8284) --cycle;
\draw(-3,0) node{$1$};
\draw(0,-3) node{$2$};
\draw(3,0) node{$3$};
\draw(0,3) node{$4$};
\foreach \x in {0,90,180,270}
\path(\x:2.8284) node [fill, circle, inner sep=1.2pt]{};
{\color{mygreen}
\quiversquare{-2.8284,0}{0,-2.8284}{2.8284,0}{0,2.8284}
    \qarrow{y242}{x342}
				\qarrow{x342}{y132}
				\qarrow{y132}{x341}
				\qarrow{x341}{x232}
				\qarrow{x232}{y132}
				\qarrow{y132}{y242}
				\qarrow{y242}{y131}
				\qarrow{y131}{x121}
				\qarrow{x121}{x412}
				\qarrow{x412}{y131}
				\qarrow{y131}{x411}
				\qarrow{x411}{y242}
				
				\qarrow{y132}{y241}
				\qarrow{y241}{y131}
				\qarrow{y131}{y132}
				\qarrow{x122}{y241}
				\qarrow{y241}{x231}
				\qarrow{x231}{x122}
}
\draw(y242) node[below left]{$i$};
\draw(y131) node[above left]{$k_4$};
\draw(x411) node[left]{$k_1$};
\draw(y132) node[above right]{$k_3$};
\draw(x342) node[below]{$k_2$};
\end{scope}

\begin{scope}[xshift=7cm]
\draw[blue] (-2,-2) -- (-2,2) -- (2,2) -- (2,-2) --cycle;
\draw[->-,red,thick] (-2,2) --node[midway,left]{$e_{k_1}$} (-2,-2);
\draw[->-,red,thick] (2,-2) --node[midway,below]{$e_{k_2}$} (-2,-2);
\CoG{-2,-2}{-2,2}{2,2}
\draw[red] (G) node[left]{$e_{k_4}$};
\draw[->-,red,thick] (-2,2) -- (G);
\draw[->-,red,thick] (2,2) -- (G);
\draw[->-,red,thick] (-2,-2) -- (G);
\CoG{-2,-2}{2,-2}{2,2}
\draw[red] (G) node[right]{$e_{k_3}$};
\draw[->-,red,thick] (2,-2) -- (G);
\draw[->-,red,thick] (2,2) -- (G);
\draw[->-,red,thick] (-2,-2) -- (G);
\end{scope}

\end{tikzpicture}
    \caption{The neighboring vertices to $i \in I^\mathrm{edge}(\mu_{i^\mathrm{op}}(\bD))$ (left) and the corresponding collection of elementary webs (right). }
    \label{fig:compatibility_figure3}
\end{figure}
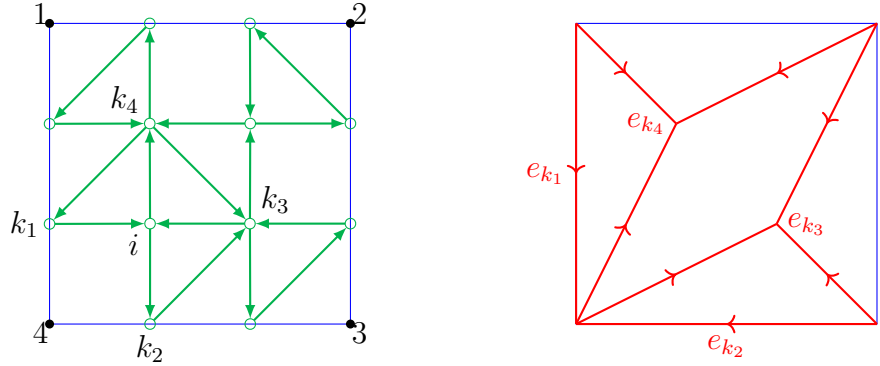

\begin{figure}
    \begin{tikzpicture}[scale=0.7]
    {\color{blue}
    \draw (3,0) -- (0,3) -- (-3,0) -- (0,-3) --cycle;
    \draw[blue] (0,-3) --(0,3);
    }
    \draw[red,thick,->-] (0,-3) to[bend right=10] (0,3);  
    \draw[red,thick,->-] (0,3) to[bend right=10] (0,-3);
    \triv{-3,0}{0,-3}{0,3};
    \triv{3,0}{0,-3}{0,3};
    \draw (-2,-2) node{$\Delta$};
    \foreach \x in {0,90,180,270}
    \path(\x:3) node [fill, circle, inner sep=1.2pt]{};
    \draw[red,dashed] (-0.3,0) ..controls (-1,-0.5) and (-2.2,-1.3).. (-2.5,-1.3) node[left]{$e_{i^\mathrm{op}}^\bD$};
	\draw[red,dashed] (0.3,0) ..controls (1,0.5) and (2.2,1.3).. (2.5,1.3) node[right]{$e_{i}^\bD$};

    \begin{scope}[xshift=4.5cm,yshift=-5cm,scale=0.8]
    \draw[blue] (3,0) -- (0,3) -- (-3,0) -- (0,-3) --cycle;
    \draw[red,thick,->-] (0,-3) to[bend right=10] (0,3);  
    \triv{-3,0}{0,-3}{0,3};
    \triv{3,0}{0,-3}{0,3};
    \triv{3,0}{0,-3}{-3,0};
    \foreach \x in {0,90,180,270}
    \path(\x:3) node [fill, circle, inner sep=1.2pt]{};
    \end{scope}
    
    \begin{scope}[xshift=4.5cm,yshift=5cm,scale=0.8]
    \draw[blue] (3,0) -- (0,3) -- (-3,0) -- (0,-3) --cycle;
    \draw[red,thick,->-] (0,3) to[bend right=10] (0,-3);
    \triv{-3,0}{0,-3}{0,3};
    \triv{3,0}{0,-3}{0,3};
    \triv{3,0}{0,3}{-3,0};
    \foreach \x in {0,90,180,270}
    \path(\x:3) node [fill, circle, inner sep=1.2pt]{};
    \end{scope}
    
    \begin{scope}[xshift=9cm]
    \draw[blue] (3,0) -- (0,3) -- (-3,0) -- (0,-3) --cycle;
    \triv{-3,0}{0,-3}{0,3};
    \triv{3,0}{0,-3}{0,3};
    \triv{3,0}{0,3}{-3,0};
    \triv{3,0}{0,-3}{-3,0};
    \foreach \x in {0,90,180,270}
    \path(\x:3) node [fill, circle, inner sep=1.2pt]{};
    \end{scope}
    
    \begin{scope}[xshift=13.5cm,yshift=5cm,scale=0.8]
    \draw[blue] (3,0) -- (0,3) -- (-3,0) -- (0,-3) --cycle;
    \triv{-3,0}{0,-3}{0,3};
    \draw[red,thick,->-] (-3,0) to[bend right=10] (3,0);
    \triv{3,0}{0,3}{-3,0};
    \triv{3,0}{0,-3}{-3,0};
    \foreach \x in {0,90,180,270}
    \path(\x:3) node [fill, circle, inner sep=1.2pt]{};
    \end{scope}
    
    \begin{scope}[xshift=13.5cm,yshift=-5cm,scale=0.8]
    \draw[blue] (3,0) -- (0,3) -- (-3,0) -- (0,-3) --cycle;
    \draw[red,thick,->-] (3,0) to[bend right=10] (-3,0);
    \triv{3,0}{0,-3}{0,3};
    \triv{3,0}{0,3}{-3,0};
    \triv{3,0}{0,-3}{-3,0};
    \foreach \x in {0,90,180,270}
    \path(\x:3) node [fill, circle, inner sep=1.2pt]{};
    \end{scope}
    
    \begin{scope}[xshift=18cm]
    {\color{blue}
    \draw (3,0) -- (0,3) -- (-3,0) -- (0,-3) --cycle;
    \draw (3,0) -- (-3,0);
    }
    \draw[red,thick,->-] (3,0) to[bend right=10] (-3,0);
    \draw[red,thick,->-] (-3,0) to[bend right=10] (3,0);
    \triv{3,0}{0,3}{-3,0};
    \triv{3,0}{0,-3}{-3,0};
    \foreach \x in {0,90,180,270}
    \path(\x:3) node [fill, circle, inner sep=1.2pt]{};
    \draw (2,-2) node{$\Delta'$};
    \draw[red,dashed] (0,0.3) ..controls (0.5,1) and (2.2,1.3).. (2.5,1.3) node[right]{$e_{j}^{\bD'}$};
	\draw[red,dashed] (0,-0.3) ..controls (-0.5,-1) and (-2.2,-1.3).. (-2.5,-1.3) node[left]{$e_{j^\mathrm{op}}^{\bD'}$};
    \end{scope}
    
    \draw[thick,->] (2,2) --node[midway,above left]{$\mu_{i}$} (3,3);
    \draw[thick,->] (2,-2) --node[midway,below left]{$\mu_{i^\mathrm{op}}$} (3,-3);
    \draw[thick,->] (6,3) --node[midway,above right]{$\mu_{i^\mathrm{op}}$} (7,2);
    \draw[thick,->] (6,-3) --node[midway,below right]{$\mu_{i}$} (7,-2);
    \draw[thick,->] (16,2) --node[midway,above right]{$\mu_{j}$} (15,3);
    \draw[thick,->] (16,-2) --node[midway,below right]{$\mu_{j^\mathrm{op}}$} (15,-3);
    \draw[thick,->] (12,3) --node[midway,above left]{$\mu_{j^\mathrm{op}}$} (11,2);
    \draw[thick,->] (12,-3) --node[midway,below left]{$\mu_{j}$} (11,-2);
    \end{tikzpicture}
        \caption{The web clusters associated with the flip sequence in \cref{fig:flip sequence}. Here the relevant elementary webs are just overwritten, not meaning the simultaneous crossing. The elementary webs on the boundary of the quadrilateral are omitted.}
        \label{fig:flip sequence_web}
    \end{figure}
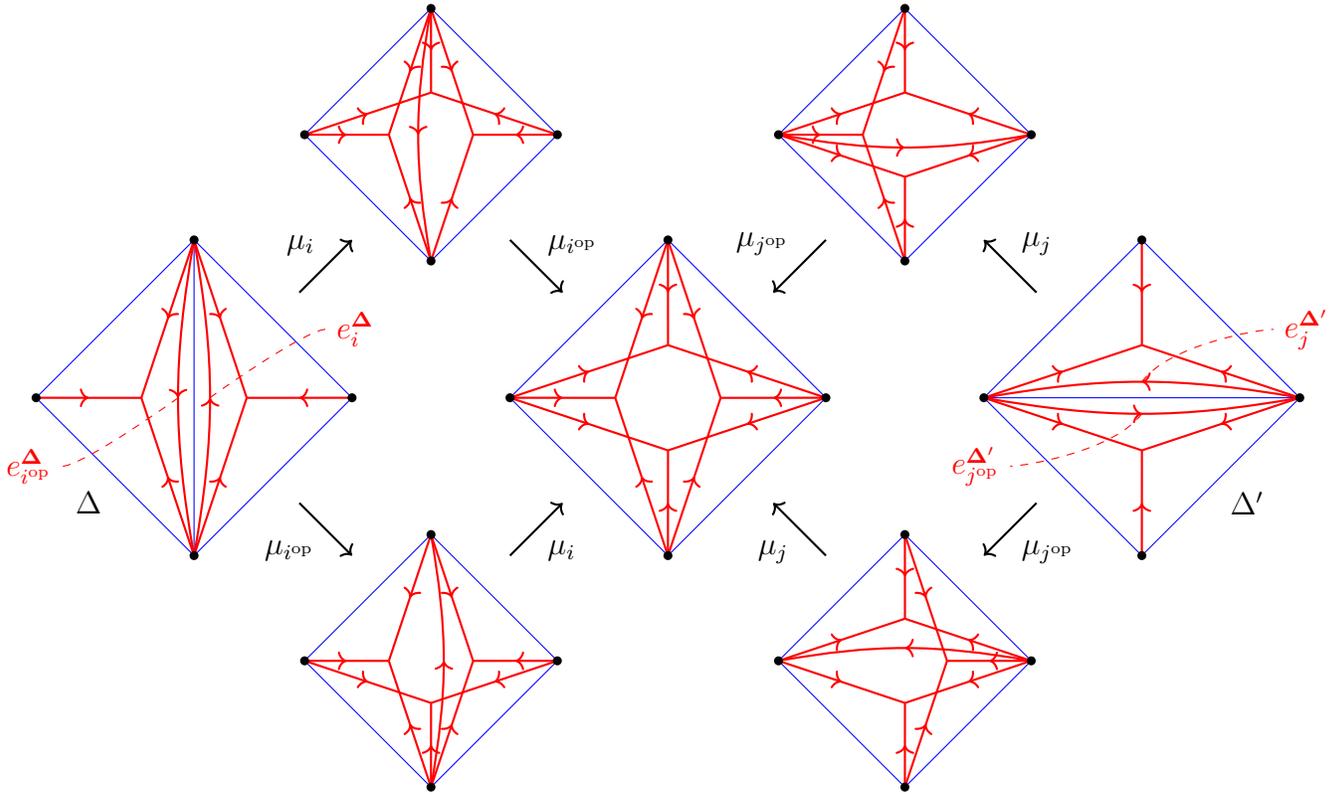

It follows from the above theorem that there exists a canonical mutation class $\sfs_q(\mathfrak{sl}_3,\Sigma)$ containing the quantum seeds $(B^{(\omega)},\Pi^{(\omega)},M^{(\omega)})$ associated with the vertices $\omega \in \Exch'_{\mathfrak{sl}_3,\Sigma}$, and we get the following algebras:
\begin{align*}
    \CA^q_{\mathfrak{sl}_3,\Sigma} \subset \UCA^q_{\mathfrak{sl}_3,\Sigma} \subset \mathrm{Frac}\Skein{\Sigma}^q.
\end{align*}
Here we write $\CA^q_{\mathfrak{sl}_3,\Sigma}:=\CA_{\sfs_q(\mathfrak{sl}_3,\Sigma)}$, $\UCA^q_{\mathfrak{sl}_3,\Sigma}:=\UCA_{\sfs_q(\mathfrak{sl}_3,\Sigma)}$ by simplifying the notation. 
By construction, we already know that the cluster variables associated with the vertices $\omega \in \Exch'_{\mathfrak{sl}_3,\Sigma}$ in the surface subgraph are realized in $\Skein{\Sigma}^q[\partial^{-1}]$. 
Comparison of the three algebras $\CA^q_{\mathfrak{sl}_3,\Sigma}, \UCA^q_{\mathfrak{sl}_3,\Sigma}, \Skein{\Sigma}^q[\partial^{-1}]$ is discussed in \cref{sec:comparison}.

\section{Comparison of quantum cluster algebras and skein algebras}\label{sec:comparison}
In this section, we give several results on the comparison of the three subalgebras $\Skein{\Sigma}^q[\partial^{-1}]$, $\CA^q_{\mathfrak{sl}_3,\Sigma}$, $\UCA^q_{\mathfrak{sl}_3,\Sigma}$ of the fraction algebra $\mathrm{Frac}\Skein{\Sigma}^q$. Recall from \cref{thm:q-Laurent} that we know the inclusion $\CA^q_{\mathfrak{sl}_3,\Sigma}\subset \UCA^q_{\mathfrak{sl}_3,\Sigma}$ as the quantum Laurent phenomenon. 

\subsection{The three algebras coincide for a triangle or a quadrilateral}
We first establish the basic cases where the marked surface is a triangle or a quadrilateral.

\begin{cor}\label{prop:comparison_T and Q}
When $\Sigma$ is a triangle or a quadrilateral, we have 
\begin{align*}
    \Skein{\Sigma}^q[\partial^{-1}] =\CA^q_{\mathfrak{sl}_3,\Sigma} =  \UCA^q_{\mathfrak{sl}_3,\Sigma}. 
\end{align*}
 
\end{cor}

\begin{proof}
Since these quantum cluster algebras have acyclic exchange types $A_1$ and $D_4$, we know $\CA^q_{\mathfrak{sl}_3,\Sigma}=\UCA^q_{\mathfrak{sl}_3,\Sigma}$. See, for instance, \cite[Proposition 8.17]{Muller16}. We are going to show the equality $\CA^q_{\mathfrak{sl}_3,\Sigma} = \Skein{\Sigma}^q[\partial^{-1}]$. Since these two are subalgebras of $\mathrm{Frac} \Skein{\Sigma}^q$, we only need to establish a correspondence between their generators: the cluster variables and the elementary webs. 

When $\Sigma=T$ is a triangle, we have two quantum clusters in $\CA^q_{\mathfrak{sl}_3,T}$, which are exactly the two web clusters $C_{(\Delta_T,\pm)}$ given in \cref{prop:Eweb T}. Since the skein algebra $\Skein{T}^q$ is generated by the elementary webs contained in these web clusters by \cref{prop:generators T}, we have $\CA^q_{\mathfrak{sl}_3,T} = \Skein{T}^q[\partial^{-1}]$.

When $\Sigma=Q$ is a quadrilateral, we can also provide a one-to-one correspondence between the quantum clusters in $\CA^q_{\mathfrak{sl}_3,Q}$ and the web clusters in $\Skein{Q}^q$, as follows. The quantum cluster algebra $\CA^q_{\mathfrak{sl}_3,Q}$ of type $D_4$ has $16$ unfrozen variables and $8$ frozen variables. Among those, we have already identified all the frozen ones and the $12$ unfrozen ones which are associated with the decorated cell decompositions of $Q$. In order to identify the remaining four cluster variables, let us consider the decorated triangulation $\bD=(\Delta_{Q}^{(13)},(+,-))$.
The associated quiver and the cluster variables are shown in \cref{fig:quadrilateral_+-cluster}. Consider the new cluster variable $A'_1 \in \mathrm{Frac} \Skein{Q}^q$ obtained from this quantum cluster by the mutation directed to $1$. Explicitly, it is defined by the quantum exchange relation
\begin{align*}
    A'_1 &= M^{\bD}(-\sff_1 + \sff_2 + \sff_3) + M^{\bD}(-\sff_1 + \sff_4 + \sff_7 + \sff_{12}) \\
    &=A_1^{-1} (q [A_2A_3] + q^{-2} [A_4A_7A_{12}]).
\end{align*}
Then by comparing with the skein relation \eqref{eq:H-web appears}, we may identify $A'_1$ with the elementary web $\tikz[scale=0.9,baseline=2mm]{\draw[blue](0,0)--(1,0)--(1,1)--(0,1)--cycle; \Hweb{0,1}{0,0}{1,0}{1,1}}$. In particular, we have $A'_1 \in \Skein{Q}^q$. The other three remaining cluster variables are obtained by rotations. Thus all the cluster variables are realized in $\Skein{Q}^q$. 
Then by \cref{prop:generators Q}, we have $\CA^q_{\mathfrak{sl}_3,Q} = \Skein{Q}^q[\partial^{-1}]$.
\end{proof}

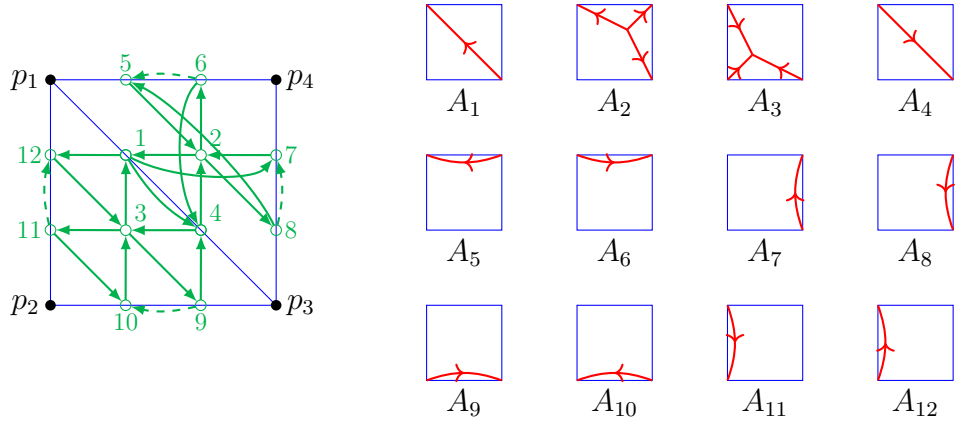
\begin{figure}
\begin{tikzpicture}
\newcommand{\smallsq}[1]{
\draw[blue] (#1) --++(1,0) --++(0,1) --++(-1,0)--cycle;}

\draw[blue] (0,0) -- (3,0) -- (3,3) -- (0,3) --cycle;
\draw[blue] (3,0) -- (0,3);
\fill(0,0) circle(2pt) node[left]{$p_2$};
\fill(3,0) circle(2pt) node[right]{$p_3$};
\fill(0,3) circle(2pt) node[left]{$p_1$};
\fill(3,3) circle(2pt) node[right]{$p_4$};
{\color{mygreen}
\quiverplus{0,0}{3,0}{0,3}
\quiverminus{3,0}{3,3}{0,3}
\begin{scope}[>=latex]
\qarrowbr{1,2}{2,1};
\qdlarrow{2,0}{1,0};
\qdlarrow{0,1}{0,2};
\qdrarrow{2,3}{1,3};
\qdrarrow{3,1}{3,2};
\end{scope}

\draw(1,2) node[above right,scale=0.8]{$1$};
\draw(2,2) node[above right,scale=0.8]{$2$};
\draw(1,1) node[above right,scale=0.8]{$3$};
\draw(2,1) node[above right,scale=0.8]{$4$};
\draw(1,3) node[above,scale=0.8]{$5$};
\draw(2,3) node[above,scale=0.8]{$6$};
\draw(3,2) node[right,scale=0.8]{$7$};
\draw(3,1) node[right,scale=0.8]{$8$};
\draw(2,0) node[below,scale=0.8]{$9$};
\draw(1,0) node[below,scale=0.8]{$10$};
\draw(0,1) node[left,scale=0.8]{$11$};
\draw(0,2) node[left,scale=0.8]{$12$};
}

\begin{scope}[xshift=5cm,yshift=3cm]
\smallsq{0,0}
\draw[red,thick,->-] (1,0) -- (0,1);
\node at (0.5,-0.3) {$A_1$};
\smallsq{2,0}
\trivop{2,1}{3,1}{3,0}
\node at (2.5,-0.3) {$A_2$};
\smallsq{4,0}
\triv{4,0}{4,1}{5,0}
\node at (4.5,-0.3) {$A_3$};
\smallsq{6,0}
\draw[red,thick,->-] (6,1) -- (7,0);
\node at (6.5,-0.3) {$A_4$};
\end{scope}

\begin{scope}[xshift=5cm,yshift=1cm]
\smallsq{0,0}
\draw[red,thick,->-,bend left=20] (1,1) to (0,1);
\node at (0.5,-0.3) {$A_5$};
\smallsq{2,0}
\draw[red,thick,->-,bend right=20] (2,1) to (3,1);
\node at (2.5,-0.3) {$A_6$};
\smallsq{4,0}
\draw[red,thick,->-,bend left=20] (5,0) to (5,1);
\node at (4.5,-0.3) {$A_7$};
\smallsq{6,0}
\draw[red,thick,->-,bend right=20] (7,1) to (7,0);
\node at (6.5,-0.3) {$A_8$};
\end{scope}

\begin{scope}[xshift=5cm,yshift=-1cm]
\smallsq{0,0}
\draw[red,thick,->-,bend left=20] (0,0) to (1,0);
\node at (0.5,-0.3) {$A_9$};
\smallsq{2,0}
\draw[red,thick,->-,bend right=20] (3,0) to (2,0);
\node at (2.5,-0.3) {$A_{10}$};
\smallsq{4,0}
\draw[red,thick,->-,bend left=20] (4,1) to (4,0);
\node at (4.5,-0.3) {$A_{11}$};
\smallsq{6,0}
\draw[red,thick,->-,bend right=20] (6,0) to (6,1);
\node at (6.5,-0.3) {$A_{12}$};
\end{scope}
\end{tikzpicture}
    \caption{The decorated triangulation $\bD=(\Delta_{Q}^{(13)},(+,-))$ and the associated quantum cluster.}
    \label{fig:quadrilateral_+-cluster}
\end{figure}

\subsection{Inclusion $\Skein{\Sigma}^q[\partial^{-1}]\subset \CA^q_{\mathfrak{sl}_3,\Sigma}$}\label{subsec:S_in_A}
Here we give the inclusion $\Skein{\Sigma}^q[\partial^{-1}]\subset \CA^q_{\mathfrak{sl}_3,\Sigma}$ and a comparison result of the additional structures: gradings, involutions and group actions.

\begin{thm}\label{thm:S in A}
For a connected unpunctured marked surface $\Sigma$ with at least two special points, we have an inclusion
\begin{align*}
    \Skein{\Sigma}^q[\partial^{-1}]\subset \CA^q_{\mathfrak{sl}_3,\Sigma}.
\end{align*}
\end{thm}

\begin{proof}
By \cref{thm:localized-generators}, the boundary-localized skein algebra $\Skein{\Sigma}^q[\partial^{-1}]$ is generated by oriented simple arcs and triads. Therefore it suffices to prove that these elements are cluster variables. For an oriented simple arc $\alpha$, there exists an ideal triangulation $\Delta$ that contains the underlying ideal arc. Then $\alpha$ is a cluster variable contained in the cluster $C_{\bD}$ assoociated with a decorated triangulation $\bD$ with underlying triangulation $\Delta$. For a triad $\tau$, let $T$ be the ideal triangle spanned by the three endpoints of $\tau$. Then one can similarly choose an ideal triangulation that contains $T$ and see that $\tau$ is a cluster variable. 
\end{proof}

\bigskip
\paragraph{\textbf{Comparison of gradings}}
Recall the lattice
\begin{align*}
    L(3):= \ker( (\bZ \times \bZ)^\bM \xrightarrow{\mathrm{aug}} \bZ \xrightarrow{\mathrm{mod}_3} \bZ_3),
\end{align*}
where $\mathrm{aug}((k_p,l_p)_{p \in \bM}):=\sum_{p \in \bM} (k_p - l_p)$. We are going to compare the lattices $\coker p^\ast$ and $L(3)$, where the former parametrizes the ensemble degree (\cref{lemdef:ensemble_grading}) and the latter does the endpoint degree (\cref{def:egrading}).
For any $\omega \in \Exch'_{\mathfrak{sl}_3,\Sigma}$, take the associated web cluster $(e_i^{(\omega)})_{i \in I(\omega)}$ and consider the map
\begin{align*}
    \mathrm{end}^{(\omega)}: \accentset{\circ}{\Lambda} {}^{(\omega)} \to L(3),\quad \sff_i^{(\omega)} \mapsto \mathrm{gr}(e_i^{(\omega)}).
\end{align*}

\begin{lem}[mutation-invariance]\label{lem:grading_mutation}
For an edge $\omega \overbar{k} \omega'$ in $\Exch'_{\mathfrak{sl}_3,\Sigma}$, we have $\mathrm{end}^{(\omega')} = \mathrm{end}^{(\omega)} \circ \mu_{k,\epsilon}^*$ 
for each $\epsilon \in \{+,-\}$. 
\end{lem}

\begin{proof}
It suffices to check the relation $\mathrm{end}^{(\omega')}(\sff_k^{(\omega')})= \mathrm{end}^{(\omega)}(\mu_{k,\epsilon}^*(\sff_k^{(\omega)}))$, which is equivalent to
\begin{align*}
    &\mathrm{gr}(e_k^{(\omega')}) = \mathrm{gr}\left((e_k^{(\omega)})^{-1} \prod_{j \in I(\omega)} (A_j^{(\omega)})^{[\epsilon b_{jk}^{(\omega)}]_+} \right), \quad \mbox{or}\\
    &\mathrm{gr}(e_k^{(\omega)}e_k^{(\omega')})=\mathrm{gr}\left(\prod_{j \in I(\omega)} (A_j^{(\omega)})^{[\epsilon b_{jk}^{(\omega)}]_+} \right).
\end{align*}
On the other hand, the comparison of the skein and quantum exchange relations obtained in the proof of \cref{thm:q-mutation-equivalence} tells us that the monomial appearing in the right-hand side is exactly one of the two terms in the corresponding skein relation. Since the skein relations are homogeneous with respect to the endpoint grading, we get the desired assertion.
\end{proof}

\begin{lem}\label{lem:grading_isomorphism}
For any $\omega \in \Exch'_{\mathfrak{sl}_3,\Sigma}$, we have an isomorphism $\overline{\mathrm{end}}^{(\omega)}: \coker p_{(\omega)}^* \xrightarrow{\sim} L(3)$
of lattices which fits into the following diagram:
\begin{equation*}
    \begin{tikzcd}
    \accentset{\circ}{\Lambda}{}^{(\omega)} \ar[r,"\mathrm{end}^{(\omega)}"] \ar[d,"\alpha_{(\omega)}^*"'] & L(3). \\
    \coker p_{(\omega)}^* \ar[ru,"\overline{\mathrm{end}}^{(\omega)}"'] &
    \end{tikzcd}
\end{equation*}
Here $\alpha_{(\omega)}^*:\accentset{\circ}{\Lambda}{}^{(\omega)} \to \coker p_{(\omega)}^*$ denotes the natural projection.
Together with the mutation-invariance (\cref{lem:grading_mutation}), we get a canonical isomorphism 
\begin{align*}
    \overline{\mathrm{end}}: \coker p^* \xrightarrow{\sim} L(3).
\end{align*}
\end{lem}

\begin{proof}
Thanks to the mutation-invariance (\cref{lem:grading_mutation}), it suffices to prove the statement for a decorated triangulation $\omega=\bD$. For the well-definedness of $\overline{\mathrm{end}}^{\bD}$, we need to check that 
\begin{align*}
    \mathrm{end}^{\bD}(p_{\bD}^*\sfe_i^{\bD})=\mathrm{end}^{\bD}\left(\sum_{j \in I(\Delta)}\varepsilon_{ij}^\bD\sff_j^{\bD}\right)=\mathrm{gr}\left(\prod_{j \in I(\Delta)}(e_j^\bD)^{\varepsilon_{ij}^\bD}\right)=0
\end{align*}
for all $i \in I(\bD)_\uf$. Let us write $e_j:=e_j^{\bD} \in \Skein{\Sigma}^q$ for simplicity. 

If $i \in I^\mathrm{edge}(\Delta)$, then label the neighboring vertices as in the left of \cref{fig:compatibility_figure1}. Then one can see $\mathrm{gr}(e_{j_2}e_{j_4}/e_{j_1}e_{j_3})=0$ by inspection into the right of \cref{fig:compatibility_figure1}. If $i \in I^\mathrm{tri}(\Delta)$, then label the neighboring vertices as in the left of \cref{fig:compatibility_figure2}. Then one can see $\mathrm{gr}(e_{j_2}e_{j_4}e_{j_6}/e_{j_1}e_{j_3}e_{j_5})=0$ by inspection into the right of \cref{fig:compatibility_figure2}. Thus the map $\overline{\mathrm{end}}^{\bD}$ is well-defined. 

For the surjectivity, first note that the lattice $L(3)$ is generated by the following degree vectors 
\begin{enumerate}
    \item $\mathrm{gr}(e_i^{\bD})$ for $i \in I^\mathrm{tri}(\bD)$;
    \item $\mathrm{gr}(e_i^{(\omega)})$ for $\omega \in \Exch'_{\mathfrak{sl}_3,\Sigma}$ and $i \in I^\mathrm{edge}(\omega)$.
\end{enumerate}
Indeed, given any vector in $L(3)$, by adding a suitable number of vectors of the the form (1), we can translate it so that $\mathrm{aug}=\sum_{p \in \bM}k_p - \sum_{p \in \bM}l_p=0$. Such a degree vector can be written as a sum of the degree vectors of the oriented arcs, which are of the form (2). Since each web cluster $C_{(\omega)}$ is mutation-equivalent to $C_\bD$, each degree vector in (1)(2) can be written as a sum of the vectors $\mathrm{gr}(e_i^{\bD})$ for $i \in I(\bD)$. Thus $\overline{\mathrm{end}}^{\bD}$ is surjective. 

Since $L(3)$ is a full-rank sub-lattice of $(\bZ\times \bZ)^\bM$, we have $\rank L(3)=2|\bM|$. On the other hand, we know from the geometric realizations of the cluster varieties as certain moduli spaces that $\rank \coker p^*=\dim \A_{SL_3,\Sigma} - \dim \X_{PGL_3,\Sigma} = 2|\bM|$ \cite[Lemma 2.4]{FG03}. It follows that $\overline{\mathrm{end}}^{\bD}$ is an isomorphism of lattices as a surjective morphism between two lattices of the same rank.
\end{proof}

\begin{prop}\label{thm:comparison_grading}
The ensemble grading on $\UCA^q_{\mathfrak{sl}_3,\Sigma}$ coincides with the endpoint grading on $\Skein{\Sigma}^q[\partial^{-1}]$. More precisely, we have
\begin{align*}
    \overline{\mathrm{end}}(\mathrm{\mathbf{gr}}(A_i^{(\omega)})) = \mathrm{gr} (e_i^{(\omega)})
\end{align*}
for any $\omega \in \Exch_{\mathfrak{sl}_3,\Sigma}$ and $i \in I(\omega)$. 
\end{prop}

\begin{proof}
From \cref{lem:grading_isomorphism}, we get
\begin{align*}
    \overline{\mathrm{end}}(\mathrm{\mathbf{gr}}(A_i^{(\omega)})) = \overline{\mathrm{end}}^{(\omega)}(\alpha_{(\omega)}^*(\sff_i^{(\omega)})) = 
    \mathrm{end}^{(\omega)} (\sff_i^{(\omega)}) =
    \mathrm{gr} (e_i^{(\omega)})
\end{align*}
as desired.
\end{proof}

\begin{rem}
This comparison result of gradings has a practical importance in finding web clusters, though it is implicit in this paper. Namely, suppose we know a web cluster $C_{(\omega)}$ corresponding to a vertex $\omega \in \Exch_{\mathfrak{sl}_3,\Sigma}$, and want to know the web cluster $C_{(\omega')}$ corresponding to an adjacent vertex $\omega'$ connected by a mutation $\mu_k$. First, we can easily compute the ensemble degree of the new cluster variable $A_k^{(\omega')}$ from those of $A_j^{(\omega)}$. Then from the comparison result, we know the endpoint degree which the new elementary web $e_k^{(\omega')}$ should have. It is also typically easy to guess the web $e_k^{(\omega')}$ from its endpoint degree by naturally connecting the prescribed endpoints.
\end{rem}

\begin{proof}[Proof of \cref{introthm:comparison}]
What remaining are the group equivariance, and the comparison of the bar-involution and the mirror-reflection. 
The group equivariance is easily seen by comparing the descriptions of the group actions given in \cref{subsub:basic_structures,subsec:classical_Teich}. For example, under the action of a mapping class $\phi \in MC(\Sigma)$, a cluster variable (resp. elementary web) associated to a decorated triangulation $\bD$ is sent to a cluster variable (resp. elementary web) associated to $\phi^{-1}(\bD)$. 
To compare the bar-involution and the mirror-reflection, note that they coincide on the elementary webs associated with a decorated triangulation. Then the assertion follows from the cluster expansion (\cref{cor:web-cluster-expansion-T}).
\end{proof}

\subsection{A direct inclusion $\Skein{\Sigma}^q[\partial^{-1}]\subset \UCA^q_{\mathfrak{sl}_3,\Sigma}$ and quantum Laurent positivity}
Although we get the inclusions $\Skein{\Sigma}^q[\partial^{-1}]\subset \CA^q_{\mathfrak{sl}_3,\Sigma} \subset \UCA^q_{\mathfrak{sl}_3,\Sigma}$ by combining \cref{thm:S in A} with the quantum Laurent phenomenon, it does not tell us how to get the quantum Laurent expressions of an $\mathfrak{sl}_3$-web. In particular, positivity of the coefficients is not clear. 
Here we give a direct way to compute the inclusion $\Skein{\Sigma}^q[\partial^{-1}]\subset \UCA^q_{\mathfrak{sl}_3,\Sigma}$ based on the cluster expansion results we have discussed in \cref{sect:skein}, which also tell us a partial result on the positivity.

\begin{thm}\label{thm:comparison_skein_cluster}
For any unpunctured marked surface $\Sigma$, we have an 
inclusion
\begin{align*}
    \Skein{\Sigma}^q[\partial^{-1}] \subset \UCA^q_{\mathfrak{sl}_3,\Sigma}.
\end{align*}
\end{thm}

\begin{proof}

Thanks to \cref{thm:q-Laurent}, it suffices to check the inclusion $\Skein{\Sigma}^q[\partial^{-1}] \subset \UCA^q_{\mathfrak{sl}_3,\Sigma}(\bD)$ for a decorated triangulation $\bD$. Moreover, note that all the vertices of $\Exch_{\mathfrak{sl}_3,\Sigma}$ adjacent to $\bD$ are decorated cell decompositions. Therefore it suffices to see that
for any $x \in \Skein{\Sigma}^q$ and a decorated cell decomposition $\omega$, there exists a monomial $J_{(\omega)}$ of elementary webs in the web cluster $C_{(\omega)}$ such that $xJ_{(\omega)} \in \langle C_{(\omega)} \rangle_\mathrm{alg}$. 

When $\omega=\bD$ is a decorated triangulation, it is exactly \cref{cor:web-cluster-expansion-T}. Indeed, we have seen that by multiplying a product of elementary webs $e_i^{\bD}$ for $i \in I^{\mathrm{edge}(\Delta)}$ to $x$, we can decompose it into a sum of webs in triangles (\emph{i.e.}, a product of $e_j^{\bD}$'s and $\ast (e_j^{\bD})$'s for $j\ \in I^\mathrm{tri}(\Delta)$), and these webs are further decomposed into a polynomial in $C_{\bD}$ by multiplying some product of $\ast (e_j^{\bD})$'s. 

When $\omega$ is a decorated cell decomposition of deficiency $1$, the assertion follows from \cref{prop:adjacent-web-cluster-expansion-Q}. Indeed, we can similarly decompose $x \in \Skein{\Sigma}^q$ into a sum of webs in triangles and the unique quadrilateral, and the latter webs can be expanded in the web cluster associated with $\omega$. Thus we get the inclusion $\Skein{\Sigma}^q[\partial^{-1}] \subset \UCA^q_{\mathfrak{sl}_3,\Sigma}(\bD) = \UCA^q_{\mathfrak{sl}_3,\Sigma}$ as desired.

\end{proof}

As a quantum counterpart of \eqref{eq:univ_Laurent}, we have the semiring
\begin{align*}
    \mathbb{L}_{\sfs_q}^+:= \bigcap_{\omega \in \Exch_\sfs}T_{(\omega)}^+ \subset \UCA_{\sfs_q}
\end{align*}
of \emph{quantum universally Laurent polynomials}, where $T_{(\omega)}^+ \subset T_{(\omega)}$ denote the semiring of quantum Laurent polynomials in the variables $A_i^{(\omega)}$ for $i \in I$ and $q^{1/2}$ with non-negative coefficients. When $\sfs_q=\sfs_q(\mathfrak{sl}_3,\Sigma)$, one may ask the existence of the following diagram:
\begin{equation*}
    \begin{tikzcd}
    \Bweb{\Sigma} \ar[r] \ar[dr,dashed,"?"'] & \UCA^q_{\mathfrak{sl}_3,\Sigma} \\
    & \mathbb{L}_{\sfs_q(\mathfrak{sl}_3,\Sigma)}^+. \ar[u,hookrightarrow]
    \end{tikzcd}
\end{equation*}
In order to state our partial result, consider a larger semiring
\begin{align*}
    \widetilde{\mathbb{L}}_{\mathfrak{sl}_3,\Sigma}^+:= \bigcap_{\bD}T_{\bD}^+ \subset \UCA_{\sfs_q},
\end{align*}
where $\bD$ runs over all the decorated triangulations of $\Sigma$. An element of $\widetilde{\mathbb{L}}_{\mathfrak{sl}_3,\Sigma}^+$ is called a \emph{quantum GS-universally positive Laurent polynomial}. 
The following is a rephrasing of \cref{cor:elevation-preserving web in cluster,cor:simple web}:

\begin{thm}[Quantum Laurent positivity of webs]\label{thm:positivity_cluster}
Any elevation-preserving web with respect to $\Delta$ is contained in $T_{\bD}^+$. In particular, the $n$-bracelet or the $n$-bangle along an oriented simple loop in $\Sigma$  for any $n$ are contained in the semiring $\widetilde{\mathbb{L}}_{\mathfrak{sl}_3,\Sigma}^+$.
\end{thm}

\subsection{$\Skein{\Sigma}^q[\partial^{-1}]=\UCA^q_{\mathfrak{sl}_3,\Sigma}$ under the covering conjecture}\label{subsec:S=U} 
Fix a marked surface $\Sigma$, and write $\sfs_q:=\sfs_q(\mathfrak{sl}_3,\Sigma)$. 

For an ideal cell decomposition $(\Delta;E)$ of deficiency $1$, let $Q_E$ be the unique quadrilateral having $E$ as a diagonal. 
Let $\mathrm{mon}(\Delta;E) \subset \Skein{\Sigma}^q$ denote the multiplicatively closed set generated by the elementary webs along the edges of $\Delta$ except for those along the edge $E$ and $A^{1/2}$. Define $\Skein{\Sigma}^q[(\Delta;E)^{-1}]$ to be the Ore localization of $\Skein{\Sigma}^q$ by the Ore set $\mathrm{mon}(\Delta;E)$. By \cref{thm:localization}, we have an inclusion $\Skein{\Sigma}^q[\partial^{-1}] \subset \Skein{\Sigma}^q[(\Delta;E)^{-1}]$ for any $(\Delta;E)$. We are going to show that $\Skein{\Sigma}^q[\partial^{-1}]=\UCA^q_{\mathfrak{sl}_3,\Sigma}$ holds under the following conjecture on these localizations:

\begin{conj}[Covering conjecture]\label{conj:covering}
We have
\begin{align*}
    \Skein{\Sigma}^q[\partial^{-1}] = \bigcap_{E \in e_{\mathrm{int}}(\Delta)} \Skein{\Sigma}^q[(\Delta;E)^{-1}]
\end{align*}
for any ideal triangulation $\Delta$ of $\Sigma$.
\end{conj}
Indeed, its classical analogue holds true from the geometry of the moduli space of decorated $SL_3$-local systems (for instance, see \cite{Shen20} for its dual counterpart). 

On the side of cluster algebra, we similarly define $\CA_{\sfs_q}[(\Delta;E)^{-1}] \subset \mathrm{Frac}\Skein{\Sigma}^q$ to be the Ore localization of the quantum cluster algebra $\CA_{\sfs_q}$ by $\mathrm{mon}(\Delta;E)$. Note that the elementary webs in $\mathrm{mon}(\Delta;E)$ are identified with some cluster variables. Let $\sfs_q[(\Delta;E)^{-1}]$ denote the mutation class of quantum seeds obtained from $\sfs_q$ by \emph{freezing} the cluster variables in $\mathrm{mon}(\Delta;E)$ at any decorated cell decomposition $\omega$ over $(\Delta;E)$. The corresponding quantum cluster algebra is generated by iterated mutations from the quantum cluster associated with $\omega$, where the mutations of the cluster variables in $\mathrm{mon}(\Delta;E)$ are prohibited. Since two decorated cell decompositions over $(\Delta;E)$ can be connected by mutations for other directions, it does not depend on $\omega$.

\begin{lem}
We have $\CA_{\sfs_q[(\Delta;E)^{-1}]}=\CA_{\sfs_q}[(\Delta;E)^{-1}]=\UCA_{\sfs_q}[(\Delta;E)^{-1}]=\UCA_{\sfs_q[(\Delta;E)^{-1}]}$. 
\end{lem}
Namely, the quantum cluster algebra $\CA_{\sfs_q[(\Delta;E)^{-1}]}$ is a \emph{cluster localization} in the sense of \cite[Section 8.3]{Muller16}. 

\begin{proof}
The cluster type of the mutation class $\sfs_q[(\Delta;E)^{-1}]$ is $A_1 \times \cdots \times A_1 \times D_4$. In particular it has acyclic exchange type, and hence $\CA_{\sfs_q[(\Delta;E)^{-1}]}=\UCA_{\sfs_q[(\Delta;E)^{-1}]}$. See, for instance, \cite[Proposition 8.17]{Muller16}. Then the assertion follows from \cite[Proposition 8.5]{Muller16}.  
\end{proof}

\begin{lem}
We have $\CA_{\sfs_q[(\Delta;E)^{-1}]}=\Skein{\Sigma}^q[(\Delta;E)^{-1}]=\UCA_{\sfs_q[(\Delta;E)^{-1}]}$.
\end{lem}

\begin{proof}
In the same way as the proof of \cref{prop:comparison_T and Q}, we can identify all the cluster variables in the mutation class $\sfs_q[(\Delta;E)^{-1}]$ with elementary webs. Hence $\CA_{\sfs_q[(\Delta;E)^{-1}]} \subset \Skein{\Sigma}^q[(\Delta;E)^{-1}]$. We also have 
\begin{align*}
    \Skein{\Sigma}^q[(\Delta;E)^{-1}] \subset \UCA_{\sfs_q}[(\Delta;E)^{-1}]=\UCA_{\sfs_q[(\Delta;E)^{-1}]}
\end{align*}
by the second inclusion in \cref{thm:comparison_skein_cluster}. Thus the assertion follows from the previous lemma.
\end{proof}




\begin{lem}
We have $\UCA_{\sfs_q}= \bigcap_{E \in e_{\mathrm{int}}(\Delta)}\UCA_{\sfs_q[(\Delta;E)^{-1}]}$ for any ideal triangulation $\Delta$ of $\Sigma$.
\end{lem}

\begin{proof}
By the quantum upper bound theorem (\cref{thm:q-Laurent}), the upper cluster algebra $\UCA_{\sfs_q}$ coincides with the upper bound at any decorated triangulation $\bD$ over $\Delta$:
\begin{align}\label{eq:upper-bound}
    \UCA_{\sfs_q}= \UCA_{\sfs_q}(\bD) = T_{\bD}\cap \bigcap_{k \in I_\uf(\Delta)} T_{\mu_k(\bD)}.
\end{align}
Here $T_{\mu_k(\bD)}$ denotes the quantum torus associated with the quantum seed obtained from that for $\bD$ by the mutation directed to $k \in I_\uf(\Delta)$. Similarly, we have
\begin{align}\label{eq:upper-bound-local}
    \UCA_{\sfs_q[\Delta;E]^{-1}} = \UCA_{\sfs_q[\Delta;E]^{-1}}(\bD) = T_{\bD} \cap \bigcap_{k} T_{\mu_k(\bD)},
\end{align}
where $k$ runs over the set $I^\mathrm{tri}(\Delta)$ of face indices and the two indices on the edge $E$. Then the assertion can be verified by comparing the expressions \eqref{eq:upper-bound} and \eqref{eq:upper-bound-local}, since the quantum torus for each unfrozen edge index $k \in I^\mathrm{edge}(\Delta)$ appears at least once in \eqref{eq:upper-bound-local} when $E$ runs over all the interior edges as well as the face indices.  
\end{proof}

\begin{prop}\label{thm:S=U}
Assuming that \cref{conj:covering} holds true, we have 
\begin{align*}
    \Skein{\Sigma}^q[\partial^{-1}] = \UCA^q_{\mathfrak{sl}_3,\Sigma}.
\end{align*}
\end{prop}

\begin{proof}
Combining the lemmas above, we get
\begin{align*}
    \UCA^q_{\mathfrak{sl}_3,\Sigma}&= \bigcap_{E \in e_{\mathrm{int}}(\Delta)}\UCA_{\sfs_q[(\Delta;E)^{-1}]} 
    = \bigcap_{E \in e_{\mathrm{int}}(\Delta)}\Skein{\Sigma}^q[(\Delta;E)^{-1}]
    = \Skein{\Sigma}^q[\partial^{-1}].
\end{align*}
Here we used \cref{conj:covering} in the last equality.
\end{proof}


\appendix
\section{The \texorpdfstring{$\mathfrak{sl}_3$}{sl3}-skein algebra for a quadrilateral}\label{subsec:quadrilateral_skein}
Let $Q$ be a quadrilateral with special points $p_1,p_2,p_3,p_4$ in this counter-clockwise order. 

In the same way as in the triangle case, we define boundary webs
\[
	\Bweb{\partial^{\times}Q}=\{e_{12},e_{21},e_{23},e_{32},e_{34},e_{43},e_{41},e_{14}\}
\]
and introduce the following 16 $\mathfrak{sl}_3$-webs:
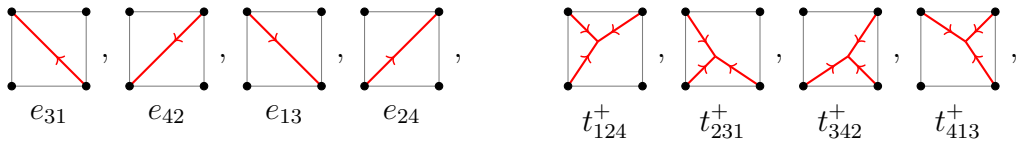
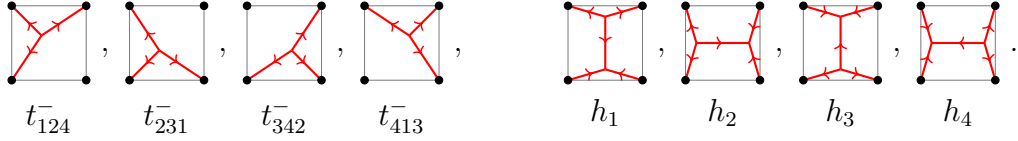
\begin{align*}
	&\Quadweb{\dweb{0}}{$e_{31}$}, \Quadweb{\dweb{90}}{$e_{42}$}, \Quadweb{\dweb{180}}{$e_{13}$}, \Quadweb{\dweb{-90}}{$e_{24}$},
	&&\Quadweb{\twebplus{0}}{$t_{124}^{+}$}, \Quadweb{\twebplus{90}}{$t_{231}^{+}$}, \Quadweb{\twebplus{180}}{$t_{342}^{+}$}, \Quadweb{\twebplus{-90}}{$t_{413}^{+}$},\\
	&\Quadweb{\twebminus{0}}{$t_{124}^{-}$}, \Quadweb{\twebminus{90}}{$t_{231}^{-}$}, \Quadweb{\twebminus{180}}{$t_{342}^{-}$}, \Quadweb{\twebminus{-90}}{$t_{413}^{-}$},
	&&\Quadweb{\hweb{0}}{$h_{1}$}, \Quadweb{\hweb{90}}{$h_{2}$}, \Quadweb{\hweb{180}}{$h_{3}$}, \Quadweb{\hweb{-90}}{$h_{4}$}.
\end{align*}

Let us denote the set of these webs by
\[
	\Eweb{Q}:=\{e_{ij},t_{123}^{\epsilon},t_{231}^{\epsilon},t_{342}^{\epsilon},t_{413}^{\epsilon},h_{i}\mid i, j\in\{1,2,3,4\}, i\neq j ,\epsilon\in\{{+},{-}\}\}.
\]
We will show that it is exactly the set of elementary webs for $Q$ in \cref{prop:Eweb Q} soon below.

\begin{lem}\label{lem:str const Q}
The complete list of relations among $\Eweb{Q}$ is given as follows\footnote{Here the left-hand sides of the above equations mean the multiplication of two elementary webs,
	the first web being located right above the second one.
	A collection of elementary webs depicted in the same quadrilateral mean the Weyl ordering of these webs.}:
	\begin{align}
		\Quadweb{\twebplus{0}}{${t_{124}^{+}}$} \Quadweb{\twebplus{90}}{${t_{231}^{+}}$}
		&= A^{-1}\Quadweb{\twebplus{90}\twebplus{0}}{$[{t_{124}^{+}}{t_{231}^{+}}]$},\\
		\Quadweb{\twebplus{0}}{${t_{124}^{+}}$} \Quadweb{\twebplus{180}}{${t_{342}^{+}}$}
		&= \Quadweb{\twebplus{0}\twebplus{180}}{$[{t_{124}^{+}}{t_{342}^{+}}]$},\\
		\Quadweb{\twebplus{0}}{${t_{124}^{+}}$} \Quadweb{\twebplus{-90}}{${t_{413}^{+}}$}
		&= A \Quadweb{\twebplus{-90}\twebplus{0}}{$[{t_{124}^{+}}{t_{413}^{+}}]$},
	\end{align}
	\smallskip
	\begin{align}
		\Quadweb{\twebplus{0}}{${t_{124}^{+}}$} \Quadweb{\twebminus{0}}{${t_{124}^{-}}$}
		&= A^{-\frac{3}{2}}\Quadweb{\bwebpos{0}\bwebpos{-90}\dweb{-90}}{$[e_{12}e_{24}e_{41}]$}
		+ A^{\frac{3}{2}}\Quadweb{\bwebneg{0}\bwebneg{-90}\dweb{90}}{$[e_{21}e_{14}e_{42}]$},\label{eq:skein_sign_change}\\
		\Quadweb{\twebplus{0}}{${t_{124}^{+}}$} \Quadweb{\twebminus{90}}{${t_{231}^{-}}$}
		&= A^{-2}\Quadweb{\bwebpos{0}\bwebpos{-90}\bwebpos{90}}{$[e_{12}e_{23}e_{41}]$}
		+ A\Quadweb{\bwebneg{0}\hweb{180}}{$[e_{21}h_{3}]$},\\
		\Quadweb{\twebplus{0}}{${t_{124}^{+}}$} \Quadweb{\twebminus{180}}{${t_{342}^{-}}$}
		&= \Quadweb{\twebplus{0}\twebminus{180}}{$[{t_{124}^{+}}{t_{342}^{-}}]$},\\
		\Quadweb{\twebplus{0}}{${t_{124}^{+}}$} \Quadweb{\twebminus{-90}}{${t_{413}^{-}}$}
		&=A^{2} \Quadweb{\bwebneg{0}\bwebneg{180}\bwebneg{-90}}{$[e_{21}e_{14}e_{43}]$}
		+A \Quadweb{\bwebpos{-90}\hweb{90}}{$[e_{41}h_{2}]$},
	\end{align}
	\smallskip
	\begin{align}
		\Quadweb{\twebplus{0}}{${t_{124}^{+}}$} \Quadweb{\dweb{0}}{$e_{31}$}
		&= A^{-\frac{3}{2}}\Quadweb{\bwebpos{-90}\twebplus{90}}{$[e_{41}{t_{231}^{+}}]$} 
		+ A^{\frac{3}{2}}\Quadweb{\bwebneg{0}\twebplus{-90}}{$[e_{21}{t_{413}^{+}}]$},\label{eq:skein_1st_mutation}\\
		\Quadweb{\twebplus{0}}{${t_{124}^{+}}$} \Quadweb{\dweb{90}}{$e_{42}$}
		&= A^{\frac{1}{2}} \Quadweb{\twebplus{0}\dwebdown{90}}{$[{t_{124}^{+}}e_{42}]$},\\
		\Quadweb{\twebplus{0}}{${t_{124}^{+}}$} \Quadweb{\dweb{180}}{$e_{13}$}
		&= \Quadweb{\twebplus{0}\dwebup{180}}{$[{t_{124}^{+}}e_{13}]$},\\
		\Quadweb{\twebplus{0}}{${t_{124}^{+}}$} \Quadweb{\dweb{-90}}{$e_{24}$}
		&= A^{-\frac{1}{2}} \Quadweb{\twebplus{0}\dwebup{-90}}{$[{t_{124}^{+}}e_{24}]$},
	\end{align}
	\smallskip
	\begin{align}
		\Quadweb{\twebplus{0}}{${t_{124}^{+}}$} \Quadweb{\hweb{0}}{$h_{1}$}
		&= A\Quadweb{\bwebneg{0}\bwebneg{-90}\twebplus{180}}{$[e_{21}e_{14}{t_{342}^{+}}]$}
		+ A^{-2}\Quadweb{\bwebpos{-90}\dwebup{-90}\twebplus{90}}{$[e_{41}e_{24}{t_{231}^{+}}]$},\\
		\Quadweb{\twebplus{0}}{${t_{124}^{+}}$} \Quadweb{\hweb{90}}{$h_{2}$}
		&= A^{-1}\Quadweb{\bwebpos{0}\bwebpos{-90}\twebminus{-90}}{$[e_{12}e_{41}{t_{413}^{-}}]$}
		+ A^{2}\Quadweb{\bwebneg{0}\dwebup{90}\twebplus{-90}}{$[e_{21}e_{42}{t_{413}^{+}}]$},\\
		\Quadweb{\twebplus{0}}{${t_{124}^{+}}$} \Quadweb{\hweb{180}}{$h_{3}$}
		&= A^{\frac{1}{2}}\Quadweb{\twebplus{0}\hweb{180}}{$[{t_{124}^{+}}h_{3}]$},\\
		\Quadweb{\twebplus{0}}{${t_{124}^{+}}$} \Quadweb{\hweb{-90}}{$h_{4}$}
		&= A^{-\frac{1}{2}}\Quadweb{\twebplus{0}\hweb{-90}}{$[{t_{124}^{+}}h_{4}]$},
	\end{align}
	\smallskip
	\begin{align}
		\Quadweb{\dweb{0}}{$e_{31}$} \Quadweb{\dweb{90}}{$e_{42}$}
		&= A^{2}\Quadweb{\bwebneg{90}\bwebpos{-90}}{$[e_{32}e_{41}]$} + A^{-1} \Quadweb{\hweb{90}}{$h_{2}$},\\
		\Quadweb{\dweb{0}}{$e_{31}$} \Quadweb{\dweb{180}}{$e_{13}$}
		&= \Quadweb{\dwebup{0}\dwebup{180}}{$[e_{31}e_{13}]$},\\
		\Quadweb{\dweb{0}}{$e_{31}$} \Quadweb{\dweb{-90}}{$e_{24}$}
		&= A^{-2} \Quadweb{\bwebneg{0}\bwebpos{180}}{$[e_{21}e_{34}]$} + A \Quadweb{\hweb{0}}{$h_{1}$},
	\end{align}
	\smallskip
	\begin{align}
		\Quadweb{\dweb{0}}{$e_{31}$} \Quadweb{\hweb{0}}{$h_{1}$} &= A \Quadweb{\dweb{0}\hweb{0}}{$[e_{31}h_{1}]$},\\
		\Quadweb{\dweb{0}}{$e_{31}$} \Quadweb{\hweb{90}}{$h_{2}$} &= A^{-1} \Quadweb{\dweb{0}\hweb{90}}{$[e_{31}h_{2}]$}\\
		\Quadweb{\dweb{0}}{$e_{31}$} \Quadweb{\hweb{180}}{$h_{3}$}
		&= A^{-1} \Quadweb{\twebminus{90}\twebplus{-90}}{$[{t_{231}^{-}}{t_{413}^{+}}]$} + A^{2} \Quadweb{\bwebneg{90}\bwebpos{-90}\dweb{180}}{$[e_{32}e_{41}e_{13}]$},\\
		\Quadweb{\dweb{0}}{$e_{31}$} \Quadweb{\hweb{-90}}{$h_{4}$}
		&= A \Quadweb{\twebplus{90}\twebminus{-90}}{$[{t_{231}^{+}}{t_{413}^{-}}]$} + A^{-2} \Quadweb{\bwebneg{0}\bwebpos{180}\dweb{180}}{$[e_{21}e_{13}e_{34}]$}, \label{eq:H-web appears}
	\end{align}
	\smallskip
	\begin{align}
		\Quadweb{\hweb{0}}{$h_{1}$}\Quadweb{\hweb{90}}{$h_{2}$}
		&= A \Quadweb{\bwebneg{0}\bwebneg{90}\bwebpos{180}\bwebpos{-90}}{$[e_{21}e_{32}e_{34}e_{41}]$}
		+ A^{-2} \Quadweb{\dwebup{0}\twebminus{0}\twebplus{180}}{$[e_{31}t_{124}^{-}t_{342}^{+}]$},\\
		\Quadweb{\hweb{0}}{$h_{1}$}\Quadweb{\hweb{180}}{$h_{3}$}
		&= \Quadweb{\bwebpos{0}\bwebpos{90}\bwebpos{180}\bwebpos{-90}}{$[e_{12}e_{23}e_{34}e_{41}]$}
		+ \Quadweb{\bwebpos{90}\bwebpos{-90}\bwebneg{90}\bwebneg{-90}}{$[e_{32}e_{23}e_{14}e_{41}]$}
		+ \Quadweb{\bwebneg{0}\bwebneg{90}\bwebneg{180}\bwebneg{-90}}{$[e_{14}e_{43}e_{32}e_{21}]$}\notag\\
		&\qquad + A^{3}\Quadweb{\bwebpos{0}\bwebneg{180}\hweb{-90}}{$[e_{12}e_{43}h_{4}]$} + \Quadweb{\bwebneg{0}\bwebpos{180}\hweb{90}}{$[e_{21}e_{34}h_{2}]$},\\
		\Quadweb{\hweb{0}}{$h_{1}$}\Quadweb{\hweb{-90}}{$h_{4}$}
		&= A^{-1} \Quadweb{\bwebneg{0}\bwebneg{-90}\bwebpos{90}\bwebpos{180}}{$[e_{21}e_{14}e_{23}e_{34}]$}
		+ A \Quadweb{\dwebup{-90}\twebplus{90}\twebminus{-90}}{$[e_{24}{t_{231}^{+}}{t_{413}^{-}}]$}.
	\end{align}
\end{lem}
Indeed, the relations between $\Bweb{\partial^{\times}Q}$ and the other webs follow from the boundary skein relations~\labelcref{rel:parallel}, \labelcref{rel:antiparallel}.
The remaining relations are obtained by applying the Dynkin involution and rotations of the quadrilateral to the above relations.

\begin{rem}
	The Weyl ordering of a $\mathfrak{sl}_{3}$-web appearing in the above relations can be represented by a flat trivalent graph obtained by the following operations:
	\begin{align*}
		&\mathord{
			\ 
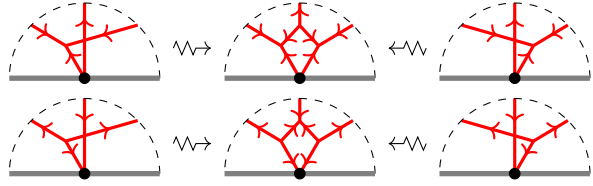
\tikz[baseline=-.6ex, scale=.1, yshift=-4cm]{
				\coordinate (P) at (0,0);
				\coordinate (B) at (90:10);
				\coordinate (A) at (135:10);
				\coordinate (C) at (45:10);
				\coordinate (A1) at (120:5);
				\coordinate (B1) at (90:5);
				\coordinate (C1) at (60:5);
				\draw[very thick, red, ->-={.8}{}] (P) -- (B);
				\draw[very thick, red, ->-={.8}{}] (P) -- (A1);
				\draw[very thick, red, ->-] (A) -- (A1);
				\draw[very thick, red, ->-] (C) -- (A1);
				\draw[dashed] (10,0) arc (0:180:10cm);
				\draw[gray, line width=2pt] (-10,0) -- (10,0);
				\draw[fill=black] (P) circle [radius=20pt];
			}
		}
		\mathord{
			\ \tikz[baseline=-.6ex, scale=.05]{
			\draw[->, decorate, decoration={zigzag,segment length=1.5mm, post length=1.5mm}] (-5,0) -- (5,0);
			}
		}
		\mathord{
			\ \tikz[baseline=-.6ex, scale=.1, yshift=-4cm]{
				\coordinate (P) at (0,0);
				\coordinate (B) at (90:10);
				\coordinate (A) at (135:10);
				\coordinate (C) at (45:10);
				\coordinate (A1) at (120:5);
				\coordinate (B1) at (90:7);
				\coordinate (C1) at (60:5);
				\draw[very thick, red, ->-={.8}{}] (P) -- (C1);
				\draw[very thick, red, ->-] (C) -- (C1);
				\draw[very thick, red, ->-={.8}{}] (B1) -- (C1);
				\draw[very thick, red, ->-={.8}{}] (B1) -- (A1);
				\draw[very thick, red, ->-={.8}{}] (B1) -- (B);
				\draw[very thick, red, ->-={.8}{}] (P) -- (A1);
				\draw[very thick, red, ->-] (A) -- (A1);
				\draw[dashed] (10,0) arc (0:180:10cm);
				\draw[gray, line width=2pt] (-10,0) -- (10,0);
				\draw[fill=black] (P) circle [radius=20pt];
			}
		}
		\mathord{
			\ \tikz[baseline=-.6ex, scale=.05]{
			\draw[->, decorate, decoration={zigzag,segment length=1.5mm, post length=1.5mm}] (5,0) -- (-5,0);
			}
		}
		\mathord{
			\ \tikz[baseline=-.6ex, scale=.1, yshift=-4cm]{
				\coordinate (P) at (0,0);
				\coordinate (B) at (90:10);
				\coordinate (C) at (135:10);
				\coordinate (A) at (45:10);
				\coordinate (C1) at (120:5);
				\coordinate (B1) at (90:5);
				\coordinate (A1) at (60:5);
				\draw[very thick, red, ->-={.8}{}] (P) -- (B);
				\draw[very thick, red, ->-={.8}{}] (P) -- (A1);
				\draw[very thick, red, ->-] (A) -- (A1);
				\draw[very thick, red, ->-] (C) -- (A1);
				\draw[dashed] (10,0) arc (0:180:10cm);
				\draw[gray, line width=2pt] (-10,0) -- (10,0);
				\draw[fill=black] (P) circle [radius=20pt];
			}
		},\\
		&\mathord{
			\ \tikz[baseline=-.6ex, scale=.1, yshift=-4cm]{
				\coordinate (P) at (0,0);
				\coordinate (B) at (90:10);
				\coordinate (A) at (135:10);
				\coordinate (C) at (45:10);
				\coordinate (A1) at (120:5);
				\coordinate (B1) at (90:5);
				\coordinate (C1) at (60:5);
				\draw[very thick, red, -<-={.8}{}] (P) -- (B);
				\draw[very thick, red, -<-={.8}{}] (P) -- (A1);
				\draw[very thick, red, -<-] (A) -- (A1);
				\draw[very thick, red, -<-] (C) -- (A1);
				\draw[dashed] (10,0) arc (0:180:10cm);
				\draw[gray, line width=2pt] (-10,0) -- (10,0);
				\draw[fill=black] (P) circle [radius=20pt];
			}
		}
		\mathord{
			\ \tikz[baseline=-.6ex, scale=.05]{
			\draw[->, decorate, decoration={zigzag,segment length=1.5mm, post length=1.5mm}] (-5,0) -- (5,0);
			}
		}
		\mathord{
			\ \tikz[baseline=-.6ex, scale=.1, yshift=-4cm]{
				\coordinate (P) at (0,0);
				\coordinate (B) at (90:10);
				\coordinate (A) at (135:10);
				\coordinate (C) at (45:10);
				\coordinate (A1) at (120:5);
				\coordinate (B1) at (90:7);
				\coordinate (C1) at (60:5);
				\draw[very thick, red, -<-] (P) -- (C1);
				\draw[very thick, red, -<-] (C) -- (C1);
				\draw[very thick, red, -<-] (B1) -- (C1);
				\draw[very thick, red, -<-] (B1) -- (A1);
				\draw[very thick, red, -<-] (B1) -- (B);
				\draw[very thick, red, -<-] (P) -- (A1);
				\draw[very thick, red, -<-] (A) -- (A1);
				\draw[dashed] (10,0) arc (0:180:10cm);
				\draw[gray, line width=2pt] (-10,0) -- (10,0);
				\draw[fill=black] (P) circle [radius=20pt];
			}
		}
		\mathord{
			\ \tikz[baseline=-.6ex, scale=.05]{
			\draw[->, decorate, decoration={zigzag,segment length=1.5mm, post length=1.5mm}] (5,0) -- (-5,0);
			}
		}
		\mathord{
			\ \tikz[baseline=-.6ex, scale=.1, yshift=-4cm]{
				\coordinate (P) at (0,0);
				\coordinate (B) at (90:10);
				\coordinate (C) at (135:10);
				\coordinate (A) at (45:10);
				\coordinate (C1) at (120:5);
				\coordinate (B1) at (90:5);
				\coordinate (A1) at (60:5);
				\draw[very thick, red, -<-={.8}{}] (P) -- (B);
				\draw[very thick, red, -<-={.8}{}] (P) -- (A1);
				\draw[very thick, red, -<-] (A) -- (A1);
				\draw[very thick, red, -<-] (C) -- (A1);
				\draw[dashed] (10,0) arc (0:180:10cm);
				\draw[gray, line width=2pt] (-10,0) -- (10,0);
				\draw[fill=black] (P) circle [radius=20pt];
			}
		}.
	\end{align*}
	This operation is an inverse operation of the \emph{arborization} in~\cite{FP16}.
\end{rem}

\begin{prop}\label{prop:generators Q}
	The skein algebra $\SK{Q}$ is generated by $\Eweb{Q}$ as a $\bZ_{A}$-algebra.
\end{prop}
\begin{proof}
	We take a point $p$ on an edge of $Q$ and a point $q$ on the opposite side.
	For any non-elliptic flat trivalent graph $G$ representing a basis web in $\Bweb{Q}$, fix a minimal cut-path $\alpha$ of $G$ from $p$ to $q$.
	We remark that the minimal cut-path is non-convex to the left and right sides by \cref{def:cut-path}, thus there exists left and right cores $\alpha_{L}$ and $\alpha_{R}$ \cref{lem:Kuperberg-lemma}~(2).
	An induction on $|\mathrm{wt}_{\alpha}(G)|$ will prove the proposition.
	The cases $\mathrm{wt}_{\alpha}(G)=0,1$ are easy.
	Assume $|\mathrm{wt}_{\alpha}(G)|=n$, by \cref{lem:Kuperberg-lemma}~(1), there is no cut-paths related to $\alpha_{L}$ (resp. $\alpha_{R}$) by $H$-moves in the left (resp. right) side, and $\alpha_{L}$ is related to $\alpha_{R}$ by $H$-moves and identity moves.
	The explicit description of the basis webs on $T$ in \cref{fig:Bweb T} and the proof of \cref{prop:generators T} imply that there exists a trivalent graph $G'$ of the form in the left of \cref{fig:generators Q} such that $G=A^kG'$ for some $k$.
	Here the web $B$ in the biangle bounded by $\alpha_L$ and $\alpha_R$ is constructed by a concatenation of $H$-webs, as shown in the center of \cref{fig:generators Q}. 
	By applying the skein relations~\labelcref{rel:parallel} and \labelcref{rel:antiparallel}, these $H$-webs can be replaced by internal crossings up to multiplication by $A$ and modulo webs with weights lower than $n$.
	Then the web $B$ can be replaced with $A^{\bullet}\sigma+\sum_{x}\lambda_{x} x$, where $x$ is a trivalent graph whose minimal cut-path between $p$ and $q$ has intersection points less than $n$, and $\sigma$ is a positive permutation braid between $\alpha_{L}$ and $\alpha_{R}$, as shown in the right of \cref{fig:generators Q}.
	By substituting $B=A^{\bullet}\sigma+\sum_{x}\lambda_{x} x$ into $G'$, we obtain an expression $G'=G'_{\sigma}+\sum_{x}\lambda_{x} G'_{x}$ and see that $G'_{\sigma}$ is described as a product of webs in $\Eweb{Q}$.
	The proof is finished by applying the induction hypothesis to $G'_{x}$.
	By the proof, notice that the webs $h_1$ and $h_3$ are not needed for the generating set.
\end{proof}

\begin{figure}
	\centering
	\begin{tikzpicture}[scale=.1]
		\node[draw, fill=black, circle, inner sep=1.5] (A) at (135:30) {};
		\node[draw, fill=black, circle, inner sep=1.5] (B) at (225:30) {};
		\node[draw, fill=black, circle, inner sep=1.5] (C) at (315:30) {};
		\node[draw, fill=black, circle, inner sep=1.5] (D) at (45:30) {};
		\coordinate (P) at ($(B)!0.5!(C)$);
		\coordinate (Q) at ($(A)!0.5!(D)$);
		\draw[blue] (A) -- (B);
		\draw[blue] (B) -- (C);
		\draw[blue] (C) -- (D);
		\draw[blue] (D) -- (A);
		\draw[very thick, red] (A) -- (-10,0); 
		\draw[very thick, red] (B) -- (-10,0); 
		\draw[very thick, red] (0,0) -- (-10,0);
		\draw[xshift=-10cm, very thick, red, fill=white] (60:3) -- (180:3) -- (300:3) -- cycle;
		\draw[very thick, red] (C) -- (10,0); 
		\draw[very thick, red] (D) -- (10,0); 
		\draw[very thick, red] (0,0) -- (10,0);
		\draw[xshift=10cm, very thick, red, fill=white] (0:3) -- (120:3) -- (240:3) -- cycle;
		\draw[very thick, red, rounded corners] (A) -- (-16,0) -- (B);
		\draw[very thick, red, rounded corners] (A) -- (-19,0) -- (B);
		\draw[very thick, red, rounded corners] (D) -- (16,0) -- (C);
		\draw[very thick, red, rounded corners] (D) -- (19,0) -- (C);
		\draw[very thick, red, rounded corners] (A) -- (0,16) -- (D);
		\draw[very thick, red, rounded corners] (A) -- (0,10) -- (D);
		\draw[very thick, red, rounded corners] (B) -- (0,-16) -- (C);
		\draw[very thick, red, rounded corners] (B) -- (0,-10) -- (C);
		\fill[pink!60] (P) to[out=north west, in=south] (-5,0) to[out=north, in=south west] (Q) to[out=south east, in=north] (5,0) to[out=south, in=north east] (P) -- cycle;
		\draw[->-={.7}{}] (P) to[out=north west, in=south] (-5,0) to[out=north, in=south west] (Q);
		\draw[->-={.7}{}] (P) to[out=north east, in=south] (5,0) to[out=north, in=south east] (Q);
		\fill (P) circle [radius=20pt];
		\fill (Q) circle [radius=20pt];
		\node at (0,0) {$B$};
		\node at (A) [above] {$p_1$};
		\node at (B) [left] {$p_2$};
		\node at (C) [below] {$p_3$};
		\node at (D) [right] {$p_4$};
		\node at (P) [below] {$p$};
		\node at (Q) [above] {$q$};
		\node at (P) [below=20pt] {$G'$};
	\end{tikzpicture}
	\begin{tikzpicture}[scale=.1]
		\coordinate (A) at (135:30);
		\coordinate (B) at (225:30);
		\coordinate (C) at (315:30);
		\coordinate (D) at (45:30);
		\coordinate (X) at (-15,20);
		\coordinate (Y) at (-15,-20);
		\coordinate (Z) at (15,-20);
		\coordinate (W) at (15,20);
		\foreach \i in {1,2,3,4,5}{
			\coordinate (L\i) at ($(X)!0.14*\i!(Y)$);
		}
		\foreach \i in {1,2,3,4,5}{
			\coordinate (R\i) at ($(W)!0.14*\i!(Z)$);
		}
		\foreach \i in {1,2,3,4,5}{
			\draw[red, very thick] (L\i) -- (R\i);
		}
		\draw[red, very thick, ->-] ($(L1)!.75!(R1)$) -- ($(L2)!.75!(R2)$);
		\draw[red, very thick, ->-] ($(L3)!.75!(R3)$) -- ($(L4)!.75!(R4)$);
		\draw[red, very thick, ->-] ($(L2)!.5!(R2)$) -- ($(L3)!.5!(R3)$);
		\draw[red, very thick, ->-] ($(L4)!.5!(R4)$) -- ($(L5)!.5!(R5)$);
		\draw[red, very thick, ->-] ($(L3)!.25!(R3)$) -- ($(L4)!.25!(R4)$);
		\coordinate (P) at ($(B)!0.5!(C)$);
		\coordinate (Q) at ($(A)!0.5!(D)$);
		\draw[->-, rounded corners] (P) -- (-15,-20) -- (-15,20) -- (Q);
		\draw[->-, rounded corners] (P) -- (15,-20) -- (15,20) -- (Q);
		\fill (P) circle [radius=20pt];
		\fill (Q) circle [radius=20pt];
		\node at (P) [rotate=90,above=25pt, red]{$\cdots$};
		\node at (L3) [left=3pt] {$\alpha_{L}$};
		\node at (R3) [right=3pt] {$\alpha_{R}$};
		\node at (P) [below] {$p$};
		\node at (Q) [above] {$q$};
		\node at (P) [below=20pt] {$B$};
	\end{tikzpicture}
	\begin{tikzpicture}[scale=.1]
		\coordinate (A) at (135:30);
		\coordinate (B) at (225:30);
		\coordinate (C) at (315:30);
		\coordinate (D) at (45:30);
		\coordinate (X) at (-15,20);
		\coordinate (Y) at (-15,-20);
		\coordinate (Z) at (15,-20);
		\coordinate (W) at (15,20);
		\foreach \i in {1,2,3,4,5}{
			\coordinate (L\i) at ($(X)!0.14*\i!(Y)$);
		}
		\foreach \i in {1,2,3,4,5}{
			\coordinate (R\i) at ($(W)!0.14*\i!(Z)$);
		}
		\draw[red, very thick, -<-={.1}{}, rounded corners] (L3) -- ($(L3)!.2!(R3)$) -- ($(L4)!.3!(R4)$) -- ($(L4)!.45!(R4)$) -- ($(L5)!.55!(R5)$) -- (R5);
		\draw[red, very thick, -<-={.1}{}, rounded corners] (L2) -- ($(L2)!.45!(R2)$) -- ($(L3)!.55!(R3)$) -- ($(L3)!.7!(R3)$) -- ($(L4)!.8!(R4)$) -- (R4);
		\draw[red, very thick, -<-={.1}{}, rounded corners] (L1) -- ($(L1)!.7!(R1)$) -- ($(L2)!.8!(R2)$) -- (R2);
		\draw[red, very thick, ->-={.9}{red}, overarc, rounded corners] (L5) -- ($(L5)!.45!(R5)$) -- ($(L4)!.55!(R4)$) -- ($(L4)!.7!(R4)$) -- ($(L3)!.8!(R3)$) -- (R3);
		\draw[red, very thick, ->-={.9}{red}, overarc, rounded corners] (L4) -- ($(L4)!.2!(R4)$) -- ($(L3)!.3!(R3)$) -- ($(L3)!.45!(R3)$) -- ($(L2)!.55!(R2)$) -- ($(L2)!.7!(R2)$) -- ($(L1)!.8!(R1)$)-- (R1);
		\coordinate (P) at ($(B)!0.5!(C)$);
		\coordinate (Q) at ($(A)!0.5!(D)$);
		\draw[->-, rounded corners] (P) -- (-15,-20) -- (-15,20) -- (Q);
		\draw[->-, rounded corners] (P) -- (15,-20) -- (15,20) -- (Q);
		\fill (P) circle [radius=20pt];
		\fill (Q) circle [radius=20pt];
		\node at (P) [rotate=90,above=25pt, red]{$\cdots$};
		\node at (L3) [left=3pt] {$\alpha_{L}$};
		\node at (R3) [right=3pt] {$\alpha_{R}$};
		\node at (P) [below] {$p$};
		\node at (Q) [above] {$q$};
		\node at (P) [below=20pt] {$\sigma$};
	\end{tikzpicture}
	\caption{A web in $G'$}
	\label{fig:generators Q}
\end{figure}

Using the above proposition, we give the set of elementary webs and the collection of web clusters for the quadrilateral.

\begin{prop}\label{prop:Eweb Q}
	$\Eweb{Q}=\{e_{ij},t_{123}^{\epsilon},t_{231}^{\epsilon},t_{342}^{\epsilon},t_{413}^{\epsilon},h_{i}\mid i,j\in\{1,2,3,4\},i\neq j,\epsilon\in\{{+},{-}\}\}$ is the set of elementary webs for $Q$.
\end{prop}
\begin{proof}
	We use the same argument as in the proof of \cref{prop:Eweb T}.
	Recall that if a basis web $G$ is decomposed into a product $G=G_1G_2$, then $\vec{\gr}(G)=\vec{\gr}(G_1)+\vec{\gr}(G_2)$. On the other hand, we know the explicit generators of $\SK{Q}$ given in \cref{prop:generators Q}.
	Observe that 
	\[
		\vec{\gr}(t_{ijk}^{+})=(3,0),\quad \vec{\gr}(t_{ijk}^{-})=(0,3), \quad
		\vec{\gr}(e_{ij})=(1,1), \quad
		\vec{\gr}(h_j)=(2,2).
	\]
	Hence except for the last $h_j$'s, one can easily see that these webs are indecomposable. 
	Therefore we only have to care about the possibility that $h_j$ is decomposed into two webs in $\{e_{13},e_{31},e_{24},e_{42}\}$, but it is impossible by \cref{lem:str const Q}. Thus each web in $\Eweb{Q}$ is indecomposable. 
	With a notice that each basis web appearing on the right-hand side of expansions given in \cref{lem:str const Q} is described as a product of webs in $\Eweb{Q}$, 
	we conclude that this set is exactly the set of elementary webs.
\end{proof}

\begin{rem}
	We can also determine the web clusters by \cref{lem:str const Q}. In fact, $\Cweb{Q}$ consists $50$ web clusters, which are in a one-to-one correspondence with the quantum seeds in the cluster algebra $\CA^q_{\mathfrak{sl}_3,Q}$ of type $D_4$. 
\end{rem}

We will consider expansions of any webs in the five web clusters $C_\nu=C'_\nu\cup\Bweb{\partial^{\times}Q}$ in $\Cweb{Q}$ for $\nu=0,1,2,3,4$, where $C'_\nu$'s are given as follows:
\begin{align*}
	C'_{0}&=\Bigg\{\Quadweb{\twebplus{0}}{${t_{124}^{+}}$}, \Quadweb{\twebplus{180}}{${t_{342}^{+}}$}, \Quadweb{\dweb{90}}{${e_{42}}$}, \Quadweb{\dweb{-90}}{${e_{24}}$}\Bigg\},&
	C'_{1}&=\Bigg\{\Quadweb{\twebplus{0}}{${t_{124}^{+}}$}, \Quadweb{\twebminus{180}}{${t_{342}^{-}}$}, \Quadweb{\dweb{90}}{${e_{42}}$}, \Quadweb{\dweb{-90}}{${e_{24}}$}\Bigg\}\\
	C'_{2}&=\Bigg\{\Quadweb{\twebminus{0}}{${t_{124}^{-}}$}, \Quadweb{\twebplus{180}}{${t_{342}^{+}}$}, \Quadweb{\dweb{90}}{${e_{42}}$}, \Quadweb{\dweb{-90}}{${e_{24}}$}\Bigg\},&
	C'_{3}&=\Bigg\{\Quadweb{\twebplus{0}}{${t_{124}^{+}}$}, \Quadweb{\twebplus{180}}{${t_{342}^{+}}$}, \Quadweb{\twebplus{90}}{${t_{231}^{+}}$}, \Quadweb{\dweb{-90}}{${e_{24}}$}\Bigg\}\\
	C'_{4}&=\Bigg\{\Quadweb{\twebplus{0}}{${t_{124}^{-}}$}, \Quadweb{\twebminus{180}}{${t_{342}^{+}}$}, \Quadweb{\dweb{90}}{${e_{42}}$}, \Quadweb{\twebplus{-90}}{${t_{413}^{+}}$}\Bigg\}.&&
\end{align*}
As clarified later in \cref{sect:correspondance}, the web clusters $C_1$, $C_2$, $C_3$, and $C_4$ are ``adjacent'' to $C_0$ by a mutation. 
It will mean that by replacing $t_{342}^{+} \in C_{0}$ with $t_{342}^{-}$ we get $C_{1}$, and they satisfy the relation
\begin{align}
	\Quadweb{\twebplus{180}}{$t_{342}^{+}$}\Quadweb{\twebminus{180}}{$t_{342}^{-}$}=A^{3/2}\Quadweb{\bwebpos{90}\bwebpos{180}\dweb{90}}{$[e_{23}e_{34}e_{42}]$}+A^{-3/2}\Quadweb{\bwebneg{90}\bwebneg{180}\dweb{-90}}{$[e_{24}e_{43}e_{32}]$}\in\langle C_{0}\cap C_{1}\rangle_{\mathrm{alg}}.
\end{align}
Similarly,
\begin{align*}
	\Quadweb{\twebplus{0}}{$t_{124}^{+}$}\Quadweb{\twebminus{0}}{$t_{124}^{-}$}
	&=A^{3/2}\Quadweb{\bwebpos{-90}\bwebpos{0}\dweb{-90}}{$[e_{12}e_{24}e_{41}]$}+A^{-3/2}\Quadweb{\bwebneg{-90}\bwebneg{0}\dweb{90}}{$[e_{14}e_{42}e_{21}]$}\in\langle C_{0}\cap C_{2}\rangle_{\mathrm{alg}},\\
	\Quadweb{\dweb{90}}{$e_{42}$}\Quadweb{\twebplus{90}}{$t_{231}^{+}$}
	&=A^{3/2}\Quadweb{\bwebpos{0}\twebplus{180}}{$[e_{12}t_{342}^{+}]$}+A^{-3/2}\Quadweb{\bwebneg{90}\twebplus{0}}{$[e_{32}t_{124}^{+}]$}\in\langle C_{0}\cap C_{3}\rangle_{\mathrm{alg}},\\
	\Quadweb{\dweb{-90}}{$e_{24}$}\Quadweb{\twebplus{-90}}{$t_{413}^{+}$}
	&=A^{3/2}\Quadweb{\bwebpos{180}\twebplus{0}}{$[e_{34}t_{124}^{+}]$}+A^{-3/2}\Quadweb{\bwebneg{-90}\twebplus{180}}{$[e_{14}t_{342}^{+}]$}\in\langle C_{0}\cap C_{4}\rangle_{\mathrm{alg}}.
\end{align*}

The following lemma gives the ``cluster expansion'' of any web in $\SK{\Sigma}$ in the web clusters $C_\nu$ for $\nu=0,1,2,3,4$.

\begin{prop}\label{prop:adjacent-web-cluster-expansion-Q}
	For any web $x\in\SK{Q}$ and $\nu=0,1,2,3,4$, there exists $J_\nu\in\mathrm{mon}(C_\nu)$ such that $xJ_{\nu}\in\langle C_\nu\rangle_{\mathrm{alg}}$.
\end{prop}

\begin{proof}
	For the web clusters $C_0$, $C_1$, and $C_2$, the assertion is already proved in \cref{cor:web-cluster-expansion-T}.
	An expansion in $C_3$ gives an expansion on $C_4$ via an automorphism induced by a rotation of $Q$.
	Therefore we only need to obtain an expansion of elementary webs in $\SK{Q}$ in the web cluster $C_3$.
	For each elementary web in 
	\[
		\Eweb{Q}\setminus C_{3}=\{{e_{31}},{e_{42}},{e_{13}},{t_{413}^{+}},t_{123}^{-},t_{231}^{-},t_{342}^{-},t_{413}^{-},h_1,h_2,h_3,h_4\},
	\]
	we can expand it as a polynomial in $C_3$ by the right multiplication of webs in $C'_3$.
	Indeed, we have
	\begin{align}
		&\Quadweb{\dweb{90}}{$e_{42}$}\Quadweb{\twebplus{90}}{$t_{231}^{+}$}\in\Bigg\langle\Quadweb{\twebplus{0}}{$t_{124}^{+}$},\Quadweb{\twebplus{180}}{$t_{342}^{+}$}\Bigg\rangle_{\mathrm{alg}}\subset\langle C_3\rangle_{\mathrm{alg}}, \label{eq:C3expansion1} \\
		&\Quadweb{\twebplus{-90}}{$t_{413}^{+}$}\Quadweb{\dweb{-90}}{$e_{24}$}\in\Bigg\langle\Quadweb{\twebplus{0}}{$t_{124}^{+}$},\Quadweb{\twebplus{180}}{$t_{342}^{+}$}\Bigg\rangle_{\mathrm{alg}}\subset\langle C_3\rangle_{\mathrm{alg}}, \label{eq:C3expansion2} \\
		&\Quadweb{\hweb{0}}{$h_{1}$}\Quadweb{\twebplus{0}}{$t_{124}^{+}$}\in\Bigg\langle\Quadweb{\twebplus{0}}{$t_{124}^{+}$},\Quadweb{\twebplus{180}}{$t_{342}^{+}$}\Bigg\rangle_{\mathrm{alg}}\subset\langle C_3\rangle_{\mathrm{alg}},\notag\\
		&\Quadweb{\hweb{-90}}{$h_{4}$}\Quadweb{\twebplus{180}}{$t_{342}^{+}$}\in\Bigg\langle\Quadweb{\twebplus{0}}{$t_{124}^{+}$},\Quadweb{\twebplus{180}}{$t_{342}^{+}$}\Bigg\rangle_{\mathrm{alg}}\subset\langle C_3\rangle_{\mathrm{alg}}.\notag
	\end{align}
	Similarly, one can confirm the followings by a straightforward computation:
	\begin{align*}
		&\Quadweb{\dweb{180}}{$e_{31}$}\Bigg(\Quadweb{\twebplus{0}}{$t_{124}^{+}$}\Quadweb{\dweb{90}}{$e_{24}$}\Bigg)\in\langle C_3\rangle_{\mathrm{alg}},
		&&\Quadweb{\dweb{0}}{$e_{13}$}\Bigg(\Quadweb{\twebplus{180}}{$t_{342}^{+}$}\Quadweb{\dweb{90}}{$e_{24}$}\Bigg)\in\langle C_3\rangle_{\mathrm{alg}},\\
		&\Quadweb{\twebminus{0}}{$t_{124}^{-}$}\Bigg(\Quadweb{\twebplus{0}}{$t_{124}^{+}$}\Quadweb{\twebplus{90}}{$t_{231}^{+}$}\Bigg)\in\langle C_3\rangle_{\mathrm{alg}},
		&&\Quadweb{\twebminus{180}}{$t_{342}^{-}$}\Bigg(\Quadweb{\twebplus{180}}{$t_{342}^{+}$}\Quadweb{\twebplus{90}}{$t_{231}^{+}$}\Bigg)\in\langle C_3\rangle_{\mathrm{alg}},\\
		&\Quadweb{\twebminus{-90}}{$t_{413}^{-}$}\Bigg(\Quadweb{\twebplus{0}}{$t_{124}^{+}$}\Quadweb{\twebplus{180}}{$t_{342}^{+}$}\Bigg)\in\langle C_3\rangle_{\mathrm{alg}},&&
	\end{align*}
	\begin{align*}
		&\Quadweb{\hweb{90}}{$h_{2}$}\Bigg(\Quadweb{\dweb{90}}{$e_{24}$}\Quadweb{\twebplus{0}}{$t_{124}^{+}$}\Quadweb{\twebplus{0}}{$t_{124}^{+}$}\Bigg)\in\langle C_3\rangle_{\mathrm{alg}},\\
		&\Quadweb{\hweb{180}}{$h_{3}$}\Bigg(\Quadweb{\dweb{90}}{$e_{24}$}\Quadweb{\twebplus{180}}{$t_{342}^{+}$}\Quadweb{\twebplus{0}}{$t_{124}^{+}$}\Bigg)\in\langle C_3\rangle_{\mathrm{alg}},\\
		&\Quadweb{\twebminus{90}}{$t_{231}^{-}$}\Bigg(\Quadweb{\twebplus{90}}{$t_{231}^{+}$}\Quadweb{\twebplus{0}}{$t_{124}^{+}$}\Quadweb{\twebplus{180}}{$t_{342}^{+}$}\Quadweb{\dweb{90}}{$e_{24}$}\Bigg)\in\langle C_3\rangle_{\mathrm{alg}}.
	\end{align*}
	For example, the most complicated one will be an expansion of ${t_{231}^{-}}({t_{231}^{+}}{t_{124}^{+}}{t_{342}^{+}}{e_{24}})$.
	Firstly,
	${t_{413}^{+}}{e_{24}}\in\langle C_3\rangle_{\mathrm{alg}}$ by \cref{eq:C3expansion1}.
	Since $e_{31}{t_{124}^{+}}$ is a sum of $e_{e_{41}}{t_{231}^{+}}$ and $e_{21}{t_{413}^{+}}$, we get
	$e_{31}({t_{124}^{+}}{e_{24}})\in\langle C_3\rangle_{\mathrm{alg}}$ by \cref{eq:C3expansion2}.
	In the same way, $e_{13}({t_{342}^{+}}{e_{24}})\in\langle C_3\rangle_{\mathrm{alg}}$.
	By \cref{lem:str const T}, ${t_{231}^{-}}{t_{231}^{+}}$ is expanded as a sum of $e_{12}e_{23}{e_{31}}$ and $e_{21}e_{32}{e_{13}}$.
	We remark that $e_{12}e_{23}{e_{31}}$ and ${t_{342}^{+}}$ are $A$-commutative, so are $e_{21}e_{32}{e_{13}}$ and ${t_{124}^{+}}$.
	Thus we get ${t_{231}^{-}}({t_{231}^{+}}{t_{124}^{+}}{t_{342}^{+}}{e_{24}})\in\langle C_3\rangle_{\mathrm{alg}}$.
\end{proof}
\section{Relation to the cluster varieties}\label{sec:FG}
Here we recall some relations between the theory of cluster algebras \cite{FZ-CA1} and that of cluster varieties \cite{FG09}. Although we  mainly deal with (the quantum aspects of) the former in the body of the paper, we borrow some notations from the latter to indicate connections to relevant geometric notions. For a comparison of their quantizations given by \cite{BZ} and \cite{FG08}, see \cite[Section 18]{GS19}.

Let $\sfs$ be a mutation class of seeds in a field $\cF$, and $\bExch_\sfs$ the associated exchange graph (see \cref{subsub:CA}). 
For $v \in \bExch_\sfs$, consider a lattice $\Lambda^{(v)}=\bigoplus_{i \in I}\bZ \sfe_i^{(v)}$ with a fixed basis and its dual $\accentset{\circ}{\Lambda}^{(v)}=\bigoplus_{i \in I}\bZ \sff_i^{(v)}$. Let 
\begin{align*}
    \X_{(v)}:=\Hom(\Lambda^{(v)},\bG_m),\quad \A_{(v)}:=\Hom(\accentset{\circ}{\Lambda}^{(v)},\bG_m)
\end{align*}
denote the associated algebraic tori of dimension $N$, where $\bG_m:=\Spec \bZ[t,t^{-1}]$ denotes the multiplicative algebraic group\footnote{A reader not familiar with such a notion may substitute any field $k$ to get $\bG_m(k)=k^\ast$, and $\A_{(v)}(k)\cong (k^*)^I$. This amounts to consider schemes over $k$, making their function rings $k$-algebras in the sequel.}. 
The basis vectors $\sfe_i^{(v)}$ and $\sff_i^{(v)}$ give rise to characters $X_i^{(v)}:\X_{(v)} \to \bG_m$ and $A_i^{(v)}:\A_{(v)} \to \bG_m$ respectively, called the \emph{cluster coordinates}. 
The quiver exchange matrix $\varepsilon^{(v)}$ defines a $\frac{1}{2}\bZ$-valued bilinear form on $\Lambda^{(v)}$ by $(\sfe_i^{(v)},\sfe_j^{(v)}):=\varepsilon^{(v)}_{ij}$. 

Now let us focus on the $\A$-side, which is directly related to the (upper) cluster algebras. The exchange relation \eqref{eq:A-transf} can be regarded as a birational map $\mu_k^a: \A_{(v)} \to \A_{(v')}$, called the \emph{cluster $\A$-transformation} \cite{FG09}. Namely, $\mu_k^a$ is defined by
\begin{align*}
    (\mu_k^a)^*A^{(v')}_i&:=\begin{cases}
    \displaystyle{ (A^{(v)}_k)^{-1} \left(\prod_{j \in I}(A^{(v)}_j)^{[ \varepsilon^{(v)}_{kj}]_+} + \prod_{j \in I}(A^{(v)}_j)^{[-\varepsilon^{(v)}_{kj}]_+}\right)} & \mbox{if $i=k$}, \\
    A^{(v)}_i & \mbox{if $i \neq k$}. 
    \end{cases}
\end{align*}
in terms of the cluster coordinates. 
Then the \emph{cluster$\A$-variety} (or the \emph{cluster $K_2$-variety} is the scheme defined as
\begin{align*}
     \A_\sfs:= \bigcup_{v \in \bExch_\sfs} \A_{(v)}.
\end{align*}
Here the (open subsets of) tori $\A_{(v)},\A_{(v')}$ are identified via the cluster transformation $\mu_k^a$ if there is an edge of the form $v \overbar{k} v'$, or via the coordinate permutation \eqref{eq:classical_permutation} if there is an edge of the form $v \overbar{\sigma} v'$. Similarly, the \emph{cluster $\X$-variety} (or the \emph{cluster Poisson variety}) $\X_\sfs$ is defined by gluing the tori $\X_{(v)}$ by the cluster $\X$-transformations \cite[(13)]{FG09}. 

\bigskip
\paragraph{\textbf{The upper cluster algebra.}}
The ring $\cO(\A_\sfs)$ of regular functions on $\A_\sfs$ is naturally identified with the upper cluster algebra $\UCA_\sfs$, as follows. 

First note that the collection $\mathbf{A}^{(v)}:=(A^{(v)}_i)_{i \in I}$ of cluster $\A$-coordinates associated to $v \in \bExch_\sfs$ can be also regarded as a collection of rational functions on the cluster $\A$-variety. 
In particular, the pair $(B^{(v)},\mathbf{A}^{(v)})$ defines a seed in the field $\cF=\mathcal{K}(\A_\sfs)$ of rational functions on $\A_\sfs$. 
Since they are related by cluster $\A$-transformations, these seeds are mutation-equivalent to each other. 
Then the upper cluster algebra $\UCA_\sfs \subset \mathcal{K}(\A_\sfs)$ consists of functions on $\A_\sfs$ whose restriction to each torus $\A_{(v)}$ are regular functions (\emph{i.e.}, Laurent polynomials). This is exactly the the ring $\cO(\A_\sfs)$ of regular functions on $\A_\sfs$. 
We remark that the Laurent phenomenon theorem \cite[Theorem 3.1]{FZ-CA1} means in this geometric setting that each cluster coordinate is in fact extended to a regular function on $\A_\sfs$, and hence $\CA_\sfs \subset \UCA_\sfs=\cO(\A_\sfs)$. 

Let 
\begin{align}\label{eq:univ_Laurent}
    \mathbb{L}_+(\A_\sfs):=\bigcap_{v \in \bExch_\sfs} \bZ_+[(A_i^{(v)})^{\pm 1} \mid i \in I] \subset \cO(\A_\sfs)
\end{align}
denote the semiring of \emph{universally positive Laurent polynomials} on $\A_\sfs$, where $\bZ_+[(A_i^{(v)})^{\pm 1} \mid i \in I]$ is the semiring of Laurent polynomials in $A_i^{(v)}$'s with non-negative integral coefficients. Then the Fock--Goncharov duality conjecture \cite[Section 4]{FG09} asserts that $\mathbb{L}_+(\A_\sfs)$ is isomorphic to the abelian semigroup generated by the set $\X_{\sfs^\vee}(\bZ^t)$, satisfying certain axioms on the coordinate expressions and the structure constants. 

\bigskip
\paragraph{\textbf{Ensemble grading.}}
The upper cluster algebra $\cO(\A_\sfs)$ has a natural grading induced by an action of an algebraic torus $H_\A$ on $\A_\sfs$. 

For $v \in \bExch_\sfs$, the bilinear form $(\ ,\ )^{(v)}$ on $\Lambda^{(v)}$ induces the \emph{ensemble map}
\begin{align*}
    p_{(v)}^*: \Lambda_\uf^{(v)} \to \accentset{\circ}{\Lambda}^{(v)}, \quad \sfe_i^{(v)} \mapsto (\sfe_i^{(v)},\ )^{(v)}=\sum_{j \in I} \varepsilon_{ij}^{(v)}\sff_j^{(v)},
\end{align*}
where $\Lambda_\uf^{(v)}:=\bigoplus_{i \in I_\uf} \bZ \sfe_i^{(v)}$. It induces a monomial morphism $p_{(v)}:\A_{(v)} \to \X^\uf_{(v)}:=\Hom(\Lambda_\uf^{(v)},\bG_m)$, which is expressed as $p_{(v)}^*X^{(v)}_i = \prod_{j \in I} (A^{(v)}_j)^{\varepsilon^{(v)}_{ij}}$. It is known that these maps $\{p_{(v)}\}_v$ commute with the cluster transformations, and thus combine to define a morphism $p: \A_\sfs \to \X^\uf_\sfs$ between the cluster varieties. 

Let us recall the \emph{decomposition of mutations} \cite[Section 2.1.2]{FG09} and its signed version (see, for instance, \cite[Section 2]{GS16}). For an edge $v \overbar{k} v'$, the lattices $\Lambda^{(v)}$ and $\Lambda^{(v')}$ are related by two linear isomorphisms
\begin{align*}
    \mu_{k,\epsilon}^*: \Lambda^{(v')} \xrightarrow{\sim} \Lambda^{(v)}, \quad \sfe_i^{(v')} \mapsto \sum_{j \in I} \sfe_j^{(v)}(F_{k,\epsilon}^{(v)})_{ji}
\end{align*}
for $\epsilon \in \{+,-\}$, which we call the \emph{signed seed mutations}. The dual lattices $\accentset{\circ}{\Lambda}^{(v)}$ are related by their contragradients $\check{\mu}_{k,\epsilon}^*:=((\mu_{k,\epsilon}^*)^\mathsf{T})^{-1}: \accentset{\circ}{\Lambda}^{(v')} \xrightarrow{\sim} \accentset{\circ}{\Lambda}^{(v)}$. These linear isomorphisms induce monomial isomorphisms between the corresponding tori. It is known that the cluster transformation $\mu_k^a$ is decomposed as $\mu_k^a=\check{\mu}_{k,\epsilon}\circ \mu_{k,\epsilon}^{\#}$, where $\mu_{k,\epsilon}^{\#}$ is a certain birational automorphism on $\A_{(v)}$, and we have a similar decomposition of the cluster $\X$-transformations. The ensemble maps commute with these monomial parts of cluster transformations:
\begin{equation}\label{eq:ensemble_commutativity}
    \begin{tikzcd}
    \Lambda_\uf^{(v')} \ar[r,"p_{(v')}^*"] \ar[d,"\mu_{k,\epsilon}^*"'] & \accentset{\circ}{\Lambda}^{(v')} \ar[d,"\check{\mu}_{k,\epsilon}^*"] \\
    \Lambda_\uf^{(v)} \ar[r,"p^*_{(v)}"'] & \accentset{\circ}{\Lambda}^{(v)}. 
    \end{tikzcd}
\end{equation}
Now let us consider the exact sequence
\begin{align*}
    0 \to \ker p_{(v)}^* \to \Lambda_\uf^{(v)} \to \accentset{\circ}{\Lambda}^{(v)} \to \coker p_{(v)}^* \to 0
\end{align*}
induced by the ensemble map. One can check that the signed mutation $\check{\mu}_{k,\epsilon}^*$ induces an isomorphism $\coker p_{(v')}^* \to \coker p_{(v)}^*$, which does not depend on the sign $\epsilon$. Via these linear isomorphisms, we identify the lattices $\coker p_{(v)}^*$ for $v \in \bExch_\sfs$ and simply denote it by $\coker p^*$. Let $H_\A:=\Hom(\coker p^*,\bG_m)$ denote the corresponding algebraic torus. 
Then the projection $\alpha_{(v)}^*:\accentset{\circ}{\Lambda}^{(v)} \to \coker p_{(v)}^*$ induces a monomial morphism $H_\A \to \A_{(v)}$, which is also regarded as a monomial action of $H_\A$ on $\A_{(v)}$. From the commutative diagram \eqref{eq:ensemble_commutativity}, these actions 
combine to give an action \cite[Lemma 2.10(a)]{FG09} 
\begin{align}\label{eq:H_A-action}
   \alpha: H_\A \times \A_\sfs \to \A_\sfs.
\end{align}

In particular the torus $H_\A$ acts on the upper cluster algebra $\cO(\A_\sfs)$ from the right, and thus defines a grading valued in the lattice $\coker p^*$. Namely, a function $f \in \cO(\A_\sfs)$ is homogeneous if there exists a character $\chi_f:H_\A \to \bG_m$ such that $f.h=\chi_f(h)\cdot f$ for all $h \in H_\A$; the character $\chi_f \in \Hom(H_\A,\bG_m) = \coker p^*$ is the grading of $f$. In particular, the grading of the cluster coordinate $A^{(v)}_i$ is given by $\alpha_{(v)}^*(\sff^{(v)}_i) \in \coker p_{(v)}^*$.

\bigskip 
\paragraph{\textbf{The cluster modular group.}}
Let $\varepsilon^\bullet: V(\bExch_\sfs) \to \mathrm{Mat}$, $v \mapsto \varepsilon^{(v)}$ be the projection that extracts the quiver exchange matrices. 
The \emph{cluster modular group} is the subgroup $\Gamma_\sfs \subset \mathrm{Aut}(\bExch_\sfs)$ consisting of graph automorphisms which preserves the fibers of $\varepsilon^\bullet$ and the labels on the edges. Then it acts on the cluster varieties $\A_\sfs$, $\X_\sfs$ by permuting the coordinate systems associated to the vertices of $\bExch_\sfs$. 
In particular it acts on the upper cluster algebra $\cO(\A_\sfs)$ from the right. Since the action sends a cluster to another cluster, it preserves the cluster algebra $\CA_\sfs \subset \cO(\A_\sfs)$. 
It also acts on the torus $H_\A$ so that the grading is $\Gamma_\sfs$-equivariant, in the sense that  $\chi_{f.\phi}=\phi^*(\chi_f)$ for $H_\A$-homogeneous functions $f$ and $\phi \in \Gamma_\sfs$.


\end{document}